\documentclass[a4paper,11pt]{amsart}
\usepackage{quiver}
\usepackage{pdflscape}

\usepackage[table,xcdraw]{xcolor}

\usepackage{amsthm,amsmath,amssymb,stmaryrd,enumerate,euscript,mathrsfs,a4,array, microtype, stmaryrd, url}
\usepackage{multirow}
\usepackage{arydshln}
\usepackage{pictexwd,dcpic,graphicx,cite}
\usepackage[export]{adjustbox}
\usepackage{subcaption}
\usepackage[latin1]{inputenc}
\usepackage{indentfirst}
\usepackage{calligra}
\usepackage{graphicx}
\usepackage{amssymb}
\usepackage{fancyhdr}
\usepackage{amsmath}
\usepackage{latexsym}
\usepackage{amsthm}
\usepackage[all]{xy} 
\usepackage{amscd}
\usepackage{hyperref}
\usepackage{tikz-cd}
\usepackage{tikz}
\usepackage{todonotes}
\usepackage{comment}

\usetikzlibrary{arrows.meta}

\usepackage[a4paper,top=3cm,bottom=3cm,inner=2.5cm,outer=2.5cm]{geometry}

\usepackage{tikz}
\usetikzlibrary{positioning,decorations.markings,arrows.meta,calc,fit,quotes,cd,math,arrows,backgrounds,shapes.geometric}
\usepackage{amsmath,mathtools}
\usepackage{xparse}

\makeatletter
\newenvironment{smallermatrix}[1][c]
{\null\,\vcenter\bgroup
  \Let@\restore@math@cr\default@tag
  \baselineskip0pt \lineskip0.4pt \lineskiplimit0pt
  \ialign\bgroup\if#1l\else\hfil\fi$\m@th\scriptstyle##$\if#1r\else\hfil\fi&&\thickspace\hfil
  $\m@th\scriptstyle##$\hfil\crcr
}{%
  \crcr\egroup\egroup\,%
}
\makeatother

\NewDocumentCommand{\ts}{O{c} e{^?_}}{
  \begin{smallermatrix}[#1]
  \mathstrut\IfValueT{#2}{#2} \\
  \mathstrut\IfValueT{#3}{#3} \\
  \mathstrut\IfValueT{#4}{#4}
  \end{smallermatrix}%
}

\makeatletter
\def\hlinewd#1{%
\noalign{\ifnum0=`}\fi\hrule \@height #1 \futurelet
\reserved@a\@xhline}
\makeatother

\graphicspath{ {images/} }
\usepackage{babel}
\usepackage{amsthm}
\usepackage{thmtools}
\usepackage{cleveref}
\crefrangeformat{equation}{(#3#1#4-#5#2#6)}

\newtheorem{thm}{Theorem}[section]
\newtheorem*{theorem*}{Theorem}

\newtheorem{cor}[thm]{Corollary}
\newtheorem{prop}[thm]{Proposition}

\newtheorem{lemma}[thm]{Lemma}

\newtheorem{defn}[thm]{Definition}

\newtheorem{rmk}[thm]{Remark}

\newtheorem*{remark*}{Remark}
\newtheorem*{remarks*}{Remarks}
\newtheorem*{notation*}{Notation}
\newtheorem*{convention*}{Convention}

\newcommand{\co}{\mathit{co}}
\newcommand{\rev}{\mathit{rev}}
\newcommand{\nb}{\mathit{nb}}

\newcommand{\cat}{\mathcal} 

\newcommand{\catA}{\cat{A}}
\newcommand{\catB}{\cat{B}}
\newcommand{\catC}{\cat{C}}
\newcommand{\catD}{\cat{D}}

\newcommand{\ca}{\cat{A}}
\newcommand{\cb}{\cat{B}}
\newcommand{\cc}{\cat{C}}
\newcommand{\cd}{\cat{D}}
\newcommand{\ce}{\cat{E}}

\newcommand{\cx}{\cat{X}}

\newcommand{\mul}{\mathbb}

\newcommand{\mc}{\mul{C}}

\newcommand{\tcat}{\mathbf}

\newcommand{\Gray}{\tcat{2}\text{-}\tcat{Cat}}
\newcommand{\GCat}{\tcat{G}\text{-}\tcat{Cat}}

\newcommand{\GCatm}{\mathbb{L}\tcat{ax}}

\newcommand{\GCatmp}{\mathbb{P}\tcat{sd}}

\newcommand{\aeq}{\textrm{ae}} 
\newcommand{\aexi}{\Theta} 

\newcommand{\Set}{\mathbf{Set}}

\newcommand{\lax}{\mathbf{Lax}}
\newcommand{\psd}{\mathbf{Psd}}
\newcommand{\psden}{\mathbf{Psd}_\nb}

\newcommand{\oplax}{\mathbf{Oplax}}

\def\red{\color{red}}
\def\blue{\color{blue}}
\def\black{\color{black}}

\def\grey{\color{lightgray}}

\makeatletter
\protected\def\xvcenter{%
  \hbox\bgroup$\everyvbox{\everyvbox{}\aftergroup\m@th\aftergroup$\aftergroup\egroup}%
  \vcenter
}

\DeclareRobustCommand{\midscript}[1]{
  \mathchoice{\mid@script\scriptstyle{#1}}
    {\mid@script\scriptstyle{#1}}
    {\mid@script\scriptscriptstyle{#1}}
    {\mid@script\scriptscriptstyle{#1}}
}
\newcommand{\mid@script}[2]{
  \vcenter{\hbox{$\m@th#1#2$}}
}

\makeatletter

\makeindex
\title{The symmetric skew monoidal structure of Gray-categories}
\author{John Bourke}
\address{Department of Mathematics and Statistics, Masaryk University, Kotl\'a\v rsk\'a 2, Brno 61137, Czech Republic}
\email{bourkej@math.muni.cz}

\author{Gabriele Lobbia}
\address{Departmental of Mathematics, University of Milan, Via Saldini, 50, 20133 Milano, Italy}
\email{gabriele.lobbia@unimi.it}

\begin{document}

\begin{abstract}
In previous work, we constructed a closed skew monoidal structure on the category of Gray-categories and Gray-functors whose internal hom $\psd(\ca,\cb)$ captures pseudomaps and pseudo-transformations.  In the present paper we give a symmetric description of the Gray-category $\psd(\ca,\cb)$ and use this to prove that the skew monoidal structure is symmetric skew monoidal.
\end{abstract}

\maketitle

\tableofcontents

\section{Introduction}

A stumbling block in the algebraic approach to higher categories is that the category $\GCat$ of Gray-categories (or semi-strict $3$-categories) does not admit a symmetric monoidal closed structure capturing higher-dimensional pseudo-transformations.  In our previous paper \cite{BourkeLobbia}, we overcame this problem by equipping $\GCat$ with two closed \emph{skew monoidal structure} capturing higher-dimensional pseudo-transformations.

In the present paper, we focus on the weaker or most flexible of these, which we call the flexible skew structure $(\GCat, \otimes_p,1)$.   Its internal hom $\psd(\ca,\cb)$ contains pseudomaps of Gray-categories, pseudo-transformations, modifications and perturbations.    

The main goal of the present paper is to prove that the flexible skew monoidal structure $(\GCat, \otimes_p,1)$ is \emph{symmetric skew monoidal} in the sense of \cite{BourkeLack:symmetric}.  Note that symmetry in the skew setting is a little different to classical symmetry, not involving a natural isomorphism $\ca \otimes \cb \to \cb \otimes \ca$ but, rather, a natural isomorphism $(\ca \otimes \cb) \otimes \cc \to (\ca \otimes \cc) \otimes \cb$ where the first variable is kept fixed.  Nonetheless, these coincide with classical braidings when the underlying skew monoidal category is monoidal.

In order to do this, we give a new symmetric description of the Gray-categories $\psd(\ca,\cb)$.  Whereas in \cite{BourkeLobbia}, pseudo-transformations are defined as certain lax transformations, we give a new \emph{unbiased} description of them, as interacting lax and oplax transformations.  In addition to being of independent interest, this lends itself naturally to proving that $(\GCat, \otimes_p,1)$ is symmetric.

Why is this important?  Our primary applications of the symmetry are in \cite{BourkeLobbia2}, where we introduce several flavours of semi-strict $4$-category, of which the strictest are called \emph{sharp} and the weakest are called \emph{loose}.  The loose $4$-categories are precisely categories enriched in the multicategory $\GCatmp^l$ underlying $(\GCat, \otimes_p,1)$.  Corresponding to the fact that $(\GCat, \otimes_p,1)$ is closed symmetric skew monoidal, $\GCatmp^l$ is a symmetric closed multicategory.  This means that categories enriched in it --- that is, loose $4$-categories --- permit many useful constructions, such as opposites and enriched left and right representable functors.  In addition, we make significant use of the symmetric structure in the homotopical aspects of \cite{BourkeLobbia2}, where we prove that the homotopy category of Gray-categories is symmetric monoidal closed.

Let us now give a more detailed overview of our results.  In Section \ref{sect:prelim-psd} we start by recalling the necessary preliminaries about Gray-categories and the flexible skew structure, as well as the Gray-categories of $\psd(\ca,\cb)$ of pseudomaps and pseudotransformations.

 Section~\ref{sect:unbiased} is devoted to developing the unbiased description of pseudo-transformations (see \Cref{prop:equ-psd-transf}) and their higher transfors of $\psd(\ca,\cb)$.  In particular, we describe the Gray-category $\psden(\ca,\cb)$ of unbiased pseudo-transformations, which in \Cref{unbiased} is shown to be isomorphic to $\psd(\ca,\cb)$.
 
 Section~\ref{sect:unbiased} describes the necessary background about skew monoidal categories and their associated skew multicategories, focusing on the example of $(\GCat,\otimes_p,1)$ and its associated skew multicategory $\GCatmp$.   This is essential in Section~\ref{sect:braiding} where we construct the braiding for the skew monoidal structure.  Note that rather than working with the skew monoidal category directly, we will instead establish the braiding on the skew multicategory $\GCatmp$ that it represents and then apply results of \cite{Lob:short-multi} to obtain the skew monoidal braiding.  The necessary axioms are verified in Section~\ref{sect:axioms}.
 
 Finally, in the appendices we describe a number of technical results used within the paper.
%
%
%
%
%

\section{Background on pseudo-transformations}\label{sect:prelim-psd}

In this section we recall the background needed to give the symmetric presentation of pseudo-transformations of Gray-categories, here called \emph{unbiased pseudo-transformations}. 
We start by considering adjoint equivalences in a 2-category and describe some useful properties of constructions involving \black mates (\Cref{rmk:mates-over-adj-equiv-calculations}). 
\black 

\subsection{Adjoint equivalences in a 2-category}

Let $\cc$ be a 2-category. 
Recall that an adjoint equivalence $(\epsilon, f\dashv g,\eta)\colon X\to Y$ in $\cc$ consists of  1-cells $g\colon X\to Y$ and $f\colon Y\to X$, together with invertible 2-cells $\epsilon\colon fg\to 1_X$ and $\eta\colon 1_Y\to gf$ satisfying the so called triangle equalities for an adjunction. 
We now set some notation which will be useful in the paper, while recalling some basic constructions involving adjoint equivalences.

\begin{notation*}
	\begin{itemize}
		\item[]
		\item For any adjoint equivalence $\aexi=(\epsilon,f\dashv g,\eta)\colon X\to Y$, also $(\eta^{-1},g\dashv f,\epsilon^{-1})\colon Y\to X$ forms an adjoint equivalence, which we refer to as the \emph{reverse} adjoint equivalence and denote it with $\aexi^\rev$. 
		
		\item For  adjoint equivalences $\aexi=(\epsilon,f\dashv g,\eta)\colon X\to Y$ and $\aexi'=(\epsilon',f'\dashv g',\eta')\colon Y\to Z$ as below
\[\begin{tikzcd}[ampersand replacement=\&]
	X \& Y \& Z
	\arrow[""{name=0, anchor=center, inner sep=0}, "g"', curve={height=12pt}, from=1-1, to=1-2]
	\arrow[""{name=1, anchor=center, inner sep=0}, "f"', curve={height=12pt}, from=1-2, to=1-1]
	\arrow[""{name=2, anchor=center, inner sep=0}, "{g'}"', curve={height=12pt}, from=1-2, to=1-3]
	\arrow[""{name=3, anchor=center, inner sep=0}, "{f'}"', curve={height=12pt}, from=1-3, to=1-2]
	\arrow["\perp"{description}, draw=none, from=1, to=0]
	\arrow["\perp"{description}, draw=none, from=3, to=2]
\end{tikzcd}\]
		there is a \emph{composite} adjoint equivalence defined by $(\epsilon'\cdot f'\epsilon g',f'f\dashv gg',g\eta' f\cdot\eta)\colon X\to Z$. 
		We will denote it by 
		$\aexi' \cdot \aexi $.
		
		\item We write $1_\aeq$ for the identity adjoint equivalence $(1_{1_X},1_X\dashv 1_X,1_{1_X})$. 
		
		
		\item Any 2-functor preserves adjoint equivalences:
		that is, given any adjoint equivalence $\aexi=(\epsilon,f\dashv g,\eta)$ in $\cc$ and 2-functor $F\colon \cc\to \cd$, then $F\aexi:= (F\epsilon,Ff\dashv Fg,F\eta)$ is an adjoint equivalence in $\cd$. 


		\black 
	\end{itemize} 
\end{notation*}

Consider adjoint pairs $f\dashv g\colon X\to Y$ and $f\dashv g\colon X'\to Y'$ in a 2-category $\cc$. Then for fixed 1-cells $a\colon X\to X'$ and $b\colon Y\to Y'$, there is a bijection between 2-cells as below 
\[\begin{tikzcd}[ampersand replacement=\&]
	Y \& {Y'} \&\& X \& {X'} \\
	X \& {X'} \&\& Y \& {Y'}
	\arrow[""{name=0, anchor=center, inner sep=0}, "b", from=1-1, to=1-2]
	\arrow["f"', from=1-1, to=2-1]
	\arrow["{f'}", from=1-2, to=2-2]
	\arrow["a", from=1-4, to=1-5]
	\arrow[""{name=1, anchor=center, inner sep=0}, "g"', from=1-4, to=2-4]
	\arrow[""{name=2, anchor=center, inner sep=0}, "{g'}", from=1-5, to=2-5]
	\arrow[""{name=3, anchor=center, inner sep=0}, "a"', from=2-1, to=2-2]
	\arrow["b"', from=2-4, to=2-5]
	\arrow["\lambda", between={0.3}{0.7}, Rightarrow, nfold, from=0, to=3]
	\arrow["\mu", between={0.3}{0.7}, Rightarrow, nfold, from=1, to=2]
\end{tikzcd}\]
obtained by pasting with the units and counits of the adjunctions (see for instance \cite{KellyStreet:rew-el-2cats}, where the corresponding 2-cells are called \emph{mates}).  Clearly, if the adjoint pairs we consider are adjoint equivalences, then the mate of an invertible 2-cell is still invertible. 

\begin{rmk}
	\label{rmk:mates-over-adj-equiv-calculations}
	Consider
	adjoint pairs $f\dashv g\colon X\to Y$ and $f'\dashv g'\colon X'\to Y'$ in a 2-category. 
	For any pair of 1-cells $a\colon X\to X'$ and $b\colon Y\to Y'$, and an invertible 2-cell $\lambda\colon g'a\to bg$, we can define $\lambda\flat$ as the mate of $\lambda^{-1}$, i.e. as the pasting below. 
	\begin{center}
		${\lambda}\flat:=$
		\begin{tikzcd}[ampersand replacement=\&,row sep=scriptsize]
			Y \& \\
			X \& Y \\
			{X'} \& {Y'} \\
			\& {X'}
			\arrow["f"', from=1-1, to=2-1]
			\arrow[""{name=0, anchor=center, inner sep=0}, curve={height=-6pt}, equals, from=1-1, to=2-2]
			\arrow[""{name=1, anchor=center, inner sep=0}, "g"{description}, from=2-1, to=2-2]
			\arrow["a"', from=2-1, to=3-1]
			\arrow["b", from=2-2, to=3-2]
			\arrow[""{name=2, anchor=center, inner sep=0}, "{g'}"{description}, from=3-1, to=3-2]
			\arrow[""{name=3, anchor=center, inner sep=0}, curve={height=6pt}, equals, from=3-1, to=4-2]
			\arrow["{f'}", from=3-2, to=4-2]
			\arrow["\eta"', between={0.3}{0.7}, Rightarrow, from=0, to=1]
			\arrow["{\lambda^{-1}}", shift right=2, between={0.4}{0.6}, Rightarrow, from=1, to=2]
			\arrow["\epsilon'", between={0.3}{0.7}, Rightarrow, from=2, to=3]
		\end{tikzcd}
	\end{center}
	Then, using the definition of $\overline{\lambda}$ and the triangle equalities of the adjoints, we can see that the following equations hold.
	
	\begin{minipage}{0.48\textwidth}
		\begin{equation}
			\begin{tikzcd}[ampersand replacement=\&]
				X \& {X'} \\
				Y \& {Y'} \\
				X \& {X'}
				\arrow[""{name=0, anchor=center, inner sep=0}, "a", from=1-1, to=1-2]
				\arrow["g"{description}, from=1-1, to=2-1]
				\arrow[""{name=1, anchor=center, inner sep=0}, curve={height=30pt}, equals, from=1-1, to=3-1]
				\arrow["{g'}", from=1-2, to=2-2]
				\arrow[""{name=2, anchor=center, inner sep=0}, "b"{description}, from=2-1, to=2-2]
				\arrow["f"{description}, from=2-1, to=3-1]
				\arrow["{f'}", from=2-2, to=3-2]
				\arrow[""{name=3, anchor=center, inner sep=0}, "a"', from=3-1, to=3-2]
				\arrow["\lambda", between={0.4}{0.6}, Rightarrow, from=0, to=2]
				\arrow["\epsilon"'{pos=0.2}, between={0}{0.7}, Rightarrow, from=2-1, to=1]
				\arrow["{\lambda\flat}", between={0.4}{0.6}, Rightarrow, from=2, to=3]
			\end{tikzcd}
			=
			\begin{tikzcd}[ampersand replacement=\&]
				X \& {X'} \\
				\& {Y'} \\
				\& {X'}
				\arrow["a", from=1-1, to=1-2]
				\arrow["{g'}", from=1-2, to=2-2]
				\arrow[""{name=0, anchor=center, inner sep=0}, curve={height=30pt}, equals, from=1-2, to=3-2]
				\arrow["{f'}", from=2-2, to=3-2]
				\arrow["{\epsilon'}"'{pos=0.2}, between={0}{0.7}, Rightarrow, from=2-2, to=0]
			\end{tikzcd}
		\end{equation}
	\end{minipage}
	\hfill
	\begin{minipage}{0.48\textwidth}
		\begin{equation}
			\begin{tikzcd}[ampersand replacement=\&]
				Y \& {Y'} \\
				X \& {X'} \\
				Y \& {Y'}
				\arrow[""{name=0, anchor=center, inner sep=0}, "b", from=1-1, to=1-2]
				\arrow["f"', from=1-1, to=2-1]
				\arrow["{f'}"{description}, from=1-2, to=2-2]
				\arrow[""{name=1, anchor=center, inner sep=0}, curve={height=-30pt}, equals, from=1-2, to=3-2]
				\arrow[""{name=2, anchor=center, inner sep=0}, "a"{description}, from=2-1, to=2-2]
				\arrow["g"', from=2-1, to=3-1]
				\arrow["{g'}"{description}, from=2-2, to=3-2]
				\arrow[""{name=3, anchor=center, inner sep=0}, "b"', from=3-1, to=3-2]
				\arrow["{\lambda\flat}", between={0.4}{0.6}, Rightarrow, from=0, to=2]
				\arrow["{\eta'}"'{pos=0.7}, between={0.3}{1}, Rightarrow, from=1, to=2-2]
				\arrow["\lambda", between={0.4}{0.6}, Rightarrow, from=2, to=3]
			\end{tikzcd}
			=
			\begin{tikzcd}[ampersand replacement=\&]
				Y \& \\
				X \\
				Y \& {Y'}
				\arrow["f"', from=1-1, to=2-1]
				\arrow[""{name=0, anchor=center, inner sep=0}, curve={height=-30pt}, equals, from=1-1, to=3-1]
				\arrow["g"', from=2-1, to=3-1]
				\arrow["b"', from=3-1, to=3-2]
				\arrow["\eta"', between={0.3}{1}, Rightarrow, from=0, to=2-1]
			\end{tikzcd}
		\end{equation}
	\end{minipage}
	
	These equations will be useful in \Cref{prop:equ-psd-mod}, where we prove the unbiased characterisation for modifications.
	
	
	Similarly, if we start with an invertible 2-cell $\gamma\colon af\to f'b$, we can define $\gamma\sharp$ as the mate of $\gamma^{-1}$ as below.
	\begin{center}
		$\gamma\sharp:=$
		\begin{tikzcd}[ampersand replacement=\&]
			X \& Y \& {Y'} \& \\
			\& X \& {X'} \& {Y'}
			\arrow[""{name=0, anchor=center, inner sep=0}, "g", from=1-1, to=1-2]
			\arrow[""{name=1, anchor=center, inner sep=0}, curve={height=12pt}, equals, from=1-1, to=2-2]
			\arrow[""{name=2, anchor=center, inner sep=0}, "b", from=1-2, to=1-3]
			\arrow["f"{description}, from=1-2, to=2-2]
			\arrow["{f'}"{description}, from=1-3, to=2-3]
			\arrow[""{name=3, anchor=center, inner sep=0}, curve={height=-12pt}, equals, from=1-3, to=2-4]
			\arrow[""{name=4, anchor=center, inner sep=0}, "a"', from=2-2, to=2-3]
			\arrow[""{name=5, anchor=center, inner sep=0}, "{g'}"', from=2-3, to=2-4]
			\arrow["\epsilon", between={0.3}{0.7}, Rightarrow, from=0, to=1]
			\arrow["{\gamma^{-1}}", shift right, between={0.3}{0.7}, Rightarrow, from=2, to=4]
			\arrow["{\eta'}"', between={0.3}{0.7}, Rightarrow, from=3, to=5]
		\end{tikzcd}
	\end{center}
Much as in the previous case, we obtain the following equations, which will be useful in \Cref{sec:ternary-braidings}.
	\begin{equation}
		\begin{tikzcd}[ampersand replacement=\&]
			Y \& X \& Y \\
			{Y'} \& {X'} \& {Y'}
			\arrow[""{name=0, anchor=center, inner sep=0}, "f"{description}, from=1-1, to=1-2]
			\arrow[""{name=1, anchor=center, inner sep=0}, curve={height=-30pt}, equals, from=1-1, to=1-3]
			\arrow["b"', from=1-1, to=2-1]
			\arrow[""{name=2, anchor=center, inner sep=0}, "g"{description}, from=1-2, to=1-3]
			\arrow["a"{description}, from=1-2, to=2-2]
			\arrow["b", from=1-3, to=2-3]
			\arrow[""{name=3, anchor=center, inner sep=0}, "{f'}"', from=2-1, to=2-2]
			\arrow[""{name=4, anchor=center, inner sep=0}, "{g'}"', from=2-2, to=2-3]
			\arrow["\eta"{pos=0.6}, shift right=2, between={0.4}{0.9}, Rightarrow, from=1, to=1-2]
			\arrow["\gamma", shift right, between={0.3}{0.7}, Rightarrow, from=0, to=3]
			\arrow["{\gamma\sharp}", shift right, between={0.3}{0.7}, Rightarrow, from=2, to=4]
		\end{tikzcd}
		=
		\begin{tikzcd}[ampersand replacement=\&]
			Y \&\& \\
			{Y'} \& {X'} \& {Y'}
			\arrow["b"', from=1-1, to=2-1]
			\arrow["{f'}"', from=2-1, to=2-2]
			\arrow[""{name=0, anchor=center, inner sep=0}, curve={height=-30pt}, equals, from=2-1, to=2-3]
			\arrow["{g'}"', from=2-2, to=2-3]
			\arrow["{\eta'}"{pos=0.6}, shift right=2, between={0.4}{0.9}, Rightarrow, from=0, to=2-2]
		\end{tikzcd}
	\end{equation}
	\begin{equation}
		\begin{tikzcd}[ampersand replacement=\&]
			X \& Y \& X \\
			{X'} \& {Y'} \& {X'}
			\arrow[""{name=0, anchor=center, inner sep=0}, "g", from=1-1, to=1-2]
			\arrow["a"', from=1-1, to=2-1]
			\arrow[""{name=1, anchor=center, inner sep=0}, "f", from=1-2, to=1-3]
			\arrow["b"{description}, from=1-2, to=2-2]
			\arrow["b", from=1-3, to=2-3]
			\arrow[""{name=2, anchor=center, inner sep=0}, "{g'}"{description}, from=2-1, to=2-2]
			\arrow[""{name=3, anchor=center, inner sep=0}, curve={height=30pt}, equals, from=2-1, to=2-3]
			\arrow[""{name=4, anchor=center, inner sep=0}, "{f'}"{description}, from=2-2, to=2-3]
			\arrow["{\gamma\sharp}", shift right, between={0.3}{0.7}, Rightarrow, from=0, to=2]
			\arrow["\gamma", shift right, between={0.3}{0.7}, Rightarrow, from=1, to=4]
			\arrow["{\epsilon'}"{pos=0.3}, shift right, between={0.1}{0.6}, Rightarrow, from=2-2, to=3]
		\end{tikzcd}
		=
		\begin{tikzcd}[ampersand replacement=\&]
			X \& Y \& X \\
			\&\& {X'}
			\arrow["g", from=1-1, to=1-2]
			\arrow[""{name=0, anchor=center, inner sep=0}, curve={height=30pt}, equals, from=1-1, to=1-3]
			\arrow["f", from=1-2, to=1-3]
			\arrow["b", from=1-3, to=2-3]
			\arrow["\epsilon"{pos=0.3}, shift right, between={0.1}{0.6}, Rightarrow, from=1-2, to=0]
		\end{tikzcd}
	\end{equation}
	
	Let us mention that these constructions were used in \cite[Section 3.5]{BourkeLobbia} for oplax transformations and also in  \cite[Appendix E]{BourkeLobbia} (repeated here in Appendix \ref{app:inc-sub-b-in-b-1st} and \ref{app:inc-sub-b-in-b-2nd}) for the incubator in the substitution of binary maps into binary.   In the applications of $\sharp$ and $\flat$ in \cite{BourkeLobbia}, the adjoint equivalences are given by invertible 1-cells. We will use the above equations in the proof of \Cref{prop:first-3-braid-eq}.
%
%
%
\end{rmk}

\black

\subsection{Gray-categories}

\black
A Gray-category $\ca$ is a $(\Gray,\otimes_p,1)$-enriched category, where $\otimes_p$ denotes the Gray-tensor product, as described in Section 3.2 of \cite{Gurskibook} or Section 4 of \cite{Gordon1995Coherence}.   
In particular, it amounts to:
\begin{enumerate}[(i)]
\item A category $\ca_1$ of $0$ and $1$-cells with composition of $f\colon x \to y$ and $g\colon y \to z$ denoted by $gf\colon x \to z$, and with identities denoted by $1_x\colon x \to x$.
\item\label{ax:strict-hom} For $0$-cells $x,y$, a $2$-category $\ca(x,y)$ whose objects are the $1$-cells of $\ca_1$, and whose $n$-cells are referred to as the $(n+1)$-cells of $\ca$.
\item Given $1$-cells $f\colon w \to x$ and $g\colon y \to z$, there are \emph{whiskering} $2$-functors $- \circ f\colon \ca(x,y) \to \ca(w,y)$ and $g \circ -\colon \ca(x,y) \to \ca(x,z)$ whose action on objects is composition in $\ca_1$.  These whiskering operations are functorial in $\ca_1$ and compatible, in the sense that $(g \circ -) \circ f = g \circ (- \circ f)$.
\item Given a pair of $2$-cells 
\[\begin{tikzcd}
	x & y & z
	\arrow[""{name=0, anchor=center, inner sep=0}, "f", curve={height=-12pt}, from=1-1, to=1-2]
	\arrow[""{name=1, anchor=center, inner sep=0}, "{f'}"', curve={height=12pt}, from=1-1, to=1-2]
	\arrow[""{name=2, anchor=center, inner sep=0}, "g", curve={height=-12pt}, from=1-2, to=1-3]
	\arrow[""{name=3, anchor=center, inner sep=0}, "{g'}"', curve={height=12pt}, from=1-2, to=1-3]
	\arrow["\phi"{xshift=0.1cm}, between={0.2}{0.8}, Rightarrow, nfold, from=0, to=1]
	\arrow["\psi"{xshift=0.1cm}, between={0.2}{0.8}, Rightarrow, nfold, from=2, to=3]
\end{tikzcd}\]\black
	there is an invertible $3$-cell
\[\begin{tikzcd}[ampersand replacement=\&]
	gf \& {g'f} \\
	{gf'} \& {g'f'}
	\arrow[""{name=0, anchor=center, inner sep=0}, "{\psi\circ f}", from=1-1, to=1-2]
	\arrow["{g\circ\phi}"', from=1-1, to=2-1]
	\arrow["{g'\circ\phi}", from=1-2, to=2-2]
	\arrow[""{name=1, anchor=center, inner sep=0}, "{\psi\circ f'}"', from=2-1, to=2-2]
	\arrow["{\psi:\phi}", shift right=2, between={0.4}{0.6}, Rightarrow, nfold, from=0, to=1]
\end{tikzcd}\]
\black
	called the interchange map.
	The interchange maps satisfy several equations, for which we refer the reader to \cite[Definition~1.1]{GambLob:psm}.\footnote{Notice that the direction of the interchange isomorphism here is the opposite with respect to both \cite{BourkeLobbia} and \cite[Definition~1.1]{GambLob:psm}.  This choice of convention matches the viewpoint of pseudo-natural transformations as lax natural transformations (rather than oplax ones) and fits better with our treatment of pseudo-transformations of Gray-categories as certain lax transformations.  Since in \cite{BourkeLobbia} we used the dashed arrows notation throughout for interchange cells, the two papers can be read consistently --- the only exception to this is in a single axiom for a pseudomap in Appendix A.1.(vii), whose interchange map (according to the present convention) should be replaced by its inverse.}
\end{enumerate}


For (vertically) composable 2-cells $\phi_1\colon f\to f'$ and $\phi_2\colon f'\to f''$, we will denote by $\phi_2\cdot\phi_1\colon f\to f''$ their vertical composition. 
The same symbol will be used also to denote whiskering of 3-cells by 2-cells, that is: for 1-cells $f,f',f''\colon x\to y$ and 3-cells 
\[\begin{tikzcd}[ampersand replacement=\&]
	f \& {f'} \& {f''}
	\arrow[""{name=0, anchor=center, inner sep=0}, "\phi", curve={height=-12pt}, from=1-1, to=1-2]
	\arrow[""{name=1, anchor=center, inner sep=0}, "{\phi'}"', curve={height=12pt}, from=1-1, to=1-2]
	\arrow[""{name=2, anchor=center, inner sep=0}, "\varphi", curve={height=-12pt}, from=1-2, to=1-3]
	\arrow[""{name=3, anchor=center, inner sep=0}, "{\varphi'}"', curve={height=12pt}, from=1-2, to=1-3]
	\arrow["\Lambda"{xshift=0.1cm}, shift right, between={0.3}{0.7}, Rightarrow, nfold, from=0, to=1]
	\arrow["\Pi"{xshift=0.1cm}, shift right, between={0.3}{0.7}, Rightarrow, nfold, from=2, to=3]
\end{tikzcd}\]
we will write $\Pi\cdot\phi\colon \varphi\cdot\phi\to\varphi'\cdot\phi$ and $\varphi\cdot\Gamma\colon\varphi\cdot\phi\to\varphi\cdot\phi'$ for the left and right whiskering. 
We will use $\Lambda'\ast\Lambda$ for vertical composition of 3-cells in $\ca(x,y)$. \black 
We will rarely need any such notation, since composites in the hom-2-categories will almost always be given using diagrams.  Our diagram style follows that of Gurski \cite{Gurskibook}.  In diagrams, we will abuse notation for whiskering $2/3$-cells by a $1$-cell and only write $1\circ-$ and $-\circ1$, as in the diagram below. Finally, in some larger diagrams, where the context will make it clear, we will avoid labelling the interchange and its inverse \black and only indicate it with a dashed double arrow again as in the diagrams below. 
\begin{center}
\begin{tikzcd}[ampersand replacement=\&]
	gf \& {g'f} \\
	{gf'} \& {g'f'}
	\arrow[""{name=0, anchor=center, inner sep=0}, "{\psi\circ1}", from=1-1, to=1-2]
	\arrow["{1\circ\phi}"', from=1-1, to=2-1]
	\arrow["{1\circ\phi}", from=1-2, to=2-2]
	\arrow[""{name=1, anchor=center, inner sep=0}, "{\psi\circ1}"', from=2-1, to=2-2]
	\arrow[between={0.3}{0.7}, Rightarrow, nfold, dashed, from=0, to=1]
\end{tikzcd}
\hspace{0.5cm}
\begin{tikzcd}
	gf & {gf'} \\
	{g'f} & {g'f'}
	\arrow[""{name=0, anchor=center, inner sep=0}, "{1\circ\phi}", from=1-1, to=1-2]
	\arrow["{\psi\circ1}"', from=1-1, to=2-1]
	\arrow["{\psi\circ1}", from=1-2, to=2-2]
	\arrow[""{name=1, anchor=center, inner sep=0}, "{1\circ\phi}"', from=2-1, to=2-2]
	\arrow[between={0.3}{0.6}, Rightarrow, nfold, dashed, from=0, to=1]
\end{tikzcd}
\end{center}
\black

%
%

In the next lemma we establish some equations that follow from the axioms regarding the interchange maps. 
These equations will be useful in \Cref{sect:unbiased}. 
Note that we will usually use these equations for units or counits of adjoint equivalences, but \Cref{lemma:interchange-property} holds for any 3-cells in a Gray-category with domain (or codomain) the identity.
\begin{lemma}
	\label{lemma:interchange-property}
	Consider $1$-cells $f,f'\colon x\rightrightarrows y$ and $g,g'\colon y\rightrightarrows z$ in a Gray-category $\ca$. \black
	 Given 2-cells $\phi\colon f\to f'$, $\mu\colon f'\to f, \psi\colon g\to g'$ and 3-cells $\Delta\colon 1_x\to \mu\cdot\phi$ and $\Theta \colon\phi\cdot\mu\to 1_y$ the following equations hold.
	\begin{equation}
		\label{eq:lem-interchange-1}
\begin{tikzcd}[ampersand replacement=\&,column sep=small]
	gf \&\& {g'f} \& {} \& \\
	{gf'} \&\& {g'f'} \&\& {g'f}
	\arrow[""{name=0, anchor=center, inner sep=0}, "{\psi\circ1}", from=1-1, to=1-3]
	\arrow["{1\circ \phi}"', from=1-1, to=2-1]
	\arrow["{1\circ \phi}"{description}, from=1-3, to=2-3]
	\arrow[curve={height=-12pt}, equals, nfold, from=1-3, to=2-5]
	\arrow[""{name=1, anchor=center, inner sep=0}, "{\psi\circ1}"', from=2-1, to=2-3]
	\arrow[""{name=2, anchor=center, inner sep=0}, "{1\circ \mu}"', from=2-3, to=2-5]
	\arrow["{(\psi:\phi)}"', shift left=5, between={0.3}{0.7}, Rightarrow, nfold, dashed, from=0, to=1]
	\arrow["{1\circ\Delta}"', shift left=3, between={0.3}{0.7}, Rightarrow, nfold, from=1-4, to=2]
\end{tikzcd}
		\hspace{0.5cm}=\hspace{0.5cm}
\begin{tikzcd}[ampersand replacement=\&,column sep=small]
	gf \& {} \& \\
	{gf'} \&\& gf \\
	{g'f'} \&\& {g'f}
	\arrow["{1\circ \phi}"', from=1-1, to=2-1]
	\arrow[curve={height=-12pt}, equals, nfold, from=1-1, to=2-3]
	\arrow[""{name=0, anchor=center, inner sep=0}, "{1\circ \mu}"{description}, from=2-1, to=2-3]
	\arrow["{\psi\circ1}"', from=2-1, to=3-1]
	\arrow["{\psi\circ1}", from=2-3, to=3-3]
	\arrow[""{name=1, anchor=center, inner sep=0}, "{1\circ \mu}"', from=3-1, to=3-3]
	\arrow["{1\circ\Delta}"{pos=0.4}, shift right=4, between={0.2}{0.6}, Rightarrow, nfold, from=1-2, to=0]
	\arrow["{(\psi:\mu)^{-1}}", shift right=4, between={0.3}{0.7}, Rightarrow, nfold, dashed, from=0, to=1]
\end{tikzcd}
	\end{equation}
	\begin{equation}
\begin{tikzcd}[ampersand replacement=\&,column sep=small]
	{gf'} \&\& gf \\
	{g'f'} \&\& {g'f} \\
	\& {} \& {g'f'}
	\arrow[""{name=0, anchor=center, inner sep=0}, "{1\circ \mu}", from=1-1, to=1-3]
	\arrow["{\psi\circ1}"', from=1-1, to=2-1]
	\arrow["{\psi\circ1}", from=1-3, to=2-3]
	\arrow[""{name=1, anchor=center, inner sep=0}, "{1\circ \mu}"{description}, from=2-1, to=2-3]
	\arrow[curve={height=12pt}, equals, nfold, from=2-1, to=3-3]
	\arrow["{1\circ \phi}", from=2-3, to=3-3]
	\arrow["{(\psi:\mu)^{-1}}", shift right=4, between={0.3}{0.7}, Rightarrow, nfold, dashed, from=0, to=1]
	\arrow["{1\circ\Theta}"{pos=0.4}, shift right=3, between={0.2}{0.6}, Rightarrow, nfold, from=1, to=3-2]
\end{tikzcd}
	\hspace{0.5cm}=\hspace{0.5cm}
\begin{tikzcd}[ampersand replacement=\&,column sep=small]
	{g'f} \&\& gf \&\& {g'f} \\
	\& {} \& {gf'} \&\& {g'f'}
	\arrow[""{name=0, anchor=center, inner sep=0}, "{1\circ \mu}", from=1-1, to=1-3]
	\arrow[""{name=1, anchor=center, inner sep=0}, "{\psi\circ1}", from=1-3, to=1-5]
	\arrow["{1\circ \phi}"{description}, from=1-3, to=2-3]
	\arrow["{1\circ \phi}", from=1-5, to=2-5]
	\arrow[curve={height=-12pt}, equals, nfold, from=2-3, to=1-1]
	\arrow[""{name=2, anchor=center, inner sep=0}, "{\psi\circ1}"', from=2-3, to=2-5]
	\arrow["{1\circ\Theta}"{pos=0.4}, shift right=3, between={0.2}{0.6}, Rightarrow, nfold, from=0, to=2-2]
	\arrow["{(\psi:\phi)}", shift right=3, between={0.3}{0.7}, Rightarrow, nfold, dashed, from=1, to=2]
\end{tikzcd}	
	\end{equation}
\end{lemma}

\begin{proof}
	The strategy to establish both equations is the same: first paste both sides with one interchange map and then use the axioms of a Gray-category. 
	Let us explicitly show it for \eqref{eq:lem-interchange-1}. 
	First, notice that proving \eqref{eq:lem-interchange-1}, by pasting with $(\psi:\phi)$, is equivalent to proving the following equality. 
	\begin{center}
\begin{tikzcd}[ampersand replacement=\&]
	gf \&\& {g'f} \\
	{gf'} \&\& {g'f'} \\
	gf \&\& {g'f}
	\arrow[""{name=0, anchor=center, inner sep=0}, "{\psi\circ1}", from=1-1, to=1-3]
	\arrow["{1\circ \phi}"', from=1-1, to=2-1]
	\arrow["{1\circ \phi}"{description}, from=1-3, to=2-3]
	\arrow[""{name=1, anchor=center, inner sep=0}, shift left=2, curve={height=-30pt}, equals, nfold, from=1-3, to=3-3]
	\arrow[""{name=2, anchor=center, inner sep=0}, "{\psi\circ1}"{description}, from=2-1, to=2-3]
	\arrow["{1\circ \mu}"', from=2-1, to=3-1]
	\arrow["{1\circ \mu}"{description}, from=2-3, to=3-3]
	\arrow[""{name=3, anchor=center, inner sep=0}, "{\psi\circ1}"', from=3-1, to=3-3]
	\arrow["{(\psi:\phi)}"', shift left=5, between={0.3}{0.7}, Rightarrow, nfold, dashed, from=0, to=2]
	\arrow["{1\circ\Delta}"', shift left=3, between={0.3}{0.7}, Rightarrow, nfold, from=1, to=2-3]
	\arrow["{(\psi:\mu)}"', shift left=5, between={0.3}{0.7}, Rightarrow, nfold, dashed, from=2, to=3]
\end{tikzcd}
$\quad=\quad$
\begin{tikzcd}[ampersand replacement=\&,column sep=small]
	gf \& {} \& \\
	{gf'} \&\& gf \\
	\&\& {g'f}
	\arrow["{1\circ \phi}"', from=1-1, to=2-1]
	\arrow[curve={height=-12pt}, equals, nfold, from=1-1, to=2-3]
	\arrow[""{name=0, anchor=center, inner sep=0}, "{1\circ \mu}"{description}, from=2-1, to=2-3]
	\arrow["{\psi\circ1}", from=2-3, to=3-3]
	\arrow["{1\circ\Delta}"{pos=0.4}, shift right=4, between={0.2}{0.6}, Rightarrow, nfold, from=1-2, to=0]
\end{tikzcd}
	\end{center}
For this, we start from the right hand side and use two axioms of the interchange (specifically (1.2) and then (1.1) as in \cite[Remark~1.2]{GambLob:psm}).
	\begin{center}
\begin{tikzcd}[ampersand replacement=\&]
	gf \&\& {g'f} \\
	{gf'} \&\& {g'f'} \\
	gf \&\& {g'f}
	\arrow[""{name=0, anchor=center, inner sep=0}, "{\psi\circ1}", from=1-1, to=1-3]
	\arrow["{1\circ \phi}"', from=1-1, to=2-1]
	\arrow["{1\circ \phi}"{description}, from=1-3, to=2-3]
	\arrow[""{name=1, anchor=center, inner sep=0}, shift left=2, curve={height=-30pt}, equals, nfold, from=1-3, to=3-3]
	\arrow["{1\circ \mu}"', from=2-1, to=3-1]
	\arrow["{1\circ \mu}"{description}, from=2-3, to=3-3]
	\arrow[""{name=2, anchor=center, inner sep=0}, "{\psi\circ1}"', from=3-1, to=3-3]
	\arrow["{(\psi:\mu\cdot\phi)}"', shift left=5, between={0.3}{0.7}, Rightarrow, nfold, dashed, from=0, to=2]
	\arrow["{1\circ\Delta}"', shift left=3, between={0.3}{0.7}, Rightarrow, nfold, from=1, to=2-3]
\end{tikzcd}
$\quad=\quad$
\begin{tikzcd}[ampersand replacement=\&]
	gf \&\& \\
	{gf'} \\
	gf \&\& {g'f}
	\arrow["{1\circ \phi}"', from=1-1, to=2-1]
	\arrow[""{name=0, anchor=center, inner sep=0}, shift left=2, curve={height=-30pt}, equals, nfold, from=1-1, to=3-1]
	\arrow["{1\circ \mu}"', from=2-1, to=3-1]
	\arrow["{\psi\circ1}"', from=3-1, to=3-3]
	\arrow["{1\circ\Delta}"'{pos=0.6}, shift left=2, between={0.4}{0.8}, Rightarrow, nfold, from=0, to=2-1]
\end{tikzcd}

	\end{center}
\end{proof}

\black 

\subsection{Pseudo-transformations}\label{sec:psd-transf}
In this section, we recall from \cite{BourkeLobbia} the Gray-categories $\psd(\ca,\cb)$, $\lax(\ca,\cb)$ and $\oplax(\ca,\cb)$ of pseudomaps, pseudo/lax/oplax-transformations and their higher cells.
These $n$-cells are sometimes called $n$-transfors.

%
%



\begin{itemize}
	\item A \emph{pseudomap} $F\colon\ca\to\cb$ of Gray-categories consists of  
\begin{itemize}
\item for each object $x \in \ca$, an object $Fx \in \cb$,
\item for each pair of objects $x,y \in \ca$, a $2$-functor $F_{x,y} \colon \ca(x,y) \to \cb(Fx,Fy)$,
	\item  for any pair of 1-cells $f_1\colon x\to y$ and $f_2\colon y\to z$ in $\ca$, an invertible 2-cell 
	$$F^2_{f_2,f_1}\colon Ff_2\circ Ff_1\to F(f_2\circ f_1)$$
	in $\cb$ satisfying the nine axioms listed in \cite[Appendix~A.1]{BourkeLobbia}. 
\end{itemize}  
We point out that pseudomaps are locally $2$-functors, preserve identity $1$-cells strictly and are very close to the notion of normal pseudofunctors in $2$-category theory.  A Gray-functor (or strict map) is a pseudomap for which the invertible $2$-cells $F^2_{f_2,f_1}$ are identities.

\item 
A \emph{lax/oplax transformation} $\alpha^l/\alpha^o\colon F\to G$ consists of:
\begin{itemize}
	\item for any object $x\in\ca$ a 1-cell
	\begin{center}
		\vspace{-0.25cm}
		\begin{tabular}{>{\centering\arraybackslash}p{0.5\textwidth}:
				>{\centering\arraybackslash}p{0.4\textwidth}}
		$\alpha^l_x\colon Fx\to Gx$ 
		& $\alpha^o_x\colon Fx\to Gx$;
		\end{tabular}
	\end{center}
		\vspace{0.1cm}
	\item for any 1-cell $f\colon x\to y$ in $\ca$, a 2-cell in $\cb$
	\begin{center}
		\vspace{-0.25cm}
		\begin{tabular}{>{\centering\arraybackslash}p{0.5\textwidth}:
				>{\centering\arraybackslash}p{0.4\textwidth}}
			\begin{tikzcd}
				{Fx} & {Gx} \\
				{Fy} & {Gy}
				\arrow[""{name=0, anchor=center, inner sep=0}, "{\alpha^l_x}", from=1-1, to=1-2]
				\arrow[""{name=1, anchor=center, inner sep=0}, "{\alpha^l_y}"', from=2-1, to=2-2]
				\arrow["Ff"', from=1-1, to=2-1]
				\arrow["{Gf}", from=1-2, to=2-2]
				\arrow["{\alpha^l_f}"{xshift=0.1cm}, shorten <=10pt, shorten >=10pt, Rightarrow, from=0, to=1]
			\end{tikzcd} 
			& \begin{tikzcd}
				{Fx} & {Gx} \\
				{Fy} & {Gy;}
				\arrow[""{name=0, anchor=center, inner sep=0}, "{\alpha^o_x}", from=1-1, to=1-2]
				\arrow[""{name=1, anchor=center, inner sep=0}, "{\alpha^o_y}"', from=2-1, to=2-2]
				\arrow["Ff"', from=1-1, to=2-1]
				\arrow["{Gf}", from=1-2, to=2-2]
				\arrow["{\alpha^o_f}"{xshift=-0.1cm}, shorten <=10pt, shorten >=10pt, Rightarrow, from=1, to=0]
			\end{tikzcd}
		\end{tabular}
	\end{center}
	\vspace{0.1cm}
	
	\item for any 2-cell $\phi\colon f\to f'$ in $\ca$ a 3-cell in $\cb$ of the form
	\begin{center}
		\vspace{-0.25cm}
		\begin{tabular}{>{\centering\arraybackslash}p{0.5\textwidth}:
				>{\centering\arraybackslash}p{0.4\textwidth}}
			\begin{tikzcd}[ampersand replacement=\&]
				{Gf\alpha^l_x} \& {\alpha^l_y Ff} \\
				{Gf\alpha^l_x} \& {\alpha^l_y Ff}
				\arrow[""{name=0, anchor=center, inner sep=0}, "{\alpha^l_f}", from=1-1, to=1-2]
				\arrow["{G\phi\circ1}"', from=1-1, to=2-1]
				\arrow["{1\circ F\phi}", from=1-2, to=2-2]
				\arrow[""{name=1, anchor=center, inner sep=0}, "{\alpha^l_{f'}}"', from=2-1, to=2-2]
				\arrow["{\alpha^l_\phi}", between={0.3}{0.7}, Rightarrow, from=0, to=1]
			\end{tikzcd} 
			& 
			\begin{tikzcd}[ampersand replacement=\&]
				{\alpha^o_y Ff} \& {\alpha^o_y Ff'} \\
				{Gf\alpha^o_x} \& {Gf'\alpha^o_x}
				\arrow[""{name=0, anchor=center, inner sep=0}, "{1\circ F\phi}", from=1-1, to=1-2]
				\arrow["{\alpha^o_f}"', from=1-1, to=2-1]
				\arrow["{\alpha^o_{f'}}", from=1-2, to=2-2]
				\arrow[""{name=1, anchor=center, inner sep=0}, "{G\phi\circ1}"', from=2-1, to=2-2]
				\arrow["{\alpha^o_\phi}", between={0.3}{0.7}, Rightarrow, from=0, to=1]
			\end{tikzcd}
		\end{tabular}
	\end{center}
	\vspace{0.1cm}
	
	\item 1-cells $f_1\colon x\to y$ and $f_2\colon y\to z$ in $\ca$, an invertible 3-cell in $\cb$
	\begin{center}
		\vspace{-0.25cm}
		\begin{tabular}{>{\centering\arraybackslash}p{0.5\textwidth}:
				>{\centering\arraybackslash}p{0.4\textwidth}}
				\scalebox{0.9}{
			\begin{tikzcd}[ampersand replacement=\&]
				{Gf_2Gf_1\alpha^l_x} \& {Gf_2\alpha^l_Y Ff_1} \& {\alpha^l_z Ff_2Ff_1} \\
				{G(f_2f_1) \alpha^l_x} \&\& {\alpha^l_zF(f_2f_1)}
				\arrow["{1\circ\alpha^l_{f_1}}", from=1-1, to=1-2]
				\arrow["{G^2_{f_2,f_1}\circ1}"', from=1-1, to=2-1]
				\arrow["{\alpha^l_{f_2}\circ1}", from=1-2, to=1-3]
				\arrow["{1\circ F^2_{f_2,f_1}}", from=1-3, to=2-3]
				\arrow[""{name=0, anchor=center, inner sep=0}, "{\alpha^l_{f_2f_1}}"', from=2-1, to=2-3]
				\arrow["{(\alpha^l)^2_{f_2,f_1}}", between={0.3}{0.7}, Rightarrow, from=1-2, to=0]
			\end{tikzcd}
		}
			& 
		\scalebox{0.9}{
			\begin{tikzcd}[ampersand replacement=\&]
				{\alpha^o_zFf_2Ff_1} \&\& {\alpha^o_zF(f_2f_1)} \\
				{Gf_2\alpha^o_yFf_1} \\
				{Gf_2Gf_1\alpha^o_x} \&\& {G(f_2f_1)\alpha_x.}
				\arrow[""{name=0, anchor=center, inner sep=0}, "{1\circ F_{f_2,f_1}^2}", from=1-1, to=1-3]
				\arrow["{{\alpha^o}_{f_2}\circ1}"', from=1-1, to=2-1]
				\arrow[""{name=1, anchor=center, inner sep=0}, "{G_{f_2,f_1}^2\circ1}"', from=3-1, to=3-3]
				\arrow["{1\circ\alpha^o_{f_1}}"', from=2-1, to=3-1]
				\arrow["{\alpha^o_{f_2f_1}}", from=1-3, to=3-3]
				\arrow["{(\alpha^o)^{f_2,f_1}_2}"{xshift=0.1cm}, shorten <=25pt, shorten >=25pt, Rightarrow, from=0, to=1]
			\end{tikzcd}
			}
		\end{tabular}
	\end{center}
	\vspace{
		0.1cm}
\end{itemize} 
The data of a lax transformation has to satisfy the seven axioms listed in \cite[Appendix~A.2]{BourkeLobbia}, whilst the data of an oplax transformation has to satisfy dual axioms, written explicitly in \Cref{subsec:ax-oplax-transf}. 

Finally, \emph{pseudo-transformations} are lax transformations $\alpha^l\colon F\to G$ equipped with, for any 1-cell $f\colon x\to y\in\ca$, an adjoint equivalence $\aexi^\alpha_f := \alpha^o_f \dashv \alpha^l_f\colon Gf\alpha^l_x \to \alpha^l_yFf$ \footnote{The notation used here hints at the fact that $\alpha^o_f$ will be seen as part of an oplax transformation in \Cref{prop:equ-psd-transf}.} such that:
\begin{enumerate}[(a)]
	\item at an object $x\in\ca$, $\alpha^o_{1_x} \dashv \alpha^l_{1_x}$ is the identity adjoint equivalence;
	\item at a $2$-cell $\phi\colon f\to f'\in\ca$, the 3-cell $\alpha_\phi$ is invertible.
\end{enumerate}

\item A \emph{lax/oplax modification} 
$\Lambda^l\colon\alpha^l\to\beta^l$/$\Lambda^o\colon\alpha^o\to\beta^o$ consists of:
\begin{itemize}
	\item for any object $x\in\ca$, a 2-cell in $\cb$
	\begin{center}
		\vspace{-0.25cm}
	\begin{tabular}{>{\centering\arraybackslash}p{0.5\textwidth}:
			>{\centering\arraybackslash}p{0.4\textwidth}}
		$\Lambda^l_x\colon \alpha^l_x\to \beta^l_x$ 
		& $\Lambda^o_x\colon \alpha^o_x\to \beta^o_x$;
	\end{tabular}
	\vspace{0.1cm}
	\end{center}
	
	\item for any 1-cell $f\colon x\to y$ in $\ca$, a 3-cell $\Lambda_f$ in $\cb$ of the form
	\begin{center}
		\vspace{-0.25cm}
	\begin{tabular}{>{\centering\arraybackslash}p{0.5\textwidth}:
			>{\centering\arraybackslash}p{0.4\textwidth}}
			\begin{tikzcd}[ampersand replacement=\&]
				{Gf\alpha_x^l} \& {Gf\beta_x^l} \\
				{\alpha_y^l Ff} \& {\beta_y^l Ff}
				\arrow[""{name=0, anchor=center, inner sep=0}, "{1\circ\Lambda^l_X}", from=1-1, to=1-2]
				\arrow["{\alpha_f^l}"', from=1-1, to=2-1]
				\arrow["{\beta^l_{f}}", from=1-2, to=2-2]
				\arrow[""{name=1, anchor=center, inner sep=0}, "{\Lambda_y^l\circ1}"', from=2-1, to=2-2]
				\arrow["{\Lambda_f^l}", between={0.3}{0.7}, Rightarrow, from=0, to=1]
			\end{tikzcd}
		& 
		\begin{tikzcd}[ampersand replacement=\&]
			{\alpha^o_yFf} \& {Gf\alpha^o_x} \\
			{\beta^o_yFf} \& {Gf\beta^o_x.}
			\arrow[""{name=0, anchor=center, inner sep=0}, "{\alpha^o_f}", from=1-1, to=1-2]
			\arrow["{\Lambda^o_y\circ1}"', from=1-1, to=2-1]
			\arrow["{1\circ\Lambda^o_x}", from=1-2, to=2-2]
			\arrow[""{name=1, anchor=center, inner sep=0}, "{\beta^o_f}"', from=2-1, to=2-2]
			\arrow["{\Lambda^o_f}", between={0.3}{0.7}, Rightarrow, from=0, to=1]
		\end{tikzcd}
	\end{tabular}
	\end{center}
	\vspace{0.1cm}
\end{itemize}
	The data of a lax modification have to satisfy the three axioms listed in \cite[Appendix~A.3]{BourkeLobbia}, whilst the data of an oplax modifaction has to satisfy dual axioms written explicitly in \Cref{subsec:ax-oplax-mod}. \black

	\emph{Pseudo-modifications} are lax modifications $\Lambda\colon \alpha\to \beta$ such that, for any 1-cell $f\colon x\to y\in\ca$, the 3-cell $\Lambda_f$ is invertible. 
	
	\item In each case, the data for a lax/pseudo/oplax-perturbation $\theta\colon\Lambda\to\Gamma$ consists of, for any object $x\in\ca$, a 3-cell $\theta_x\colon \Lambda_x\to\Gamma_x$ in $\cb$.  \emph{Lax perturbations} have to satisfy the axiom in \cite[Appendix~A.4]{BourkeLobbia} whilst the corresponding axiom for an \emph{oplax perturbation} is given in \Cref{ax:oplax-pert}.  A \emph{pseudo-perturbation} between pseudo-modifications is simply a lax perturbation.
\end{itemize}

Following Section 3.2 \cite{BourkeLobbia}, we let $\GCat$ denote the category of Gray-categories and Gray-functors and $\GCat_p$ denote the category of Gray-categories and pseudomaps.

\begin{rmk}\label{plusminus}
There is a natural isomorphism $(-)_{\ca,\cb}^{+}\colon\oplax(\ca,\cb) \cong \lax(\ca^{co},\cb^{co})^{co}$ where $\ca^{co}$ and $\cb^{co}$ are the Gray-categories with 2-cells reversed.  At an oplax transformation $\alpha \colon F \to G$, the lax transformation $\alpha^+ \colon F^{co} \to G^{co}$ has $(\alpha^+)_x = \alpha_x$, $(\alpha^+)_f = \alpha_f$ (now with the reverse orientation), $(\alpha^+)_\phi = \alpha_\phi$ and with $(\alpha^+)^2_{f_2,f_1} = (\alpha_2^{f_2,f_1})\sharp$ for the the $(-)\sharp$ operation described in \Cref{rmk:mates-over-adj-equiv-calculations}, in which the adjoint equivalences are given by the invertible 1-cells $1\circ F^2_{f_2,f_1}$ and $G^2_{f_2,f_1}\circ1$.   Full details are described in \cite[Section~3.5]{BourkeLobbia}. 

The inverse operation to $(-)^{+}$ is called $(-)^{-}$ in \cite[Section~3.5]{BourkeLobbia}.  It involves the $(-)\flat$ operation from  \Cref{rmk:mates-over-adj-equiv-calculations}.  These combine to natural isomorphisms $(-)_{\ca,\cb}^{-} \colon \colon\lax(\ca,\cb) \cong \oplax(\ca^{co}.\cb^{co})^{co}$. 
\end{rmk}

\section{An unbiased presentation of pseudo-transformations}\label{sect:unbiased}

 
In this section, as hinted at in \cite[Remark~6.1]{BourkeLobbia}, we give a symmetric presentation of pseudo-transformations by combining lax and oplax transformations.  This presentation will be needed in \Cref{sec:bin-braiding} to construct the binary symmetry for the skew multicategory $\GCatmp$ and hence, in \Cref{sec:symm-skw-mon-structure}, to prove that the flexible skew monoidal structure $(\GCat,\otimes_p,1)$ is symmetric skew monoidal.
\subsection{Characterisation of $n$-cells in $\psd(\ca,\cb)$}

We begin with 1-cells (i.e. pseudo-transformations) in \Cref{prop:equ-psd-transf}, before moving on to 2- and 3-cells in  Propositions~\ref{prop:equ-psd-mod} and \ref{prop:equ-psd-pert}. 


\begin{prop}
	\label{prop:equ-psd-transf}
	Let $F,G\colon\ca \rightrightarrows \cb$ be pseudomaps and consider a family $\alpha_x\colon Fx\to Gx$ of 1-cells in $\cb$ equipped with, for any 1-cell $f\colon x\to y$ in $\ca$, an adjoint pair $\aexi^\alpha_f$
	\begin{center}
		\begin{tikzcd}[ampersand replacement=\&]
			{Gf\alpha_x} \& {\alpha_yFf}
			\arrow[""{name=0, anchor=center, inner sep=0}, "{\alpha_f^l}", curve={height=-12pt}, from=1-1, to=1-2]
			\arrow[""{name=1, anchor=center, inner sep=0}, "{\alpha_f^o}", curve={height=-12pt}, from=1-2, to=1-1]
			\arrow["\top"{description}, draw=none, from=1, to=0]
		\end{tikzcd}\hspace{0.25cm} (with unit $\eta_f$ and counit $\epsilon_f$) 
	\end{center}
	in $\cb(Fx,Gy)$ such that for any $x\in\ca$, $\aexi^\alpha_{1_x}=1_\aeq$. 
	Then, there is a bijective correspondence between:
	\begin{enumerate}[(1)]
		\item Extensions of $(\alpha_x,\alpha_f^l)$ to a lax transformation $\alpha^l$.
		
		\item Extensions of $(\alpha_x,\alpha_f^o)$ to an oplax transformation $\alpha^o$.

		\item Extensions of $(\alpha_x,\alpha_f^l)$ to a lax transformation $\alpha^l$ and of $(\alpha_x,\alpha_f^o)$ to an oplax transformation $\alpha^o$ such that 		
		\begin{center}
			$\alpha^o_\phi=$
\begin{tikzcd}[ampersand replacement=\&]
	{\alpha_yFf} \& \\
	{Gf\alpha_x} \& {\alpha_yFf} \\
	{Gf'\alpha_x} \& {\alpha_yFf'} \\
	\& {Gf'\alpha_x}
	\arrow["{{\alpha^o_f}}"', from=1-1, to=2-1]
	\arrow[""{name=0, anchor=center, inner sep=0}, curve={height=-12pt}, equals, from=1-1, to=2-2]
	\arrow[""{name=1, anchor=center, inner sep=0}, "{\alpha^l_f}"{description}, from=2-1, to=2-2]
	\arrow["{G\phi\circ1}"', from=2-1, to=3-1]
	\arrow["{1\circ F\phi}", from=2-2, to=3-2]
	\arrow[""{name=2, anchor=center, inner sep=0}, "{\alpha^l_{f'}}"{description}, from=3-1, to=3-2]
	\arrow[""{name=3, anchor=center, inner sep=0}, curve={height=12pt}, equals, from=3-1, to=4-2]
	\arrow["{\alpha^o_{f'}}", from=3-2, to=4-2]
	\arrow["{\eta_{f}}"'{pos=0.6}, between={0.3}{0.6}, Rightarrow, from=0, to=1]
	\arrow["{\alpha^l_\phi}", between={0.4}{0.6}, Rightarrow, from=1, to=2]
	\arrow["{\epsilon_{f'}}"{pos=0.4}, between={0.3}{0.6}, Rightarrow, from=2, to=3]
\end{tikzcd}
			\hspace{1cm} and \\
			${(\alpha^o)}_2^{f_2,f_1}=$
\begin{tikzcd}[ampersand replacement=\&]
	{\alpha_Z Ff_2Ff_1} \&\& {} \&\& \\
	{Gf_2\alpha_Y Ff_1} \& {} \\
	{G{f_2}G{f_1}\alpha_X} \&\& {Gf_2\alpha_Y Ff_1} \&\& {\alpha_Z Ff_2Ff_1} \\
	{G({f_2}{f_1}) \alpha_X} \&\& {} \& {} \& {\alpha_ZF(f_2f_1)} \\
	\&\&\& {} \& {G({f_2}{f_1}) \alpha_X}
	\arrow["{{\alpha^o_{f_2}}\circ1}"', from=1-1, to=2-1]
	\arrow[curve={height=-12pt}, equals, from=1-1, to=3-5]
	\arrow["{\eta_{f_2}\circ1}"'{pos=0.6}, between={0.5}{0.8}, Rightarrow, from=1-3, to=3-3]
	\arrow["{1\circ\alpha^o_{f_1}}"', from=2-1, to=3-1]
	\arrow[curve={height=-12pt}, equals, from=2-1, to=3-3]
	\arrow[""{name=0, anchor=center, inner sep=0}, "{1\circ\alpha^l_{f_1}}"{description}, from=3-1, to=3-3]
	\arrow["{G^2_{f_2,f_1}\circ1}"', from=3-1, to=4-1]
	\arrow["{\alpha^l_{f_2}\circ1}"{description}, from=3-3, to=3-5]
	\arrow["{{(\alpha^l)}^2_{f_2,f_1}}", between={0.3}{0.7}, Rightarrow, from=3-3, to=4-3]
	\arrow["{1\circ F^2_{f_2,f_1}}", from=3-5, to=4-5]
	\arrow["{\alpha^l_{f_2f_1}}"{description}, from=4-1, to=4-5]
	\arrow[curve={height=12pt}, equals, from=4-1, to=5-5]
	\arrow["{\epsilon_{f_2f_1}}"', between={0.1}{0.7}, Rightarrow, from=4-4, to=5-4]
	\arrow["{{\alpha^o_{f_2f_1}}}", from=4-5, to=5-5]
	\arrow["{1\circ\eta_{f_1}}"'{pos=0.4}, between={0.3}{0.7}, Rightarrow, from=2-2, to=0]
\end{tikzcd}
		\end{center}
In other words, $\alpha^o_\phi$ and $\alpha^l_\phi$ are mates with respect to the adjoint equivalences $(\aexi^\alpha_f,\aexi^\alpha_{f'})$ and 
and $({\alpha^o})_2^{f_2,f_1}$ and $({\alpha^l})^2_{f_2,f_1}$ are mates with the respect to the adjoint equivalences $(\,(\aexi^\alpha_{f_2}\circ1)\cdot (1\circ\aexi^\alpha_{f_1}),\aexi^{\alpha}_{f_2,f_1}\,)$.

	\end{enumerate}
\end{prop}

\begin{proof}
	To get from $(3)$ to either $(1)$ or $(2)$ we consider the mapping which forgets either the oplax or lax part in $(3)$, respectively. 
	Let us call these $u_1\colon(3)\to(1)$ and $u_2\colon(3)\to(2)$. 
	
	We start by constructing the inverse mapping $v_1\colon(1)\to(3)$ to $u_1$.  Given a lax transformation $\alpha^l$ as in (1), we define $v_1(\alpha^l):=(\alpha^l,\,(\alpha_x,\alpha_f^o,\alpha^o_\phi,{(\alpha^o)}_2^{f_2,f_1})\,)$ where $\alpha^o_\phi$ and ${(\alpha^o)}_2^{f_2,f_1}$ are the mates as in $(3)$.  Let us assume for a moment that this data actually defines an oplax transformation so that $v_1$ is defined. Then $u_1v_1=1$ since neither $u_1$ nor $v_1$ change  the lax part, whilst $v_1u_1=1$ since in $(3)$ both $\alpha^o_\phi$ and ${(\alpha^o)}_2^{f_2,f_1}$ must be defined following those equations. 
	
It remains then to prove that defining $\alpha^o_\phi$ and ${(\alpha^o)}_2^{f_2,f_1}$ as in $(3)$ 
       produces an oplax transformation.  We will see that the axioms for an oplax transformation follow from the axioms for the lax transformation, triangle identities for the adjoint pairs and the axioms regarding the interchange isomorphism in a Gray-category. \black
First, the degeneracy equations (Axioms (\ref{ax:oplax-tr.deg-2-cell}), (\ref{ax:oplax-tr.deg-3-cell}) and (\ref{ax:oplax-tr.deg-alpha2}) in \Cref{subsec:ax-oplax-transf})
follow by the degeneracies for $\alpha_l$ as a lax transformation and the assumption $\aexi^\alpha_{1_x}=1_\aeq$. 
Then, we proceed with the compatibility with 3-cells, i.e. Axiom (\ref{ax:oplax-tr.3-cell}) in \Cref{subsec:ax-oplax-transf}. 
This asserts that at a 3-cell $\Gamma\colon\phi\to\phi'$ in $\ca$, the first and last diagrams below are equal.
In order to prove it, we notice that: the first and last equation below are the definition of $\alpha^o_\phi$; the equation in the middle follows by the corresponding axiom for $\alpha$ as a lax transformation (Axiom A.2 (ii) of \cite{BourkeLobbia}).\footnote{
		This axiom and the next one follow by the properties of mates inducing isomorphism of double categories (see \cite[Proposition ~2.2]{KellyStreet:rew-el-2cats}) and the corresponding axioms for a lax transformation. }
	\begin{center}
		\scalebox{0.9}{
			\begin{tikzcd}[ampersand replacement=\&]
				{\alpha_yFf} \&\& {\alpha_yFf'} \\
				\\
				{Gf\alpha_x} \&\& {Gf'\alpha_x}
				\arrow[""{name=0, anchor=center, inner sep=0}, "{1\circ F\phi}", curve={height=-12pt}, from=1-1, to=1-3]
				\arrow[""{name=1, anchor=center, inner sep=0}, "{1\circ F\phi'}"', curve={height=12pt}, from=1-1, to=1-3]
				\arrow["{\alpha^o_f}"', from=1-1, to=3-1]
				\arrow["{\alpha^o_{f'}}", from=1-3, to=3-3]
				\arrow[""{name=2, anchor=center, inner sep=0}, "{G\phi'\circ1}"', curve={height=12pt}, from=3-1, to=3-3]
				\arrow["{1\circ F\Gamma}", shift right=3, between={0.3}{0.7}, Rightarrow, from=0, to=1]
				\arrow["{\alpha^o_{\phi'}}", between={0.4}{0.6}, Rightarrow, from=1, to=2]
			\end{tikzcd}
		}
		=
		\scalebox{0.9}{
		\begin{tikzcd}[ampersand replacement=\&]
			\&\& {\alpha_yFf} \&\& {\alpha_yFf'} \&\& {Gf'\alpha_x} \\
			\\
			{\alpha_yFf} \&\& {Gf\alpha_x} \&\& {Gf'\alpha_x}
			\arrow[""{name=0, anchor=center, inner sep=0}, "{1\circ F\phi}", curve={height=-12pt}, from=1-3, to=1-5]
			\arrow[""{name=1, anchor=center, inner sep=0}, "{1\circ F\phi'}"', curve={height=12pt}, from=1-3, to=1-5]
			\arrow["{\alpha^o_{f'}}", from=1-5, to=1-7]
			\arrow[""{name=2, anchor=center, inner sep=0}, equals, nfold, from=3-1, to=1-3]
			\arrow["{\alpha^o_f}"', from=3-1, to=3-3]
			\arrow[""{name=3, anchor=center, inner sep=0}, "{\alpha_f^l}"{description}, from=3-3, to=1-3]
			\arrow["{G\phi'\circ1}"', curve={height=12pt}, from=3-3, to=3-5]
			\arrow[""{name=4, anchor=center, inner sep=0}, "{\alpha_{f'}^l}"{description}, from=3-5, to=1-5]
			\arrow[""{name=5, anchor=center, inner sep=0}, equals, nfold, from=3-5, to=1-7]
			\arrow["{1\circ F\Gamma}", shift right=3, between={0.3}{0.7}, Rightarrow, nfold, from=0, to=1]
			\arrow["{\eta_{f}}", shift right=5, between={0.3}{0.6}, Rightarrow, nfold, from=2, to=3]
			\arrow["{\alpha_{\phi}^l}", shift right=4, between={0.4}{0.6}, Rightarrow, nfold, from=3, to=4]
			\arrow["{\epsilon_{f'}}", between={0.4}{0.6}, Rightarrow, nfold, from=4, to=5]
		\end{tikzcd}
		}
		= \\
		\scalebox{0.9}{
		\begin{tikzcd}[ampersand replacement=\&]
			\&\& {\alpha_yFf} \&\& {\alpha_yFf'} \&\& {Gf'\alpha_x} \\
			\\
			{\alpha_yFf} \&\& {Gf\alpha_x} \&\& {Gf'\alpha_x}
			\arrow["{1\circ F\phi}", curve={height=-12pt}, from=1-3, to=1-5]
			\arrow["{\alpha^o_{f'}}", from=1-5, to=1-7]
			\arrow[""{name=0, anchor=center, inner sep=0}, equals, nfold, from=3-1, to=1-3]
			\arrow["{\alpha^o_f}"', from=3-1, to=3-3]
			\arrow[""{name=1, anchor=center, inner sep=0}, "{\alpha_f^l}"{description}, from=3-3, to=1-3]
			\arrow[""{name=2, anchor=center, inner sep=0}, "{G\phi'\circ1}"', curve={height=12pt}, from=3-3, to=3-5]
			\arrow[""{name=3, anchor=center, inner sep=0}, "{G\phi\circ1}", curve={height=-12pt}, from=3-3, to=3-5]
			\arrow[""{name=4, anchor=center, inner sep=0}, "{\alpha_{f'}^l}"{description}, from=3-5, to=1-5]
			\arrow[""{name=5, anchor=center, inner sep=0}, equals, nfold, from=3-5, to=1-7]
			\arrow["{\eta_{f}}", shift right=5, between={0.3}{0.6}, Rightarrow, nfold, from=0, to=1]
			\arrow["{\alpha_{\phi}^l}", shift left=4, between={0.4}{0.6}, Rightarrow, nfold, from=1, to=4]
			\arrow["{G\Gamma\circ1}", shift right=3, between={0.3}{0.7}, Rightarrow, nfold, from=3, to=2]
			\arrow["{\epsilon_{f'}}", between={0.4}{0.6}, Rightarrow, nfold, from=4, to=5]
		\end{tikzcd}
		}
		=
		\scalebox{0.9}{
			\begin{tikzcd}[ampersand replacement=\&]
				{\alpha_yFf} \&\& {\alpha_yFf'} \\
				\\
				{Gf\alpha_x} \&\& {Gf'\alpha_x}
				\arrow[""{name=0, anchor=center, inner sep=0}, "{1\circ F\phi}", curve={height=-12pt}, from=1-1, to=1-3]
				\arrow["{\alpha^o_f}"', from=1-1, to=3-1]
				\arrow["{\alpha^o_{f'}}", from=1-3, to=3-3]
				\arrow[""{name=1, anchor=center, inner sep=0}, "{G\phi'\circ1}"', curve={height=12pt}, from=3-1, to=3-3]
				\arrow[""{name=2, anchor=center, inner sep=0}, "{G\phi\circ1}", curve={height=-12pt}, from=3-1, to=3-3]
				\arrow["{\alpha^o_\phi}", between={0.4}{0.6}, Rightarrow, from=0, to=2]
				\arrow["{G\Gamma\circ1}", shift right=3, between={0.3}{0.7}, Rightarrow, from=2, to=1]
			\end{tikzcd}
		}
	\end{center}
	
	We now turn to compatibility with vertical composition of 2-cells $\phi_1\colon f\to f'$ and $\phi_2\colon f'\to f''$, i.e. Axiom (\ref{ax:oplax-tr.comp-2-cell}) in \Cref{subsec:ax-oplax-transf}. \black

	This axiom follows by the equation below (which is true by one of the triangle identities for $\alpha_{f'}^o\dashv\alpha^l_{f'}$) and the corresponding axiom of $\alpha$ (i.e. Axiom A.2 (iii) of \cite{BourkeLobbia}). 
	\begin{center}
\[\begin{tikzcd}[ampersand replacement=\&]
	{\alpha_yFf} \&\&\&\& \\
	{Gf\alpha_x} \&\& {\alpha_yFf} \\
	{Gf'\alpha_x} \&\& {\alpha_yFf'} \\
	\&\& {Gf'\alpha_x} \&\& {\alpha_yFf'} \\
	\&\& {Gf'\alpha_x} \&\& {\alpha_yFf''} \\
	\&\&\&\& {Gf'\alpha_x}
	\arrow["{\alpha^o_f}"', from=1-1, to=2-1]
	\arrow[""{name=0, anchor=center, inner sep=0}, curve={height=-6pt}, equals, nfold, from=1-1, to=2-3]
	\arrow[""{name=1, anchor=center, inner sep=0}, "{\alpha_f}"{description}, from=2-1, to=2-3]
	\arrow["{G\phi_1\circ1}"', from=2-1, to=3-1]
	\arrow["{1\circ F\phi_1}", from=2-3, to=3-3]
	\arrow[""{name=2, anchor=center, inner sep=0}, "{\alpha_{f'}}"{description}, from=3-1, to=3-3]
	\arrow[""{name=3, anchor=center, inner sep=0}, curve={height=6pt}, equals, nfold, from=3-1, to=4-3]
	\arrow["{\alpha^o_{f'}}"{description}, from=3-3, to=4-3]
	\arrow[""{name=4, anchor=center, inner sep=0}, curve={height=-6pt}, equals, nfold, from=3-3, to=4-5]
	\arrow[""{name=5, anchor=center, inner sep=0}, "{\alpha_{f'}}"{description}, from=4-3, to=4-5]
	\arrow["{G\phi_2\circ1}"', from=4-3, to=5-3]
	\arrow["{1\circ F\phi_2}", from=4-5, to=5-5]
	\arrow[""{name=6, anchor=center, inner sep=0}, from=5-3, to=5-5]
	\arrow[""{name=7, anchor=center, inner sep=0}, curve={height=6pt}, equals, nfold, from=5-3, to=6-5]
	\arrow["{\alpha^o_{f''}}", from=5-5, to=6-5]
	\arrow["{\eta_f}"', shift right=5, between={0.3}{0.7}, Rightarrow, nfold, from=0, to=1]
	\arrow["{\alpha^l_{\phi_1}}", between={0.4}{0.6}, Rightarrow, nfold, from=1, to=2]
	\arrow["{\epsilon_{f'}}", shift left=5, between={0.3}{0.7}, Rightarrow, nfold, from=2, to=3]
	\arrow["{\eta_{f'}}"', shift right=5, between={0.3}{0.7}, Rightarrow, nfold, from=4, to=5]
	\arrow["{\alpha^l_{\phi_2}}", between={0.4}{0.6}, Rightarrow, nfold, from=5, to=6]
	\arrow["{\epsilon_{f'}}", shift left=5, between={0.3}{0.7}, Rightarrow, nfold, from=6, to=7]
\end{tikzcd}		=
\begin{tikzcd}[ampersand replacement=\&]
	{\alpha_yFf} \&\& \\
	{Gf\alpha_x} \&\& {\alpha_yFf} \\
	{Gf'\alpha_x} \&\& {\alpha_yFf'} \\
	{Gf'\alpha_x} \&\& {\alpha_yFf''} \\
	\&\& {Gf'\alpha_x}
	\arrow["{\alpha^o_f}"', from=1-1, to=2-1]
	\arrow[""{name=0, anchor=center, inner sep=0}, curve={height=-6pt}, equals, nfold, from=1-1, to=2-3]
	\arrow[""{name=1, anchor=center, inner sep=0}, "{\alpha_f}"{description}, from=2-1, to=2-3]
	\arrow["{G\phi_1\circ1}"', from=2-1, to=3-1]
	\arrow["{1\circ F\phi_1}", from=2-3, to=3-3]
	\arrow[""{name=2, anchor=center, inner sep=0}, "{\alpha_{f'}}"{description}, from=3-1, to=3-3]
	\arrow["{G\phi_2\circ1}"', from=3-1, to=4-1]
	\arrow["{1\circ F\phi_2}", from=3-3, to=4-3]
	\arrow[""{name=3, anchor=center, inner sep=0}, from=4-1, to=4-3]
	\arrow[""{name=4, anchor=center, inner sep=0}, curve={height=6pt}, equals, nfold, from=4-1, to=5-3]
	\arrow["{\alpha^o_{f''}}", from=4-3, to=5-3]
	\arrow["{\eta_f}"', shift right=5, between={0.3}{0.7}, Rightarrow, nfold, from=0, to=1]
	\arrow["{\alpha^l_{\phi_1}}", between={0.4}{0.6}, Rightarrow, nfold, from=1, to=2]
	\arrow["{\alpha^l_{\phi_2}}", between={0.4}{0.6}, Rightarrow, nfold, from=2, to=3]
	\arrow["{\epsilon_{f'}}", shift left=5, between={0.3}{0.7}, Rightarrow, nfold, from=3, to=4]
\end{tikzcd}\]
\end{center}
	Then, let us consider compatibility with composition of three 1-cells $w\xrightarrow{f_1}x\xrightarrow{f_2}y\xrightarrow{f_3}z$, i.e. Axiom (\ref{ax:oplax-tr.comp-1-cell}) in \Cref{subsec:ax-oplax-transf}.\black 
	We start from the right hand side and move to the left hand side using the definition of $(\alpha^o)^2)_{-,-}$ and (in order): triangle identity for $\alpha^o_{f_2f_1}\dashv\alpha^l_{f_2f_1}$; \Cref{lemma:interchange-property} for $\eta_{f_3}\circ1$; the corresponding axiom (A.2(v) of \cite{BourkeLobbia}) for $\alpha$ as a lax transformation; \Cref{lemma:interchange-property} for $1\circ\eta_{f_1}$; triangle identity for $\alpha^o_{f_3f_2}\dashv\alpha^l_{f_3f_2}$.
	For the details see \Cref{app:big-diagram}.

	Finally, we need to prove the axioms regarding whiskering, i.e. Axioms (\ref{ax:oplax-tr.whisk-alpha2}).(L)/(R) in \Cref{subsec:ax-oplax-transf}. \black
	We will give the calculations for the right whiskering, the left case is completely analogous. 
	Consider a 2-cell $\phi\colon f_1\to f_2\in\ca(x,y)$ and a 1-cell $f_2\colon y\to z\in\ca$.
	\black 
	We start with the left hand side and obtain the following equality using (in order): the definitions of $({\alpha^o})_2^{f_2,f_1}$ and $({\alpha^o})_{f_2\phi}$; one triangle equality of $\alpha^o_{f_2f_1}\dashv\alpha^l_{f_2f_1}$.
	
	\begin{center}
		\scalebox{0.9}{
			\begin{tikzcd}[ampersand replacement=\&]
				{\alpha_Z Ff_2Ff_1} \&\&\&\&\&\&\& \\
				{Gf_2\alpha_Y Ff_1} \&\& {} \\
				{G{f_2}G{f_1}\alpha_X} \&\& {Gf_2\alpha_Y Ff_1} \&\& {\alpha_Z Ff_2Ff_1} \\
				{G({f_2}{f_1}) \alpha_X} \&\& {} \& {} \& {\alpha_ZF(f_2f_1)} \\
				\&\&\& {} \& {G({f_2}{f_1}) \alpha_X} \&\&\& {\alpha_ZF(f_2f_1)} \\
				\&\&\&\& {G({f_2}{f'_1}) \alpha_x} \&\&\& {\alpha_zF(f_2f'_1)} \\
				\&\&\&\&\&\&\& {G({f_2}{f'_1}) \alpha_x}
				\arrow["{{\alpha^o_{f_2}}\circ1}"', from=1-1, to=2-1]
				\arrow[equals, from=1-1, to=3-5]
				\arrow["{1\circ\alpha^o_{f_1}}"', from=2-1, to=3-1]
				\arrow[""{name=0, anchor=center, inner sep=0}, equals, from=2-1, to=3-3]
				\arrow["{\eta_{f_2}\circ1}"', between={0.3}{0.7}, Rightarrow, from=2-3, to=3-3]
				\arrow[""{name=1, anchor=center, inner sep=0}, "{1\circ\alpha^l_{f_1}}"{description}, from=3-1, to=3-3]
				\arrow["{G^2_{f_2,f_1}\circ1}"', from=3-1, to=4-1]
				\arrow["{\alpha^l_{f_2}\circ1}"{description}, from=3-3, to=3-5]
				\arrow["{({\alpha^l})^2_{f_2,f_1}}", between={0.3}{0.7}, Rightarrow, from=3-3, to=4-3]
				\arrow["{1\circ F^2_{f_2,f_1}}", from=3-5, to=4-5]
				\arrow["{\alpha^l_{f_2f_1}}"{description}, from=4-1, to=4-5]
				\arrow[equals, from=4-1, to=5-5]
				\arrow["{\epsilon_{f_2f_1}}"{pos=0.2}, between={0.2}{0.4}, Rightarrow, from=4-4, to=5-4]
				\arrow["{{\alpha^o_{f_2f_1}}}"{description}, from=4-5, to=5-5]
				\arrow[""{name=2, anchor=center, inner sep=0}, equals, from=4-5, to=5-8]
				\arrow[""{name=3, anchor=center, inner sep=0}, "{{\alpha^l_{f_2f_1}}}"{description, pos=0.6}, from=5-5, to=5-8]
				\arrow["{1\circ G(f_2\phi)}"', from=5-5, to=6-5]
				\arrow["{1\circ F(f_2\phi)}", from=5-8, to=6-8]
				\arrow[""{name=4, anchor=center, inner sep=0}, "{{\alpha^l_{f_2f'_1}}}"{description, pos=0.4}, from=6-5, to=6-8]
				\arrow[""{name=5, anchor=center, inner sep=0}, Rightarrow, from=6-5, to=7-8]
				\arrow["{{\alpha^o_{f_2f'_1}}}", from=6-8, to=7-8]
				\arrow["{1\circ\eta_{f_1}}"'{pos=0.3}, between={0.3}{0.7}, Rightarrow, from=0, to=1]
				\arrow["{\eta_{f_2f_1}}"'{pos=0.4}, shift right=5, between={0.3}{0.7}, Rightarrow, from=2, to=3]
				\arrow["{\alpha^l_\phi}", between={0.4}{0.6}, Rightarrow, from=3, to=4]
				\arrow["{\epsilon_{f_2f'_1}}"{pos=0.6}, shift left=5, between={0.3}{0.7}, Rightarrow, from=4, to=5]
			\end{tikzcd}
		} \\
		$=$ 
		\scalebox{1}{
			\begin{tikzcd}[ampersand replacement=\&]
				{\alpha_Z Ff_2Ff_1} \&\&\&\& \\
				{Gf_2\alpha_Y Ff_1} \&\& {} \\
				{G{f_2}G{f_1}\alpha_X} \&\& {Gf_2\alpha_Y Ff_1} \&\& {\alpha_Z Ff_2Ff_1} \\
				{G({f_2}{f_1}) \alpha_X} \&\& {} \&\& {\alpha_ZF(f_2f_1)} \\
				{G({f_2}{f'_1}) \alpha_x} \&\&\&\& {\alpha_zF(f_2f'_1)} \\
				\&\&\&\& {G({f_2}{f'_1}) \alpha_x}
				\arrow["{{\alpha^o_{f_2}}\circ1}"', from=1-1, to=2-1]
				\arrow[equals, from=1-1, to=3-5]
				\arrow["{1\circ\alpha^o_{f_1}}"', from=2-1, to=3-1]
				\arrow[""{name=0, anchor=center, inner sep=0}, equals, from=2-1, to=3-3]
				\arrow["{\eta_{f_2}\circ1}"', between={0.3}{0.7}, Rightarrow, from=2-3, to=3-3]
				\arrow[""{name=1, anchor=center, inner sep=0}, "{1\circ\alpha^l_{f_1}}"{description}, from=3-1, to=3-3]
				\arrow["{G^2_{f_2,f_1}\circ1}"', from=3-1, to=4-1]
				\arrow["{\alpha^l_{f_2}\circ1}"{description}, from=3-3, to=3-5]
				\arrow["{({\alpha^l})^2_{f_2,f_1}}", between={0.3}{0.7}, Rightarrow, from=3-3, to=4-3]
				\arrow["{1\circ F^2_{f_2,f_1}}", from=3-5, to=4-5]
				\arrow[""{name=2, anchor=center, inner sep=0}, "{\alpha^l_{f_2f_1}}"{description}, from=4-1, to=4-5]
				\arrow["{G(f_2\phi)\circ1}"', from=4-1, to=5-1]
				\arrow["{1\circ F(f_2\phi)}", from=4-5, to=5-5]
				\arrow[""{name=3, anchor=center, inner sep=0}, "{{\alpha^l_{f_2f'_1}}}"{description, pos=0.4}, from=5-1, to=5-5]
				\arrow[""{name=4, anchor=center, inner sep=0}, Rightarrow, from=5-1, to=6-5]
				\arrow["{{\alpha^o_{f_2f'_1}}}", from=5-5, to=6-5]
				\arrow["{1\circ\eta_{f_1}}"'{pos=0.3}, between={0.3}{0.7}, Rightarrow, from=0, to=1]
				\arrow["{\alpha^l_\phi}", between={0.4}{0.6}, Rightarrow, from=2, to=3]
				\arrow["{\epsilon_{f_2f'_1}}"{pos=0.6}, shift left=5, between={0.3}{0.7}, Rightarrow, from=3, to=4]
			\end{tikzcd}
		} 
		$=:(A)$
	\end{center}
	
	On the other hand, let us start with the left hand side using (in order): the definition of ${(\alpha^o)}_2^{f_2,f'_1}$ and $\alpha^o_{\phi}$; 
	one triangle identity of $\alpha^o_{f_1}\dashv\alpha^l_{f_1}$; \Cref{lemma:interchange-property} for $\eta_{f_2}$. 
	
	\begin{center}
		\begin{tikzcd}[ampersand replacement=\&,column sep=scriptsize,row sep=small]
			{\alpha_Z Ff_2Ff_1} \&\&\&\&\&\& \\
			{Gf_2\alpha_Y Ff_1} \& {} \&\&\& {\alpha_Z Ff_2Ff'_1} \\
			{G{f_2}G{f_1}\alpha_X} \&\& {Gf_2\alpha_Y Ff_1} \\
			{G{f_2}G{f'_1}\alpha_X} \&\& {Gf_2\alpha_Y Ff'_1} \& {} \\
			\& {} \& {G{f_2}G{f'_1}\alpha_X} \&\& {Gf_2\alpha_Y Ff'_1} \&\& {\alpha_Z Ff_2Ff'_1} \\
			\&\& {G({f_2}{f'_1}) \alpha_x} \&\&\& {} \& {\alpha_zF(f_2f'_1)} \\
			\&\&\&\&\& {} \& {G({f_2}{f'_1}) \alpha_x}
			\arrow["{{\alpha^o_{f_2}}\circ1}"', from=1-1, to=2-1]
			\arrow[""{name=0, anchor=center, inner sep=0}, "{1\circ F\phi}", curve={height=-12pt}, from=1-1, to=2-5]
			\arrow["{1\circ\alpha^o_{f_1}}"', from=2-1, to=3-1]
			\arrow[curve={height=-12pt}, equals, from=2-1, to=3-3]
			\arrow["{{\alpha^o_{f_2}}\circ1}"{description}, from=2-5, to=4-3]
			\arrow["{\eta_{f_2}\circ1}", between={0.4}{0.6}, Rightarrow, from=2-5, to=5-5]
			\arrow[curve={height=-12pt}, equals, from=2-5, to=5-7]
			\arrow[""{name=1, anchor=center, inner sep=0}, "{1\circ\alpha^l_{f_1}}"{description}, from=3-1, to=3-3]
			\arrow["{1\circ G\phi\circ1}"', from=3-1, to=4-1]
			\arrow["{1\circ F\phi}"', from=3-3, to=4-3]
			\arrow[""{name=2, anchor=center, inner sep=0}, "{1\circ\alpha^l_{f'_1}}"{description}, from=4-1, to=4-3]
			\arrow[curve={height=12pt}, equals, from=4-1, to=5-3]
			\arrow[from=4-3, to=5-3]
			\arrow[curve={height=-12pt}, equals, from=4-3, to=5-5]
			\arrow[""{name=3, anchor=center, inner sep=0}, "{1\circ\alpha^l_{f'_1}}"{description}, from=5-3, to=5-5]
			\arrow["{G^2_{f_2,f'_1}\circ1}"', from=5-3, to=6-3]
			\arrow["{\alpha^l_{f_2}\circ1}", from=5-5, to=5-7]
			\arrow["{1\circ F^2_{f_2,f'_1}}", from=5-7, to=6-7]
			\arrow[""{name=4, anchor=center, inner sep=0}, "{{\alpha^l_{f_2f'_1}}}"{description, pos=0.4}, from=6-3, to=6-7]
			\arrow[curve={height=12pt}, Rightarrow, from=6-3, to=7-7]
			\arrow["{\epsilon_{f_2f'_1}}"'{pos=0.4}, shift right=5, between={0.2}{0.8}, Rightarrow, from=6-6, to=7-6]
			\arrow["{{\alpha^o_{f_2f'_1}}}", from=6-7, to=7-7]
			\arrow[between={0.3}{0.7}, Rightarrow, dashed, from=0, to=3-3]
			\arrow["{1\circ\eta_{f_1}}"'{pos=0.3}, between={0.2}{0.6}, Rightarrow, from=2-2, to=1]
			\arrow["{\alpha^l_\phi}"', between={0.4}{0.6}, Rightarrow, from=1, to=2]
			\arrow["{1\circ\epsilon_{f_1'}}"{pos=0.6}, between={0.4}{0.8}, Rightarrow, from=2, to=5-2]
			\arrow["{1\circ\eta_{f_1'}}"'{pos=0.3}, between={0.1}{0.6}, Rightarrow, from=4-4, to=3]
			\arrow["{{\alpha^l}^2_{f_2,f'_1}}"{pos=0.4}, between={0.2}{0.7}, Rightarrow, from=5-5, to=4]
		\end{tikzcd}\\
		$=$
		\begin{tikzcd}[ampersand replacement=\&]
			{\alpha_Z Ff_2Ff_1} \&\&\&\& \\
			{Gf_2\alpha_Y Ff_1} \& {} \&\& {\alpha_Z Ff_2Ff'_1} \\
			{G{f_2}G{f_1}\alpha_X} \&\& {Gf_2\alpha_Y Ff_1} \\
			{G{f_2}G{f'_1}\alpha_X} \&\& {Gf_2\alpha_Y Ff'_1} \& {} \& {\alpha_Z Ff_2Ff'_1} \\
			{G({f_2}{f'_1}) \alpha_x} \&\& {} \& {} \& {\alpha_zF(f_2f'_1)} \\
			\&\&\& {} \& {G({f_2}{f'_1}) \alpha_x}
			\arrow["{{\alpha^o_{f_2}}\circ1}"', from=1-1, to=2-1]
			\arrow[""{name=0, anchor=center, inner sep=0}, "{1\circ F\phi}", curve={height=-12pt}, from=1-1, to=2-4]
			\arrow["{1\circ\alpha^o_{f_1}}"', from=2-1, to=3-1]
			\arrow[curve={height=-12pt}, equals, from=2-1, to=3-3]
			\arrow["{{\alpha^o_{f_2}}\circ1}"{description, pos=0.3}, from=2-4, to=4-3]
			\arrow["{\eta_{f_2}\circ1}", between={0.4}{0.6}, Rightarrow, from=2-4, to=4-4]
			\arrow[curve={height=-12pt}, equals, from=2-4, to=4-5]
			\arrow[""{name=1, anchor=center, inner sep=0}, "{1\circ\alpha^l_{f_1}}"{description}, from=3-1, to=3-3]
			\arrow["{1\circ G\phi\circ1}"', from=3-1, to=4-1]
			\arrow["{1\circ F\phi}"', from=3-3, to=4-3]
			\arrow[""{name=2, anchor=center, inner sep=0}, "{1\circ\alpha^l_{f'_1}}"{description}, from=4-1, to=4-3]
			\arrow["{G^2_{f_2,f'_1}\circ1}"', from=4-1, to=5-1]
			\arrow["{\alpha^l_{f_2}\circ1}", from=4-3, to=4-5]
			\arrow["{{\alpha^l}^2_{f_2,f'_1}}"{pos=0.4}, between={0.2}{0.7}, Rightarrow, from=4-3, to=5-3]
			\arrow["{1\circ F^2_{f_2,f'_1}}", from=4-5, to=5-5]
			\arrow["{{\alpha^l_{f_2f'_1}}}"{description, pos=0.4}, from=5-1, to=5-5]
			\arrow[curve={height=12pt}, Rightarrow, from=5-1, to=6-5]
			\arrow["{\epsilon_{f_2f'_1}}"'{pos=0.4}, shift right=5, between={0.2}{0.8}, Rightarrow, from=5-4, to=6-4]
			\arrow["{{\alpha^o_{f_2f'_1}}}", from=5-5, to=6-5]
			\arrow[between={0.3}{0.7}, Rightarrow, dashed, from=0, to=3-3]
			\arrow["{1\circ\eta_{f_1}}"'{pos=0.3}, between={0.2}{0.6}, Rightarrow, from=2-2, to=1]
			\arrow["{\alpha^l_\phi}"', between={0.4}{0.6}, Rightarrow, from=1, to=2]
		\end{tikzcd}\\
		$=$
		\scalebox{1}{
			\begin{tikzcd}[ampersand replacement=\&]
				{\alpha_Z Ff_2Ff_1} \&\&\&\& \\
				{Gf_2\alpha_Y Ff_1} \&\& {} \\
				{G{f_2}G{f_1}\alpha_X} \&\& {Gf_2\alpha_Y Ff_1} \&\& {\alpha_Z Ff_2Ff_1} \\
				{G{f_2}G{f'_1}\alpha_X} \&\& {Gf_2\alpha_Y Ff'_1} \&\& {\alpha_Z Ff_2Ff'_1} \\
				{G({f_2}{f'_1}) \alpha_x} \&\&\&\& {\alpha_zF(f_2f'_1)} \\
				\&\&\&\& {G({f_2}{f'_1}) \alpha_x}
				\arrow["{{\alpha^o_{f_2}}\circ1}"', from=1-1, to=2-1]
				\arrow[equals, from=1-1, to=3-5]
				\arrow["{1\circ\alpha^o_{f_1}}"', from=2-1, to=3-1]
				\arrow[""{name=0, anchor=center, inner sep=0}, equals, from=2-1, to=3-3]
				\arrow["{\eta_{f_2}\circ1}"', between={0.3}{0.7}, Rightarrow, from=2-3, to=3-3]
				\arrow[""{name=1, anchor=center, inner sep=0}, "{1\circ\alpha^l_{f_1}}"{description}, from=3-1, to=3-3]
				\arrow["{1\circ G\phi\circ1}"', from=3-1, to=4-1]
				\arrow[""{name=2, anchor=center, inner sep=0}, "{\alpha^l_{f_2}\circ1}"{description}, from=3-3, to=3-5]
				\arrow["{1\circ F\phi}"', from=3-3, to=4-3]
				\arrow["{1\circ F\phi}", from=3-5, to=4-5]
				\arrow[""{name=3, anchor=center, inner sep=0}, "{1\circ\alpha^l_{f'_1}}"', from=4-1, to=4-3]
				\arrow["{G^2_{f_2,f'_1}\circ1}"', from=4-1, to=5-1]
				\arrow[""{name=4, anchor=center, inner sep=0}, "{\alpha^l_{f_2}\circ1}"', from=4-3, to=4-5]
				\arrow["{1\circ F^2_{f_2,f'_1}}", from=4-5, to=5-5]
				\arrow[""{name=5, anchor=center, inner sep=0}, "{{\alpha^l_{f_2f'_1}}}"{description, pos=0.4}, from=5-1, to=5-5]
				\arrow[""{name=6, anchor=center, inner sep=0}, Rightarrow, from=5-1, to=6-5]
				\arrow["{{\alpha^o_{f_2f'_1}}}", from=5-5, to=6-5]
				\arrow["{1\circ\eta_{f_1}}"'{pos=0.3}, between={0.3}{0.7}, Rightarrow, from=0, to=1]
				\arrow["{\alpha^l_\phi}"', between={0.4}{0.6}, Rightarrow, from=1, to=3]
				\arrow[between={0.3}{0.7}, Rightarrow, dashed, from=2, to=4]
				\arrow["{{\alpha^l}^2_{f_2,f'_1}}"{pos=0.4}, between={0.3}{0.6}, Rightarrow, from=4-3, to=5]
				\arrow["{\epsilon_{f_2f'_1}}"{pos=0.6}, shift left=5, between={0.3}{0.7}, Rightarrow, from=5, to=6]
			\end{tikzcd}
		}
	\end{center}
	We conclude noticing that this diagram is equal to $(A)$ above by the corresponding axiom of $\alpha$ as a lax transformation (Axiom A.2 (vii).(L) of \cite{BourkeLobbia}). \\
	
	
The construction of the inverse $v_2\colon(2)\to(3)$ to $u_2$ is dual and follows formally using the isomorphism $\lax(\ca^\co,\cb^\co)^\co\cong\oplax(\ca,\cb)$.\black

\end{proof}

\begin{rmk}
	\label{rmk:inv-transf}
	Following the above notation\black, it is straightforward \black to check that $\alpha^l_\phi$ is invertible if and only if $\alpha^o_\phi$ is (since in an adjoint equivalence the unit and counit are invertible).
\end{rmk}

Thanks to \Cref{prop:equ-psd-transf} and \Cref{rmk:inv-transf}, we can give an equivalent notion of pseudo-transformation as in the following definition. \black

\begin{defn}\label{def:unb-psd-transf}
	Let $F,G\colon\ca\rightrightarrows \cb$ be  pseudomaps. 
	An \emph{unbiased pseudo-transformation} $\alpha=(\alpha^l,\alpha^o,\aexi^\alpha_f)\colon F\to G$ consists of:
	\begin{itemize}
		\item for any object $x\in\ca$, a 1-cell $\alpha_x\colon Fx\to Gx$;
		\item for any 1-cell $f\colon x\to y\in\ca$, an adjoint equivalence $\aexi^\alpha_f=(\alpha^o_f\dashv\alpha^l_f\colon Gf\alpha_x\to\alpha_yFf)$ such that $\aexi^\alpha_{1_x}=1_\aeq$;
		\item extensions of $(\alpha_x,\alpha_f^l)$ to a lax transformation and of $(\alpha_x,\alpha_f^o)$ to an oplax transformation such that the components $\alpha^o_\phi$ and ${(\alpha^o)}_2^{f_2,f_1}$ are mates of $\alpha^l_\phi$ and ${(\alpha^l)}^2_{f_2,f_1}$ (see point (3) of \Cref{prop:equ-psd-transf} for more details).
		Moreover, we require $\alpha^l_\phi$ (equivalently its mate $\alpha^o_\phi$) to be invertible.\footnote{That one is invertible just when the other is is because they are mates through adjoint equivalences, as remarked in \Cref{rmk:inv-transf}.}
		\end{itemize} 
\end{defn}


\begin{notation*}
	Given an unbiased pseudo-transformation $\alpha$, we write $\alpha^l$ and $\alpha^o$ for the corresponding lax and oplax transformations.  
\end{notation*}

We now turn to 2-cells in $\psd(\ca,\cb)$, i.e. pseudo-modifications.

\begin{prop}
	\label{prop:equ-psd-mod}
	Let $\alpha,\beta \colon F \rightrightarrows G \in \psd(\ca,\cb)$ be unbiased pseudo-transformations.
	Given a family of 2-cells $\Lambda_x\colon\alpha_x\to\beta_x$, there is a bijective correspondence between:
	\begin{enumerate}[(1)]
		\item For any 1-cell $f\colon x\to y$ in $\ca$, an invertible 3-cell
		\[\begin{tikzcd}[ampersand replacement=\&]
			{Gf\alpha_x} \& {Gf\beta_x} \\
			{\alpha_yFf} \& {\beta_yFf}
			\arrow[""{name=0, anchor=center, inner sep=0}, "{1\circ\Lambda_x}", from=1-1, to=1-2]
			\arrow["{\beta^l_f}", from=1-2, to=2-2]
			\arrow["{\alpha^l_f}"', from=1-1, to=2-1]
			\arrow[""{name=1, anchor=center, inner sep=0}, "{\Lambda_y\circ1}"', from=2-1, to=2-2]
			\arrow["{\Lambda_f^l}", shorten <=9pt, shorten >=9pt, Rightarrow, from=0, to=1]
		\end{tikzcd}\]
		making $(\Lambda_x,\Lambda_f^l)$ a modification $\alpha^l\to\beta^l$ in $\lax(\ca,\cb)$. 
		\item For any 1-cell $f\colon x\to y$ in $\ca$, an invertible 3-cell
		\[\begin{tikzcd}[ampersand replacement=\&]
			{\alpha_yFf} \& {Gf\alpha_x} \\
			{\beta_yFf} \& {Gf\beta_x}
			\arrow["{1\circ\Lambda_x}", from=1-2, to=2-2]
			\arrow[""{name=0, anchor=center, inner sep=0}, "{\beta^o_f}"', from=2-1, to=2-2]
			\arrow[""{name=1, anchor=center, inner sep=0}, "{\alpha^o_f}", from=1-1, to=1-2]
			\arrow["{\Lambda_y\circ1}"', from=1-1, to=2-1]
			\arrow["{\Lambda_f^o}", shorten <=9pt, shorten >=9pt, Rightarrow, from=1, to=0]
		\end{tikzcd}\]
		making $(\Lambda_x,\Lambda_f^o)$ a modification $\alpha^o\to\beta^o$ in $\oplax(\ca,\cb)$. 
		\item For any 1-cell $f\colon x\to y$ in $\ca$,  invertible 3-cells $\Lambda_f^l$ and $\Lambda^o_f$ as above making $(\Lambda_x,\Lambda_f^l)$ and $(\Lambda_x,\Lambda_f^o)$ modifications in $\lax(\ca,\cb)$ and $\oplax(\ca,\cb)$ respectively, such that $\Lambda_f^o$ and $\Lambda_f^l$ are mates with respect to $(({\aexi^\alpha_f})^{\rev}(,{\aexi^\beta_{f}})^{\rev})$, i.e.
		\begin{center}
			$\Lambda_f^o=$
\begin{tikzcd}[ampersand replacement=\&]
	{\alpha_yFf} \&\& {Gf\alpha_x} \&\& {Gf\beta_x} \&\& \\
	\&\& {\alpha_yFf} \&\& {\beta_yFf} \&\& {Gf\beta_x}
	\arrow[""{name=0, anchor=center, inner sep=0}, "{\alpha^o_f}", from=1-1, to=1-3]
	\arrow[""{name=1, anchor=center, inner sep=0}, curve={height=12pt}, equals, nfold, from=1-1, to=2-3]
	\arrow[""{name=2, anchor=center, inner sep=0}, "{1\circ \Lambda_x}", from=1-3, to=1-5]
	\arrow["{\alpha^l_f}"{description}, from=1-3, to=2-3]
	\arrow["{\beta^l_f}"{description}, from=1-5, to=2-5]
	\arrow[""{name=3, anchor=center, inner sep=0}, curve={height=-12pt}, equals, nfold, from=1-5, to=2-7]
	\arrow[""{name=4, anchor=center, inner sep=0}, "{\Lambda_y\circ1}"', from=2-3, to=2-5]
	\arrow[""{name=5, anchor=center, inner sep=0}, "{\beta^o_f}"', from=2-5, to=2-7]
	\arrow["{({\eta_f^\alpha})^{-1}}", between={0.3}{0.7}, Rightarrow, nfold, from=0, to=1]
	\arrow["{\Lambda_f^l}", between={0.4}{0.6}, Rightarrow, nfold, from=2, to=4]
	\arrow["{({\epsilon_f^{\beta}})^{-1}}"', between={0.3}{0.7}, Rightarrow, nfold, from=3, to=5]
\end{tikzcd}
		\end{center}
	\end{enumerate}
\end{prop}

\begin{proof}
As in the proof of \Cref{prop:equ-psd-transf}, the maps $u_1\colon(3)\rightarrow(1)$ and $u_2\colon(3)\rightarrow(2)$ forget, respectively, the oplax and lax part in $(3)$.  Proving that $u_1$ and $u_2$ are bijections is identical in form to the proof of \Cref{prop:equ-psd-transf}.  In particular, the inverse $v_1$ to $u_1$ is defined as $v_1(\Lambda):=(\Lambda,\,(\Lambda_x,\Lambda^o_f)\,)$ where $\Lambda^o_f$ is now the mate described in $(3)$ and only non-trivial thing to check is that the data $(\Lambda_x,\Lambda^o_f)\,)$ defines an oplax modification.  This is what we now prove.\black

%
%

Firstly, the degeneracy axiom for $\Lambda^o_f$ (i.e. Axiom (\ref{ax:op-mod.deg-3}) in \Cref{subsec:ax-oplax-mod})
\black
 follows from the degeneracy axioms for $\aexi^\alpha_f,\,\aexi^\beta_f$, and $\Lambda_f$.  
The non-trivial axioms are those concerning compatibility with composition of 1-cells and compatibility with 2-cells. \black
	For both of these, we will use \Cref{rmk:mates-over-adj-equiv-calculations} applied to $\Lambda^o_f$, noticing that $\Lambda^o_f$ is the inverse of $\Lambda_f\flat$. 

	Given composable 1-cells $x\xrightarrow{f_1}y\xrightarrow{f_2}z$ in $\ca$, we must establish Axiom (\ref{ax:op-mod.comp-1}) in \Cref{subsec:ax-oplax-mod}. \black
	Firstly, we notice that axiom is equivalent to the equation below (by pasting both sides with $({\Lambda^o_{f_2}})^{-1}\circ1$, $1\circ ({\Lambda^o_{f_1}})^{-1}$, $({\Lambda^o_{f_2f_1}})^{-1}$, and the appropriate interchange isomorphisms). 
	
	\begin{center}
		\scalebox{0.75}{
			\begin{tikzcd}[ampersand replacement=\&,column sep=small,row sep=scriptsize]
				{\alpha_z Ff_2Ff_1} \&\& {\beta_z Ff_2Ff_1} \&\& \\
				\&\& {\alpha_zF(f_2f_1)} \&\& {\beta_zF(f_2f_1)} \\
				{Gf_2\alpha_y Ff_1} \\
				\&\&\&\& {} \\
				{G{f_2}G{f_1}\alpha_x} \\
				\&\& {G({f_2}{f_1}) \alpha_x} \&\& {G({f_2}{f_1}) \beta_x}
				\arrow[""{name=0, anchor=center, inner sep=0}, "{\Lambda_z\circ1}", from=1-1, to=1-3]
				\arrow[""{name=1, anchor=center, inner sep=0}, "{1\circ F^2_{f_2,f_1}}"'{pos=0.6}, from=1-1, to=2-3]
				\arrow["{{\alpha^o_{f_2}}\circ1}"', from=1-1, to=3-1]
				\arrow["{1\circ F^2_{f_2,f_1}}", from=1-3, to=2-5]
				\arrow[""{name=2, anchor=center, inner sep=0}, "{\Lambda_z\circ1}"', from=2-3, to=2-5]
				\arrow["{{\alpha^o_{f_2f_1}}}", from=2-3, to=6-3]
				\arrow["{{\beta^o_{f_2f_1}}}", from=2-5, to=6-5]
				\arrow["{1\circ\alpha^o_{f_1}}"', from=3-1, to=5-1]
				\arrow[""{name=3, anchor=center, inner sep=0}, "{G^2_{f_2,f_1}\circ1}"'{pos=0.6}, from=5-1, to=6-3]
				\arrow[""{name=4, anchor=center, inner sep=0}, "{1\circ\Lambda_x}"', from=6-3, to=6-5]
				\arrow[between={0.3}{0.7}, Rightarrow, dashed, from=0, to=2]
				\arrow["{({\alpha^o})_2^{f_2,f_1}}", between={0.4}{0.6}, Rightarrow, from=1, to=3]
				\arrow["{({\Lambda^o_{f_2f_1}})^{-1}}", between={0.4}{0.6}, Rightarrow, from=2, to=4]
			\end{tikzcd}
		}
		$=$
		\scalebox{0.75}{
			\begin{tikzcd}[ampersand replacement=\&,column sep=small,row sep=scriptsize]
				{\alpha_z Ff_2Ff_1} \&\& {\beta_z Ff_2Ff_1} \&\&\& \\
				\&\&\&\& {\beta_zF(f_2f_1)}  \\
				{Gf_2\alpha_y Ff_1} \&\& {Gf_2\beta_y Ff_1} \\
				\&\&\&\& {} \\
				{G{f_2}G{f_1}\alpha_x} \&\& {G{f_2}G{f_1}\beta_x} \\
				\&\& {G({f_2}{f_1}) \alpha_x} \&\& {G({f_2}{f_1}) \beta_x}
				\arrow[""{name=0, anchor=center, inner sep=0}, "{\Lambda_z\circ1}", from=1-1, to=1-3]
				\arrow["{\Lambda_z\circ1}"', from=1-1, to=3-1]
				\arrow[""{name=1, anchor=center, inner sep=0}, "{1\circ F^2_{f_2,f_1}}"{pos=0.4}, from=1-3, to=2-5]
				\arrow["{{\beta^o_{f_2}}\circ1}"{description}, from=1-3, to=3-3]
				\arrow["{{\beta^o_{f_2f_1}}}", from=2-5, to=6-5]
				\arrow[""{name=2, anchor=center, inner sep=0}, "{1\circ\Lambda_y\circ1}", from=3-1, to=3-3]
				\arrow["{1\circ\alpha^o_{f_1}}"', from=3-1, to=5-1]
				\arrow["{1\circ\beta^o_{f_1}}"{description}, from=3-3, to=5-3]
				\arrow[""{name=3, anchor=center, inner sep=0}, "{1\circ\Lambda_x}", from=5-1, to=5-3]
				\arrow["{G^2_{f_2,f_1}\circ1}"'{pos=0.4}, from=5-1, to=6-3]
				\arrow[""{name=4, anchor=center, inner sep=0}, "{G^2_{f_2,f_1}\circ1}"{pos=0.4}, from=5-3, to=6-5]
				\arrow[""{name=5, anchor=center, inner sep=0}, "{1\circ\Lambda_x}"', from=6-3, to=6-5]
				\arrow["{({\Lambda^o_{f_2})^{-1} \circ1}}"', between={0.4}{0.6}, Rightarrow, from=0, to=2]
				\arrow["{({\beta^o})_2^{f_2,f_1}}", between={0.4}{0.6}, Rightarrow, from=1, to=4]
				\arrow["{{1\circ (\Lambda^o_{f_1})^{-1}}}"', between={0.4}{0.6}, Rightarrow, from=2, to=3]
				\arrow[between={0.3}{0.7}, Rightarrow, dashed, from=3, to=5]
			\end{tikzcd}
		}
	\end{center}
To verify this equation, we move from the left hand side and to the right hand side using (in order): the definition of $({\alpha^o})^{f_2,f_1}_2$; \Cref{rmk:mates-over-adj-equiv-calculations} applied to $1\circ\eta^\alpha_{f_1}$ and $\eta^\alpha_{f_2}\circ1$; the corresponding axiom of $\Lambda$ as a lax modification; \Cref{rmk:mates-over-adj-equiv-calculations} applied to $\epsilon^\beta_{f_2f_1}$; and finally the definition of ${(\beta^o)}^{f_2,f_1}_2$.
	
	\begin{center}
		\begin{tikzcd}[ampersand replacement=\&,column sep=small,row sep=scriptsize]
			{\alpha_z Ff_2Ff_1} \&\&\&\&\&\& \\
			\\
			{Gf_2\alpha_y Ff_1} \&\& {} \\
			\& {} \\
			{G{f_2}G{f_1}\alpha_x} \&\& {Gf_2\alpha_y Ff_1} \&\& {\alpha_z Ff_2Ff_1} \&\& {\beta_z Ff_2Ff_1} \\
			{G({f_2}{f_1}) \alpha_x} \&\& {} \&\& {\alpha_zF(f_2f_1)} \&\& {\beta_zF(f_2f_1)} \\
			\\
			\&\& {} \&\& {G({f_2}{f_1}) \alpha_x} \&\& {G({f_2}{f_1}) \beta_x}
			\arrow["{{\alpha^o_{f_2}}\circ1}"', from=1-1, to=3-1]
			\arrow[curve={height=-18pt}, equals, from=1-1, to=5-5]
			\arrow["{1\circ\alpha^o_{f_1}}"', from=3-1, to=5-1]
			\arrow[curve={height=-12pt}, equals, from=3-1, to=5-3]
			\arrow["{\eta^\alpha_{f_2}\circ1}"{pos=0.3}, between={0.2}{0.5}, Rightarrow, from=3-3, to=5-3]
			\arrow[""{name=0, anchor=center, inner sep=0}, "{1\circ\alpha^l_{f_1}}"', from=5-1, to=5-3]
			\arrow["{G^2_{f_2,f_1}\circ1}"', from=5-1, to=6-1]
			\arrow["{{\alpha^l_{f_2}}\circ1}", from=5-3, to=5-5]
			\arrow["{({\alpha^l)}^2_{f_2,f_1}}", shift left=3, between={0.2}{0.6}, Rightarrow, from=5-3, to=6-3]
			\arrow[""{name=1, anchor=center, inner sep=0}, "{\Lambda_z\circ1}", from=5-5, to=5-7]
			\arrow["{1\circ F^2_{f_2,f_1}}"'{pos=0.6}, from=5-5, to=6-5]
			\arrow["{1\circ F^2_{f_2,f_1}}", from=5-7, to=6-7]
			\arrow["{{\alpha^l_{f_2f_1}}}"', from=6-1, to=6-5]
			\arrow[curve={height=12pt}, equals, from=6-1, to=8-5]
			\arrow["{\epsilon^\alpha_{f_2f_1}}", between={0.3}{0.6}, Rightarrow, from=6-3, to=8-3]
			\arrow[""{name=2, anchor=center, inner sep=0}, "{\Lambda_z\circ1}"', from=6-5, to=6-7]
			\arrow["{{\alpha^o_{f_2f_1}}}", from=6-5, to=8-5]
			\arrow["{{\beta^o_{f_2f_1}}}", from=6-7, to=8-7]
			\arrow[""{name=3, anchor=center, inner sep=0}, "{1\circ\Lambda_x}"', from=8-5, to=8-7]
			\arrow["{1\circ\eta^\alpha_{f_1}}"{pos=0.3}, shift right=2, between={0}{0.5}, Rightarrow, from=4-2, to=0]
			\arrow[between={0.3}{0.7}, Rightarrow, dashed, from=1, to=2]
			\arrow["{({\Lambda^o_{f_2f_1})}^{-1}}", between={0.4}{0.6}, Rightarrow, from=2, to=3]
		\end{tikzcd}
		$=$\\
		\adjustbox{width=\textwidth}{
\begin{tikzcd}[ampersand replacement=\&]
	{\alpha_z Ff_2Ff_1} \&\& {\beta_z Ff_2Ff_1} \&\&\&\& \\
	\\
	{Gf_2\alpha_y Ff_1} \&\& {Gf_2\beta_y Ff_1} \\
	\&\&\& {} \& {} \\
	{G{f_2}G{f_1}\alpha_x} \&\& {G{f_2}G{f_1}\beta_x} \& {} \& {} \\
	\&\& {Gf_2\alpha_y Ff_1} \&\& {Gf_2\beta_y Ff_1} \\
	{G({f_2}{f_1}) \alpha_x} \&\&\&\& {\alpha_z Ff_2Ff_1} \&\& {\beta_z Ff_2Ff_1} \\
	\&\& \textcolor{white}{{Gf_2\alpha_y Ff_1}} \& {} \\
	\&\&\&\& {\alpha_zF(f_2f_1)} \&\& {\beta_zF(f_2f_1)} \\
	\&\&\& {} \\
	\&\&\&\& {G({f_2}{f_1}) \alpha_x} \&\& {G({f_2}{f_1}) \beta_x}
	\arrow[""{name=0, anchor=center, inner sep=0}, "{\Lambda_z\circ1}", from=1-1, to=1-3]
	\arrow["{{\alpha^o_{f_2}}\circ1}"', from=1-1, to=3-1]
	\arrow["{{\beta^o_{f_2}}\circ1}"{description}, from=1-3, to=3-3]
	\arrow[curve={height=-18pt}, equals, nfold, from=1-3, to=7-7]
	\arrow[""{name=1, anchor=center, inner sep=0}, "{1\circ\Lambda_y\circ1}", from=3-1, to=3-3]
	\arrow["{1\circ\alpha^o_{f_1}}"', from=3-1, to=5-1]
	\arrow["{1\circ\beta^o_{f_1}}"{description}, from=3-3, to=5-3]
	\arrow[curve={height=-12pt}, equals, nfold, from=3-3, to=6-5]
	\arrow["{1\circ\eta^\beta_{f_1}}"', shift right=2, between={0.3}{1}, Rightarrow, nfold, from=4-4, to=5-4]
	\arrow["{\eta^\beta_{f_2}\circ1}"', Rightarrow, nfold, from=4-5, to=5-5]
	\arrow[""{name=2, anchor=center, inner sep=0}, "{1\circ\Lambda_x}", from=5-1, to=5-3]
	\arrow["{1\circ\alpha^l_{f_1}}"', from=5-1, to=6-3]
	\arrow["{G^2_{f_2,f_1}\circ1}"', from=5-1, to=7-1]
	\arrow["{1\circ\beta^l_{f_1}}", from=5-3, to=6-5]
	\arrow[""{name=3, anchor=center, inner sep=0}, "{1\circ\Lambda_y\circ1}"{description}, from=6-3, to=6-5]
	\arrow["{{\alpha^l_{f_2}}\circ1}"{description}, from=6-3, to=7-5]
	\arrow["{{(\alpha^l)}^2_{f_2,f_1}}"{pos=0.4}, shift left=3, between={0.2}{0.6}, Rightarrow, nfold, from=6-3, to=8-3]
	\arrow["{{\beta^l_{f_2}}\circ1}", from=6-5, to=7-7]
	\arrow["{{\alpha^l_{f_2f_1}}}"{description}, from=7-1, to=9-5]
	\arrow[curve={height=12pt}, equals, nfold, from=7-1, to=11-5]
	\arrow[""{name=4, anchor=center, inner sep=0}, "{\Lambda_z\circ1}"', from=7-5, to=7-7]
	\arrow["{1\circ F^2_{f_2,f_1}}"', from=7-5, to=9-5]
	\arrow["{1\circ F^2_{f_2,f_1}}", from=7-7, to=9-7]
	\arrow["{\epsilon^\alpha_{f_2f_1}}"'{pos=0.6}, between={0.4}{0.8}, Rightarrow, nfold, from=8-4, to=10-4]
	\arrow[""{name=5, anchor=center, inner sep=0}, "{\Lambda_z\circ1}"', from=9-5, to=9-7]
	\arrow["{{\alpha^o_{f_2f_1}}}", from=9-5, to=11-5]
	\arrow["{{\beta^o_{f_2f_1}}}", from=9-7, to=11-7]
	\arrow[""{name=6, anchor=center, inner sep=0}, "{1\circ\Lambda_x}"', from=11-5, to=11-7]
	\arrow["{{(\Lambda^o_{f_2})}^{-1}\circ1}"', shift left=3, between={0.4}{0.6}, Rightarrow, nfold, from=0, to=1]
	\arrow["{1\circ{(\Lambda^o_{f_1})}^{-1}}"', shift left=2, between={0.4}{0.6}, Rightarrow, nfold, from=1, to=2]
	\arrow["{1\circ\Lambda_{f_1}}", shift right=2, between={0.4}{0.6}, Rightarrow, nfold, from=2, to=3]
	\arrow["{\Lambda_{f_1}\circ1}", shift right=2, between={0.4}{0.6}, Rightarrow, nfold, from=3, to=4]
	\arrow[between={0.3}{0.7}, Rightarrow, nfold, dashed, from=4, to=5]
	\arrow["{{(\Lambda^o_{f_2f_1})}^{-1}}", between={0.4}{0.6}, Rightarrow, nfold, from=5, to=6]
\end{tikzcd}		}\\
		$=$
		\scalebox{0.8}{
			\begin{tikzcd}[ampersand replacement=\&]
				{\alpha_z Ff_2Ff_1} \&\& {\beta_z Ff_2Ff_1} \&\&\&\& \\
				\\
				{Gf_2\alpha_y Ff_1} \&\& {Gf_2\beta_y Ff_1} \\
				\&\&\& {} \& {} \\
				{G{f_2}G{f_1}\alpha_x} \&\& {G{f_2}G{f_1}\beta_x} \& {} \& {} \\
				\&\&\&\& {Gf_2\beta_y Ff_1} \\
				{G({f_2}{f_1}) \alpha_x} \&\& {G({f_2f_1})\beta_x} \&\&\&\& {\beta_z Ff_2Ff_1} \\
				\&\& \textcolor{white}{{Gf_2\alpha_y Ff_1}} \&\& {} \\
				\&\&\& {} \& {\alpha_zF(f_2f_1)} \&\& {\beta_zF(f_2f_1)} \\
				\&\&\& {} \\
				\&\&\&\& {G({f_2}{f_1}) \alpha_x} \&\& {G({f_2}{f_1}) \beta_x}
				\arrow[""{name=0, anchor=center, inner sep=0}, "{\Lambda_z\circ1}", from=1-1, to=1-3]
				\arrow["{{\alpha^o_{f_2}}\circ1}"', from=1-1, to=3-1]
				\arrow["{{\beta^o_{f_2}}\circ1}"{description}, from=1-3, to=3-3]
				\arrow[curve={height=-18pt}, equals, from=1-3, to=7-7]
				\arrow[""{name=1, anchor=center, inner sep=0}, "{1\circ\Lambda_y\circ1}", from=3-1, to=3-3]
				\arrow["{1\circ\alpha^o_{f_1}}"', from=3-1, to=5-1]
				\arrow["{1\circ\beta^o_{f_1}}"{description}, from=3-3, to=5-3]
				\arrow[curve={height=-12pt}, equals, from=3-3, to=6-5]
				\arrow["{1\circ\eta^\beta_{f_1}}"', shift right=2, between={0.3}{1}, Rightarrow, from=4-4, to=5-4]
				\arrow["{\eta^\beta_{f_2}\circ1}"', Rightarrow, from=4-5, to=5-5]
				\arrow[""{name=2, anchor=center, inner sep=0}, "{1\circ\Lambda_x}", from=5-1, to=5-3]
				\arrow["{G^2_{f_2,f_1}\circ1}"', from=5-1, to=7-1]
				\arrow["{1\circ\beta^l_{f_1}}"', from=5-3, to=6-5]
				\arrow["{G^2_{f_2,f_1}\circ1}"{description}, from=5-3, to=7-3]
				\arrow["{{\beta^l_{f_2}}\circ1}", from=6-5, to=7-7]
				\arrow["{({\beta^l)}^2_{f_2,f_1}}", between={0.3}{0.7}, Rightarrow, from=6-5, to=8-5]
				\arrow[""{name=3, anchor=center, inner sep=0}, "{1\circ\Lambda_x}", from=7-1, to=7-3]
				\arrow["{{\alpha^l_{f_2f_1}}}"{description}, from=7-1, to=9-5]
				\arrow[curve={height=12pt}, equals, from=7-1, to=11-5]
				\arrow["{\Lambda_{f_2f_1}}"{pos=0.6}, shift right, between={0.4}{0.6}, Rightarrow, from=7-3, to=9-5]
				\arrow["{\beta^l_{f_2f_1}}"{description}, from=7-3, to=9-7]
				\arrow["{1\circ F^2_{f_2,f_1}}", from=7-7, to=9-7]
				\arrow["{\epsilon^\alpha_{f_2f_1}}", Rightarrow, from=9-4, to=10-4]
				\arrow[""{name=4, anchor=center, inner sep=0}, "{\Lambda_z\circ1}"', from=9-5, to=9-7]
				\arrow["{{\alpha^o_{f_2f_1}}}", from=9-5, to=11-5]
				\arrow["{{\beta^o_{f_2f_1}}}", from=9-7, to=11-7]
				\arrow[""{name=5, anchor=center, inner sep=0}, "{1\circ\Lambda_x}"', from=11-5, to=11-7]
				\arrow["{({\Lambda^o_{f_2}})^{-1}\circ1}"', shift left=3, between={0.4}{0.6}, Rightarrow, from=0, to=1]
				\arrow["{1\circ ({\Lambda^o_{f_1}})^{-1}}"', shift left=2, between={0.4}{0.6}, Rightarrow, from=1, to=2]
				\arrow[between={0.3}{0.7}, Rightarrow, dashed, from=2, to=3]
				\arrow["{({\Lambda^o_{f_2f_1}})^{-1}}", between={0.4}{0.6}, Rightarrow, from=4, to=5]
			\end{tikzcd}
		} \\

		\adjustbox{width=\textwidth}{
			$=$
			\begin{tikzcd}[ampersand replacement=\&]
				{\alpha_z Ff_2Ff_1} \&\& {\beta_z Ff_2Ff_1} \&\&\&\& \\
				\\
				{Gf_2\alpha_y Ff_1} \&\& {Gf_2\beta_y Ff_1} \\
				\&\&\& {} \& {} \\
				{G{f_2}G{f_1}\alpha_x} \&\& {G{f_2}G{f_1}\beta_x} \& {} \& {} \\
				\&\&\&\& {Gf_2\beta_y Ff_1} \\
				{G({f_2}{f_1}) \alpha_x} \&\& {G({f_2f_1})\beta_x} \&\&\&\& {\beta_z Ff_2Ff_1} \\
				\&\& \textcolor{white}{{Gf_2\alpha_y Ff_1}} \&\& {} \\
				\&\&\&\&\& {} \& {\beta_zF(f_2f_1)} \\
				\&\&\&\&\& {} \\
				\&\&\&\&\&\& {G({f_2}{f_1}) \beta_x}
				\arrow[""{name=0, anchor=center, inner sep=0}, "{\Lambda_z\circ1}", from=1-1, to=1-3]
				\arrow["{{\alpha^o_{f_2}}\circ1}"', from=1-1, to=3-1]
				\arrow["{{\beta^o_{f_2}}\circ1}"{description}, from=1-3, to=3-3]
				\arrow[curve={height=-18pt}, equals, from=1-3, to=7-7]
				\arrow[""{name=1, anchor=center, inner sep=0}, "{1\circ\Lambda_y\circ1}", from=3-1, to=3-3]
				\arrow["{1\circ\alpha^o_{f_1}}"', from=3-1, to=5-1]
				\arrow["{1\circ\beta^o_{f_1}}"{description}, from=3-3, to=5-3]
				\arrow[curve={height=-12pt}, equals, from=3-3, to=6-5]
				\arrow["{1\circ\eta^\beta_{f_1}}"', shift right=2, between={0.3}{1}, Rightarrow, from=4-4, to=5-4]
				\arrow["{\eta^\beta_{f_2}\circ1}"', Rightarrow, from=4-5, to=5-5]
				\arrow[""{name=2, anchor=center, inner sep=0}, "{1\circ\Lambda_x}", from=5-1, to=5-3]
				\arrow["{G^2_{f_2,f_1}\circ1}"', from=5-1, to=7-1]
				\arrow["{1\circ\beta^l_{f_1}}"', from=5-3, to=6-5]
				\arrow["{G^2_{f_2,f_1}\circ1}"{description}, from=5-3, to=7-3]
				\arrow["{{\beta^l_{f_2}}\circ1}", from=6-5, to=7-7]
				\arrow["{{(\beta^l)}^2_{f_2,f_1}}", between={0.3}{0.7}, Rightarrow, from=6-5, to=8-5]
				\arrow[""{name=3, anchor=center, inner sep=0}, "{1\circ\Lambda_x}", from=7-1, to=7-3]
				\arrow["{\beta^l_{f_2f_1}}"{description}, from=7-3, to=9-7]
				\arrow[curve={height=12pt}, equals, from=7-3, to=11-7]
				\arrow["{1\circ F^2_{f_2,f_1}}", from=7-7, to=9-7]
				\arrow["{\epsilon^\beta_{f_2f_1}}"', Rightarrow, from=9-6, to=10-6]
				\arrow["{{\beta^o_{f_2f_1}}}", from=9-7, to=11-7]
				\arrow["{({\Lambda^o_{f_2}})^{-1}\circ1}"', shift left=3, between={0.4}{0.6}, Rightarrow, from=0, to=1]
				\arrow["{1\circ ({\Lambda^o_{f_1}})^{-1}}"', shift left=2, between={0.4}{0.6}, Rightarrow, from=1, to=2]
				\arrow[between={0.3}{0.7}, Rightarrow, dashed, from=2, to=3]
			\end{tikzcd}
		}
	\end{center}
	
	
	
Finally, we consider the compatibility with 2-cells $\phi \colon f \to f' \in \ca(x,y)$, that is Axiom (\ref{ax:op-mod.2-cell}) in \Cref{subsec:ax-oplax-mod}.
	Since the $\Lambda_f$'s are components of a lax modification, the following equation holds.
	\begin{center}
		\adjustbox{scale=0.85}{\begin{tikzcd}
				{Gf\alpha_x} && {Gf\beta_x} \\
				&& {\alpha_yFf} && {\beta_yFf} \\
				{Gf'\alpha_x} \\
				&& {\alpha_yFf'} && {\beta_yFf'}
				\arrow[""{name=0, anchor=center, inner sep=0}, "{1\circ\Lambda_x}", from=1-1, to=1-3]
				\arrow[""{name=1, anchor=center, inner sep=0}, "{\alpha_f}"{description}, from=1-1, to=2-3]
				\arrow["{\beta_f}", from=1-3, to=2-5]
				\arrow[""{name=2, anchor=center, inner sep=0}, "{\Lambda_y\circ1}"{description}, from=2-3, to=2-5]
				\arrow["{1\circ F\phi}"{description}, from=2-3, to=4-3]
				\arrow["{1\circ F\phi}", from=2-5, to=4-5]
				\arrow["G\phi\circ1"', from=1-1, to=3-1]
				\arrow[""{name=3, anchor=center, inner sep=0}, "{\alpha_{f'}}"', from=3-1, to=4-3]
				\arrow[""{name=4, anchor=center, inner sep=0}, "{\Lambda_y\circ1}"', from=4-3, to=4-5]
				\arrow["{\Lambda_f}"{xshift=0.1cm,yshift=-0.1cm}, shorten <=30pt, shorten >=30pt, Rightarrow, from=0, to=2]
				\arrow[shorten <=20pt, shorten >=20pt, Rightarrow, dashed, from=2, to=4]
				\arrow["{\alpha_\phi}"{xshift=0.1cm}, shorten <=20pt, shorten >=20pt, Rightarrow, from=1, to=3]
			\end{tikzcd} $=$
			\begin{tikzcd}[ampersand replacement=\&]
				{Gf\alpha_x} \&\& {Gf\beta_x} \\
				\&\&\&\& {\beta_yFf} \\
				{Gf'\alpha_x} \&\& {Gf'\beta_x} \\
				\&\& {\alpha_yFf'} \&\& {\beta_yFf'}
				\arrow[""{name=0, anchor=center, inner sep=0}, "{1\circ\Lambda_x}", from=1-1, to=1-3]
				\arrow[""{name=1, anchor=center, inner sep=0}, "{\beta_f}", from=1-3, to=2-5]
				\arrow["{1\circ F\phi}", from=2-5, to=4-5]
				\arrow["{\alpha_{f'}}"', from=3-1, to=4-3]
				\arrow[""{name=2, anchor=center, inner sep=0}, "{\Lambda_Y\circ1}"', from=4-3, to=4-5]
				\arrow[""{name=3, anchor=center, inner sep=0}, "{1\circ\Lambda_x}", from=3-1, to=3-3]
				\arrow["G\phi\circ1"{description}, from=1-3, to=3-3]
				\arrow[""{name=4, anchor=center, inner sep=0}, "{\beta_{f'}}"{description}, from=3-3, to=4-5]
				\arrow["G\phi\circ1"', from=1-1, to=3-1]
				\arrow[shorten <=20pt, shorten >=20pt, Rightarrow, dashed, from=0, to=3]
				\arrow["{\beta_\phi}"{xshift=0.1cm}, shorten <=20pt, shorten >=20pt, Rightarrow, from=1, to=4]
				\arrow["{\Lambda_{f'}}"{xshift=0.1cm,yshift=-0.1cm}, shorten <=30pt, shorten >=30pt, Rightarrow, from=3, to=2]
		\end{tikzcd}}
	\end{center}
	Pasting the invertible 3-cells $({\Lambda^o_f})^{-1}$ and $\eta^\beta_f$ and using \Cref{rmk:mates-over-adj-equiv-calculations} on the left hand side, the above equation becomes the equation below.
	\begin{center}
		\adjustbox{scale=0.825}{
			\begin{tikzcd}[ampersand replacement=\&,column sep=scriptsize]
				{\alpha_yFf} \&\&\&\& \\
				\& {} \\
				{Gf\alpha_x} \\
				\&\& {\alpha_YFf} \&\& {\beta_yFf} \\
				{Gf'\alpha_x} \\
				\&\& {\alpha_yFf'} \&\& {\beta_yFf'}
				\arrow["{\alpha^o_f}"', from=1-1, to=3-1]
				\arrow[curve={height=-24pt}, equals, from=1-1, to=4-3]
				\arrow[""{name=0, anchor=center, inner sep=0}, "{\alpha_f}"{description}, from=3-1, to=4-3]
				\arrow["{G\phi\circ1}"', from=3-1, to=5-1]
				\arrow[""{name=1, anchor=center, inner sep=0}, "{\Lambda_y\circ1}", from=4-3, to=4-5]
				\arrow["{1\circ F\phi}"{description}, from=4-3, to=6-3]
				\arrow["{1\circ F\phi}", from=4-5, to=6-5]
				\arrow[""{name=2, anchor=center, inner sep=0}, "{\alpha_{f'}}"', from=5-1, to=6-3]
				\arrow[""{name=3, anchor=center, inner sep=0}, "{\Lambda_y\circ1}"', from=6-3, to=6-5]
				\arrow["{\eta^\alpha_f}"{pos=0.3}, between={0.2}{0.5}, Rightarrow, from=2-2, to=0]
				\arrow["{\alpha_\phi}", between={0.4}{0.6}, Rightarrow, from=0, to=2]
				\arrow[between={0.4}{0.6}, Rightarrow, dashed, from=1, to=3]
		\end{tikzcd}}
		=
		\adjustbox{scale=0.825}{
			\begin{tikzcd}[ampersand replacement=\&,column sep=scriptsize]
				{\alpha_yFf} \&\& {\beta_yFf} \&\& \\
				\&\&\& {} \\
				{Gf\alpha_x} \&\& {Gf\beta_x} \\
				\&\&\&\& {\beta_yFf} \\
				{Gf'\alpha_x} \&\& {Gf'\beta_x} \\
				\&\& {\alpha_yFf'} \&\& {\beta_yFf'}
				\arrow[""{name=0, anchor=center, inner sep=0}, "{\Lambda_y\circ 1}", from=1-1, to=1-3]
				\arrow["{\alpha^o_f}"', from=1-1, to=3-1]
				\arrow["{\beta^o_f}"{description}, from=1-3, to=3-3]
				\arrow[curve={height=-24pt}, equals, from=1-3, to=4-5]
				\arrow[""{name=1, anchor=center, inner sep=0}, "{1\circ\Lambda_x}", from=3-1, to=3-3]
				\arrow["{G\phi\circ1}"', from=3-1, to=5-1]
				\arrow[""{name=2, anchor=center, inner sep=0}, "{\beta_f}", from=3-3, to=4-5]
				\arrow["{G\phi\circ1}"{description}, from=3-3, to=5-3]
				\arrow["{1\circ F\phi}", from=4-5, to=6-5]
				\arrow[""{name=3, anchor=center, inner sep=0}, "{1\circ\Lambda_x}", from=5-1, to=5-3]
				\arrow["{\alpha_{f'}}"', from=5-1, to=6-3]
				\arrow[""{name=4, anchor=center, inner sep=0}, "{\beta_{f'}}"{description}, from=5-3, to=6-5]
				\arrow[""{name=5, anchor=center, inner sep=0}, "{\Lambda_y\circ1}"', from=6-3, to=6-5]
				\arrow["{{\Lambda^o_f}^{-1}}", between={0.4}{0.6}, Rightarrow, from=0, to=1]
				\arrow["{\eta^\beta_f}"{pos=0.3}, between={0.2}{0.5}, Rightarrow, from=2-4, to=2]
				\arrow[between={0.4}{0.6}, Rightarrow, dashed, from=1, to=3]
				\arrow["{\beta_\phi}", between={0.4}{0.6}, Rightarrow, from=2, to=4]
				\arrow["{\Lambda_{f'}}", between={0.4}{0.6}, Rightarrow, from=3, to=5]
		\end{tikzcd}}
	\end{center}
Again using  \Cref{rmk:mates-over-adj-equiv-calculations} (on the right hand side this time) and pasting the invertible 3-cells $\epsilon^\alpha_{f'}$ and ${\Lambda^o_{f'}}^{-1}$, the above equation implies the equation below.
	\begin{center}
		\adjustbox{scale=0.825}{
			\begin{tikzcd}[ampersand replacement=\&,column sep=scriptsize]
				{\alpha_yFf} \&\&\&\& \\
				\& {} \\
				{Gf\alpha_x} \\
				\&\& {\alpha_yFf} \&\& {\beta_yFf} \\
				{Gf'\alpha_x} \\
				\&\& {\alpha_yFf'} \&\& {\beta_yFf'} \\
				\& {} \\
				\&\& {Gf'\alpha_x} \&\& {Gf'\beta_x}
				\arrow["{\alpha^o_f}"', from=1-1, to=3-1]
				\arrow[curve={height=-24pt}, equals, from=1-1, to=4-3]
				\arrow[""{name=0, anchor=center, inner sep=0}, "{\alpha_f}"{description}, from=3-1, to=4-3]
				\arrow["{G\phi\circ1}"', from=3-1, to=5-1]
				\arrow[""{name=1, anchor=center, inner sep=0}, "{\Lambda_y\circ1}", from=4-3, to=4-5]
				\arrow["{1\circ F\phi}"{description}, from=4-3, to=6-3]
				\arrow["{1\circ F\phi}", from=4-5, to=6-5]
				\arrow[""{name=2, anchor=center, inner sep=0}, "{\alpha_{f'}}"{description}, from=5-1, to=6-3]
				\arrow[curve={height=24pt}, equals, from=5-1, to=8-3]
				\arrow[""{name=3, anchor=center, inner sep=0}, "{\Lambda_y\circ1}"{description}, from=6-3, to=6-5]
				\arrow["{\alpha^o_{f'}}"{description}, from=6-3, to=8-3]
				\arrow["{\beta^o_{f'}}", from=6-5, to=8-5]
				\arrow[""{name=4, anchor=center, inner sep=0}, "{1\circ\Lambda_x}"', from=8-3, to=8-5]
				\arrow["{\eta^\alpha_f}"{pos=0.3}, between={0.2}{0.5}, Rightarrow, from=2-2, to=0]
				\arrow["{\alpha_\phi}", between={0.4}{0.6}, Rightarrow, from=0, to=2]
				\arrow[between={0.4}{0.6}, Rightarrow, dashed, from=1, to=3]
				\arrow["{\epsilon^\alpha_{f'}}"{pos=0.7}, between={0.5}{0.8}, Rightarrow, from=2, to=7-2]
				\arrow["{{\Lambda^o_f}^{-1}}", between={0.4}{0.6}, Rightarrow, from=3, to=4]
			\end{tikzcd}	
		}
		=
		\adjustbox{scale=0.825}{
			\begin{tikzcd}[ampersand replacement=\&,column sep=scriptsize]
				{\alpha_yFf} \&\& {\beta_yFf} \&\& \\
				\&\&\& {} \\
				{Gf\alpha_x} \&\& {Gf\beta_x} \\
				\&\&\&\& {\beta_yFf} \\
				{Gf'\alpha_x} \&\& {Gf'\beta_x} \\
				\&\&\&\& {\beta_yFf'} \\
				\&\&\& {} \\
				\&\&\&\& {Gf'\beta_x}
				\arrow[""{name=0, anchor=center, inner sep=0}, "{\Lambda_y\circ 1}", from=1-1, to=1-3]
				\arrow["{\alpha^o_f}"', from=1-1, to=3-1]
				\arrow["{\beta^o_f}"{description}, from=1-3, to=3-3]
				\arrow[curve={height=-24pt}, equals, from=1-3, to=4-5]
				\arrow[""{name=1, anchor=center, inner sep=0}, "{1\circ\Lambda_x}", from=3-1, to=3-3]
				\arrow["{G\phi\circ1}"', from=3-1, to=5-1]
				\arrow[""{name=2, anchor=center, inner sep=0}, "{\beta_f}", from=3-3, to=4-5]
				\arrow["{G\phi\circ1}"{description}, from=3-3, to=5-3]
				\arrow["{1\circ F\phi}", from=4-5, to=6-5]
				\arrow[""{name=3, anchor=center, inner sep=0}, "{1\circ\Lambda_x}", from=5-1, to=5-3]
				\arrow[""{name=4, anchor=center, inner sep=0}, "{\beta_{f'}}"{description}, from=5-3, to=6-5]
				\arrow[curve={height=24pt}, equals, from=5-3, to=8-5]
				\arrow["{\beta^o_{f'}}", from=6-5, to=8-5]
				\arrow["{{\Lambda^o_f}^{-1}}", between={0.4}{0.6}, Rightarrow, from=0, to=1]
				\arrow["{\eta^\beta_f}"{pos=0.3}, between={0.2}{0.5}, Rightarrow, from=2-4, to=2]
				\arrow[between={0.4}{0.6}, Rightarrow, dashed, from=1, to=3]
				\arrow["{\beta_\phi}", between={0.4}{0.6}, Rightarrow, from=2, to=4]
				\arrow["{\epsilon^\beta_{f'}}"{pos=0.7}, between={0.5}{0.8}, Rightarrow, from=4, to=7-4]
			\end{tikzcd}	
		}
	\end{center}
Finally, we can rewrite the equation above as the equation below, using the definition of $\alpha^o_\phi$ and $\beta^o_\phi$. 
	\begin{center}
		\adjustbox{scale=0.85}{
			\begin{tikzcd}[ampersand replacement=\&,column sep=scriptsize]
				{\alpha_yFf} \&\& {\beta_yFf} \&\& \\
				\&\& {\alpha_yFf'} \&\& {\beta_yFf'} \\
				{Gf\alpha_x} \\
				\&\& {Gf'\alpha_x} \&\& {Gf'\beta_x}
				\arrow[""{name=0, anchor=center, inner sep=0}, "{\Lambda_y\circ1}", from=1-1, to=1-3]
				\arrow[""{name=1, anchor=center, inner sep=0}, "{1\circ F\phi}"{description}, from=1-1, to=2-3]
				\arrow["{\alpha^o_f}"', from=1-1, to=3-1]
				\arrow["{1\circ F\phi}", from=1-3, to=2-5]
				\arrow[""{name=2, anchor=center, inner sep=0}, "{\Lambda_y\circ1}"{description}, from=2-3, to=2-5]
				\arrow["{\alpha^o_{f'}}"{description}, from=2-3, to=4-3]
				\arrow["{\beta^o_{f'}}", from=2-5, to=4-5]
				\arrow[""{name=3, anchor=center, inner sep=0}, "{G\phi\circ1}"', from=3-1, to=4-3]
				\arrow[""{name=4, anchor=center, inner sep=0}, "{1\circ\Lambda_x}"', from=4-3, to=4-5]
				\arrow[between={0.4}{0.6}, Rightarrow, dashed, from=0, to=2]
				\arrow["{\alpha^o_\phi}", between={0.4}{0.6}, Rightarrow, from=1, to=3]
				\arrow["{{\Lambda^o_{f'}}^{-1}}", between={0.4}{0.6}, Rightarrow, from=2, to=4]
			\end{tikzcd}
		}
		=
		\adjustbox{scale=0.825}{
			\begin{tikzcd}[ampersand replacement=\&,column sep=scriptsize]
				{\alpha_yFf} \&\& {\beta_yFf} \&\& \\
				\&\&\&\& {\beta_yFf'} \\
				{Gf\alpha_x} \&\& {Gf\beta_x} \\
				\&\& {Gf'\alpha_x} \&\& {Gf'\beta_x}
				\arrow[""{name=0, anchor=center, inner sep=0}, "{\Lambda_y\circ 1}", from=1-1, to=1-3]
				\arrow["{\alpha^o_f}"', from=1-1, to=3-1]
				\arrow[""{name=1, anchor=center, inner sep=0}, "{1\circ F\phi}", from=1-3, to=2-5]
				\arrow["{\beta^o_f}"{description}, from=1-3, to=3-3]
				\arrow["{\beta^o_{f'}}", from=2-5, to=4-5]
				\arrow[""{name=2, anchor=center, inner sep=0}, "{1\circ\Lambda_x}", from=3-1, to=3-3]
				\arrow["{G\phi\circ1}"', from=3-1, to=4-3]
				\arrow[""{name=3, anchor=center, inner sep=0}, "{G\phi\circ1}"{description}, from=3-3, to=4-5]
				\arrow[""{name=4, anchor=center, inner sep=0}, "{1\circ\Lambda_x}", from=4-3, to=4-5]
				\arrow["{{\Lambda^o_f}^{-1}}", between={0.4}{0.6}, Rightarrow, from=0, to=2]
				\arrow["{\beta^o_\phi}", between={0.4}{0.6}, Rightarrow, from=1, to=3]
				\arrow[between={0.4}{0.6}, Rightarrow, dashed, from=2, to=4]
			\end{tikzcd}
		}
	\end{center}
	
This equation is equivalent to the sought-for axiom for $\Lambda^o$ (pasting both sides of the equation with ${\Lambda^o_f}^{-1}$, ${\Lambda^o_{f'}}^{-1}$ and two interchange isomorphisms). 

\end{proof}


\begin{defn}\label{def:unb-psd-mod}
	Let $\alpha,\beta\colon F\rightrightarrows G$ be unbiased pseudo-transformations between pseudomaps $F,G\colon\ca\rightrightarrows \cb$. 
	An \emph{unbiased pseudo-modification} $\Gamma\colon \alpha\to\beta$ consists of:
	\begin{itemize}
	\item for any object $x\in\ca$, a 2-cell $\Gamma_x\colon\alpha_x\to\beta_x$;
	\item for any 1-cell $f\colon x\to y\in\ca$, invertible 3-cells $\Gamma_f^l$ and $\Gamma_f^o$ making $(\Gamma_x,\Gamma^l_f)\colon \alpha^l\to\beta^l$ and $(\Gamma_x,\Gamma^o_f)\colon \alpha^o\to\beta^o$ lax and oplax modifications. 
	\end{itemize}
Moreover, we require that $\Gamma_f^l$ and $\Gamma_f^o$ are mates of each other (see point (3) of \Cref{prop:equ-psd-mod} for more details).\footnote{Note here that again, since $\Gamma_f^l$ and $\Gamma_f^o$ are mates along adjoint equivalences, it suffices for one of them to be invertible to guarantee that the other one will be as well.}
\end{defn}

\begin{notation*}
	Given an unbiased pseudo-modification $\Lambda$, we write $\Lambda^l$ and $\Lambda^o$ for its lax and oplax modifications. 
\end{notation*}

Finally, we consider 3-cells in $\psd(\ca,\cb)$, i.e. pseudo perturbations. 
In this case, there won't be extra structure to define, only two different equivalent axioms. 

\begin{prop}\label{prop:equ-psd-pert}
	Consider unbiased pseudo-modifications $\Lambda,\Gamma\colon\alpha\to\beta$ and a family of 3-cells $\theta_x\colon\Lambda_x\to\Gamma_x$. Then $\theta$ is a perturbation $\Lambda^l\to\Gamma^l$ in $\lax(\ca,\cb)$ if and only if it is a perturbation $\Lambda^o\to\Gamma^o$ in $\oplax(\ca,\cb)$. 
\end{prop}

\begin{proof}
	Let us start with $\theta$ a perturbation in $\lax(\ca,\cb)$. 
	We need to prove that $\theta$ is also a perturbation in $\oplax(\ca,\cb)$, i.e. the axiom in \Cref{ax:oplax-pert} holds. 
	Substituting the explicit definitions of $\Lambda^o_f$ and $\Gamma^o_f$ from \Cref{def:unb-psd-mod}, the equation above can be written as the equation below, which is true because we assumed $\theta\in\lax(\ca,\cb)$. 
	\begin{center}
		\begin{tikzcd}[ampersand replacement=\&,column sep=scriptsize,row sep=2.25em]
			{\alpha_yFf} \&\& {Gf\alpha_x} \&\& {Gf\beta_x} \&\& \\
			\&\& {\alpha_yFf} \&\& {\beta_yFf} \&\& {Gf\beta_x}
			\arrow[""{name=0, anchor=center, inner sep=0}, "{\alpha^o_f}", from=1-1, to=1-3]
			\arrow[""{name=1, anchor=center, inner sep=0}, equals, from=1-1, to=2-3]
			\arrow[""{name=2, anchor=center, inner sep=0}, "{1\circ \Lambda_x}", curve={height=-12pt}, from=1-3, to=1-5]
			\arrow["{\alpha^l_f}"{description}, from=1-3, to=2-3]
			\arrow["{\beta^l_f}"{description}, from=1-5, to=2-5]
			\arrow[""{name=3, anchor=center, inner sep=0}, equals, from=1-5, to=2-7]
			\arrow[""{name=4, anchor=center, inner sep=0}, "{\Lambda_y\circ1}"{description}, curve={height=-12pt}, from=2-3, to=2-5]
			\arrow[""{name=5, anchor=center, inner sep=0}, "{\Gamma_y\circ1}"', curve={height=12pt}, from=2-3, to=2-5]
			\arrow[""{name=6, anchor=center, inner sep=0}, "{\beta^o_f}"', from=2-5, to=2-7]
			\arrow["{({\eta_f^\alpha})^{-1}}", between={0.3}{0.7}, Rightarrow, from=0, to=1]
			\arrow["{\Lambda_f}", between={0.4}{0.6}, Rightarrow, from=2, to=4]
			\arrow["{({\epsilon_f^{\beta}})^{-1}}"', between={0.3}{0.7}, Rightarrow, from=3, to=6]
			\arrow["{\theta_y\circ1}", shift right=2, between={0.3}{0.7}, Rightarrow, from=4, to=5]
		\end{tikzcd} = \\ \vspace{0.2cm}
		\begin{tikzcd}[ampersand replacement=\&,column sep=scriptsize,row sep=2.25em]
			{\alpha_yFf} \&\& {Gf\alpha_x} \&\& {Gf\beta_x} \&\& \\
			\&\& {\alpha_yFf} \&\& {\beta_yFf} \&\& {Gf\beta_x}
			\arrow[""{name=0, anchor=center, inner sep=0}, "{\alpha^o_f}", from=1-1, to=1-3]
			\arrow[""{name=1, anchor=center, inner sep=0}, equals, from=1-1, to=2-3]
			\arrow[""{name=2, anchor=center, inner sep=0}, "{1\circ \Lambda_x}", curve={height=-12pt}, from=1-3, to=1-5]
			\arrow[""{name=3, anchor=center, inner sep=0}, "{1\circ\Gamma_x}"{description}, curve={height=12pt}, from=1-3, to=1-5]
			\arrow["{\alpha^l_f}"{description}, from=1-3, to=2-3]
			\arrow["{\beta^l_f}"{description}, from=1-5, to=2-5]
			\arrow[""{name=4, anchor=center, inner sep=0}, equals, from=1-5, to=2-7]
			\arrow[""{name=5, anchor=center, inner sep=0}, "{\Gamma_y\circ1}"', curve={height=12pt}, from=2-3, to=2-5]
			\arrow[""{name=6, anchor=center, inner sep=0}, "{\beta^o_f}"', from=2-5, to=2-7]
			\arrow["{({\eta_f^\alpha})^{-1}}", between={0.3}{0.7}, Rightarrow, from=0, to=1]
			\arrow["{1\circ\theta_x}", shift right=2, between={0.3}{0.7}, Rightarrow, from=2, to=3]
			\arrow["{\Gamma_f}", between={0.4}{0.6}, Rightarrow, from=3, to=5]
			\arrow["({{\epsilon_f^{\beta}})^{-1}}"', between={0.3}{0.7}, Rightarrow, from=4, to=6]
		\end{tikzcd}
	\end{center}
	Analogously one can prove that if $\theta \colon \Lambda^o\to\Gamma^o$ is a perturbation, then it is a perturbation $\Lambda^l\to\Gamma^l$.
\end{proof}

\begin{defn}\label{def:unb-psd-prt}
	Let $\Lambda,\Gamma\colon\alpha\to\beta$ be unbiased pseudo-modifications.  An \emph{unbiased perturbation} $\theta \colon \Lambda \to \Gamma$ consists of a family of 3-cells $\theta_x\colon\Lambda_x\to\Gamma_x$ forming a perturbation $\Lambda^l\to\Gamma^l$ in $\lax(\ca,\cb)$ and a perturbation $\Lambda^o\to\Gamma^o$ in $\oplax(\ca,\cb)$.
\end{defn}
 
\subsection{The Gray category of unbiased pseudo-transformations}
\label{sec:gray-struct-en-psd} 

Let $\psden(\ca,\cb)$ be the $3$-graph (a.k.a. $3$-dimensional globular set) of Gray-categories, pseudomaps, unbiased pseudo-transformations (\Cref{def:unb-psd-transf}), unbiased pseudo-modifications (\Cref{def:unb-psd-mod}) and unbiased perturbations (\Cref{def:unb-psd-prt}).  
There is an evident forgetful morphism of $3$-graphs $U^p_\nb \colon \psden(\ca,\cb) \to \psden(\ca,\cb)$.

\begin{prop}
$U^p_\nb \colon \psden(\ca,\cb) \to \psd(\ca,\cb)$ is an isomorphism of $3$-graphs.
\end{prop}
\begin{proof}
This combines Propositions \ref{prop:equ-psd-transf}, \ref{prop:equ-psd-mod}, and \ref{prop:equ-psd-pert}. 
\end{proof}

Transporting the structure of a Gray-category along the isomorphism $U^p_\nb$, we immediately obtain:

\begin{cor}\label{unbiased}
$\psden(\ca,\cb)$ admits the unique structure of a Gray-category such that 
$$U^p_\nb \colon \psden(\ca,\cb) \to \psd(\ca,\cb)$$ 
is an isomorphism in $\GCat$.
\end{cor}

In elementary terms, this means that the various compositions in $\psden(\ca,\cb)$ combine composition in $\lax(\ca,\cb)$ with composition of adjoint equivalences. Indeed, postcomposing with the forgetful Gray-functor $U^l\colon\psd(\ca,\cb)\to\lax(\ca,\cb)$ we obtain another forgetful Gray-functor $U^l_\nb:=U^l\circ U^p_\nb\colon\psden(\ca,\cb)\to\lax(\ca,\cb)$.  Now we would also like to show that the Gray-category structure on $\psden(\ca,\cb)$ also enhances composition in $\oplax(\ca,\cb)$.  To this end, observe that there is also an evident forgetful morphism of $3$-graphs $U^o_\nb \colon \psden(\ca,\cb) \to \oplax(\ca,\cb)$, forgetting the lax part of unbiased $n$-transfors.  

\begin{prop}\label{prop:forg-fct-psden-oplax}
$U^o_\nb \colon \psden(\ca,\cb) \to \oplax(\ca,\cb)$ is a $Gray$-functor.
\end{prop}

\begin{proof}
By the uniqueness of the Gray-categorical structure on $\psden(\ca,\cb)$ in \Cref{unbiased}, we just need to prove that computing each operation  both in  $\psd(\ca,\cb)$ and in $\oplax(\ca,\cb)$ gives a $n$-transfor in $\psden(\ca,\cb)$. 
%
%
%
%
%

\begin{itemize}
  \item \emph{Vertical composition of unbiased pseudotransformations.} 
   Given $\alpha=(\alpha^l,\alpha^o)\colon F\to G$ and $\beta=(\beta^l,\beta^o)\colon G \to H$ unbiased pseudotransformations, we need to show that $(\beta^l\cdot\alpha^l,\beta^o\cdot\alpha^o)$ together with the adjoint equivalence below is an unbiased pseudotransformation. 
	\[\begin{tikzcd}[ampersand replacement=\&]
		{Hf\beta_x\alpha_x} \& {\beta_yGf\alpha_x} \& {\beta_y\alpha_yFf}
		\arrow[""{name=0, anchor=center, inner sep=0}, "{{\beta_f}^l\circ1}", curve={height=-12pt}, from=1-1, to=1-2]
		\arrow[""{name=1, anchor=center, inner sep=0}, "{{\beta_f}^o\circ1}", curve={height=-12pt}, from=1-2, to=1-1]
		\arrow[""{name=2, anchor=center, inner sep=0}, "{1\circ{\alpha_f}^l}", curve={height=-12pt}, from=1-2, to=1-3]
		\arrow[""{name=3, anchor=center, inner sep=0}, "{1\circ{\alpha_f}^o}", curve={height=-12pt}, from=1-3, to=1-2]
		\arrow["\top"{description}, draw=none, from=1, to=0]
		\arrow["\top"{description}, draw=none, from=3, to=2]
	\end{tikzcd}\]
By definition of a pseudotransformation, we must show that $(\beta^o\cdot\alpha^o)_\phi$ is invertible and that the components of $\beta^l\cdot\alpha^l$ and $\beta^o\cdot\alpha^o$ at 2-cells and composable 1-cells are mates.
	\begin{itemize}
		\item Now the component of $(\beta^o\cdot\alpha^o)_\phi$ is the pasting
\[\begin{tikzcd}[ampersand replacement=\&]
	{Hf\beta_x\alpha_x} \&\& {\beta_yGf\alpha_x} \&\& {\beta_y\alpha_yFf} \\
	{Hg\beta_x\alpha_x} \&\& {\beta_yGg\alpha_x} \&\& {\beta_y\alpha_yFg}
	\arrow["{H\phi \circ 1 \circ 1}"', from=1-1, to=2-1]
	\arrow[""{name=0, anchor=center, inner sep=0}, "{\beta_f^o \circ 1}"', from=1-3, to=1-1]
	\arrow["{1 \circ G\phi \circ 1}"{description}, from=1-3, to=2-3]
	\arrow[""{name=1, anchor=center, inner sep=0}, "{1 \circ \alpha^o_f}"', from=1-5, to=1-3]
	\arrow["{1\circ 1 \circ F\phi}", from=1-5, to=2-5]
	\arrow[""{name=2, anchor=center, inner sep=0}, "{\beta^o_g \circ 1}", from=2-3, to=2-1]
	\arrow[""{name=3, anchor=center, inner sep=0}, "{1 \circ \alpha^o_g}", from=2-5, to=2-3]
	\arrow["{\beta^o_\phi \circ 1}"', between={0.2}{0.8}, Rightarrow, nfold, from=0, to=2]
	\arrow["{1 \circ \alpha^o_\phi}", between={0.2}{0.8}, Rightarrow, nfold, from=1, to=3]
\end{tikzcd}\]
of the components of $\alpha^o_\phi$ and $\beta^o_\phi$ as depicted, whilst component of $(\beta^l\cdot\alpha^l)_\phi$ is likewise obtained by pasting (see \cite[Appendix B.1]{BourkeLobbia}).  In particular, $(\beta^o\cdot\alpha^o)_\phi$ is invertible as a pasting of invertible $3$-cells.  Furthermore that $(\beta^l\cdot\alpha^l)_\phi$ and $(\beta^o\cdot\alpha^o)_\phi$ are mates, as required, follows from the functoriality of the mates correspondence  \cite[Proposition~2.2]{KellyStreet:rew-el-2cats} (i.e. the pasting of mates is the mate of the pasting).  The components of $\beta^l\cdot\alpha^l$ and $\beta^o\cdot\alpha^o$ at a composable pair of $1$-cells are likewise obtained by pasting (again see \cite[Appendix B.1]{BourkeLobbia} for the lax case) and therefore the required mate condition again follows from the functoriality of the mates correspondence.
\end{itemize}

\item \emph{Vertical composition of pseudo-modifications.} Let $\Lambda=(\Lambda^l,\Lambda^o)\colon\alpha\to\beta$ and $\Gamma=(\Gamma^l,\Gamma^o)\colon\beta\to\gamma$ be unbiased pseudo-modifications. 
	Then, we want to show that $(\Gamma^l\ast\Lambda^l,\Gamma^o\ast\Lambda^o)$ is again an unbiased pseudo-modification.
	For any 1-cell $f$, $(\Gamma^o\ast\Lambda^o)_f$ is given by the pasting
\[\begin{tikzcd}[ampersand replacement=\&]
	{\alpha_yFf} \&\& {\beta_yFf} \&\& {\gamma_yFf} \\
	{Gf\alpha_x} \&\& {Gf\beta_x} \&\& {Gf\gamma_x}
	\arrow["{\Lambda^o_y \circ 1}", from=1-1, to=1-3]
	\arrow[""{name=0, anchor=center, inner sep=0}, "{\alpha^o_f}"', from=1-1, to=2-1]
	\arrow["{\Gamma^o_y \circ 1}", from=1-3, to=1-5]
	\arrow[""{name=1, anchor=center, inner sep=0}, "{\beta^o_f}"', from=1-3, to=2-3]
	\arrow[""{name=2, anchor=center, inner sep=0}, "{\gamma^o_f}", from=1-5, to=2-5]
	\arrow["{1 \circ \Lambda^o_x}"', from=2-1, to=2-3]
	\arrow["{1 \circ \Gamma^o_x}"', from=2-3, to=2-5]
	\arrow["{\Lambda^o_f}", between={0.3}{0.7}, Rightarrow, nfold, from=0, to=1]
	\arrow["{\Gamma^o_f}", between={0.3}{0.7}, Rightarrow, nfold, from=1, to=2]
\end{tikzcd}\]
Accordingly it is invertible since $\Lambda^o_f$ and $\Gamma^o_f$ are by assumption.  Furthermore, since $(\Gamma^o\ast\Lambda^o)_f$ is obtained as the analogous pasting (see Appendix B.2.1 of \cite{BourkeLobbia}), it follows that $(\Gamma^l\ast\Lambda^l)_f$ and $(\Gamma^o\ast\Lambda^o)_f$ are again mates by the functoriality of the mates correspondence \cite[Proposition~2.2]{KellyStreet:rew-el-2cats}, as required.

\item \emph{Whiskering of unbiased pseudo-modifications by unbiased pseudo-transformations.} Let $\alpha,\alpha'\colon F\to G$ and $\beta,\beta'\colon G\to H$ be unbiased pseudo-transformations, and $\Lambda\colon\alpha\to\alpha'$ and $\Gamma\colon \beta\to\beta'$ unbiased pseudo-modifications. 
	We must show that $(\Gamma^l\cdot\alpha^l,\Gamma^o\cdot\alpha^o)$ and $(\beta^l\cdot\Lambda^l,\beta^o\cdot\Lambda^o)$ are also unbiased pseudo-modifications. 
	We give the explicit calculations only for the left whiskering, since the right one is completely analogous.
	Let us consider $(\Gamma^l\cdot\alpha^l,\Gamma^o\cdot\alpha^o)$, where $\Gamma^l\cdot\alpha^l$ and $\Gamma^o\cdot\alpha^o$ are defined in $\lax(\ca,\cb)[F,G]$ and $\oplax(\ca,\cb)[F,G]$, respectively. 
	\begin{itemize}
		\item  Recalling the definitions of $(\Gamma^l\cdot\alpha^l)_f$ and $(\Gamma^o\cdot\alpha^o)_f$ (see below), it is clear that they are invertible, since they are pastings of invertible 3-cells by assumption.
		\begin{center}
			$(\Gamma^l\cdot\alpha^l)_f:=$ 
			\scalebox{0.7}{
				\begin{tikzcd}[ampersand replacement=\&,column sep=small]
					{Hf\beta_x\alpha_x} \&\& {Hf\beta'_x\alpha_x} \\
					{\beta_y Gf\alpha_x} \&\& {\beta'_y Gf\alpha_x} \\
					{\beta_y\alpha_y Ff} \&\& {\beta'_y\alpha_y Ff}
					\arrow[""{name=0, anchor=center, inner sep=0}, "{1\circ\Gamma_x\circ1}", from=1-1, to=1-3]
					\arrow["{\beta^l_f\circ1}"', from=1-1, to=2-1]
					\arrow["{(\beta')^l_f\circ1}", from=1-3, to=2-3]
					\arrow[""{name=1, anchor=center, inner sep=0}, "{\Gamma_y\circ1}"', from=2-1, to=2-3]
					\arrow["{1\circ\alpha^l_f}"', from=2-1, to=3-1]
					\arrow["{1\circ\alpha^l_f}", from=2-3, to=3-3]
					\arrow[""{name=2, anchor=center, inner sep=0}, "{\Gamma_y\circ1}"', from=3-1, to=3-3]
					\arrow["{\Gamma^l_f\circ1}", between={0.3}{0.7}, Rightarrow, from=0, to=1]
					\arrow[between={0.4}{0.8}, Rightarrow, dashed, from=1, to=2]
				\end{tikzcd} 
			}
			\hspace{0.5cm} $(\Gamma^o\cdot\alpha^o)_f:=$
			\scalebox{0.7}{
				\begin{tikzcd}[ampersand replacement=\&,column sep=small]
					{\beta_y\alpha_y Ff} \&\& {\beta_y Gf\alpha_x} \&\& {Hf\beta_x\alpha_x} \\
					{\beta'_y\alpha_y Ff} \&\& {\beta'_y Gf\alpha_x} \&\& {Hf\beta'_x\alpha_x}
					\arrow[""{name=0, anchor=center, inner sep=0}, "{1\circ\alpha^o_f}", from=1-1, to=1-3]
					\arrow["{\Gamma_y\circ1}"', from=1-1, to=2-1]
					\arrow[""{name=1, anchor=center, inner sep=0}, "{\beta^o_f\circ1}", from=1-3, to=1-5]
					\arrow["{\Gamma_y\circ1}"', from=1-3, to=2-3]
					\arrow["{1\circ\Gamma_x\circ1}", from=1-5, to=2-5]
					\arrow[""{name=2, anchor=center, inner sep=0}, "{1\circ\alpha^o_f}"', from=2-1, to=2-3]
					\arrow[""{name=3, anchor=center, inner sep=0}, "{(\beta')^o_f\circ1}"', from=2-3, to=2-5]
					\arrow[between={0.3}{0.7}, Rightarrow, dashed, from=0, to=2]
					\arrow["{\Gamma^o_f\circ1}", between={0.3}{0.7}, Rightarrow, from=1, to=3]
				\end{tikzcd}
			}
		\end{center}
		\item We need to prove that the 3-cells in the previous point are mates of each other. 
		We start considering the mate of $(\Gamma^l\cdot\alpha^l)_f$ and then use (in order): \Cref{lemma:interchange-property} for $({\eta_f^{\alpha}})^{-1}$ and a triangle equation; the property of $\Gamma_f^o$ to be the mate of $\Gamma_f^l$ (since $\Gamma$ is an unbiased pseudo-modification).
		\begin{center}
			\scalebox{0.8}{
				\begin{tikzcd}[ampersand replacement=\&,column sep=small]
					{\beta_y\alpha_y Ff} \&\& {\beta_y Gf\alpha_x} \&\& {Hf\beta_x\alpha_x} \&\& {Hf\beta'_x\alpha_x} \&\& {} \&\& \\
					\&\&\& {} \& {\beta_y Gf\alpha_x} \&\& {\beta'_y Gf\alpha_x} \& {} \\
					\&\& {} \&\& {\beta_y\alpha_y Ff} \&\& {\beta'_y\alpha_y Ff} \&\& {\beta'_y Gf\alpha_x} \&\& {Hf\beta'_x\alpha_x}
					\arrow["{1\circ\alpha^o_f}", from=1-1, to=1-3]
					\arrow[curve={height=12pt}, equals, from=1-1, to=3-5]
					\arrow[""{name=0, anchor=center, inner sep=0}, "{\beta^o_f\circ1}", from=1-3, to=1-5]
					\arrow[curve={height=12pt}, equals, from=1-3, to=2-5]
					\arrow["{1\circ{\eta^\alpha_f}^{-1}}"{pos=0.3}, shift right=3, between={0.2}{0.4}, Rightarrow, from=1-3, to=3-3]
					\arrow[""{name=1, anchor=center, inner sep=0}, "{1\circ\Gamma_x\circ1}", from=1-5, to=1-7]
					\arrow["{\beta^l_f\circ1}", from=1-5, to=2-5]
					\arrow["{(\beta')^l_f\circ1}", from=1-7, to=2-7]
					\arrow[curve={height=-12pt}, equals, from=1-7, to=3-11]
					\arrow["{{\epsilon_f^{\beta'}}^{-1}}"'{pos=0.6}, shift left=4, between={0.5}{0.7}, Rightarrow, from=1-9, to=3-9]
					\arrow[""{name=2, anchor=center, inner sep=0}, "{\Gamma_y\circ1}"', from=2-5, to=2-7]
					\arrow["{1\circ\alpha^l_f}"', from=2-5, to=3-5]
					\arrow["{1\circ\alpha^l_f}"', from=2-7, to=3-7]
					\arrow[curve={height=-12pt}, equals, from=2-7, to=3-9]
					\arrow[""{name=3, anchor=center, inner sep=0}, "{\Gamma_y\circ1}"', from=3-5, to=3-7]
					\arrow[""{name=4, anchor=center, inner sep=0}, "{1\circ\alpha^o_f}"', from=3-7, to=3-9]
					\arrow["{(\beta')^o_f\circ1}"', from=3-9, to=3-11]
					\arrow["{{\eta^\beta_f}^{-1}\circ1}"{pos=0.4}, shift right=2, between={0.2}{0.6}, Rightarrow, from=0, to=2-4]
					\arrow["{\Gamma^l_f\circ1}", between={0.3}{0.7}, Rightarrow, from=1, to=2]
					\arrow[between={0.4}{0.8}, Rightarrow, dashed, from=2, to=3]
					\arrow["{1\circ{\epsilon_f^\alpha}^{-1}}"', shift left=3, between={0.3}{0.7}, Rightarrow, from=2-8, to=4]
				\end{tikzcd}
			} \\
			$=$
			\scalebox{0.8}{
				\begin{tikzcd}[ampersand replacement=\&]
					{\beta_y\alpha_y Ff} \& {\beta_y Gf\alpha_x} \&\& {Hf\beta_x\alpha_x} \&\& {Hf\beta'_x\alpha_x} \& {} \& \\
					\&\& {} \& {\beta_y Gf\alpha_x} \& {} \\
					{\beta'_y\alpha_y Ff} \&\& {} \&\&\& {\beta'_y Gf\alpha_x} \&\& {Hf\beta'_x\alpha_x}
					\arrow["{1\circ\alpha^o_f}", from=1-1, to=1-2]
					\arrow["{\Gamma_y\circ1}"', from=1-1, to=3-1]
					\arrow[""{name=0, anchor=center, inner sep=0}, "{\beta^o_f\circ1}", from=1-2, to=1-4]
					\arrow[curve={height=12pt}, equals, from=1-2, to=2-4]
					\arrow[""{name=1, anchor=center, inner sep=0}, "{1\circ\Gamma_x\circ1}", from=1-4, to=1-6]
					\arrow["{\beta^l_f\circ1}", from=1-4, to=2-4]
					\arrow["{(\beta')^l_f\circ1}", from=1-6, to=3-6]
					\arrow[curve={height=-12pt}, equals, from=1-6, to=3-8]
					\arrow[between={0.1}{0.9}, Rightarrow, dashed, from=2-3, to=3-3]
					\arrow[""{name=2, anchor=center, inner sep=0}, "{\Gamma_y\circ1}"{description}, from=2-4, to=3-6]
					\arrow["{1\circ\alpha^o_f}"', from=3-1, to=3-6]
					\arrow[""{name=3, anchor=center, inner sep=0}, "{(\beta')^o_f\circ1}"', from=3-6, to=3-8]
					\arrow["{{\eta^\beta_f}^{-1}\circ1}"{pos=0.4}, shift right=2, between={0.2}{0.6}, Rightarrow, from=0, to=2-3]
					\arrow["{\Gamma^l_f\circ1}", between={0.3}{0.7}, Rightarrow, from=1, to=2]
					\arrow["{{\epsilon_f^{\beta'}}^{-1}}"'{pos=0.6}, shift left=4, between={0.5}{0.7}, Rightarrow, from=1-7, to=3]
				\end{tikzcd}
			} \\
			$=$
			\scalebox{0.8}{
				\begin{tikzcd}[ampersand replacement=\&,column sep=small]
					{\beta_y\alpha_y Ff} \&\& {\beta_y Gf\alpha_x} \&\& {Hf\beta_x\alpha_x} \\
					{\beta'_y\alpha_y Ff} \&\& {\beta'_y Gf\alpha_x} \&\& {Hf\beta'_x\alpha_x}
					\arrow[""{name=0, anchor=center, inner sep=0}, "{1\circ\alpha^o_f}", from=1-1, to=1-3]
					\arrow["{\Gamma_y\circ1}"', from=1-1, to=2-1]
					\arrow[""{name=1, anchor=center, inner sep=0}, "{\beta^o_f\circ1}", from=1-3, to=1-5]
					\arrow["{\Gamma_y\circ1}"', from=1-3, to=2-3]
					\arrow["{1\circ\Gamma_x\circ1}", from=1-5, to=2-5]
					\arrow[""{name=2, anchor=center, inner sep=0}, "{1\circ\alpha^o_f}"', from=2-1, to=2-3]
					\arrow[""{name=3, anchor=center, inner sep=0}, "{(\beta')^o_f\circ1}"', from=2-3, to=2-5]
					\arrow[between={0.3}{0.7}, Rightarrow, dashed, from=0, to=2]
					\arrow["{\Gamma^o_f\circ1}", between={0.3}{0.7}, Rightarrow, from=1, to=3]
				\end{tikzcd}
			}
		\end{center}
	\end{itemize}
	

  \item \emph{Perturbations and interchange.} 
Unbiased pseudo perturbations in $\psden(\ca,\cb)$ are simply families of $3$-cells $\Theta_x$ of $\cb$, equally giving a perturbation between their lax or oplax parts.  Since all compositions involving perturbations in $\psden(\ca,\cb)$ and likewise in $\oplax(\ca,\cb)$ are simply computed componentwise in $\cb$, it follows immediately that they are preserved by $U^o_\nb$.  Similarly the interchange cells are preserved by $U^o_\nb$ since they are simply computed componentwise in $\cb$.
\end{itemize}
\end{proof}

Furthermore, recall from Section 6.2 of \cite{BourkeLobbia} that $\psd(-,-) \colon \GCat_p^{op} \times \GCat_p \to \GCat_p$ is a functor.  Therefore transporting the functor structure along the componentwise isomorphisms $U^p_\nb \colon \psden(\ca,\cb) \to \psd(\ca,\cb)$ we immediately obtain:

\begin{cor}\label{uniquenatural}
The Gray-categories $\psden(\ca,\cb)$ form the object assignment of a unique functor $\psden(-,-)\colon\GCat_p^{op} \times \GCat_p \to \GCat_p$ such that the $U^p_\nb \colon \psden(\ca,\cb) \to \psd(\ca,\cb)$ are natural in $\ca$ and $\cb$.
\end{cor}

\begin{prop}\label{naturality}
The components $U^o_\nb \colon \psden(\ca,\cb) \to \oplax(\ca,\cb)$ are natural in $\ca$ and $\cb$.  
\end{prop}
\begin{proof}
Both $\psden(-,-)$ and $\oplax(-,-)$ act by whiskering $n$-transfors by pseudomaps.   Just like in \Cref{prop:forg-fct-psden-oplax}, by uniqueness of the functorial structure on $\psden(-,-)$ in \Cref{uniquenatural}, it suffices to show that computing left/right whiskering as for $\psd(-,-)$ and $\oplax(-,-)$ results in unbiased pseudo $n$-transfors. 
%
%
%
%
%
\begin{itemize}
	\item \emph{Unbiased pseudo-transformations.} Let 	$\alpha=(\alpha^l,\alpha^o,\aexi^\alpha_{-})\colon F\to G \in \psden(\ca,\cb)$ be an unbiased pseudo-transformation and let $P\colon\ca'\to\ca$ and $R\colon\cb\to\cb'$ be pseudomaps.  Then, we need to show that $(\alpha^l\circ P,\alpha^o\circ P,\aexi^\alpha_{P-})\colon FP\to GP$ and $(R\circ\alpha^l,R\circ\alpha^o,R\circ\aexi^\alpha_{-})\colon RF\to RG$ are unbiased pseudo-transformations, where $R\circ\aexi^\alpha_{-}$ is the adjoint equivalences below.\footnote{See \cite[Section 6.1]{BourkeLobbiaCorrected}, where these are defined in the actions of $\psd(P,\cb)$ and $\psd(\ca,R)$.}
	\[\begin{tikzcd}[ampersand replacement=\&]
		{RGfR\alpha_x} \& {R(Gf\alpha_x)} \&\& {R(\alpha_yFf)} \& {R\alpha_yRFf}
		\arrow[""{name=0, anchor=center, inner sep=0}, "{R^2}", curve={height=-12pt}, from=1-1, to=1-2]
		\arrow[""{name=1, anchor=center, inner sep=0}, "{({R^2})^{-1}}", curve={height=-12pt}, from=1-2, to=1-1]
		\arrow[""{name=2, anchor=center, inner sep=0}, "{R(\alpha^l_f)}", curve={height=-12pt}, from=1-2, to=1-4]
		\arrow[""{name=3, anchor=center, inner sep=0}, "{R({\alpha^o_f})}", curve={height=-12pt}, from=1-4, to=1-2]
		\arrow[""{name=4, anchor=center, inner sep=0}, "{({R^2})^{-1}}", curve={height=-12pt}, from=1-4, to=1-5]
		\arrow[""{name=5, anchor=center, inner sep=0}, "{R^2}", curve={height=-12pt}, from=1-5, to=1-4]
		\arrow["\top"{description}, draw=none, from=1, to=0]
		\arrow["\top"{description}, draw=none, from=3, to=2]
		\arrow["\top"{description}, draw=none, from=5, to=4]
	\end{tikzcd}\]
	We need to check that the components of $\alpha\circ P$ and $R\circ\alpha$ at 2-cells $\phi$ and at composable 1-cells $f_1,f_2$ (in $\ca'$ and $\ca$ respectively), are mates. 
	\begin{itemize}
		\item For $\alpha\circ P$, we have $(\alpha^l\circ P)_\phi := \alpha^l_{P\phi}$ and $(\alpha^o\circ P)_\phi := \alpha^o_{P\phi}$ and so the mate condition holds as since $\alpha\in\psden(\ca,\cb)$.    The component $(\alpha^l\circ P)^2_{f_2,f_1}$ is the pasting of $(\alpha^l)^2_{Pf_2,Pf_1}$ and $\alpha^l_{P^2_{f_2,f_1}}$, whilst $(\alpha^o\circ P)_2^{f_2,f_1}$ is the pasting of $(\alpha^o)^2_{Pf_2,Pf_1}$ and $\alpha^o_{P^2_{f_2,f_1}}$ (see \cite[Appendix C]{BourkeLobbia} for the lax case).
%
		Since these are mates of each other (because $\alpha\in\psden(\ca,\cb)$), then to conclude that $\alpha^l\circ P$ and $\alpha^o\circ P$ are mates it suffices to use one triangle identity of the adjunction $\aexi^\alpha_{Pf_2Pf_1}$. 
		\item For $R\circ\alpha$ we use that the components of both $R\circ\alpha^l$ and $R\circ\alpha^o$ are defined as images under $R$ of the corresponding components whiskered by instances of $R^2_{-.-}$ or its inverse (see \cite[Appendix C]{BourkeLobbia} for the lax case). 
		Then, we can conclude that the required mate equations hold because they hold for $\alpha$ and because  $R$ is locally a 2-functor. 
	\end{itemize}
	
	\item \emph{Unbiased pseudo-modifications.} For $n=2,3$, let $\Lambda=(\Lambda^l,\Lambda^o)\in\psden(\ca,\cb)$ be an unbiased pseudo-modification, and $P\colon\ca'\to\ca$ and $R\colon\cb\to\cb'$ be pseudomaps. 
	We need to show that $(\Lambda^l\circ P,\Lambda^o\circ P)\in\psden(\ca',\cb)$ and $(R\circ\Lambda^l,R\circ\Lambda^o)\in\psden(\ca,\cb')$ are unbiased pseudo-transformations.
	\begin{itemize}
		\item For $\Lambda\circ P$ it is enough to note that all the components of $\Lambda^l\circ P$ and $\Lambda^o\circ P$ are components of $\Lambda^l$ and $\Lambda^o$ at a $0/1$-cell in the image of $P$.  See \cite[Appendix C]{BourkeLobbia} for the lax case.  As such, they satisfy the required mate condition as a special case of the mate condition for $\Lambda^l$ and $\Lambda^o$.  
		\item For $R\circ\Lambda$ we recall that all the components of both $R\circ\Lambda^l$ and $R\circ\Lambda^o$ are defined as images of the corresponding components whiskered by instances of $R^2_{-.-}$ or its inverse (see \cite[Appendix C]{BourkeLobbia} for the lax case). 
		Then, we can conclude that the equality hold because they hold for $\Lambda$ and $R$ is locally a 2-functor. 
\end{itemize}
	\item \emph{Unbiased pseudo-perturbations.} This case is trivial since there is no mate condition to verify.
\end{itemize}

\end{proof}

\black

\black 

\section{Background on symmetric skew monoidal categories, skew multicategories etc}\label{sect:prelim-skew-brd}

In this section, we review the necessary background to prove that the flexible skew structure $(\GCat,\otimes_p,\mathbf{1})$ is symmetric. 
First, we recall the notions of symmetric skew monoidal categories and symmetric skew multicategories \cite{BourkeLack:symmetric}. 
We will also need the results of \cite{Lob:short-multi}, which show that, in our setting, it suffices to construct three isomorphisms and verify a small number of equations in order to obtain a symmetric skew multicategory. 
We then recall the main definitions of the skew multicategory $\GCatmp$ corresponding to the flexible skew monoidal structure under the equivalence \cite[Theorem~6.1]{BourkeLack:skew}. 
Finally, we adapt the compact definition of pseudo-binary map using the unbiased notions introduced in \Cref{sect:unbiased}.

\subsection{Symmetric skew monoidal categories}\label{sec:brd-skew-multi-mon}

Skew monoidal categories were introduced independently in \cite{Altenkirch} and \cite{Szlachanyi2012Skew}.   A (left) skew monoidal category is a category $\cc$ together with a functor $\otimes\colon\cc\times \cc \to \cc$, a unit object $i \in \cc$, and natural families $\alpha \colon(a \otimes b) \otimes c \to a \otimes (b \otimes c)$, $l \colon i \otimes a \to a$ and $r \colon a \to a \otimes i$ satisfying five axioms described in \cite{Szlachanyi2012Skew}.  A monoidal category is precisely a skew monoidal category in which of the constraints $\alpha, l$ and $r$ are invertible.


Whilst a braiding on a monoidal category consists of a natural isomorphism $a\otimes b\cong b\otimes a$ satisfying two equations, a braiding for a skew monoidal category consists of an invertible natural transformation $s_{x,a,b}\colon(x\otimes a)\otimes b\to (x\otimes b)\otimes a$
satisfying four equations (see \cite[Definition~2.2]{BourkeLack:symmetric} for the details). 
As in the classical case, a braiding on a skew monoidal category is a symmetry when $s$ is the inverse of itself, i.e. if $s_{x,b,z}s_{x,a,b}=1_{(x\otimes a)\otimes b}$. 

\subsection {Symmetric skew monoidal categories versus symmetric skew multicategories}
Just as monoidal categories correspond to certain multicategories (the representable ones \cite{Hermida2000Representable}), skew monoidal categories correspond to certain \emph{skew multicategories} (the left representable ones) \cite{BourkeLack:skew}.  

Skew multicategories, introduced in \cite{BourkeLack:skew} differ from classical multicategories in that they have two kinds of multimap (tight and loose) together with a function that views tight multimaps as loose ones, all satisfying some axioms.   The concept simplifies significantly when these functions are subset inclusions, which is the case of interest here. 
In that case (which we hence assume) a skew multicategory $\mc$ consists of a multicategory $\mc^l$, whose multimaps we call \emph{loose}, together with a subset $$\mc^t_n(a_1,\ldots,a_n;b) \subseteq \mc^l_n(a_1,\ldots,a_n;b)$$ for any $n >0$ of \emph{tight} $n$-ary multimaps.   Furthermore the unary identities are required to be tight and certain substitutions respect tightness\begin{footnote}{Although the reader won't need to know these conditions here, they state that $f \circ_i g$ is tight if $f$ and $g$ are tight and $i=1$ or $f$ is tight and $i >1$.}\end{footnote}. See \cite[2.3]{BourkeLobbia} or \cite[Definition~4.2]{BourkeLack:skew} for further details on skew multicategories.

A braiding (or symmetry) for a skew multicategory $\mc$ is a simply a braiding (or symmetry) on the the multicategory $\mc^l$ of loose multimaps satisfying an additional condition concerning the preservation of tight multimaps.  In the case of a symmetry, this can be easily stated: the actions $\sigma.- \colon \mc^l_n(a_1,\ldots,a_n;b) \to \mc^l_n(a_{\sigma1},\ldots,a_{\sigma n};b)$ of elements of the symmetric group $\sigma \in S_n$ satisfying $\sigma 1 = 1$ preserve tight multimaps.

In fact braidings (or symmetries) on a skew monoidal category correspond bijectively to braidings (or symmetries) on the associated skew multicategory.  This was proved in \cite[Theorem~5.9]{BourkeLack:symmetric}.

%
%

\subsection{Short braidings}
In this paper our goal is to describe a symmetry on the flexible skew monoidal structure $(\GCat,\otimes_p,1)$ or, equivalently, by  \cite[Theorem~5.9]{BourkeLack:symmetric} a symmetry on its associated skew multicategory $\GCatmp$.  Now in \cite{BourkeLobbiaCorrected} we described $\GCatmp$ in elementary terms for $n$-ary multimaps up to dimension $4$.  And in fact the second-named author showed in \cite{Lob:short-multi} that symmetries on such left representable skew multicategories can be recognised by their actions on multimaps up to dimension $4$.  This approach will be the key to our establishing the symmetry on $(\GCat,\otimes_p,1)$ and we now describe it in detail.

%
%
%
%
%
Essentially, \cite[Theorem~5.22]{Lob:short-multi} shows that in order to define a braiding (or symmetry) on a left representable skew multicategory $\mc$ it is enough to define a \emph{short braiding}, which consists of natural isomorphisms
\begin{equation}\label{short-brd-data}
\renewcommand{\arraystretch}{1.25}
	\begin{tabular}{>{$}c<{$} >{$}c<{$} >{$}c<{$} >{$}c<{$}}
		\beta^3_2\colon &
		\mc_3^t(a_1,a_2,a_3;b) &
		\to &
		\mc_3^t(a_1,a_3,a_2;b) \\
		\beta^4_2\colon &
		\mc_4^t(a_1,a_2,a_3,a_4;b) &
		\to &
		\mc_4^t(a_1,a_3,a_2,a_4;b) \\
		\beta^4_3\colon &
		\mc_4^t(a_1,a_2,a_3,a_4;b) &
		\to &
		\mc_4^t(a_1,a_2,a_4,a_3;b)
	\end{tabular}
\end{equation}
satisfying the following six axioms
\begin{itemize}
	\item for any tight 4-ary map $h\in\mc_4^l(a_1,a_2,a_3,a_4;b)$, 
	\begin{equation}\label{ax:sh-braid-4-ary}
		\beta^4_2\beta^4_3\beta^4_2(h)=\beta^4_3\beta^4_2\beta^4_3(h)
	\end{equation}
	
	\item for any tight binary map $g\colon b_1,b_2\to c$ and tight ternary map $f\colon a_1,a_2,a_3\to b_i$,
	\begin{minipage}{0.5\textwidth}
		\begin{equation}\label{ax:sh-braid-3-in-2-first}
			g\circ_1\beta^3_2(f)=\beta^4_2(g\circ_1f)
		\end{equation}
	\end{minipage} 
	\begin{minipage}{0.5\textwidth}
		\begin{equation}\label{ax:sh-braid-3-in-2-secd}
			g\circ_2\beta^3_2(f)=\beta^4_3(g\circ_2f)
		\end{equation}
	\end{minipage}\vspace{0.2cm}
	
	\item for any tight ternary map $g\colon b_1,b_2,b_3\to c$ and tight binary map $f\colon a_1,a_2\to$~$b_i$,
	\begin{minipage}{0.5\textwidth}
		\begin{equation}\label{ax:sh-braid-2-in-3-first}
			\beta^4_3(g\circ_1f)=\beta^3_2(g)\circ_1f
		\end{equation}
	\end{minipage}
	\begin{minipage}{0.5\textwidth}
		\begin{equation}\label{ax:sh-braid-2-in-3-secd}
			\beta^4_2\beta^4_3(g\circ_2f)=\beta^3_2(g)\circ_3f
		\end{equation}
	\end{minipage}
	\begin{equation}\label{ax:sh-braid-2-in-3-thrd}
		\beta^4_3\beta^4_2(g\circ_3f)=\beta^3_2(g)\circ_2f
	\end{equation}
\end{itemize}
The short braiding is called a \textbf{short symmetry} if, moreover, $\beta^3_2\beta^3_2=1$.  For further details on the low-dimensional approach to skew multicategories, we refer the reader to \cite{Lob:short-multi}.

\red


\black 

\subsection{The Closed Skew Multicategory $\GCatmp$}\label{sec:multi-psd}

Now, we review the definition of the left representable skew multicategory $\GCatmp$ corresponding to the flexible skew monoidal structure on $\GCat$.
Since we will only need the structure up to 4-ary multimaps, we will recall these definitions.
In \cite{BourkeLobbia}, pseudo multimaps were constructed starting from lax ones, which form the skew multicategory $\GCatm$. 
For this reason let us first recall the data involved in these lax multimaps.
In particular, for binary maps we use the presentation called \emph{the compact definition} in \cite[Section~4.5]{BourkeLobbia}.
\begin{itemize}
	\item Nullary lax maps $a\colon\diamond\to\ca$ are just objects $a\in\ca$.
	\item A (loose) lax unary map $F\colon\ca\to\cb$ is a pseudomap (see \Cref{sec:psd-transf}).  The tight lax unary maps are the Gray-functors. 
	\item A (loose) lax binary map $F\colon\ca,\cb\to\cc$ consists of
	\begin{enumerate}
		\item[(1.1)] at an object $a\in\catA$, a pseudomap $F^a\colon\catB\to\catC$,
		\item[(1.2)] at an object $b\in\catB$, a pseudomap $F_b\colon\catA\to\catC$,
		\item[(2.1)] at a 1-cell $A\colon a\to a'\in\catA$, a lax transformation $F^A\colon F^a\to F^{a'}$,
		\item[(2.2)] at a 1-cell $B\colon b\to b'\in\catB$, an oplax transformation $F_B\colon F_b\to F_{b'}$, 
	\end{enumerate}
	satisfying the axioms listed in \cite[Section~4.5]{BourkeLobbia} or \Cref{sec:psd-bin-cmpt} below for full details.  
	
One of the axioms required is that, on objects, we have $F^a(b)=F_b(a)=:F(a,b)$.  Another is that at $1$-cells $A \colon a \to a^{\prime}$ and $B \colon b \to b^{\prime}$, we have $F^A(B) = F_B(A) := A_B$, so that their common value is a $2$-cell as below
\[\begin{tikzcd}
	{F(a,b)} & {F(a',b)} \\
	{F(a,b')} & {F(a',b')}
	\arrow[""{name=0, anchor=center, inner sep=0}, "{F_b(A)}", from=1-1, to=1-2]
	\arrow["{F^a(B)}"', from=1-1, to=2-1]
	\arrow["{F^{a'}(B)}", from=1-2, to=2-2]
	\arrow[""{name=1, anchor=center, inner sep=0}, "{F_{b'}(A)}"', from=2-1, to=2-2]
	\arrow["{A_B}"', between={0.2}{0.8}, Rightarrow, from=0, to=1]
\end{tikzcd}\]
This is a central piece of structure for the lax bimap $F$.

Lax binary maps are called tight when each $F_b$ is a Gray-functor and each component at composable 1-cells $(F_B)^{A_2,A_1}_2$ of the oplax transformation $F_B$ is the identity.  
	
 \item A (loose) lax ternary map $F\colon\ca,\cb,\cc\to\cd$ consists of 
	\begin{enumerate}[(1)]
		\item at an object $a\in\catA$ a loose binary map $F^a\colon\catB,\catC\to\catD$, at $b\in\catB$ a loose binary map $F\midscript{b}\colon\catA,\catC\to\catD$ and at $c\in\catC$ a loose binary map $F_c\colon\catA,\catB\to\catD$ such that
		\begin{center} 
			$F^a(b,-)=F\ts?{b}(a,-)=:F^a_b\quad F^a(-,c)=F_c(a,-)=:F^a_c$ \\ 
			$F\midscript{b}(-,c)=F_c(-,b)=:F^b_c$,
		\end{center}
		
		\item at 1-cells $A\colon a\to a'$ in $\ca$, $B\colon b\to b'$ in $\cb$ and $C\colon c\to c'$ in $\cc$, a 3-cell 
		\begin{center}
			\begin{tikzpicture}
				\node (T) at (0,4) {$F^{a'}_{b'}(C)F^{a'}_{c}(B)F^{b}_{c}(A)$};
				\node (N) at (-2,2) {$F^{a'}_{c'}(B)F^{a'}_{b}(C)F^{b}_{c}(A)$};
				\node (T') at (5,4) {$F^{a'}_{b'}(C)F^{b'}_{c}(A)F^{a}_{c}(B)$};
				\node (TS0Y) at (7,2) {$F^{b'}_{c'}(A)F^{a}_{b'}(C)F^{a}_{c}(B)$};
				\node (ST') at (0,0) {$F^{a'}_{c'}(B)F^{b}_{c'}(A)F^{a}_{b}(C)$};
				\node (TS0') at (5,0) {$F^{b'}_{c'}(A)F^{a}_{c'}(B)F^{a}_{b}(C)$,};
				
				\path[->] 
				(ST') edge node[scale=.7] (y) [below] {${A_B}^{c'}\circ 1$} (TS0')
				(T) edge node[scale=.7] [left, xshift=-0.2cm] {${B_C}^{a'}\circ 1$} (N)
				edge node[scale=.7] (x) [above] {$1\circ {A_B}^{c}$} (T')
				(N) edge node[scale=.7] [left, xshift=-0.2cm] {$1\circ {A_C}^{b}$} (ST')
				(T') edge node[scale=.7] [right, xshift=0.2cm] {${A_C}^{b'}\circ1$} (TS0Y)
				(TS0Y) edge node[scale=.7] [right, xshift=0.2cm] {$1\circ {B_C}^{a}$} (TS0');
				\draw[-{Implies},double distance=1.5pt,shorten >=45pt,shorten <=45pt] (x) to node[scale=.7] [right, xshift=0.1cm] {$(A\mid B\mid C)$} (y);
			\end{tikzpicture}
		\end{center}
	\end{enumerate}
	called the incubator, satisfying the nine axioms described in \cite[Appendix~D.2]{BourkeLobbia}. 
	We say that the lax ternary map $F$ is tight if $F\midscript{b}$ and $F_c$ are tight lax binary maps.
	\item A (loose) lax 4-ary map $F\colon\ca,\cb,\cc,\cd\rightarrow\ce$ consists of:
	\begin{itemize}
		\item at $a\in\ca$, a lax ternary map $F^a\colon\cb,\ca,\cd\to\ce$;
		\item at $b\in\cb$, a lax ternary map $F\midscript{b}\colon\ca,\cc,\cd\to\ce$;
		\item at $c\in\ca$, a lax ternary map $F\midscript{c}\colon\ca,\cb,\cd\to\ce$;
		\item at $d\in\cd$, a lax ternary map $F_d\colon\ca,\cb,\cc\to\ce$;
	\end{itemize}
	compatible with each other (i.e. satisfying the axioms described in \cite[Section~4.17]{BourkeLobbia}) and satisfying the mecon axiom (see \cite[Appendix~D.3]{BourkeLobbia}). 
	A lax 4-ary map $F$ is called tight when all $F\midscript{b},\,F\midscript{c}$ and $F_d$ are tight lax ternary maps. 
	\end{itemize}

Now we are ready to review the definitions regarding the skew multicategory $\GCatmp$.
Objects, nullary and unary maps are the same as in $\GCatm$. 
Then, 2-, 3-, and 4-ary maps are defined as follows. 
\begin{itemize}
	\item A pseudo binary map $\mathbf F\colon\ca,\cb\to\cc$ consists of a lax binary map $F\colon\ca,\cb\to\cc$ equipped with adjoint equivalences $\aexi^A_B:=({A_B}^o\dashv A_B)$ \black 
	\[\begin{tikzcd}[ampersand replacement=\&]
		{F^{a'}_{B}F^{A}_{b}} \&\& {F^{A}_{b'}F^{a}_{B}}
		\arrow[""{name=0, anchor=center, inner sep=0}, "{A_B}", curve={height=-12pt}, from=1-1, to=1-3]
		\arrow[""{name=1, anchor=center, inner sep=0}, "{{A^o_B}}", curve={height=-12pt}, from=1-3, to=1-1]
		\arrow["\top"{description}, draw=none, from=1, to=0]
	\end{tikzcd}\]
	in $\cc(F^a_b,F^{a'}_{b'})$ such that:
	\begin{enumerate}[(i)]
		\item given objects $a\in\ca$ and $b\in\cb$, the adjoint equivalences $\aexi^A_{1_b}$ and $\aexi^{1_a}_B$ are identities;;
		\item given 1-cells $A\in\ca$ and $B\in\cb$ and 2-cells $\alpha\in\ca$ and $\beta\in\cb$, the $3$-cells $\alpha_B$ and $A_\beta$ in $\cc$ are invertible. 
	\end{enumerate}

	\item A pseudo ternary map $\mathbf F\colon\ca,\cb,\cc\to\cd$ consists of a lax ternary map $\mathbf F\colon\ca,\cb,\cc\to\cd$  with invertible incubators $(A\mid B\mid C)$ together with pseudo-binary maps $\mathbf{F}^a$, $\mathbf{F}\midscript{b}$ and $\mathbf{F}_{c}$ whose underlying lax binary maps are ${F}^a$, ${F}\midscript{b}$ and ${F}_{c}$
	\item A pseudo 4-ary map $\mathbf{F}\colon\ca,\cb,\cc,\cd\to\ce$ consists of a lax 4-ary map $F\colon\ca,\cb,\cc,\cd\to\ce$ together with pseudo-ternary maps $\mathbf{F}^a$, $\mathbf{F}\midscript{b}$, $\mathbf{F}\midscript{c}$ and $\mathbf{F}_d$ compatible with each other (in the same sense as a lax 4-ary map, see \cite[Section 6.3]{BourkeLobbiaCorrected}) 
	and whose underlying lax ternary maps are $F^a$, $F\midscript{b}$, $F\midscript{c}$ and $F_d$.\black 
\end{itemize} 

A pseudo binary map $\mathbf{F}$ is tight if its underlying lax binary map $F$ is, and the following equation of adjoint equivalences holds
$$\aexi^{A_2A_1}_B=(1\circ\aexi^{A_1}_B)\cdot(\aexi^{A_2}_B\circ1).$$ 
A pseudo ternary map $\mathbf{F}$ is tight just when $\mathbf{F}\midscript{b}$ and $\mathbf{F}_{c}$ are tight pseudo binary maps.  Similarly, a pseudo ternary map $\mathbf{F}$ is tight just when $\mathbf{F}\midscript{b}$, $\mathbf{F}\midscript{c}$ and $\mathbf{F}_d$ are tight ternary maps.  (\emph{Note that for the definitions of tight pseudo multimaps, we emphasise that we are using the corrected definition of \cite[Section 6.3]{BourkeLobbiaCorrected}, in which the equation concerning the equality of adjoint equivalences was added.})

Finally, we recall in \Cref{Tab:binbin} the substitution of binary maps into binary maps.
\begin{table}[ht]
	\caption{Explicit substitutions of binary into binary}
	\begin{center}
					\begin{tabular}{cc:c c c c c c c c}
					$\mathbf{G} \circ_1 \mathbf{F}$ &&
					$\mathbf{G} \circ_1 \mathbf{F}^a$ && $\mathbf{G} \circ_1 \mathbf{F}_b$ && $\mathbf{G}_c \circ_1 \mathbf{F}$ && $\textnormal{Incubator in Appendix}~\ref{app:inc-sub-b-in-b-1st}$ \\
					$\mathbf{G} \circ_2 \mathbf{F}$ &&
					$\mathbf{G}^a \circ_1 \mathbf{F}$ && $\mathbf{G} \circ_2 \mathbf{F}^b$ && $\mathbf{G} \circ_2 \mathbf{F}_c$ && $\textnormal{Incubator in Appendix}~\ref{app:inc-sub-b-in-b-2nd}$ \\
				\end{tabular}
		\end{center}
	\label{Tab:binbin}
\end{table}

Next we recall those properties we will need about the left closed structure on $\GCatmp$, whose internal hom is $\psd(\ca,\cb)$.

Before introducing the left closedness bijections, note that given an object $a \in \ca$, there is an evaluation Gray-functor $ev_a\colon\psd(\ca,\cb)\to\cb$: $ev_a$ sends a pseudomap $F$ to its value $Fa$, a pseudo-transformation $\alpha$ to its component $\alpha_a$, a pseudo-modification $\Gamma$ to the 2-cell $\Gamma_a$, and a pseudo-perturbation $\theta$ to the 3-cell $\theta_a$.

\begin{lemma}[Theorem~6.2 of \cite{BourkeLobbiaCorrected}]
	\label{lem:left-closedness-psd}
There are natural bijections
\[
\begin{aligned}
\lambda_2\colon\GCatmp_1^l(\ca;\psd(\cb,\cx))
  &\cong \GCatmp_2^l(\ca,\cb;\cx),
  \\
\lambda_3\colon\GCatmp_2^l(\,\ca,\cb;\psd(\cc,\cx)\,)
  &\cong \GCatmp_3^l(\ca,\cb,\cc;\cx) \quad\text{and} \\
\lambda_4\colon\GCatmp_3^l(\,\ca,\cb,\cc;\psd(\cd,\cx)\,)
  &\cong \GCatmp_4^l(\ca,\cb,\cc,\cd;\cx) \footnotemark
\end{aligned}
\]
\footnotetext{Notice that in \cite{BourkeLobbiaCorrected}
these isomorphisms were denoted by $\lambda_n$ (for $n=2,3$)
to distinguish them from the lax versions.
We don't do it here since we only consider the pseudo version.}
which restrict to tight maps on both sides. 
Moreover, for $F\colon\ca\to\psd(\cb,\cx)$, $G\colon\ca,\cb\to\psd(\cc,\cx)$ and $\mathbf{H}\colon\ca,\cb,\cc,\cd\to\cx$, they satisfy the following properties by construction:
\begin{enumerate}[(i)]
	\item $(\lambda_2\mathbf{F})^{a}=\mathbf{F}(a)$ and $(\lambda_2\mathbf{F})_b=ev_b\circ \mathbf{F}$;
	
	\item $(\lambda_3\mathbf{G})^{a}=\lambda_2(\mathbf{G}^{a})$, $(\lambda_3\mathbf{G})\midscript{b}=\lambda_2(\mathbf{G}_{b})$ and $(\lambda_3\mathbf{G})_c=ev_c\circ\mathbf{G}$;
	
	\item the incubator $(A\mid B\mid C)$ of $\lambda_3\mathbf{G}$ is given as the component of the modification $A_B\in\psd(\cc,\cd)$ at a 1-cell $C$;\footnote{Here, $A_B$ is the adjoint equivalence of $\mathbf{G}$ as a pseudo binary map, hence an adjoint equivalence in $\psd(\cc,\cd)[G^a_b,G^{a'}_{b'}]$. See Proposition~4.6 and Theorem~6.2 of \cite{BourkeLobbiaCorrected} for more details.}
	
	\item $(\lambda_4\mathbf{H})^{a}=\lambda_2(\mathbf{H}^{a})$, $(\lambda_4\mathbf{H})\midscript{b}=\lambda_3(\mathbf{H}\midscript{b})$, $(\lambda_4\mathbf{H})\midscript{c}=\lambda_3(\mathbf{H}_{c})$ and $(\lambda_4\mathbf{H})_d=ev_d\circ\mathbf{H}$.
\end{enumerate}
\end{lemma}

\begin{rmk}
	\label{rmk:subst-in-terms-of-lambdas}
In Theorem~6.2 of \cite{BourkeLobbiaCorrected} the definition of substitution of binary into the first variable of a binary map was defined in terms using the left closedness isomorphisms of \Cref{lem:left-closedness-psd}.  Namely, for any loose pseudo binary maps $G\colon \cx,\cc\to\cd$ and $F\colon\ca,\cb\to\cx$, $G\circ_1F:=\lambda_3(\lambda_2^{-1}G\circ F)$. 
\end{rmk}
\black

\subsection{Unbiased pseudo binary maps}
In this section we describe the notion of an \emph{unbiased} pseudo binary map.  
This is equivalent to the definition of a pseudo binary map, as described above, but obviously symmetric. 
As such, it will be useful in constructing the binary braiding in \Cref{sec:bin-braiding} below.

\subsubsection{Pseudo binary maps in compact form}
\label{sec:psd-bin-cmpt}
We will give the definition of an \emph{unbiased} pseudo binary map in a \emph{compact form}, building on the compact definition of lax binary maps in given \cite[Section~4.5]{BourkeLobbiaCorrected}.  

Recall from \Cref{sec:multi-psd} that a loose lax binary map $F\colon\catA,\catB\rightarrow\catC$ \emph{in compact form} consists of:
\begin{enumerate}
	\item[(1.1)] at an object $a\in\catA$, a pseudomap $F^a\colon\catB\to\catC$,
	\item[(1.2)] at an object $b\in\catB$, a pseudomap $F_b\colon\catA\to\catC$,
\end{enumerate}
such that $F^a(b)=F_b(a)$, 
together with:
\begin{enumerate}
	\item[(2.1)] at a 1-cell $A\colon a\to a'\in\catA$, a lax transformation $F^A\colon F^a\to F^{a'}$,
	\item[(2.2)] at a 1-cell $B\colon b\to b'\in\catB$, an oplax transformation $F_B\colon F_b\to F_{b'}$, 
\end{enumerate}
such that 
\begin{center}
	$F^A(b)=F_b(A)$, $F_B(a)=F^a(B)$, \\
	$(F^A)_B=(F_B)_A=:A_B$.
\end{center}
Moreover we require the data
\begin{equation}
	\label{eq:compct-data-mod}
	\begin{aligned}
		(F^\alpha)_b      &:= F_b(\alpha), &
		(F_\beta)^a       &:= F^a(\beta), &
		(F^{A_2,A_1})^b  &:= (F_b)^{A_2,A_1}, &
		(F_{B_2,B_1})^a  &:= (F^a)_{B_2,B_1}, \\
		(F^\alpha)_B      &:= F_B(\alpha), &
		(F_\beta)^A       &:= (F^A)_\beta, &
		(F^{A_2,A_1})^B  &:= (F_B)^{A_2,A_1}, &
		(F_{B_2,B_1})^A  &:= (F^A)_{B_2,B_1}.
	\end{aligned}
\end{equation}
to form two modifications 
\begin{center}
	$F^\alpha\colon F^A\to F^{A'}$ in $\lax(\cb,\cc)$, \\
	$F_\beta\colon F_B\to F_{B'}$ in $\oplax(\ca,\cc)$, 
\end{center}
and two invertible modifications
\begin{center}
	$F^{A_2,A_1}\colon F^{A_2}\cdot F^{A_1}\to F^{A_2A_1}$ in $\lax(\cb,\cc)$, \\
	$F_{B_2,B_1}\colon F_{B_2}\cdot F_{B_1}\to F_{B_2B_1}$ in $\oplax(\ca,\cc)$.
\end{center} 

A loose pseudo binary map in compact form is specified by the above together with
$\aexi^A_B:=({A_B}^o\dashv A_B)$ such that:
 \begin{enumerate}[(i)]
 	\item given objects $a\in\ca$ and $b\in\cb$, the adjoint equivalences $\aexi^A_{1_b}$ and $\aexi^{1_a}_B$ are identities;
 	\item  the modifications $F^\alpha$ and $F_\beta$ have invertible components at $1$-cells.
 \end{enumerate}
  
 \subsubsection{Unbiased pseudo binary maps}\label{sec:compact-enh-psd-bin}
 
 We give the definition in compact form.
 
\label{sec:compact-enh-psd-bin}
  \begin{defn}
 	\label{def:unb-bin-maps}
	Let $\ca,\cb$ and $\cc$ be Gray-categories. An unbiased loose pseudo binary map $\textbf{F}\colon\ca,\cb\to\cc$ consists of:
\begin{enumerate}
	\item[(1.1)] for any object $a\in\catA$, a pseudomap $F^a\colon\catB\to\catC$,
	\item[(1.2)] for any object $b\in\catB$, a pseudomap $F_b\colon\catA\to\catC$,
\end{enumerate}
such that $F^a(b)=F_b(a)$, 
together with
\begin{enumerate}
	\item[(2.1)$_{nb}$] for any 1-cell $A\colon a\to a'\in\catA$, an unbiased pseudo-transformation $$\mathbf{F}^A=(F^A,(F^A)^o,\aexi^A_{-})\colon F^a\to F^{a'},$$
	\item[(2.2)$_{nb}$] at a 1-cell $B\colon b\to b'\in\catB$, an unbiased pseudo-transformation $$\mathbf{F}_B=((F_B)^l,F_B,\aexi^B_{-})\colon F_b\to F_{b'},$$ 
\end{enumerate}
such that 
\begin{align*}
	F^A(b)&=F_b(A)  &F_B(a)&=F^a(B)  	&(F^A)_B=(F_B)_A=:A_B\\
	& &  \aexi^A_B &={\aexi^B_A}^\rev &
\end{align*}
Moreover, we require the data as in \eqref{eq:compct-data-mod}\footnote{By \Cref{prop:equ-psd-mod}, the data specified is enough to determine an unbiased pseudo-modification.}
to form two unbiased pseudo-modifications 
\[
\begin{alignedat}{2}
	\mathbf{F}^\alpha
	&\colon (F^\alpha,(F^\alpha)^o)\colon F^A \to F^{A'}
	&& \text{in } \psd_{nb}(\cb,\cc), \\
	\mathbf{F}_\beta
	&\colon ((F_\beta)^l,F_\beta)\colon F_B \to F_{B'}
	&& \text{in } \psd_{nb}(\ca,\cc).
\end{alignedat}
\]
and two invertible unbiased pseudo-modifications
\[
\begin{alignedat}{2}
	F^{A_2,A_1}
	&\colon F^{A_2}\cdot F^{A_1}\to F^{A_2A_1}
	&& \text{in } \psden(\cb,\cc), \\
	F_{B_2,B_1}
	&\colon F_{B_2}\cdot F_{B_1}\to F_{B_2B_1}
	&& \text{in } \psden(\cb,\cc).
\end{alignedat}
\]

\end{defn}

Let $(\GCatmp_{nb})_2^l(\ca,\cb;\cc)$ denote the set of unbiased loose pseudo binary maps.  There is an evident forgetful function $U \colon (\GCatmp_{nb})_2^l(\ca,\cb;\cc)\to\GCatmp_2^l(\ca,\cb;\cc)$ sending $\mathbf F$ to the pseudo binary map with components $(F^a,F_b,F^A,F_B,\aexi^A_B)$.


\begin{prop}\label{prop:compact-unbiased-psd-binary}
	The forgetful function $U \colon (\GCatmp_{nb})_2^l(\ca,\cb;\cc)\to\GCatmp_2^l(\ca,\cb;\cc)$ is a bijection.
\end{prop}

\begin{proof}
We must describe the inverse to $U$.  Consider a pseudo binary map $F$ with data $(F^a,F_b,F^A,F_B,\aexi^A_B)$.  Thanks to the adjoint equivalences $\aexi^A_B$, by \Cref{prop:equ-psd-transf}, the lax and oplax transformations $F^A$ and $F_B$ admit unique extensions to unbiased pseudo-transformations $\mathbf{F}^A=(F^A,(F^A)^o,\aexi^A_{-})\colon F^a\to F^{a'}$ and $\mathbf{F}_B=((F_B)^l,F_B,\aexi^B_{-})\colon F_b\to F_{b'}$.  Since the construction starting from an oplax transformation takes the left adjoint of the adjoint equivalence as its lax transformation component, we obtain the equality $\aexi^A_B ={\aexi^B_A}^\rev$.  To show that this specifies an unbiased pseudo binary map, it remains to check that the data in \Cref{def:unb-bin-maps} for the unbiased pseudo-modifications do indeed specify unbiased pseudo-modifications.  But by \Cref{prop:equ-psd-mod}, it is enough to check that their lax/oplax parts form modifications, and this is assumed in the (compact) definition of a pseudo binary map (see \Cref{sec:psd-bin-cmpt}).  
Hence $U$ is surjective and indeed injective since, again by \Cref{prop:equ-psd-transf}, an unbiased pseudo-transformation is determined by its underlying lax/oplax transformation equipped with the adjoint equivalences $\aexi^A_B$.
\end{proof}

Now the sets $\psd_2^l(\ca,\cb;\cc)$ form the action on objects of a functor $\psd_2^l(-,-;-) \colon (\GCat^2_p)^{op} \times \GCat_p \to \Set$
by \cite[Section 6.3]{BourkeLobbiaCorrected}.  Therefore transporting the functor structure along the bijections in \Cref{prop:compact-unbiased-psd-binary}, we immediately obtain:

\begin{cor}
	\label{cor:enhanched-fct-psd_2}
	Let $\ca,\cb$ and $\cc$ be Gray-categories. 
	The sets $(\GCatmp_{nb})_2^l(\ca,\cb;\cc)$ admit the unique structure of a functor $(\GCatmp_{nb})_2^l(-,-;-) \colon (\GCat^2_p)^{op} \times \GCat_p \to \Set$ such that the forgetful bijections $U \colon (\GCatmp_{nb})_2^l(\ca,\cb;\cc)\to\GCatmp_2^l(\ca,\cb;\cc)$ in \Cref{prop:compact-unbiased-psd-binary} form the components of a natural isomorphism.

	Moreover, using the compact definition of \Cref{prop:compact-unbiased-psd-binary}, we can describe the functoriality of $(\GCatmp_{nb})_2^l(-,-;-)$ following \Cref{tab:compact-sub-psd}.
	\begin{table}[ht]
		\caption{Substitutions in compact form}
		\label{tab:compact-sub-psd}
		\centering
		\begin{tabular}{c:c c}
			$\textbf{F}\circ_1P$ &
			$(\textbf{F}\circ_1P)^a=F^{Pa}$ &
			$(\textbf{F}\circ_1P)_b=F_b\circ P$ \\
			&
			$(\textbf{F}\circ_1P)^A=\textbf{F}^{PA}$ &
			$(\textbf{F}\circ_1P)_B=\textbf{F}_B\circ P$ \\[0.25cm]
			
			$\textbf{F}\circ_2Q$ &
			$(\textbf{F}\circ_2Q)^a=F^a\circ Q$ &
			$(\textbf{F}\circ_2Q)_b=F_{Qb}$ \\
			&
			$(\textbf{F}\circ_2Q)^A=\textbf{F}^A\circ Q$ &
			$(\textbf{F}\circ_2Q)_B=\textbf{F}_{QB}$ \\[0.25cm]
			
			$G\circ_1 \textbf{F}$ &
			$(G\circ_1 \textbf{F})^a=G\circ F^a$ &
			$(G\circ \textbf{F})_b=G\circ F_b$ \\
			&
			$(G\circ_1 \textbf{F})^A=G\circ \textbf{F}^A$ &
			$(G\circ \textbf{F})_B=G \circ \textbf{F}_B$ \\
		\end{tabular}
	\end{table}
\end{cor}

\begin{proof}
It remains to prove only that the descriptions in $\Cref{tab:compact-sub-psd}$ capture the action of $(\GCatmp_{nb})_2^l(-,-;-)$.   Since the forgetful mapping $U^l \colon \GCatmp_2^l(\ca,\cb;\cc) \to \GCatm_2^l(\ca,\cb;\cc)$ is natural in each variable by construction in \cite[Section 6.3]{BourkeLobbiaCorrected},  the action of $\psd_2^l(-,-;-)$ at $\mathbf F \colon \ca, \cb \to \cc$ enhances that of $\lax_2^l(-,-;-)$.  By Table 3 in \cite[Section 4.6]{BourkeLobbiaCorrected}, this means that it acts by whiskering the lax transformations $F^A$ and oplax transformations $F_B$ by the given pseudomap $P$, $Q$ or $G$.  Likewise, the natural isomorphism $\lambda_2 \colon \GCat_p(\ca,\psd(\cb,\cc)) \cong \GCatmp_2^l(\ca,\cb;\cc)$ of \cite[Section 6.3]{BourkeLobbiaCorrected} ensures that $\psd_2^l(-,-;-)$ inherits its action on the adjoint equivalence $\aexi^A_B$ from $\psd(-,-)$ (described in \cite[Section 6.2]{BourkeLobbiaCorrected}).  Explicitly, the adjoint equivalences for $\mathbf F \circ_1 P$, $\mathbf F \circ_2 Q$ and $G \circ_1 \mathbf F$ are $\aexi^{PA}_B$, $\aexi^A_{QB}$ and 

\[\begin{tikzcd}[ampersand replacement=\&]
	{GF^{a'}_{B}GF^{A}_{b}} \&\& {G(F^{a'}_{B}F^{A}_{b})} \&\& {G(F^{A}_{b'}F^{a}_{B})} \&\& {GF^{A}_{b'}GF^{a}_{B}}
	\arrow[""{name=0, anchor=center, inner sep=0}, "{G^2}", curve={height=-12pt}, from=1-1, to=1-3]
	\arrow[""{name=1, anchor=center, inner sep=0}, "{(G^2)^{-1}}", curve={height=-12pt}, from=1-3, to=1-1]
	\arrow[""{name=2, anchor=center, inner sep=0}, "{G(A_B)}", curve={height=-12pt}, from=1-3, to=1-5]
	\arrow[""{name=3, anchor=center, inner sep=0}, "{G(A_B^o)}", curve={height=-12pt}, from=1-5, to=1-3]
	\arrow[""{name=4, anchor=center, inner sep=0}, "{(G^2)^{-1}}", curve={height=-12pt}, from=1-5, to=1-7]
	\arrow[""{name=5, anchor=center, inner sep=0}, "{G^2}", curve={height=-12pt}, from=1-7, to=1-5]
	\arrow["\dashv"{anchor=center, rotate=90}, draw=none, from=1, to=0]
	\arrow["\top"{description}, draw=none, from=3, to=2]
	\arrow["\dashv"{anchor=center, rotate=90}, draw=none, from=5, to=4]
\end{tikzcd}\]
respectively.  

That $(\GCatmp_{nb})_2^l(-,-;-)$ acts as in $\Cref{tab:compact-sub-psd}$ now follows from naturality of the mappings $U^p_\nb \colon \psden(\ca,\cb) \to \psd(\ca,\cb)$ and $U^o_\nb \colon \psden(\ca,\cb) \to \oplax(\ca,\cb)$ as shown in \Cref{uniquenatural} and \Cref{naturality} respectively.
\end{proof}

\section{Constructing the braidings}\label{sect:braiding}

\subsection{Binary braiding}\label{sec:bin-braiding}

We now construct the binary symmetry.   First, we prove a symmetry result (\Cref{prop:unb-symmetry}) for unbiased loose pseudo binary maps (\Cref{def:unb-bin-maps}).   Then, we use \Cref{prop:compact-unbiased-psd-binary} to obtain the binary symmetry for $\GCatmp$ (\Cref{prop:beta_2-def}).

\begin{prop}\label{prop:unb-symmetry}
	There is a natural bijection $\beta^2_{nb}\colon(\GCatmp_{nb})_2^l(\ca,\cb;\cc)\cong(\GCatmp_{nb})_2^l(\cb,\ca;\cc)$.
\end{prop}

\black 

\begin{proof}
	This follows directly from the compact definition of unbiased loose pseudo binary maps in \Cref{def:unb-bin-maps}, which is symmetric.  
	Explicitly, given an unbiased loose pseudo binary map $\mathbf{F}\colon\ca,\cb\to\cc$, $\beta^2\mathbf{F}\colon\cb,\ca\to\cc$ is defined by the data:
	\begin{itemize}
		\item for any object $b\in\cb$, $(\beta^2\mathbf{F})^b:=F_b\colon\ca\to\cc$;
		\item for any object $a\in\ca$, $(\beta^2\mathbf{F})_a:=F^a\colon\ca\to\cc$; 
		\item for any 1-cell $B\colon b\to b'\in\cb$; $(\beta^2\mathbf{F})^B:=\mathbf{F}_B$;
		\item for any 1-cell $A\colon a\to a'$, $(\beta^2\mathbf{F})_A:=\mathbf{F}
		^A$.
	\end{itemize}
	 
	 The data above satisfy the axioms of \Cref{def:unb-bin-maps} because $\mathbf{F}$ does, indeed:\footnote{Notice that in \eqref{def:unb-bin-maps} one consider the lax part of the pseudo-transformation in the first variable and the oplax one of the pseudo-transformation in the second variable. Hence, $\beta^2$ switches them.}
	\begin{align*}
		(F_B)^l(a)=
		&F_B(a)=F^a(B)=(\beta^2\mathbf{F})_a(B) 
		&
		(F^A)^o(b)=
		&F^A(b)=F_b(A)=(\beta^2\mathbf{F})^b(A) \\	
		(F^A)^o_B=
		&A_B^o=((F^A)^o)_B
		& 
	  	\aexi^B_A=
	  	&{\aexi^A_B}^{\rev} 
	\end{align*} 
The axioms for $\beta^2\mathbf{F}$ to be an unbiased pseudo binary map coincide with those for $\mathbf{F}$.
	
	Naturality is straightforward from the description of functoriality of $(\GCatmp_{nb})_2^l(-,-;-)$ in \Cref{cor:enhanched-fct-psd_2} and the fact that operations in $\psden(\ca,\cb)$ are as in $\psd(\ca,\cb)$ and $\oplax(\ca,\cb)$ (see  \Cref{uniquenatural} and \Cref{naturality}).
	For instance, let $P\colon\ca'\to\ca$ be a pseudomap: 
	\begin{itemize}
		\item $\beta^2(\mathbf{F}\circ_1P)\colon\cb,\ca'\to\cc$ has components at $b,B\in\cb$ and $a,A'\in\ca$ 
		\begin{center}
			$(\beta^2(\mathbf{F}\circ_1P))^b=(\mathbf{F}\circ_1P)_b=F_b\circ P$, \hspace{0.5cm} $(\beta^2(\mathbf{F}\circ_1P))_a=(\mathbf{F}\circ_1P)^a=F^{Pa}$\\
			\vspace{0.1cm}
			$(\beta^2(\mathbf{F}\circ_1P))^B=(\mathbf{F}\circ_1P)_B=\mathbf{F}_B\circ P$,\hspace{0.5cm} $(\beta^2(\mathbf{F}\circ_1P))_A=(\mathbf{F}\circ_1P)^A=\mathbf{F}^{PA}$;
		\end{center}
		
		\item  $(\beta^2\mathbf{F})\circ_2P\colon\cb,\ca'\to\cc$ has components at $b,B\in\cb$ and $a,A'\in\ca$ 
		\begin{center}
			$((\beta^2\mathbf{F})\circ_2P)^b=(\beta^2\mathbf{F})^b\circ P=F_b\circ P$, \hspace{0.5cm} $((\beta^2\mathbf{F})\circ_2P)_a=(\beta^2\mathbf{F})_{Pa}=F^{Pa}$ \\
			$((\beta^2\mathbf{F})\circ_2P)^B=(\beta^2\mathbf{F})^B\circ P=\mathbf{F}_B\circ P$, \hspace{0.5cm} $((\beta^2\mathbf{F})\circ_2P)_A=(\beta^2\mathbf{F})_{PA}=\mathbf{F}^{PA}$. 
		\end{center}
	\end{itemize}
	
\end{proof}

\begin{cor}\label{prop:beta_2-def}
	There is a natural bijection $\beta^2\colon\GCatmp_2^l(\ca,\cb;\cc)\cong\GCatmp_2^l(\cb,\ca;\cc)$.
\end{cor}

\begin{proof}
	This is a direct consequence of \Cref{prop:unb-symmetry}, precisely the required natural bijection is obtain through the following chain of natural bijections.
	\begin{align*}
	\GCatmp_2^l(\ca,\cb;\cc) 
	&\cong(\GCatmp_{nb})_2^l(\ca,\cb;\cc) 
	& (\text{\Cref{prop:compact-unbiased-psd-binary}}) & \\
	&\cong(\GCatmp_{nb})_2^l(\cb,\ca;\cc) 
	& (\text{\Cref{prop:unb-symmetry}}) & \\
	&\cong\GCatmp_2^l(\cb,\ca;\cc) 
	& (\text{\Cref{prop:compact-unbiased-psd-binary}}) &\qedhere
	\end{align*}
\end{proof}
\black

\begin{rmk}\label{rmk:beta2action}
Tracing through the bijection of \Cref{prop:compact-unbiased-psd-binary}, we obtain the following description of $\beta^2$.  Given a loose pseudo binary map $\mathbf{F}\colon\ca,\cb\to\cc$,  the loose pseudo binary map $\beta^2\mathbf{F}\colon\cb,\ca\to\cc$ has underlying lax binary map defined by $(F_b,F^a,(F_B)^l,(F^A)^o,(F_\beta)^l,(F^\alpha)^o)$ and specified adjoint equivalence $({\aexi^A_B})^{\rev}$.
%
%
\end{rmk}

\begin{cor}\label{cor:beta-2-self-inv}
	The inverse of $\beta^2$ is $\beta^2$ itself. 
\end{cor}

\begin{proof}
This follows from the description of $\beta^2$ in \Cref{rmk:beta2action} using that taking mates (back and forth) produces the original cell and since the reverse of the reverse adjoint equivalence is the original.
\end{proof}

\subsection{Duality for Pseudo Multimaps}\label{sec:duality}

In order to describe the higher braidings it will be useful to first obtain duality bijections for pseudo multimaps.  These enhance the duality bijections for lax multimaps established in \cite{BourkeLobbia}, which now collect in the following summary proposition.

\begin{prop}[Propositions~4.4,~4.7,~4.11 of \cite{BourkeLobbia}]\label{prop:duality-lax}
For $n\in\lbrace2,3,4\rbrace$, and Gray-categories $\ca_1,\ldots,\ca_n,\cb$, there is a bijection 
$$d^n \colon \GCatm^l_n(\ca_1,\ldots,\ca_n;\cb)\to\GCatm^l_n(\ca_n^\co,\ldots,\ca_1^\co;\cb^\co)$$
natural in each variable.  Furthermore, the $d^n$'s are self-inverse in the evident sense.
\end{prop} 

Since we will need their explicit description, let us recall these dualities in detail.  They make use of the $(-)^{+}$ and $(-)^{-}$ operations described in \Cref{plusminus} (with full details in \cite[Section~3.5]{BourkeLobbia}).
\begin{itemize}
	\item Given a lax bimap $F\colon\ca,\cb\to\cc$,  the lax bimap $d^2F\colon\cb^\co,\ca^\co\to\cc^\co$, in compact form (\Cref{sec:psd-bin-cmpt}, has:
	 \begin{itemize}
	 	\item pseudomaps $(d^2F)^b = (F_b)^{co} \colon \ca^{co} \to \cc^{co}$ and  $d^2F_{a} = (F^a)^{co}\colon\cb^{co} \to \cc^{co}$ for $b \in \cb^{co}, a \in \ca^{co}$;
	 	\item  lax transformation $(d^2F)^{B} = (F_B)^{+}\colon(F_b)^{co} \to (F_{b'})^{co}$ and oplax transformation $(d^2F)_{A} = ({F^A})^{-}\colon(F^a)^{co} \to (F^{a'})^{co}$ for $B \colon b \to b' \in \cb^{co}, A\colon a \to a' \in \ca^{co}$.
	 \end{itemize}
	 \item Given a lax ternary map $F\colon\ca,\cb,\cc\to\ce$, then $(d^3 F) \colon \cc^{co},\cb^{co},\ca^{co} \to \cd^{co}$ has associated bimaps $d^2(F_c)$, $d^2(F\midscript{b})$ and $d^2(F^a)$, and incubator $d^3 F(C|B|A) = (A|B|C)$.
	 \item Given a lax 4-ary map $F\colon\ca,\cb,\cc,\cd\rightarrow\ce$, $d^4 F \colon  \cd^{co},\cc^{co},\cb^{co},\ca^{co} \to \cd^{co}$ has associated ternary maps $d^3(F_d), d^3(F\midscript{c}), d^3(F\midscript{b})$ and $d^3(F^a)$.
\end{itemize}
In the next proposition, we show that these dualities lift to $\GCatmp$.

\begin{prop}
	\label{prop:dualities-psd}
For $n\in\lbrace2,3,4\rbrace$, and Gray-categories $\ca_1,\ldots,\ca_n,\cb$, the duality bijections of \Cref{prop:duality-lax} lift to duality bijections
$$d^n \colon \GCatmp^l_n(\ca_1,\ldots,\ca_n;\cb)\to\GCatmp^l_n(\ca_n^\co,\ldots,\ca_1^\co;\cb^\co)$$
natural in each variable.  
%
Furthermore, the $d^n$'s are self-inverse in the evident sense.
\end{prop}

\begin{proof}
	Firstly, we must lift the duality $d^2$ for lax bimaps described above to the setting of pseudo bimaps.  Now a pseudo bimap $\mathbf F = (F,\aexi^{A}_{B})$ consists of a lax binary map $F\colon\ca,\cb\to\cc$ together with adjoint equivalences $\aexi^{A}_{B} = (A_B)^{o} \dashv A_B$ such that
	\begin{itemize}
		\item $\aexi^{A}_{B}$ is the identity adjoint equivalence whenever $A$ or $B$ is an identity $1$-cell;
		\item each $3$-cell $F^{A}(\beta)$ and $F_B(\alpha)$ is invertible where $\alpha$ and $\beta$ are $2$-cells of $\ca$ and $\cb$ respectively. 
	\end{itemize}
Therefore given a pseudo binary map $\mathbf F$, we must enhance $d^2F$ as in \Cref{prop:duality-lax} with such a structure.   Now by definition $(d^2F^{B})_A = (F_B)^{+}(A) = A_B$.  However applying $(-)^{co}$ to $(A_B)^{o} \dashv A_B$ produces an adjoint equivalence $(A_B) \dashv (A_B)^{o}$ in $\cc^{co}$ since it interchanges the left and right adjoints; thus we must take the reverse adjoint equivalence $(\aexi^A_B)^{\rev} =( A_B)^{o} \dashv A_B$ to obtain the correct form.  Accordingly we set $d^2 \mathbf F = (d^2 F, (\aexi^A_B)^{\rev})$.  Certainly $(\aexi^A_B)^{\rev}$ is an identity adjoint equivalence just when $\aexi^A_B$ is -- in particular, whenever $A$ or $B$ is an identity $1$-cell.  Furthermore we have $(F_B)^{+}(\alpha) = F_B(\alpha)$ and $(F^A)^{-}(\beta) = F^A(\beta)$ which are invertible by assumption.  Therefore $d^2 \mathbf F \colon \cb^{\co},\ca^{\co} \to \ca^{\co}$ is a pseudo binary map.   This completes the definition of the function $d^2_{\ca,\cb,\cc} \colon \GCatmp(\ca,\cb;\cc) \to \GCatmp(\cb^{co},\ca^{co};\cc^{co})$ which evidently has inverse $d^2_{\cb^{co},\ca^{co},\cc^{co}}$ since reversing an adjoint equivalence twice produces the original.  Naturality of $d^2_{\ca,\cb,\cc} \colon \GCatm(\ca,\cb;\cc) \to \GCatm(\cb^{co},\ca^{co};\cc^{co})$ in $\cc$ follows from the fact the pseudomaps are locally $2$-functors and so preserve adjoint equivalences in hom $2$-categories whilst naturality in the pre-composition variables $\ca$ and $\cb$ is trivial.
	
	
	To lift the ternary duality $d^3$ 
	recall that a ternary pseudomap $\mathbf F$ consists of a lax ternary map together with enhancements of its associated system of lax bimaps $F^a$, $F\midscript{b}$ and $F_c$ to pseudo bimaps $\mathbf F^a$, $\mathbf{F} \midscript{b}$ and $\mathbf{F}_c$ such that the incubator $3$-cell $(A|B|C)$ is invertible.

Now the lax ternary map $(d^3 F) \colon \cc^{co},\cb^{co},\ca^{co} \to \cd^{co}$ has associated bimaps $d^2(F_c)$, $d^2(F\midscript{b})$ and $d^2(F^a)$, each of which lift to $d^2(\mathbf{F}_c)$, $d^2(\mathbf{F}\midscript{b})$ and $d^2(\mathbf{F}^a)$ using the binary duality established above.  Furthermore, the incubator of $d^3 F$ is exactly $(A|B|C)$, which is therefore invertible too, so that setting $d^3\mathbf F = (d^2(\mathbf{F}_c), d^2(\mathbf{F}\midscript{b}), d^2(\mathbf{F}^a))$ defines a pseudo ternary map. 
	The resulting function $d^3_{\ca,\cb,\cc,\cd} \colon \GCatmp(\ca,\cb,\cc;\cd) \to \GCatmp(\cc^{co},\cb^{co},\ca^{co};\cd^{co})$ has inverse $d^3_{\cc^{co},\cb^{co},\ca^{co},\cd^{co}}$ and its naturality follows from naturality of $d^3$ in the lax case and of $d^2$ in the pseudo case.
	
	The case of lifting the $4$-ary duality is similar but easier.  Indeed, a pseudo $4$-ary pseudomap is simply a lax $4$-ary map equipped with extensions of its lax $3$-ary maps to pseudo-$3$-ary maps satisfying the standard compatibility conditions.  As such, the lifted duality $d^4$ combines the duality $d^4$ in the lax setting of \Cref{prop:duality-lax} with the duality $d^3$ in the pseudo-setting described above.
\end{proof}


\begin{prop}\label{prop:interaction-d-and-beta-2}
	The symmetry $\beta^2$ commutes with $d^2$, i.e. the following diagram is commutative. 
	\[\begin{tikzcd}[ampersand replacement=\&]
		{\GCatmp(\ca,\cb;\cc)} \&\& {\GCatmp(\cb^{co},\ca^{co};\cc^{co})} \\
		{\GCatmp(\cb,\ca;\cc)} \&\& {\GCatmp(\ca^{co},\cb^{co};\cc^{co})}
		\arrow["{d^2_{\ca,\cb,\cc}}", from=1-1, to=1-3]
		\arrow["{\beta^2_{\ca,\cb}}"', from=1-1, to=2-1]
		\arrow["{\beta^2_{\cb^\co,\ca^\co}}", from=1-3, to=2-3]
		\arrow["{d^2_{\cb,\ca,\cc}}"', from=2-1, to=2-3]
	\end{tikzcd}\]
\end{prop}

\begin{proof}
	This follows directly by the definitions of $d^2$ and $\beta^2$. 
\end{proof}

\begin{rmk}\label{rmk:d-vs-beta-2}
	Since $d^2$ is self-inverse, from \Cref{prop:interaction-d-and-beta-2} it follows that $d^2\beta^2 d^2=\beta^2$.  
\end{rmk}

%
%

\subsection{Definition of the higher braidings}\label{sec:higher-braidings}

We are now ready to define all of the higher braidings using $\beta^2_1$, the dualities $d^2,d^3$ and $d^4$, and the left closedness isomorphisms $\lambda_2$ and $\lambda_3$ of \Cref{lem:left-closedness-psd}. \black   
Notice that, technically, we only need the restrictions of $\beta^3_2$, $\beta^4_2$ and $\beta^4_2$ to tight multimaps, i.e. the data of a short braiding as in \eqref{short-brd-data}.
However, in order to define these data, it will be useful to define all higher braidings up to 4-ary multimaps. 
Here we give a concise definition of each braiding, while in \Cref{sect:axioms} we give an explicit description of each.
 From their explicit descriptions, it will be clear that $\beta^3_2$, $\beta^4_2$ and $\beta^4_2$ restrict to tight multimaps (Remarks~\ref{prop:beta-3-2-restr-to-tight} and~\ref{prop:beta-4-2_3-restr-to-tight}), as required for a short braiding. \black 

\begin{enumerate}
	\item We obtain $\beta^3_1$ as the composite below using left closedness, $\beta^2_1$ and left closedness again. 
	$$\GCatmp_3^l(\ca,\cb,\cc;\cd) \stackrel{\lambda_3^{-1}}{\cong} \GCatmp_2^l(\ca,\cb;\psd(\cc,\cd)) \stackrel{\beta^2}{\cong} \GCatmp_2^l(\cb,\ca;\psd(\cc,\cd)) \stackrel{\lambda_3}{\cong} \GCatmp_3^l(\cb,\ca,\cc;\cd)$$ 
	\item We obtain $\beta^3_2$ as the composite below, using duality, $\beta^3_1$ and duality again.
	$$\GCatmp_3^l(\ca,\cb,\cc;\cd) \stackrel{d^3}{\cong} \GCatmp_3^l(\cc^{co},\cb^{co},\ca^{co};\cd^{co}) \stackrel{\beta^3_1}{\cong} \GCatmp_3^l(\cb^{co},\cc^{co},\ca^{co};\cd^{co}) \stackrel{d^3}{\cong} \GCatmp_3^l(\ca,\cc,\cb;\cd)$$ 
	\item We obtain $\beta^4_1$ using left closedness, $\beta^3_1$ and left closedness. 
		$$\GCatmp_4^l(\ca,\cb,\cc,\cd;\ce)\stackrel{\lambda_4^{-1}}{\cong}\GCatmp_3^l(\ca,\cb,\cc;\psd(\cd,\ce))\stackrel{\beta^3_1}{\cong}\GCatmp_3^l(\cb,\ca,\cc;\psd(\cd,\ce))\stackrel{\lambda_4}{\cong}\GCatmp_4(\cb,\ca,\cc,\cd;\ce)$$
	\item We obtain $\beta^4_2$ using left closedness, $\beta^3_2$ and left closedness. 
	$$\GCatmp_4^l(\ca,\cb,\cc,\cd;\ce)\stackrel{\lambda_4^{-1}}{\cong}\GCatmp_3^l(\ca,\cb,\cc;\psd(\cd,\ce))
	\stackrel{\beta^3_2}{\cong}\GCatmp_3^l(\ca,\cc,\cb;\psd(\cd,\ce))\stackrel{\lambda_4}{\cong}\GCatmp_4(\ca,\cc,\cb,\cd;\ce)$$
	\item We obtain $\beta^4_3$ using duality, $\beta^4_1$ and duality.
	\begin{align*}
		\GCatmp_4^l(\ca,\cb,\cc,\cd;\ce)\stackrel{d^4}{\cong}\GCatmp_4^l(\cd^\co,\cc^\co,\cb^\co,\ca^\co;\ce^\co)
		&\stackrel{\beta^4_1}{\cong}\GCatmp_4^l(\cc^\co,\cd^\co,\cb^\co,\ca^\co;\ce^\co) \\
		&\stackrel{d^4}{\cong}\GCatmp_4^l(\ca,\cb,\cd,\cc;\ce)
	\end{align*}
\end{enumerate}

In the next result we show how, using $\beta^3_2$ and left closedness, we obtain another important isomorphism, which will be quite useful in  \cite{BourkeLobbia2}. 

\begin{cor}\label{cor:biclosedness-iso}
	We have a natural isomorphism $\Upsilon\colon\psd(\ca,\psd(\cb,\cc)) \cong \psd(\cb,\psd(\ca,\cc))$. 
\end{cor}
\begin{proof}
	This follows by Yoneda, using left closedness and $\beta^3_2$ as shown below.
\[\begin{tikzcd}[ampersand replacement=\&]
	{\psd(\cx,\,\psd(\ca,\psd(\cb,\cc))\,)} \& {\GCatmp_2^l(\cx,\ca;\,\psd(\cb,\cc)\,)} \& {\GCatmp_3^l(\cx,\ca,\cb;\cc)} \\
	{\psd(\cx,\,\psd(\cb,\psd(\ca,\cc))\,)} \& {\GCatmp_2^l(\cx,\cb;\,\psd(\ca,\cc)\,)} \& {\GCatmp_3^l(\cx,\cb,\ca;\cc)}
	\arrow["{\lambda_2}", from=1-1, to=1-2]
	\arrow["\cong"', dashed, from=1-1, to=2-1]
	\arrow["{\lambda_3}", from=1-2, to=1-3]
	\arrow["{\beta^3_2}", from=1-3, to=2-3]
	\arrow["{{\lambda_2}^{-1}}", from=2-2, to=2-1]
	\arrow["{{\lambda_3}^{-1}}", from=2-3, to=2-2]
\end{tikzcd}\]
\end{proof}


\begin{rmk}
	We can give a more explicit description of the action of the biclosedness isomorphism in \Cref{cor:biclosedness-iso}.
	First, let us notice that on pseudomaps $F\colon\ca\to\psd(\cb,\cc)$ the action of the biclosedness isomorphism is 
	$\Upsilon(F)=\lambda_2^{-1}\beta^2\lambda_2(F)$. 
	
	Therefore, using the properties of $\lambda_2$ and $\lambda_3$ (\Cref{lem:left-closedness-psd}), and the definition of $\beta^2$ (\Cref{prop:beta_2-def}), $\Upsilon(F)$ is defined as:
	\begin{enumerate}
		\item for a 0-cell $b\in\cb$, $\Upsilon(F)(b)=(\lambda_2^{-1}\beta^2\lambda_2(F))(b)=(\beta^2\lambda_2(F))^b=(\lambda_2F)_b$;

		\item for $n$-cells $x\in\cb$, with $n=1,2,3$, $\Upsilon(F)(x)=(\lambda_2^{-1}\beta^2\lambda_2(F))(x)=(\beta^2\lambda_2(F))^x=(\lambda_2F)_x^l$, i.e. the pseudo $n$-transfor (in lax form) corresponding to the oplax $n$-transfor $(\lambda_2F)_x$ through, for each $n=1,2,3$ respectively, Propositions \ref{prop:equ-psd-transf}, \ref{prop:equ-psd-mod} and \ref{prop:equ-psd-pert};

		\item for composable 1-cells $b\xrightarrow{B_1}b'\xrightarrow{B_2}b''$, the invertible pseudo modifcation $(\Upsilon F)^2_{B_1,B_2}$ is defined as
		\begin{itemize}
			\item at 0-cells $a\in\ca$, $((\Upsilon F)^2_{B_1,B_2})_a=(\lambda_2F)^a_{B_2,B_1}=(F(a))^2_{B_2,B_1}$;
			\item at a 1-cell $A\colon a\to a'\in\ca$, $((\Upsilon F)^2_{B_1,B_2})_A$ is the mate of $((\lambda_2F)_{B_1,B_2})^A=(F(A))^2_{B_1,B_2}$.
		\end{itemize}
	\end{enumerate}
\end{rmk}
\black 

%

\section{Verifying the axioms for a symmetric skew monoidal category}\label{sect:axioms}

The main result, \Cref{maintheorem}, of this section establish that the flexible skew monoidal structure $(\GCat,\otimes_p,1)$ is symmetric skew monoidal.  This will be done by proving that the braidings $\beta^3_2$, $\beta^4_2$, and $\beta^4_3$ described above form a short braiding on $\GCatmp$ in the sense described in \Cref{sec:brd-skew-multi-mon}.

%
%
%
%
%
%

\subsection{Explicit description of the ternary braidings}\label{sec:ternary-braidings}

In this section, we calculate explicit descriptions for the two ternary braidings $\beta_1^3$ and $\beta_2^3$, summarised in Tables~\ref{tab:beta-3-1} and~\ref{tab:beta-3-2}.   

These descriptions will be crucial for establishing the short braiding.  To this end, let $F\colon\ca,\cb,\cc\to\cd$ be a pseudo-ternary map.
\begin{itemize}
	\item The binary map components of $\beta^3_1F$ are $(\beta^3_1F)^b =F\midscript{b}$, $(\beta^3_1F)\midscript{a}=F^a$ and $(\beta^3_1F)_c=\beta^2(F_c)$. 
	We give the calculation only for the first component since the other two follow in a similar way using the definitions of $\lambda_3$, $\lambda_2$ and $\beta^2$. 
\begin{align*}
(\beta^3_1F)^b
    &= [\lambda_3\beta^2{\lambda_3}^{-1}F]^b
    & \text{(definition of $\beta^3_1$)} \\
    &= \lambda_2[(\beta^2{\lambda_3}^{-1}F)^b]
    & \text{(\Cref{lem:left-closedness-psd}.(ii))} \\
    &= \lambda_2[({\lambda_3}^{-1}F)_b]
    & \text{(definition of $\beta^2$, see Prop.\ref{prop:beta_2-def})} \\
    &= \lambda_2[{\lambda_2}^{-1}(F_b)]
    & \text{(\Cref{lem:left-closedness-psd}.(ii))} \\
    &= F\midscript{b}
    & 
\end{align*}

	
To calculate the incubator  for $\beta^3_1F$, firstly observe that by the description of ${\lambda_3}$ in \Cref{lem:left-closedness-psd}.(iii), the incubator $(B \mid A \mid C)$ for $\beta^3_1F$ is the component at $C$ of the right adjoint modification $B_A \in \psd(\cc,\cd)$ for the pseudo binary map $(\beta^2{\lambda_3}^{-1}F)\colon\cb,\ca\to\psd(\cc,\cd)$.  Now by the description of $\beta^2$ in \Cref{rmk:beta2action}, $\beta^2$ takes the reverse adjoint equivalence, so that this is in the fact the left adjoint modification $(A_B)^o$ of the adjoint equivalence $\Theta^A_B$ for the pseudo binary map ${\lambda_3}^{-1}F$.  Now the left adjoint component $(A_B)^o$ at $C$ is determined as the mate of the inverse of the right adjoint $A_B$ at $C$ --- see \eqref{eq:loca-adj-psd-final} where we establish this for a general adjoint equivalence in $\psd(\cc,\cd)$.   Now by definition of $\lambda_3$, the value of the modification $A_B$ at $C$ is the incubator $(A\mid B\mid C)$.  Therefore we obtain the following formula for the incubator $(B \mid A \mid C)$ of $\beta^3_1F$:

%
%
	%
	\begin{equation}
		\label{eq:incubator-beta-3-1}
		\scalebox{0.8}{
			\begin{tikzcd}[ampersand replacement=\&,column sep=small,row sep=scriptsize]
				{F^{a'}_{b'}(C)F^{b'}_{c}(A)F^{a}_{c}(B)} \&\& {F^{a'}_{b'}(C)F^{a'}_{c}(B)F^{b}_{c}(A)} \&\& {F^{a'}_{c'}(B)F^{a'}_{b}(C)F^{b}_{c}(A)} \&\& \\
				\& {} \\
				\&\& {F^{a'}_{b'}(C)F^{b'}_{c}(A)F^{a}_{c}(B)} \&\& {F^{a'}_{c'}(B)F^{b}_{c'}(A)F^{a}_{b}(C)} \\
				\&\&\&\&\& {} \\
				\&\& {F^{b'}_{c'}(A)F^{a}_{b'}(C)F^{a}_{c}(B)} \&\& {F^{b'}_{c'}(A)F^{a}_{c'}(B)F^{a}_{b}(C)} \&\& {F^{a'}_{c'}(B)F^{b}_{c'}(A)F^{a}_{b}(C)}
				\arrow[""{name=0, anchor=center, inner sep=0}, "{1\circ{(A_B^o)}^c}", from=1-1, to=1-3]
				\arrow[curve={height=12pt}, equals, from=1-1, to=3-3]
				\arrow[""{name=1, anchor=center, inner sep=0}, "{{B_C}^{a'}\circ1}", from=1-3, to=1-5]
				\arrow["{1\circ {A_B}^c}"{description}, from=1-3, to=3-3]
				\arrow["{1\circ {A_C}^b}", from=1-5, to=3-5]
				\arrow["{{A_C}^{b'}\circ1}"', from=3-3, to=5-3]
				\arrow["{{A_B}^{c'}\circ1}"{description}, from=3-5, to=5-5]
				\arrow[curve={height=-12pt}, equals, from=3-5, to=5-7]
				\arrow[""{name=2, anchor=center, inner sep=0}, "{1\circ {B_C}^{a}}"', from=5-3, to=5-5]
				\arrow[""{name=3, anchor=center, inner sep=0}, "{{(A_B^o)}^{c'}\circ1}"', from=5-5, to=5-7]
				\arrow["{1\circ\eta_{A,B,c}^{-1}}"{pos=0.7}, between={0.4}{0.9}, Rightarrow, from=0, to=2-2]
				\arrow["{(A\mid B\mid C)^{-1}}"{pos=0.3}, shift right=3, between={0.2}{0.4}, Rightarrow, from=1, to=2]
				\arrow["{\epsilon_{A,B,c'}^{-1}\circ1}"'{pos=0.2}, between={0.1}{0.6}, Rightarrow, from=4-6, to=3]
			\end{tikzcd}
		}
	\end{equation}

\item	The binary maps defining $\beta^3_2F$ are $(\beta^3_2F)^a=\beta^2(F^a)$, $(\beta^3_2F)\midscript{c}=F_c$, and $(\beta^3_2F)_b=F\midscript{b}$.
	We give the calculation of the first component since the other two follow similarly by definitions of $d^3$, $\beta^3_1$, and Propositions~\ref{prop:interaction-d-and-beta-2} and~\ref{prop:dualities-psd}. 
\begin{align*}
(\beta^3_2F)^a
    &= [d^3\beta^3_1d^3(F)]^a
    & \text{(definition of $\beta^3_2$)} \\
    &= d^2[(\beta^3_1d^3F)_a]
    & \text{(definition or $d^3$, see \Cref{prop:dualities-psd})} \\
    &= d^2[\beta^2(d^3F)^a]
    & \text{(description of $\beta^3_1$, see \Cref{tab:beta-3-1})} \\
    &= d^2\beta^2d^2(F^a)
    & \text{(definition of $d^3$)} \\
    &= \beta^2(F^a)
    & \text{(\Cref{rmk:d-vs-beta-2})}
\end{align*}

	For the incubator, we start by noticing that $d^3$ twice reverses the direction of 2-cells, so it goes back to the starting point.
	Therefore, the incubator $(A\mid C\mid B)$ for $\beta^3_2F$ is obtained in the same way as $\beta^3_1$ but taking the mate along $(B_C)^o \dashv B_C$ \black as shown below.
	
	\begin{equation}
		\label{eq:incubator-beta-3-2}
		\scalebox{0.9}{
			\begin{tikzcd}[ampersand replacement=\&]
				{F^{a'}_{c'}(B)F^{a'}_{b}(C)F^{b}_{c}(A)} \& {} \& \\
				{F^{a'}_{b'}(C)F^{a'}_{c}(B)F^{b}_{c}(A)} \&\& {F^{a'}_{c'}(B)F^{a'}_{b}(C)F^{b}_{c}(A)} \\
				{F^{a'}_{b'}(C)F^{b'}_{c}(A)F^{a}_{c}(B)} \&\& {F^{a'}_{c'}(B)F^{b}_{c'}(A)F^{a}_{b}(C)} \\
				{F^{b'}_{c'}(A)F^{a}_{b'}(C)F^{a}_{c}(B)} \&\& {F^{b'}_{c'}(A)F^{a}_{c'}(B)F^{a}_{b}(C)} \\
				\& {} \& {F^{b'}_{c'}(A)F^{a}_{b'}(C)F^{a}_{c}(B)}
				\arrow["{{(B_C^o)}^{a'}\circ1}"', from=1-1, to=2-1]
				\arrow[curve={height=-12pt}, equals, from=1-1, to=2-3]
				\arrow[""{name=0, anchor=center, inner sep=0}, "{{B_C}^{a'}\circ1}"{description}, from=2-1, to=2-3]
				\arrow["{1\circ {A_B}^c}"', from=2-1, to=3-1]
				\arrow["{1\circ {A_C}^b}", from=2-3, to=3-3]
				\arrow["{{A_C}^{b'}\circ1}"', from=3-1, to=4-1]
				\arrow["{{A_B}^{c'}\circ1}", from=3-3, to=4-3]
				\arrow[""{name=1, anchor=center, inner sep=0}, "{1\circ {B_C}^{a}}"{description}, from=4-1, to=4-3]
				\arrow[curve={height=12pt}, equals, from=4-1, to=5-3]
				\arrow["{1\circ (B_C^o)^{a}}", from=4-3, to=5-3]
				\arrow["{\eta_{a',B,C}\circ1}"'{pos=0.4}, between={0.2}{0.6}, Rightarrow, from=1-2, to=0]
				\arrow["{(A\mid B\mid C)^{-1}}"{pos=0.4}, shift right=5, between={0.2}{0.5}, Rightarrow, from=0, to=1]
				\arrow["{1\circ\epsilon_{a,B,C}}"{pos=0.6}, between={0.3}{0.7}, Rightarrow, from=1, to=5-2]
			\end{tikzcd}
		}
	\end{equation}

\end{itemize}

\begin{table}
	\begin{minipage}{.5\linewidth}
		\caption{$\beta^3_1$}
		\label{tab:beta-3-1}
		\centering
		\renewcommand\arraystretch{1.25}
		\begin{tabular}{| c : c |}
			\hline
			at $b\in\cb$ & $F\midscript{b}$\\
			at $a\in\ca$ & $F^a$\\
			at $c\in\cc$ & $\beta^2F_c$\\
			incubator & \eqref{eq:incubator-beta-3-1}\\
			\hline
		\end{tabular}
		
	\end{minipage}%
	\begin{minipage}{.5\linewidth}
		\centering
		\caption{$\beta^3_2$}
		\label{tab:beta-3-2}
		\renewcommand\arraystretch{1.25}
		\begin{tabular}{| c : c |}
			\hline
			at $a\in\ca$ & $\beta^2F^a$\\
			at $c\in\cc$ & $F_c$\\
			at $b\in\cb$ & $F\midscript{b}$\\
			incubator & \eqref{eq:incubator-beta-3-2}\\
			\hline
		\end{tabular}
	\end{minipage} 
\end{table}

For a short braiding, we shall need the following.
\begin{prop}\label{prop:beta-3-2-restr-to-tight}
If $F$ is a tight pseudo-ternary map, then so is $\beta^3_2 F$.  In particular, the ternary braidings restrict to bijections $\beta^3_2\colon\GCatmp_3^t(\ca,\cb,\cc;\cd)\cong\GCatmp_3^t(\ca,\cc,\cb;\cd).$
\end{prop}
\begin{proof}
By definition, $F$ is a tight pseudo ternary map just when $F\midscript{b}$ and $F_c$ are tight pseudo binary.  Now by \ref{tab:beta-3-2}, we have $(\beta^3_2F)\midscript{c}=F_c$ and  $(\beta^3_2F)_b=F\midscript{b}$.  Therefore if $F$ is tight, so is $\beta^3_2F$.
\end{proof}

For a short braiding, we also need the following equation.

\begin{prop}
	\label{prop:beta-3-self-inverses}
	The isomorphisms $\beta^3_1$ and $\beta^3_2$ are self inverses, i.e. $\beta^3_1\beta^3_1=1$ and $\beta^3_2\beta^3_2=1$.
\end{prop}

\begin{proof}
Let $F\colon\ca,\cb,\cc\to\cd$ be a pseudo ternary map.  The binary map components of $\beta^3_1\beta^3_1F$ and $\beta^3_2\beta^3_2F$ coincide with those for $F$ by the descriptions of $\beta^3_1$ and $\beta^3_2$ in Table \Cref{tab:beta-3-1,tab:beta-3-2} combined with the equation $\beta^2\beta^2=1$ (\Cref{cor:beta-2-self-inv}).

It remains to check that they have the same incubators.  Firstly we consider $\beta^3_1\beta^3_1F$, which, applying \eqref{eq:incubator-beta-3-1} twice, has incubator given by the pasting below. 
In fact, when applying $\beta^3_1$ for the second time, the inverses of the unit and counits which we have to use are those of $\aexi^B_A=({\aexi^A_B})^{\rev}$, hence $\epsilon_{A,B,c}$ and $\eta_{A,B,c'}$ respectively.
\begin{center}
	
\adjustbox{width=\textwidth}{
\begin{tikzcd}[ampersand replacement=\&]
	{F^{a'}_{b'}(C)F^{a'}_{c}(B)F^{b}_{c}(A)} \&\& {F^{a'}_{b'}(C)F^{b'}_{c}(A)F^{a}_{c}(B)} \& {} \&\&\& \\
	\\
	\& {} \& {F^{a'}_{b'}(C)F^{a'}_{c}(B)F^{b}_{c}(A)} \&\& {F^{a'}_{b'}(C)F^{b'}_{c}(A)F^{a}_{c}(B)} \\
	\&\& {F^{a'}_{c'}(B)F^{a'}_{b}(C)F^{b}_{c}(A)} \&\& {F^{b'}_{c'}(A)F^{a}_{b'}(C)F^{a}_{c}(B)} \\
	\&\& {F^{a'}_{c'}(B)F^{b}_{c'}(A)F^{a}_{b}(C)} \&\& {F^{b'}_{c'}(A)F^{a}_{c'}(B)F^{a}_{b}(C)} \& {} \\
	\\
	\&\&\& {} \& {F^{a'}_{c'}(B)F^{b}_{c'}(A)F^{a}_{b}(C)} \&\& {F^{b'}_{c'}(A)F^{a}_{c'}(B)F^{a}_{b}(C)}
	\arrow[""{name=0, anchor=center, inner sep=0}, "{1\circ {A_B}^c}", from=1-1, to=1-3]
	\arrow[curve={height=12pt}, equals, from=1-1, to=3-3]
	\arrow["{1\circ{(A_B^o)}^c}"{description}, from=1-3, to=3-3]
	\arrow[curve={height=-12pt}, equals, from=1-3, to=3-5]
	\arrow[""{name=1, anchor=center, inner sep=0}, "{1\circ {A_B}^c}"{description}, from=3-3, to=3-5]
	\arrow["{{B_C}^{a'}\circ1}"', from=3-3, to=4-3]
	\arrow["{{A_C}^{b'}\circ1}", from=3-5, to=4-5]
	\arrow["{1\circ {A_C}^b}"', from=4-3, to=5-3]
	\arrow["{1\circ {B_C}^{a}}", from=4-5, to=5-5]
	\arrow[""{name=2, anchor=center, inner sep=0}, "{{A_B}^{c'}\circ1}"{description}, from=5-3, to=5-5]
	\arrow[curve={height=12pt}, equals, from=5-3, to=7-5]
	\arrow["{{(A_B^o)}^{c'}\circ1}"{description}, from=5-5, to=7-5]
	\arrow[curve={height=-12pt}, equals, from=5-5, to=7-7]
	\arrow[""{name=3, anchor=center, inner sep=0}, "{{A_B}^{c'}\circ1}"', from=7-5, to=7-7]
	\arrow["{1\circ\epsilon_{A,B,c}}"{pos=0.3}, between={0.2}{0.5}, Rightarrow, from=0, to=3-2]
	\arrow["{1\circ\eta_{A,B,c}}"'{pos=0.6}, between={0.4}{0.8}, Rightarrow, from=1-4, to=1]
	\arrow["{(A\mid B\mid C)}"{description}, between={0.3}{0.7}, Rightarrow, from=1, to=2]
	\arrow["{\epsilon_{A,B,c'}\circ1}"{pos=0.4}, between={0.3}{0.5}, Rightarrow, from=2, to=7-4]
	\arrow["{\eta_{A,B,c'}\circ1}"'{pos=0.6}, between={0.5}{0.8}, Rightarrow, from=5-6, to=3]
\end{tikzcd}
	}
\end{center}
Using the triangle equalities for the adjunctions $(A^o_B)^c\dashv{A_B}^c$ and $(A^o_B)^{c'}\dashv{A_B}^{c'}$, we see that the pasting above is indeed equal to $(A\mid B\mid C)$. 
This ends the proof that $\beta^3_1\beta^3_1=1$.

Let us now consider $\beta^3_2\beta^3_2$. 
Also here we need to recall that $\aexi^C_B=({\aexi^B_C})^{\rev}$, and so the incubator of $\beta^3_2\beta^3_2 F$ is the pasting below.
\begin{center}
\adjustbox{width=\textwidth}{
\begin{tikzcd}[ampersand replacement=\&,sep=scriptsize]
	{F^{a'}_{b'}(C)F^{a'}_{c}(B)F^{b}_{c}(A)} \& {} \&\&\&\&\& \\
	{F^{a'}_{c'}(B)F^{a'}_{b}(C)F^{b}_{c}(A)} \&\& {F^{a'}_{b'}(C)F^{a'}_{c}(B)F^{b}_{c}(A)} \&\& {F^{a'}_{b'}(C)F^{b'}_{c}(A)F^{a}_{c}(B)} \\
	\& {} \& {F^{a'}_{c'}(B)F^{a'}_{b}(C)F^{b}_{c}(A)} \&\& {F^{b'}_{c'}(A)F^{a}_{b'}(C)F^{a}_{c}(B)} \& {} \\
	\&\& {F^{a'}_{c'}(B)F^{b}_{c'}(A)F^{a}_{b}(C)} \&\& {F^{b'}_{c'}(A)F^{a}_{c'}(B)F^{a}_{b}(C)} \&\& {F^{b'}_{c'}(A)F^{a}_{b'}(C)F^{a}_{c}(B)} \\
	\&\&\&\&\& {} \& {F^{b'}_{c'}(A)F^{a}_{c'}(B)F^{a}_{b}(C)}
	\arrow["{{B_C}^{a'}\circ1}"', from=1-1, to=2-1]
	\arrow[curve={height=-18pt}, equals, from=1-1, to=2-3]
	\arrow[""{name=0, anchor=center, inner sep=0}, "{{(B_C^o)}^{a'}\circ1}"{description}, from=2-1, to=2-3]
	\arrow[curve={height=12pt}, equals, from=2-1, to=3-3]
	\arrow[""{name=1, anchor=center, inner sep=0}, "{1\circ {A_B}^c}", from=2-3, to=2-5]
	\arrow["{{B_C}^{a'}\circ1}", from=2-3, to=3-3]
	\arrow["{{A_C}^{b'}\circ1}", from=2-5, to=3-5]
	\arrow["{1\circ {A_C}^b}"', from=3-3, to=4-3]
	\arrow["{1\circ {B_C}^{a}}"', from=3-5, to=4-5]
	\arrow[curve={height=-12pt}, equals, from=3-5, to=4-7]
	\arrow[""{name=2, anchor=center, inner sep=0}, "{{A_B}^{c'}\circ1}"', from=4-3, to=4-5]
	\arrow[""{name=3, anchor=center, inner sep=0}, "{1\circ (B_C^o)^{a}}"{description}, from=4-5, to=4-7]
	\arrow[curve={height=12pt}, equals, from=4-5, to=5-7]
	\arrow["{1\circ {B_C}^{a}}", from=4-7, to=5-7]
	\arrow["{\eta_{a',B,C}\circ1}"'{pos=0.3}, between={0.1}{0.5}, Rightarrow, from=1-2, to=0]
	\arrow["{\epsilon_{a',B,C}\circ1}", between={0.3}{0.7}, Rightarrow, from=0, to=3-2]
	\arrow["{(A\mid B\mid C)}"{description}, between={0.3}{0.7}, Rightarrow, from=1, to=2]
	\arrow["{1\circ\eta_{a,B,C}}"'{pos=0.4}, between={0.2}{0.6}, Rightarrow, from=3-6, to=3]
	\arrow["{1\circ\epsilon_{a,B,C}}", between={0.3}{0.7}, Rightarrow, from=3, to=5-6]
\end{tikzcd}
	}
\end{center}
	Using the triangle equalities for the adjunctions $(B^o_C)^{a'}\dashv{B_C}^{a'}$ and $(B^o_C)^{a}\dashv{B_C}^{a}$, we see that the pasting above is indeed equal to $(A\mid B\mid C)$. 
	This ends the proof that $\beta^3_2\beta^3_2=1$.
\end{proof}

\subsection{Further useful equations about the ternary braidings}
The following further equations about the ternary braidings will be crucial in the following section when we verify the short braiding.

\begin{prop}
	\label{prop:symm-3}
	For any loose pseudo ternary map $G\colon\ca,\cb,\cc\to\cd$, 
	$$\beta^3_1\beta^3_2\beta^3_1(G)=\beta^3_2\beta^3_1\beta^3_2(G).$$ 
\end{prop}

\begin{proof}
Using the explicit descriptions of $\beta^3_1$ and $\beta^3_2$ (\Cref{tab:beta-3-1,tab:beta-3-2}) one can check (with a one line equation for each component) that $\beta^3_1\beta^3_2\beta^3_1(G)$ and $\beta^3_2\beta^3_1\beta^3_2(G)$ have the same pseudo binary maps components. For instance, at $a\in\ca$, 
\begin{center}
$[\beta^3_1\beta^3_2\beta^3_1(G)]_a
=\beta^2(\,[\beta^3_2\beta^3_1(G)]_a\,)
=\beta^2(\,[\beta^3_1(G)]\midscript{a}\,)
=\beta^2(G^a)$\\ \vspace{0.2cm}
and \hspace{0.2cm} $[\beta^3_2\beta^3_1\beta^3_2(G)]_a
=[\beta^3_1\beta^3_2(G)]\midscript{a}
=[\beta^3_2(G)]^a
=\beta^2(G^a)$.
\end{center}

For the incubator, we start with calculating the one for $\beta^3_1\beta^3_2\beta^3_1F$. 
First we recall that the incubator $(B\mid A\mid C)$ of $\beta^3_1F$ is defined in \eqref{eq:incubator-beta-3-1}, then we can use \eqref{eq:incubator-beta-3-2} to see that the incubator of $\beta^3_2\beta^3_1F$ is the 3-cell below.
\begin{equation}
	\label{diag:incub-beta-2-beta-1-of-F}
\scalebox{0.9}{
\begin{tikzcd}[ampersand replacement=\&,column sep=scriptsize,row sep=2.25em]
	{F^{b'}_{c'}(A)F^{a}_{b'}(C)F^{a}_{c}(B)} \&\& {} \&\& \\
	{F^{a'}_{b'}(C)F^{b'}_{c}(A)F^{a}_{c}(B)} \& {} \\
	{F^{a'}_{b'}(C)F^{a'}_{c}(B)F^{b}_{c}(A)} \&\& {F^{a'}_{b'}(C)F^{b'}_{c}(A)F^{a}_{c}(B)} \&\& {F^{b'}_{c'}(A)F^{a}_{b'}(C)F^{a}_{c}(B)} \\
	{F^{a'}_{c'}(B)F^{a'}_{b}(C)F^{b}_{c}(A)} \&\& {F^{a'}_{c'}(B)F^{b}_{c'}(A)F^{a}_{b}(C)} \&\& {F^{b'}_{c'}(A)F^{a}_{c'}(B)F^{a}_{b}(C)} \\
	\&\&\& {} \& {F^{a'}_{c'}(B)F^{b}_{c'}(A)F^{a}_{b}(C)} \\
	\&\& {} \&\& {F^{a'}_{c'}(B)F^{a'}_{b}(C)F^{b}_{c}(A)}
	\arrow["{{(A_C^o)}^{b'}\circ1}"', from=1-1, to=2-1]
	\arrow[curve={height=-18pt}, equals, from=1-1, to=3-5]
	\arrow["{\eta_{A,b',C}\circ1}"'{pos=0.6}, shift left=5, between={0.6}{0.8}, Rightarrow, from=1-3, to=3-3]
	\arrow["{1\circ{(A_B^o)}^c}"', from=2-1, to=3-1]
	\arrow[curve={height=-18pt}, equals, from=2-1, to=3-3]
	\arrow[""{name=0, anchor=center, inner sep=0}, "{1\circ {A_B}^c}"{description}, from=3-1, to=3-3]
	\arrow["{{B_C}^{a'}\circ1}"', from=3-1, to=4-1]
	\arrow["{{A_C}^{b'}\circ1}", from=3-3, to=3-5]
	\arrow["{(A\mid B\mid C)}", between={0.3}{0.7}, Rightarrow, from=3-3, to=4-3]
	\arrow["{1\circ {B_C}^{a}}", from=3-5, to=4-5]
	\arrow["{1\circ {A_C}^b}"', from=4-1, to=4-3]
	\arrow[curve={height=18pt}, equals, from=4-1, to=6-5]
	\arrow[""{name=1, anchor=center, inner sep=0}, "{{A_B}^{c'}\circ1}"{description}, from=4-3, to=4-5]
	\arrow[curve={height=18pt}, equals, from=4-3, to=5-5]
	\arrow["{1\circ\epsilon_{A,b,C}}"{pos=0.3}, shift right=5, between={0.2}{0.4}, Rightarrow, from=4-3, to=6-3]
	\arrow["{{(A_B^o)}^{c'}\circ1}", from=4-5, to=5-5]
	\arrow["{1\circ {(A_C^o)}^b}", from=5-5, to=6-5]
	\arrow["{1\circ\eta_{A,B,c}}"'{pos=0.4}, between={0.2}{0.6}, Rightarrow, from=2-2, to=0]
	\arrow["{\epsilon_{A,B,c'}\circ1}", between={0.3}{0.7}, Rightarrow, from=1, to=5-4]
\end{tikzcd}
}
\end{equation}
Then, using the description of the incubator of $\beta^3_1$, we see that the incubator of $\beta^3_1\beta^3_2\beta^3_1F$ is the inverse of the 3-cell below. 
\begin{center}
\adjustbox{width=\textwidth}{
	\begin{tikzcd}[ampersand replacement=\&,sep=scriptsize]
		\&\& {F^{b'}_{c'}(A)F^{a}_{c'}(B)F^{a}_{b}(C)} \&\&\& \\
		{F^{b'}_{c'}(A)F^{a}_{b'}(C)F^{a}_{c}(B)} \&\& {} \&\& {} \\
		{F^{a'}_{b'}(C)F^{b'}_{c}(A)F^{a}_{c}(B)} \& {} \&\& {} \\
		{F^{a'}_{b'}(C)F^{a'}_{c}(B)F^{b}_{c}(A)} \&\& {F^{a'}_{b'}(C)F^{b'}_{c}(A)F^{a}_{c}(B)} \&\& {F^{b'}_{c'}(A)F^{a}_{b'}(C)F^{a}_{c}(B)} \\
		\&\& {F^{a'}_{c'}(B)F^{a'}_{b}(C)F^{b}_{c}(A)} \&\& {F^{a'}_{c'}(B)F^{b}_{c'}(A)F^{a}_{b}(C)} \& {F^{b'}_{c'}(A)F^{a}_{c'}(B)F^{a}_{b}(C)} \\
		\&\&\&\&\& {F^{a'}_{c'}(B)F^{b}_{c'}(A)F^{a}_{b}(C)} \\
		\&\& {} \&\& {} \& {F^{a'}_{c'}(B)F^{a'}_{b}(C)F^{b}_{c}(A)} \\
		\&\&\&\& {F^{a'}_{b'}(C)F^{a'}_{c}(B)F^{b}_{c}(A)}
		\arrow["{1\circ {(B_C^o)}^{a}}"', from=1-3, to=2-1]
		\arrow[curve={height=-30pt}, equals, from=1-3, to=5-6]
		\arrow["{{(A_C^o)}^{b'}\circ1}"', from=2-1, to=3-1]
		\arrow[curve={height=-18pt}, equals, from=2-1, to=4-5]
		\arrow["{\eta_{A,b',C}\circ1}"'{pos=0.7}, shift left=5, between={0.6}{0.9}, Rightarrow, from=2-3, to=4-3]
		\arrow["{1\circ\eta_{a,B,C}}"{pos=0.6}, between={0.4}{0.8}, Rightarrow, from=2-5, to=3-4]
		\arrow["{1\circ{(A_B^o)}^c}"', from=3-1, to=4-1]
		\arrow[curve={height=-18pt}, equals, from=3-1, to=4-3]
		\arrow[""{name=0, anchor=center, inner sep=0}, "{1\circ {A_B}^c}"{description}, from=4-1, to=4-3]
		\arrow["{{B_C}^{a'}\circ1}"', from=4-1, to=5-3]
		\arrow[curve={height=30pt}, equals, from=4-1, to=8-5]
		\arrow["{{A_C}^{b'}\circ1}", from=4-3, to=4-5]
		\arrow["{(A\mid B\mid C)}"'{pos=0.3}, between={0.3}{0.7}, Rightarrow, from=4-3, to=5-5]
		\arrow["{1\circ {B_C}^{a}}", from=4-5, to=5-6]
		\arrow["{1\circ {A_C}^b}"', from=5-3, to=5-5]
		\arrow["{\epsilon_{a',B,C}\circ1}"{pos=0.4}, between={0.2}{0.6}, Rightarrow, from=5-3, to=7-3]
		\arrow[curve={height=18pt}, equals, from=5-3, to=7-6]
		\arrow[""{name=1, anchor=center, inner sep=0}, "{{A_B}^{c'}\circ1}"{description}, from=5-5, to=5-6]
		\arrow[""{name=2, anchor=center, inner sep=0}, curve={height=18pt}, equals, from=5-5, to=6-6]
		\arrow["{1\circ\epsilon_{A,b,C}}"'{pos=0.3}, shift right=5, between={0.2}{0.4}, Rightarrow, from=5-5, to=7-5]
		\arrow["{{(A_B^o)}^{c'}\circ1}", from=5-6, to=6-6]
		\arrow["{1\circ {(A_C^o)}^b}", from=6-6, to=7-6]
		\arrow["{{(B_C^o)}^{a'}\circ1}"{pos=0.3}, from=7-6, to=8-5]
		\arrow["{1\circ\eta_{A,B,c}}"'{pos=0.4}, between={0.2}{0.6}, Rightarrow, from=3-2, to=0]
		\arrow["{\epsilon_{A,B,c'}\circ1}"{pos=0.3}, between={0.3}{0.7}, Rightarrow, from=1, to=2]
	\end{tikzcd}
}

\end{center}
Similarly, one can check directly from the descriptions of $\beta^3_2$ and $\beta^3_1$ that also the incubator of $\beta^3_2\beta^3_1\beta^3_2F$ is definitionally the 3-cell above. 
\end{proof}

\begin{prop}
\label{prop:first-3-braid-eq}
For any pair of loose binary maps $G\colon \ca,\cx\to\cd$ and $F\colon\cb,\cc\to\cx$,  $$\beta^2G\circ_1F=\beta^3_2\beta^3_1(G\circ_2F).$$
\end{prop}

\begin{proof}
We need to check that $\beta^2G\circ_1F$ and $\beta^3_2\beta^3_1(G\circ_2F)$ have the same binary maps components and the same incubator. 
The binary maps component can be directly checked 
using the descriptions (and naturality) of $\beta^2$, $\beta^3_1$, and $\beta^3_2$ respectively in \Cref{prop:beta_2-def}, and Tables \ref{tab:beta-3-1} and \ref{tab:beta-3-2}, together with the explicit descriptions of substitutions $\circ_1$ and $\circ_2$ (see \Cref{tab:compact-sub-psd}). 
For instance, at $b\in\cb$:
\begin{align*}
	(\beta^2G\circ_1F)^b
	&= \beta^2G\circ F^b=\beta^2(G\circ F\midscript{b});\\
	[\beta^3_2\beta^3_1(G\circ_2F)]^b
	& = \beta^2 [\beta^3_1(G\circ_2F)]^b
	=\beta^2 [(G\circ_2F)\midscript{b}]
	=\beta^2(G\circ F^b).
\end{align*}
Now let us look at the incubators at 1-cells $B\colon b\to b'\in\cb$, $C\colon c\to c'\in\cc$ and $A\colon a\to a'\in\ca$.
We first consider $\beta^2G\circ_1 F$ and use the notation $\overline{G}:=\beta^2G$ in order to make the diagram below more readable. 
Using the description of the incubator for substitution $\circ_1$ of binary into binary (see \cite[Appendix E.1]{BourkeLobbia}), we see that the incubator of $\beta^2G\circ_1F$ is the pasting below.
\begin{center}
\scalebox{0.9}{
	\begin{tikzcd}[ampersand replacement=\&]
		{\overline{G}^{F^{b'}_{c'}}_{A}\overline{G}_{a}^{F^{b'}_C}\overline{G}^{F^B_{c}}_{a}} \&\& {\overline{G}^{F^{b'}_{c'}}_{A}\overline{G}_{a}^{F^{b'}_CF^B_{c}}} \&\& {\overline{G}^{F^{b'}_{c'}}_{A}\overline{G}_{a}^{F^B_{c'}F^{b}_C}} \&\& {\overline{G}^{F^{b'}_{c'}}_{A}\overline{G}_{a}^{F^B_{c'}}\overline{G}_{a}^{F^{b}_C}} \\
		{\overline{G}_{a'}^{F^{b'}_C}\overline{G}^{F^{b}_{c'}}_{A}\overline{G}^{F^B_{c}}_{a}} \&\&\&\&\&\& {\overline{G}_{a'}^{F^B_{c'}}\overline{G}^{F^{b}_{c'}}_{A}\overline{G}_{a}^{F^{b}_C}} \\
		{\overline{G}_{a'}^{F^{b'}_C}\overline{G}^{F^B_{c}}_{a'}\overline{G}^{F^{b}_{c}}_{A}} \&\& {\overline{G}_{a'}^{F^{b'}_CF^B_{c}}\overline{G}^{F^{b}_{c}}_{A}} \&\& {\overline{G}_{a'}^{F^B_{c'}F^{b}_C}\overline{G}^{F^{b}_{c}}_{A}} \&\& {\overline{G}_{a'}^{F^B_{c'}}\overline{G}_{a'}^{F^{b}_C}\overline{G}^{F^{b}_{c}}_{A}}
		\arrow[""{name=0, anchor=center, inner sep=0}, "{1\circ \overline{G}^{F^{b'}_C,F^B_{c}}_{a}}", from=1-1, to=1-3]
		\arrow["{\overline{({F_C^{b'})}_A}\circ1}"', from=1-1, to=2-1]
		\arrow[""{name=1, anchor=center, inner sep=0}, "{1\circ \overline{G}_a(A_B)}", from=1-3, to=1-5]
		\arrow["{\overline{{(F^{b'}_CF^B_{c})}_A}}"{description}, from=1-3, to=3-3]
		\arrow[""{name=2, anchor=center, inner sep=0}, "{1\circ (\overline{G}^{F^{B}_c,F^b_{C}}_{a})^{-1}}", from=1-5, to=1-7]
		\arrow["{\overline{(F^B_{c'}F^{b'}_C)_A}}"{description}, from=1-5, to=3-5]
		\arrow["{\overline{(F^B_{c'})_A}\circ1}", from=1-7, to=2-7]
		\arrow["{1\circ \overline{({F_c^B})_A}}"', from=2-1, to=3-1]
		\arrow["{1\circ \overline{(F^b_C)_A}}", from=2-7, to=3-7]
		\arrow[""{name=3, anchor=center, inner sep=0}, "{\overline{G}^{F^{b'}_C,F^B_{c}}_{a'}\circ1}"', from=3-1, to=3-3]
		\arrow[""{name=4, anchor=center, inner sep=0}, "{\overline{G}_{a'}(B_C)\circ1}"', from=3-3, to=3-5]
		\arrow[""{name=5, anchor=center, inner sep=0}, "{(\overline{G}^{F^{B}_c,F^b_{C}}_{a'})^{-1}\circ1}"', from=3-5, to=3-7]
		\arrow["{\overline{G}^{F^{b'}_C,F^B_{c}}_{A}}"{description}, between={0.3}{0.7}, Rightarrow, from=0, to=3]
		\arrow["{\overline{(B_C)_A}}"{description}, between={0.3}{0.7}, Rightarrow, from=1, to=4]
		\arrow["{(\overline{G}_A^{F^B_{c'},F^{b}_C})\sharp}"{description}, between={0.3}{0.7}, Rightarrow, from=2, to=5]
	\end{tikzcd}
}
=
\scalebox{0.9}{
	\begin{tikzcd}[ampersand replacement=\&]
		{G_{F^{b'}_{c'}}^{A}G^{a}_{F^{b'}_C}G_{F^B_{c}}^{a}} \&\& {G_{F^{b'}_{c'}}^{A}G^{a}_{F^{b'}_CF^B_{c}}} \&\& {G_{F^{b'}_{c'}}^{A}G^{a}_{F^B_{c'}F^{b}_C}} \&\& {G_{F^{b'}_{c'}}^{A}G^{a}_{F^B_{c'}}G^{a}_{F^{b}_C}} \\
		{G^{a'}_{F^{b'}_C}G_{F^{b}_{c'}}^{A}{G}_{F^B_{c}}^{a}} \&\&\&\&\&\& {G^{a'}_{F^B_{c'}}G_{F^{b}_{c'}}^{A}G^{a}_{F^{b}_C}} \\
		{G^{a'}_{F^{b'}_C}G_{F^B_{c}}^{a'}G_{F^{b}_{c}}^{A}} \&\& {G^{a'}_{F^{b'}_CF^B_{c}}G_{F^{b}_{c}}^{A}} \&\& {G^{a'}_{F^B_{c'}F^{b}_C}G_{F^{b}_{c}}^{A}} \&\& {G^{a'}_{F^B_{c'}}G^{a'}_{F^{b}_C}G_{F^{b}_{c}}^{A}}
		\arrow[""{name=0, anchor=center, inner sep=0}, "{1\circ G_{F^{b'}_C,F^B_{c}}^{a}}", from=1-1, to=1-3]
		\arrow["{A^o_{F_C^{b'}}\circ1}"', from=1-1, to=2-1]
		\arrow[""{name=1, anchor=center, inner sep=0}, "{1\circ G^a(A_B)}", from=1-3, to=1-5]
		\arrow["{A^o_{(F^{b'}_CF^B_{c})}}"{description}, from=1-3, to=3-3]
		\arrow[""{name=2, anchor=center, inner sep=0}, "{1\circ (G_{F^{B}_c,F^b_{C}}^{a})^{-1}}", from=1-5, to=1-7]
		\arrow["{A^o_{(F^B_{c'}F^{b'}_C)}}"{description}, from=1-5, to=3-5]
		\arrow["{A^o_{F^B_{c'}}\circ1}", from=1-7, to=2-7]
		\arrow["{1\circ A^o_{F_c^B}}"', from=2-1, to=3-1]
		\arrow["{1\circ A^o_{F^b_C}}", from=2-7, to=3-7]
		\arrow[""{name=3, anchor=center, inner sep=0}, "{G_{F^{b'}_C,F^B_{c}}^{a'}\circ1}"', from=3-1, to=3-3]
		\arrow[""{name=4, anchor=center, inner sep=0}, "{G^{a'}(B_C)\circ1}"', from=3-3, to=3-5]
		\arrow[""{name=5, anchor=center, inner sep=0}, "{(G_{F^{B}_c,F^b_{C}}^{a'})^{-1}\circ1}"', from=3-5, to=3-7]
		\arrow["{\overline{G}^{F^{b'}_C,F^B_{c}}_{A}}"{description}, between={0.3}{0.7}, Rightarrow, from=0, to=3]
		\arrow["{\overline{(B_C)_A}}"{description}, between={0.3}{0.7}, Rightarrow, from=1, to=4]
		\arrow["{(\overline{G}_A^{F^B_{c'},F^{b}_C})\sharp}"{description}, between={0.3}{0.7}, Rightarrow, from=2, to=5]
	\end{tikzcd}
}
\end{center}
We recall that, by definition, $\overline{G}^{F^{b'}_C,F^B_{c}}_{A}$ and $\overline{(B_C)_A}$ are mates of ${G}_{F^{b'}_C,F^B_{c}}^{A}$ and $A_{(B_C)}$, respectively, and $(\overline{G}_A^{F^B_{c'},F^{b}_C})\sharp$ is given by the construction described in \Cref{rmk:mates-over-adj-equiv-calculations} on $\overline{G}_A^{F^B_{c'},F^{b}_C}$, with respect to the invertible 2-cells $1\circ G_{F^{B}_c,F^b_{C}}^{a}$ and $G_{F^{B}_c,F^b_{C}}^{a'}\circ1$ seen as adjoint equivalences. 
On the other hand, following the first calculation in the proof of \Cref{prop:symm-3} (see \eqref{diag:incub-beta-2-beta-1-of-F}) and the description of the incubator of $G\circ_2F$ (see Appendix~\ref{app:inc-sub-b-in-b-1st}), we see that the incubator of $\beta^3_2\beta^3_1(G\circ F)$ is given by the following pasting diagram.
\begin{center}
	\scalebox{0.8}{
		\begin{tikzcd}[ampersand replacement=\&,column sep=scriptsize]
			{G_{F^{b'}_{c'}}^{A}G^{a}_{F^{b'}_C}G_{F^B_{c}}^{a}} \&\& {} \&\& \\
			{G^{a'}_{F^{b'}_C}G_{F^{b}_{c'}}^{A}{G}_{F^B_{c}}^{a}} \& {} \\
			{G^{a'}_{F^{b'}_C}G_{F^B_{c}}^{a'}G_{F^{b}_{c}}^{A}} \&\& {G^{a'}_{F^{b'}_C}G_{F^{b}_{c'}}^{A}{G}_{F^B_{c}}^{a}} \&\& {G_{F^{b'}_{c'}}^{A}G^{a}_{F^{b'}_C}G_{F^B_{c}}^{a}} \\
			{G^{a'}_{F^{b'}_CF^B_{c}}G_{F^{b}_{c}}^{A}} \&\&\&\& {G_{F^{b'}_{c'}}^{A}G^{a}_{F^{b'}_CF^B_{c}}} \\
			{G^{a'}_{F^B_{c'}F^{b}_C}G_{F^{b}_{c}}^{A}} \&\&\&\& {G_{F^{b'}_{c'}}^{A}G^{a}_{F^B_{c'}F^{b}_C}} \\
			{G^{a'}_{F^B_{c'}}G^{a'}_{F^{b}_C}G_{F^{b}_{c}}^{A}} \&\& {G^{a'}_{F^B_{c'}}G_{F^{b}_{c'}}^{A}G^{a}_{F^{b}_C}} \&\& {G_{F^{b'}_{c'}}^{A}G^{a}_{F^B_{c'}}G^{a}_{F^{b}_C}} \\
			\&\&\& {} \& {G^{a'}_{F^B_{c'}}G_{F^{b}_{c'}}^{A}G^{a}_{F^{b}_C}} \\
			\&\& {} \&\& {G^{a'}_{F^B_{c'}}G^{a'}_{F^{b}_C}G_{F^{b}_{c}}^{A}}
			\arrow["{A^o_{F_C^{b'}}\circ1}"', from=1-1, to=2-1]
			\arrow[curve={height=-18pt}, equals, from=1-1, to=3-5]
			\arrow["{\eta^A_{F_C^{b'}}\circ1}"'{pos=0.6}, shift left=5, between={0.6}{0.8}, Rightarrow, from=1-3, to=3-3]
			\arrow["{1\circ A^o_{F_c^B}}"', from=2-1, to=3-1]
			\arrow[curve={height=-18pt}, equals, from=2-1, to=3-3]
			\arrow[""{name=0, anchor=center, inner sep=0}, "{1\circ A_{F_c^B}}"{description}, from=3-1, to=3-3]
			\arrow["{G_{F^{b'}_C,F^B_{c}}^{a'}\circ1}"', from=3-1, to=4-1]
			\arrow["{A_{F_C^{b'}}\circ1}", from=3-3, to=3-5]
			\arrow["{1\circ G_{F^{b'}_C,F^B_{c}}^{a}}", from=3-5, to=4-5]
			\arrow[""{name=1, anchor=center, inner sep=0}, "{A_{(F^{b'}_CF^B_{c})}}"{description}, from=4-1, to=4-5]
			\arrow["{G^{a'}(B_C)\circ1}"', from=4-1, to=5-1]
			\arrow["{1\circ G^a(A_B)}", from=4-5, to=5-5]
			\arrow[""{name=2, anchor=center, inner sep=0}, "{A_{(F^B_{c'}F^{b}_C)}}"{description}, from=5-1, to=5-5]
			\arrow["{(G_{F^{B}_c,F^b_{C}}^{a'})^{-1}\circ1}"', from=5-1, to=6-1]
			\arrow["{1\circ (G_{F^{B}_c,F^b_{C}}^{a})^{-1}}", from=5-5, to=6-5]
			\arrow["{1\circ A_{F^b_C}}"', from=6-1, to=6-3]
			\arrow[curve={height=18pt}, equals, from=6-1, to=8-5]
			\arrow[""{name=3, anchor=center, inner sep=0}, "{A_{F^B_{c'}}\circ1}"{description}, from=6-3, to=6-5]
			\arrow[curve={height=18pt}, equals, from=6-3, to=7-5]
			\arrow["{1\circ\epsilon^A_{F^b_C}}"{pos=0.3}, shift right=5, between={0.2}{0.4}, Rightarrow, from=6-3, to=8-3]
			\arrow["{A^o_{F^B_{c'}}\circ1}", from=6-5, to=7-5]
			\arrow["{1\circ A^o_{F^b_C}}", from=7-5, to=8-5]
			\arrow["{1\circ\eta^A_{F_c^B}}"'{pos=0.4}, between={0.2}{0.6}, Rightarrow, from=2-2, to=0]
			\arrow["{G^A_{F^{b'}_C,F^B_{c}}}"{pos=0.4}, between={0.2}{0.6}, Rightarrow, from=3-3, to=1]
			\arrow["{A_{(B_C)}}", between={0.3}{0.7}, Rightarrow, from=1, to=2]
			\arrow["{(G^A_{F^B_{c'},F^{b}_C})\flat}"{pos=0.6}, between={0.4}{0.8}, Rightarrow, from=2, to=6-3]
			\arrow["{\epsilon^A_{F^B_{c'}}\circ1}", between={0.3}{0.7}, Rightarrow, from=3, to=7-4]
		\end{tikzcd}
	}
\end{center}
The equality we need to prove is equivalent to showing that pasting on both sides with $\overline{G}^{F^B_{c'},F^b_{C}}_A$ produces the same composite. 
This, using \Cref{rmk:mates-over-adj-equiv-calculations} for $(\overline{G}^{F^B_{c'},F^b_{C}}_A)\sharp$, reduced to prove the following equality. 
\begin{center}
	(LHS):=
		\scalebox{0.9}{
			\begin{tikzcd}[ampersand replacement=\&]
				{G_{F^{b'}_{c'}}^{A}G^{a}_{F^{b'}_C}G_{F^B_{c}}^{a}} \&\& {G_{F^{b'}_{c'}}^{A}G^{a}_{F^{b'}_CF^B_{c}}} \&\& {G_{F^{b'}_{c'}}^{A}G^{a}_{F^B_{c'}F^{b}_C}} \\
				{G^{a'}_{F^{b'}_C}G_{F^{b}_{c'}}^{A}{G}_{F^B_{c}}^{a}} \\
				{G^{a'}_{F^{b'}_C}G_{F^B_{c}}^{a'}G_{F^{b}_{c}}^{A}} \&\& {G^{a'}_{F^{b'}_CF^B_{c}}G_{F^{b}_{c}}^{A}} \&\& {G^{a'}_{F^B_{c'}F^{b}_C}G_{F^{b}_{c}}^{A}}
				\arrow[""{name=0, anchor=center, inner sep=0}, "{1\circ G_{F^{b'}_C,F^B_{c}}^{a}}", from=1-1, to=1-3]
				\arrow["{A^o_{F_C^{b'}}\circ1}"', from=1-1, to=2-1]
				\arrow[""{name=1, anchor=center, inner sep=0}, "{1\circ G^a(A_B)}", from=1-3, to=1-5]
				\arrow["{A^o_{(F^{b'}_CF^B_{c})}}"{description}, from=1-3, to=3-3]
				\arrow["{A^o_{(F^B_{c'}F^{b'}_C)}}"{description}, from=1-5, to=3-5]
				\arrow["{1\circ A^o_{F_c^B}}"', from=2-1, to=3-1]
				\arrow[""{name=2, anchor=center, inner sep=0}, "{G_{F^{b'}_C,F^B_{c}}^{a'}\circ1}"', from=3-1, to=3-3]
				\arrow[""{name=3, anchor=center, inner sep=0}, "{G^{a'}(B_C)\circ1}"', from=3-3, to=3-5]
				\arrow["{\overline{G}^{F^{b'}_C,F^B_{c}}_{A}}"{description}, between={0.3}{0.7}, Rightarrow, from=0, to=2]
				\arrow["{\overline{(B_C)_A}}"{description}, between={0.3}{0.7}, Rightarrow, from=1, to=3]
			\end{tikzcd}
		}
		= \\
\adjustbox{width=\textwidth}{
		\begin{tikzcd}[ampersand replacement=\&,sep=scriptsize]
			{G_{F^{b'}_{c'}}^{A}G^{a}_{F^{b'}_C}G_{F^B_{c}}^{a}} \&\& {} \&\&\&\& \\
			{G^{a'}_{F^{b'}_C}G_{F^{b}_{c'}}^{A}{G}_{F^B_{c}}^{a}} \& {} \\
			{G^{a'}_{F^{b'}_C}G_{F^B_{c}}^{a'}G_{F^{b}_{c}}^{A}} \&\& {G^{a'}_{F^{b'}_C}G_{F^{b}_{c'}}^{A}{G}_{F^B_{c}}^{a}} \&\& {G_{F^{b'}_{c'}}^{A}G^{a}_{F^{b'}_C}G_{F^B_{c}}^{a}} \\
			{G^{a'}_{F^{b'}_CF^B_{c}}G_{F^{b}_{c}}^{A}} \&\&\&\& {G_{F^{b'}_{c'}}^{A}G^{a}_{F^{b'}_CF^B_{c}}} \\
			{G^{a'}_{F^B_{c'}F^{b}_C}G_{F^{b}_{c}}^{A}} \&\&\&\& {G_{F^{b'}_{c'}}^{A}G^{a}_{F^B_{c'}F^{b}_C}} \\
			{G^{a'}_{F^B_{c'}}G^{a'}_{F^{b}_C}G_{F^{b}_{c}}^{A}} \&\& {G^{a'}_{F^B_{c'}}G_{F^{b}_{c'}}^{A}G^{a}_{F^{b}_C}} \&\& {G_{F^{b'}_{c'}}^{A}G^{a}_{F^B_{c'}}G^{a}_{F^{b}_C}} \&\& {G_{F^{b'}_{c'}}^{A}G^{a}_{F^B_{c'}F^{b}_C}} \\
			\&\&\& {} \& {G^{a'}_{F^B_{c'}}G_{F^{b}_{c'}}^{A}G^{a}_{F^{b}_C}} \\
			\&\& {} \&\& {G^{a'}_{F^B_{c'}}G^{a'}_{F^{b}_C}G_{F^{b}_{c}}^{A}} \&\& {G^{a'}_{F^B_{c'}F^{b}_C}G_{F^{b}_{c}}^{A}}
			\arrow["{A^o_{F_C^{b'}}\circ1}"', from=1-1, to=2-1]
			\arrow[curve={height=-18pt}, equals, from=1-1, to=3-5]
			\arrow["{\eta^A_{F_C^{b'}}\circ1}"'{pos=0.6}, shift left=5, between={0.6}{0.8}, Rightarrow, from=1-3, to=3-3]
			\arrow["{1\circ A^o_{F_c^B}}"', from=2-1, to=3-1]
			\arrow[curve={height=-18pt}, equals, from=2-1, to=3-3]
			\arrow[""{name=0, anchor=center, inner sep=0}, "{1\circ A_{F_c^B}}"{description}, from=3-1, to=3-3]
			\arrow["{G_{F^{b'}_C,F^B_{c}}^{a'}\circ1}"', from=3-1, to=4-1]
			\arrow["{A_{F_C^{b'}}\circ1}", from=3-3, to=3-5]
			\arrow["{1\circ G_{F^{b'}_C,F^B_{c}}^{a}}", from=3-5, to=4-5]
			\arrow[""{name=1, anchor=center, inner sep=0}, "{A_{(F^{b'}_CF^B_{c})}}"{description}, from=4-1, to=4-5]
			\arrow["{G^{a'}(B_C)\circ1}"', from=4-1, to=5-1]
			\arrow["{1\circ G^a(A_B)}", from=4-5, to=5-5]
			\arrow[""{name=2, anchor=center, inner sep=0}, "{A_{(F^B_{c'}F^{b}_C)}}"{description}, from=5-1, to=5-5]
			\arrow["{(G_{F^{B}_c,F^b_{C}}^{a'})^{-1}\circ1}"', from=5-1, to=6-1]
			\arrow["{1\circ (G_{F^{B}_c,F^b_{C}}^{a})^{-1}}", from=5-5, to=6-5]
			\arrow["{1\circ A_{F^b_C}}"', from=6-1, to=6-3]
			\arrow[curve={height=18pt}, equals, from=6-1, to=8-5]
			\arrow[""{name=3, anchor=center, inner sep=0}, "{A_{F^B_{c'}}\circ1}"{description}, from=6-3, to=6-5]
			\arrow[curve={height=18pt}, equals, from=6-3, to=7-5]
			\arrow["{1\circ\epsilon^A_{F^b_C}}"{pos=0.3}, shift right=5, between={0.2}{0.4}, Rightarrow, from=6-3, to=8-3]
			\arrow[""{name=4, anchor=center, inner sep=0}, "{1\circ G_{F^{B}_c,F^b_{C}}^{a}}"', from=6-5, to=6-7]
			\arrow["{A^o_{F^B_{c'}}\circ1}", from=6-5, to=7-5]
			\arrow["{A^o_{(F^B_{c'}F^{b'}_C)}}", from=6-7, to=8-7]
			\arrow["{1\circ A^o_{F^b_C}}", from=7-5, to=8-5]
			\arrow[""{name=5, anchor=center, inner sep=0}, "{G_{F^{B}_c,F^b_{C}}^{a'}\circ1}"', from=8-5, to=8-7]
			\arrow["{1\circ\eta^A_{F_c^B}}"'{pos=0.4}, between={0.2}{0.6}, Rightarrow, from=2-2, to=0]
			\arrow["{G^A_{F^{b'}_C,F^B_{c}}}"{pos=0.4}, between={0.2}{0.6}, Rightarrow, from=3-3, to=1]
			\arrow["{A_{(B_C)}}", between={0.3}{0.7}, Rightarrow, from=1, to=2]
			\arrow["{(G^A_{F^B_{c'},F^{b}_C})\flat}"{pos=0.6}, between={0.4}{0.8}, Rightarrow, from=2, to=6-3]
			\arrow["{\epsilon^A_{F^B_{c'}}\circ1}", between={0.3}{0.7}, Rightarrow, from=3, to=7-4]
			\arrow["{\overline{G}_A^{F^B_{c'},F^{b}_C}}", between={0.4}{0.6}, Rightarrow, from=4, to=5]
		\end{tikzcd}		
		=:(RHS)	}
	\end{center}
				
First, we consider the left hand side and use the definitions of $\overline{G}^{F^{b'}_C,F^B_{c}}_{A}$ and $\overline{(B_C)_A}$ as mates, and one triangle equality, to obtain the following equalities. 
\begin{center}
(LHS)=
\scalebox{0.8}{
\begin{tikzcd}[ampersand replacement=\&,column sep=scriptsize,row sep=2.25em]
{G_{F^{b'}_{c'}}^{A}G^{a}_{F^{b'}_C}G_{F^B_{c}}^{a}} \&\& {} \&\&\&\& \\
{G^{a'}_{F^{b'}_C}G_{F^{b}_{c'}}^{A}{G}_{F^B_{c}}^{a}} \& {} \\
{G^{a'}_{F^{b'}_C}G_{F^B_{c}}^{a'}G_{F^{b}_{c}}^{A}} \&\& {G^{a'}_{F^{b'}_C}G_{F^{b}_{c'}}^{A}{G}_{F^B_{c}}^{a}} \&\& {G_{F^{b'}_{c'}}^{A}G^{a}_{F^{b'}_C}G_{F^B_{c}}^{a}} \\
{G^{a'}_{F^{b'}_CF^B_{c}}G_{F^{b}_{c}}^{A}} \&\&\&\& {G_{F^{b'}_{c'}}^{A}G^{a}_{F^{b'}_CF^B_{c}}} \& {} \\
\&\& {} \&\& {G^{a'}_{F^{b'}_CF^B_{c}}G_{F^{b}_{c}}^{A}} \&\& {G_{F^{b'}_{c'}}^{A}G^{a}_{F^{b'}_CF^B_{c}}} \\
\&\&\&\& {G^{a'}_{F^B_{c'}F^{b}_C}G_{F^{b}_{c}}^{A}} \&\& {G_{F^{b'}_{c'}}^{A}G^{a}_{F^B_{c'}F^{b}_C}} \\
\&\&\&\&\& {} \& {G^{a'}_{F^B_{c'}F^{b}_C}G_{F^{b}_{c}}^{A}}
\arrow["{A^o_{F_C^{b'}}\circ1}"', from=1-1, to=2-1]
\arrow[curve={height=-18pt}, equals, from=1-1, to=3-5]
\arrow["{\eta^A_{F_C^{b'}}\circ1}"'{pos=0.6}, shift left=5, between={0.6}{0.8}, Rightarrow, from=1-3, to=3-3]
\arrow["{1\circ A^o_{F_c^B}}"', from=2-1, to=3-1]
\arrow[curve={height=-18pt}, equals, from=2-1, to=3-3]
\arrow[""{name=0, anchor=center, inner sep=0}, "{1\circ A_{F_c^B}}"{description}, from=3-1, to=3-3]
\arrow["{G_{F^{b'}_C,F^B_{c}}^{a'}\circ1}"', from=3-1, to=4-1]
\arrow["{A_{F_C^{b'}}\circ1}", from=3-3, to=3-5]
\arrow["{1\circ G_{F^{b'}_C,F^B_{c}}^{a}}", from=3-5, to=4-5]
\arrow[""{name=1, anchor=center, inner sep=0}, "{A_{(F^{b'}_CF^B_{c})}}"{description}, from=4-1, to=4-5]
\arrow[curve={height=12pt}, equals, from=4-1, to=5-5]
\arrow["{A_{(F^{b'}_CF^B_{c})}^o}"{description}, from=4-5, to=5-5]
\arrow[curve={height=-12pt}, equals, from=4-5, to=5-7]
\arrow[""{name=2, anchor=center, inner sep=0}, "{A_{(F^{b'}_CF^B_{c})}}"{description}, from=5-5, to=5-7]
\arrow["{G^{a'}(B_C)\circ1}"', from=5-5, to=6-5]
\arrow["{1\circ G^a(A_B)}", from=5-7, to=6-7]
\arrow[""{name=3, anchor=center, inner sep=0}, "{A_{(F^B_{c'}F^{b}_C)}}"{description}, from=6-5, to=6-7]
\arrow[curve={height=12pt}, equals, from=6-5, to=7-7]
\arrow["{A_{(F^B_{c'}F^{b}_C)}^o}", from=6-7, to=7-7]
\arrow["{1\circ\eta^A_{F_c^B}}"'{pos=0.4}, between={0.2}{0.6}, Rightarrow, from=2-2, to=0]
\arrow["{G^A_{F^{b'}_C,F^B_{c}}}"{pos=0.2}, between={0.2}{0.6}, Rightarrow, from=3-3, to=1]
\arrow["{\epsilon^A_{(F^{b'}_CF^B_{c})}}"{pos=0.6}, between={0.3}{0.6}, Rightarrow, from=1, to=5-3]
\arrow["{\eta^A_{(F^{b'}_CF^B_{c})}}"'{pos=0.7}, shift left=2, between={0.4}{0.7}, Rightarrow, from=4-6, to=2]
\arrow["{A_{(B_C)}}", between={0.3}{0.7}, Rightarrow, from=2, to=3]
\arrow["{\epsilon^A_{(F^B_{c'}F^{b}_C)}}"{pos=0.6}, shift right=2, between={0.2}{0.6}, Rightarrow, from=3, to=7-6]
\end{tikzcd}
} \\
=	
\scalebox{0.9}{
\begin{tikzcd}[ampersand replacement=\&,sep=scriptsize]
{G_{F^{b'}_{c'}}^{A}G^{a}_{F^{b'}_C}G_{F^B_{c}}^{a}} \&\& {} \&\& \\
{G^{a'}_{F^{b'}_C}G_{F^{b}_{c'}}^{A}{G}_{F^B_{c}}^{a}} \& {} \\
{G^{a'}_{F^{b'}_C}G_{F^B_{c}}^{a'}G_{F^{b}_{c}}^{A}} \&\& {G^{a'}_{F^{b'}_C}G_{F^{b}_{c'}}^{A}{G}_{F^B_{c}}^{a}} \&\& {G_{F^{b'}_{c'}}^{A}G^{a}_{F^{b'}_C}G_{F^B_{c}}^{a}} \\
{G^{a'}_{F^{b'}_CF^B_{c}}G_{F^{b}_{c}}^{A}} \&\&\&\& {G_{F^{b'}_{c'}}^{A}G^{a}_{F^{b'}_CF^B_{c}}} \\
{G^{a'}_{F^B_{c'}F^{b}_C}G_{F^{b}_{c}}^{A}} \&\&\&\& {G_{F^{b'}_{c'}}^{A}G^{a}_{F^B_{c'}F^{b}_C}} \\
\&\& {} \&\& {G^{a'}_{F^B_{c'}F^{b}_C}G_{F^{b}_{c}}^{A}}
\arrow["{A^o_{F_C^{b'}}\circ1}"', from=1-1, to=2-1]
\arrow[curve={height=-18pt}, equals, from=1-1, to=3-5]
\arrow["{\eta^A_{F_C^{b'}}\circ1}"'{pos=0.6}, shift left=5, between={0.6}{0.8}, Rightarrow, from=1-3, to=3-3]
\arrow["{1\circ A^o_{F_c^B}}"', from=2-1, to=3-1]
\arrow[curve={height=-18pt}, equals, from=2-1, to=3-3]
\arrow[""{name=0, anchor=center, inner sep=0}, "{1\circ A_{F_c^B}}"{description}, from=3-1, to=3-3]
\arrow["{G_{F^{b'}_C,F^B_{c}}^{a'}\circ1}"', from=3-1, to=4-1]
\arrow["{A_{F_C^{b'}}\circ1}", from=3-3, to=3-5]
\arrow["{1\circ G_{F^{b'}_C,F^B_{c}}^{a}}", from=3-5, to=4-5]
\arrow[""{name=1, anchor=center, inner sep=0}, "{A_{(F^{b'}_CF^B_{c})}}"{description}, from=4-1, to=4-5]
\arrow["{G^{a'}(B_C)\circ1}"', from=4-1, to=5-1]
\arrow["{1\circ G^a(A_B)}", from=4-5, to=5-5]
\arrow[""{name=2, anchor=center, inner sep=0}, "{A_{(F^B_{c'}F^{b}_C)}}"{description}, from=5-1, to=5-5]
\arrow[curve={height=12pt}, equals, from=5-1, to=6-5]
\arrow["{A_{(F^B_{c'}F^{b}_C)}^o}", from=5-5, to=6-5]
\arrow["{1\circ\eta^A_{F_c^B}}"'{pos=0.4}, between={0.2}{0.6}, Rightarrow, from=2-2, to=0]
\arrow["{G^A_{F^{b'}_C,F^B_{c}}}"{pos=0.1}, between={0.1}{0.5}, Rightarrow, from=3-3, to=1]
\arrow["{A_{(B_C)}}", between={0.3}{0.7}, Rightarrow, from=1, to=2]
\arrow["{\epsilon^A_{(F^B_{c'}F^{b}_C)}}"{pos=0.7}, between={0.3}{0.6}, Rightarrow, from=2, to=6-3]
\end{tikzcd}
}
\end{center}

Hence to prove the desired equality, it suffices to prove that the bottom part of the right hand side (RHS) is equal to $\epsilon^A_{F^B_{c'}F^b_C}$. 
We show this using, in order: definition of $\overline{G}_A^{F^B_{c'},F^{b}_C}$; two triangle identities; \Cref{rmk:mates-over-adj-equiv-calculations} for $(G^A_{F^B_{c'},F^{b}_C})\flat$. 
\begin{center}
\adjustbox{width=\textwidth}{
\begin{tikzcd}[ampersand replacement=\&]
	\&\& {} \&\&\&\& \\
	\& {} \\
	\\
	\\
	{G^{a'}_{F^B_{c'}F^{b}_C}G_{F^{b}_{c}}^{A}} \&\&\&\& {G_{F^{b'}_{c'}}^{A}G^{a}_{F^B_{c'}F^{b}_C}} \\
	{G^{a'}_{F^B_{c'}}G^{a'}_{F^{b}_C}G_{F^{b}_{c}}^{A}} \&\& {G^{a'}_{F^B_{c'}}G_{F^{b}_{c'}}^{A}G^{a}_{F^{b}_C}} \&\& {G_{F^{b'}_{c'}}^{A}G^{a}_{F^B_{c'}}G^{a}_{F^{b}_C}} \&\& {G_{F^{b'}_{c'}}^{A}G^{a}_{F^B_{c'}F^{b}_C}} \\
	\&\&\& {} \& {G^{a'}_{F^B_{c'}}G_{F^{b}_{c'}}^{A}G^{a}_{F^{b}_C}} \\
	\&\& {} \&\& {G^{a'}_{F^B_{c'}}G^{a'}_{F^{b}_C}G_{F^{b}_{c}}^{A}} \&\& {G^{a'}_{F^B_{c'}F^{b}_C}G_{F^{b}_{c}}^{A}}
	\arrow[""{name=0, anchor=center, inner sep=0}, "{A_{(F^B_{c'}F^{b}_C)}}", from=5-1, to=5-5]
	\arrow["{(G_{F^{B}_c,F^b_{C}}^{a'})^{-1}\circ1}"', from=5-1, to=6-1]
	\arrow["{1\circ (G_{F^{B}_c,F^b_{C}}^{a})^{-1}}", from=5-5, to=6-5]
	\arrow["{1\circ A_{F^b_C}}"', from=6-1, to=6-3]
	\arrow[curve={height=18pt}, equals, from=6-1, to=8-5]
	\arrow[""{name=1, anchor=center, inner sep=0}, "{A_{F^B_{c'}}\circ1}"{description}, from=6-3, to=6-5]
	\arrow[curve={height=18pt}, equals, from=6-3, to=7-5]
	\arrow["{1\circ\epsilon^A_{F^b_C}}"{pos=0.3}, shift right=5, between={0.2}{0.4}, Rightarrow, from=6-3, to=8-3]
	\arrow[""{name=2, anchor=center, inner sep=0}, "{1\circ G_{F^{B}_c,F^b_{C}}^{a}}"', from=6-5, to=6-7]
	\arrow["{A^o_{F^B_{c'}}\circ1}", from=6-5, to=7-5]
	\arrow["{A^o_{(F^B_{c'}F^{b'}_C)}}", from=6-7, to=8-7]
	\arrow["{1\circ A^o_{F^b_C}}", from=7-5, to=8-5]
	\arrow[""{name=3, anchor=center, inner sep=0}, "{G_{F^{B}_c,F^b_{C}}^{a'}\circ1}"', from=8-5, to=8-7]
	\arrow["{(G^A_{F^B_{c'},F^{b}_C})\flat}"{pos=0.6}, between={0.4}{0.8}, Rightarrow, from=0, to=6-3]
	\arrow["{\epsilon^A_{F^B_{c'}}\circ1}", between={0.3}{0.7}, Rightarrow, from=1, to=7-4]
	\arrow["{\overline{G}_A^{F^B_{c'},F^{b}_C}}", between={0.4}{0.6}, Rightarrow, from=2, to=3]
\end{tikzcd}
=} \\
\adjustbox{width=\textwidth}{
	\begin{tikzcd}[ampersand replacement=\&]
		\&\& {} \&\&\&\&\&\& \\
		\& {} \\
		\\
		\\
		{G^{a'}_{F^B_{c'}F^{b}_C}G_{F^{b}_{c}}^{A}} \&\&\&\& {G_{F^{b'}_{c'}}^{A}G^{a}_{F^B_{c'}F^{b}_C}} \\
		{G^{a'}_{F^B_{c'}}G^{a'}_{F^{b}_C}G_{F^{b}_{c}}^{A}} \&\& {G^{a'}_{F^B_{c'}}G_{F^{b}_{c'}}^{A}G^{a}_{F^{b}_C}} \&\& {G_{F^{b'}_{c'}}^{A}G^{a}_{F^B_{c'}}G^{a}_{F^{b}_C}} \\
		\&\&\& {} \& {G^{a'}_{F^B_{c'}}G_{F^{b}_{c'}}^{A}G^{a}_{F^{b}_C}} \& {} \\
		\&\& {} \&\& {G^{a'}_{F^B_{c'}}G^{a'}_{F^{b}_C}G_{F^{b}_{c}}^{A}} \&\& {G^{a'}_{F^B_{c'}}G_{F^{b}_{c'}}^{A}G^{a}_{F^{b}_C}} \&\& {G_{F^{b'}_{c'}}^{A}G^{a}_{F^B_{c'}}G^{a}_{F^{b}_C}} \\
		\&\&\&\& {G^{a'}_{F^B_{c'}F^{b}_C}G_{F^{b}_{c}}^{A}} \&\&\&\& {G_{F^{b'}_{c'}}^{A}G^{a}_{F^B_{c'}F^{b}_C}} \\
		\&\&\&\&\&\& {} \&\& {G^{a'}_{F^B_{c'}F^{b}_C}G_{F^{b}_{c}}^{A}}
		\arrow[""{name=0, anchor=center, inner sep=0}, "{A_{(F^B_{c'}F^{b}_C)}}", from=5-1, to=5-5]
		\arrow["{(G_{F^{B}_c,F^b_{C}}^{a'})^{-1}\circ1}"', from=5-1, to=6-1]
		\arrow["{1\circ (G_{F^{B}_c,F^b_{C}}^{a})^{-1}}", from=5-5, to=6-5]
		\arrow["{1\circ A_{F^b_C}}"', from=6-1, to=6-3]
		\arrow[curve={height=18pt}, equals, from=6-1, to=8-5]
		\arrow[""{name=1, anchor=center, inner sep=0}, "{A_{F^B_{c'}}\circ1}"{description}, from=6-3, to=6-5]
		\arrow[curve={height=18pt}, equals, from=6-3, to=7-5]
		\arrow["{1\circ\epsilon^A_{F^b_C}}"{pos=0.3}, shift right=5, between={0.2}{0.4}, Rightarrow, from=6-3, to=8-3]
		\arrow["{A^o_{F^B_{c'}}\circ1}", from=6-5, to=7-5]
		\arrow[""{name=2, anchor=center, inner sep=0}, curve={height=-12pt}, equals, from=6-5, to=8-9]
		\arrow["{1\circ A^o_{F^b_C}}", from=7-5, to=8-5]
		\arrow[curve={height=-12pt}, equals, from=7-5, to=8-7]
		\arrow[""{name=3, anchor=center, inner sep=0}, from=8-5, to=8-7]
		\arrow["{G_{F^{B}_c,F^b_{C}}^{a'}\circ1}"', from=8-5, to=9-5]
		\arrow[from=8-7, to=8-9]
		\arrow["{1\circ G_{F^{B}_c,F^b_{C}}^{a}}", from=8-9, to=9-9]
		\arrow[""{name=4, anchor=center, inner sep=0}, "{A_{(F^B_{c'}F^{b'}_C)}}"{description}, from=9-5, to=9-9]
		\arrow[curve={height=12pt}, equals, from=9-5, to=10-9]
		\arrow["{A^o_{(F^B_{c'}F^{b'}_C)}}", from=9-9, to=10-9]
		\arrow["{(G^A_{F^B_{c'},F^{b}_C})\flat}"{pos=0.6}, between={0.4}{0.8}, Rightarrow, from=0, to=6-3]
		\arrow["{\epsilon^A_{F^B_{c'}}\circ1}", between={0.3}{0.7}, Rightarrow, from=1, to=7-4]
		\arrow["{\eta^A_{F^B_{c'}}\circ1}"', shift left=3, between={0.4}{0.6}, Rightarrow, from=2, to=8-7]
		\arrow["{1\circ\eta^A_{F^b_C}}"', shift left=4, between={0.4}{0.7}, Rightarrow, from=7-6, to=3]
		\arrow["{\overline{G}_A^{F^B_{c'},F^{b}_C}}"{pos=0}, between={0.1}{0.6}, Rightarrow, from=8-7, to=4]
		\arrow["{\epsilon^A_{(F^B_{c'}F^{b'}_C)}}"{pos=0.7}, between={0.3}{0.6}, Rightarrow, from=4, to=10-7]
	\end{tikzcd}
}  \\
=
	\scalebox{0.9}{
		\begin{tikzcd}[ampersand replacement=\&,sep=scriptsize]
			\&\& {} \&\&\&\& \\
			\& {} \\
			\\
			\\
			{G^{a'}_{F^B_{c'}F^{b}_C}G_{F^{b}_{c}}^{A}} \&\&\&\& {G_{F^{b'}_{c'}}^{A}G^{a}_{F^B_{c'}F^{b}_C}} \\
			{G^{a'}_{F^B_{c'}}G^{a'}_{F^{b}_C}G_{F^{b}_{c}}^{A}} \&\& {G^{a'}_{F^B_{c'}}G_{F^{b}_{c'}}^{A}G^{a}_{F^{b}_C}} \&\& {G_{F^{b'}_{c'}}^{A}G^{a}_{F^B_{c'}}G^{a}_{F^{b}_C}} \&\& {=\epsilon^A_{(F^B_{c'}F^{b'}_C)}} \\
			{G^{a'}_{F^B_{c'}F^{b}_C}G_{F^{b}_{c}}^{A}} \&\&\&\& {G_{F^{b'}_{c'}}^{A}G^{a}_{F^B_{c'}F^{b}_C}} \\
			\&\& {} \&\& {G^{a'}_{F^B_{c'}F^{b}_C}G_{F^{b}_{c}}^{A}}
			\arrow[""{name=0, anchor=center, inner sep=0}, "{A_{(F^B_{c'}F^{b}_C)}}", from=5-1, to=5-5]
			\arrow["{(G_{F^{B}_c,F^b_{C}}^{a'})^{-1}\circ1}"', from=5-1, to=6-1]
			\arrow["{1\circ (G_{F^{B}_c,F^b_{C}}^{a})^{-1}}", from=5-5, to=6-5]
			\arrow["{1\circ A_{F^b_C}}"', from=6-1, to=6-3]
			\arrow["{G_{F^{B}_c,F^b_{C}}^{a'}\circ1}"', from=6-1, to=7-1]
			\arrow["{A_{F^B_{c'}}\circ1}"{description}, from=6-3, to=6-5]
			\arrow["{1\circ G_{F^{B}_c,F^b_{C}}^{a}}", from=6-5, to=7-5]
			\arrow[""{name=1, anchor=center, inner sep=0}, "{A_{(F^B_{c'}F^{b'}_C)}}"{description}, from=7-1, to=7-5]
			\arrow[curve={height=12pt}, equals, from=7-1, to=8-5]
			\arrow["{A^o_{(F^B_{c'}F^{b'}_C)}}", from=7-5, to=8-5]
			\arrow["{(G^A_{F^B_{c'},F^{b}_C})\flat}"{pos=1}, between={0.4}{0.8}, Rightarrow, from=0, to=6-3]
			\arrow["{\overline{G}_A^{F^B_{c'},F^{b}_C}}"{pos=0}, between={0.2}{0.6}, Rightarrow, from=6-3, to=1]
			\arrow["{\epsilon^A_{(F^B_{c'}F^{b'}_C)}}"{pos=0.7}, between={0.3}{0.6}, Rightarrow, from=1, to=8-3]
		\end{tikzcd}
	}
\end{center}
This concludes the proof that the incubator of $\beta^2G\circ_1F$ is the same as the one of $\beta^3_2\beta^3_1(G\circ_2F)$, and therefore that these pseudo ternary maps are equal. 
			
\end{proof}

\begin{rmk}
\label{rmk:def-right-closedness}
In order to establish a number more equations relating the ternary braidings and substitution of binary into binary, it will be useful to first define the  2- and 3-ary right closedness isomorphisms.  This is done using left closedness (\Cref{lem:left-closedness-psd}), $\beta^2$ (for $\rho_2$) and $\beta^3_i$ for $\rho_3$.  The definitions are as follows.
\begin{align*}
	\rho_2:= & \psd(\cb;\psd(\ca,\cc))\stackrel{\lambda_2}{\cong}\GCatmp_2^l(\cb,\ca;\cc)\stackrel{\beta^2}{\cong}\GCatmp_2(\ca,\cb;\cc) \\
	\rho_3:=&\GCatmp_2^l(\cb,\cc;\psd(\ca,\cd))\stackrel{\lambda_3}{\cong}\GCatmp_3^l(\cb,\cc,\ca;\cd)\stackrel{\beta^3_2}{\cong}\GCatmp_3^l(\cb,\ca,\cc;\cd)\stackrel{\beta^3_1}{\cong}\GCatmp_3^l(\ca,\cb,\cc;\cd)
\end{align*} 
Using \Cref{cor:beta-2-self-inv} and \Cref{prop:beta-3-self-inverses}, we can describe the inverses of $\rho_2$ and $\rho_3$ as below.
\begin{center}
	$\rho_2^{-1}=\lambda_2^{-1}\beta^2$ \hspace{0.5cm}
	$\rho_3^{-1}=\lambda_3^{-1}\beta^3_2\beta^3_1$
\end{center}
\end{rmk}

\begin{cor}
	\label{cor:rho-preserves-subst}
	The right closedness isomorphism $\rho_3$ respects substitution, i.e. 
	$$\rho_3^{-1}(G\circ_2F)=\rho_2^{-1}G\circ F.$$
\end{cor}

\begin{proof}
	We use (in order): the definition of $\rho_3$; \Cref{prop:first-3-braid-eq}; definition of $\circ_1$ (\Cref{rmk:subst-in-terms-of-lambdas}); definition of $\rho_2$.
	$$\rho_3(G\circ_2F)=\lambda_3^{-1}\beta^3_2\beta^3_1(G\circ_2F)
	=\lambda_3^{-1}(\beta^2G\circ_1F)
	=\lambda_2^{-1}\beta^2G\circ F=\rho_2^{-1}G\circ F$$
\end{proof}
%
%
\begin{prop}
	\label{prop:symm-2-2}
	For any pair of loose binary maps $G\colon \ca,\cx\to\cd$ and $F\colon\cb,\cc\to\cx$, the following equations hold. 
	\begin{enumerate}[(i)]
		\item $\beta^2G\circ_1F=\beta^3_2\beta^3_1(G\circ_2F)$
		\item $G\circ_1\beta^2F=\beta^3_2(G\circ_2F)$
	\end{enumerate}
	For any pair of loose binary maps $G\colon \cx,\cc\to\cd$ and $F\colon\ca,\cb\to\cx$, the following equations hold. 
	\begin{enumerate}
		\item[(iii)] $\beta^2G\circ_2F=\beta^3_1\beta^3_2(G\circ_1F)$
		\item[(iv)] $G\circ_1\beta^2F=\beta^3_1(G\circ_1F)$
	\end{enumerate}
\end{prop}
\begin{proof}
	
\begin{enumerate}[(i)]
\item[]
\item This is \Cref{prop:first-3-braid-eq}.


\item We will prove that both sides become the same under composition by $\rho_3^{-1}$.
Let us start with the left hand side.
\begin{align*}
	\rho_3^{-1}(G\circ_2\beta^2F)
	&=\rho_2^{-1}G\circ\beta^2F
	& \text{(\Cref{cor:rho-preserves-subst})} \\
	&=\beta^2(\rho_2^{-1}G\circ F)
	& \text{(naturality of $\beta^2$)} 
\end{align*}
On the other hand:
\begin{align*}
	\rho_3^{-1}\beta^3_2(G\circ_2F)
	&=\lambda_3^{-1}\beta^3_2\beta^3_1\beta^3_2(G\circ_2F)
	& \text{(definition of $\rho_3$, see \ref{rmk:def-right-closedness})} \\
	&=\lambda_3^{-1}\beta^3_1\beta^3_2\beta^3_1(G\circ_2F)
	& \text{(\Cref{prop:symm-3})} \\
	&=\beta^2\lambda_3^{-1}\beta^3_2\beta^3_1(G\circ_2F)
	& \text{(definition of $\beta^3_1$, see \Cref{sec:higher-braidings})} \\
	&=\beta^2\rho_3^{-1}(G\circ_2F)
	& \text{(definition of $\rho_3^{-1}$, see \ref{rmk:def-right-closedness})} \\
	&=\beta^2(\rho_2^{-1}G\circ_2F)
	& \text{(\Cref{cor:rho-preserves-subst})}.
\end{align*}


\item This time we show that the equation holds after composing both sides with $\rho_3^{-1}$.  We start with the left hand side:
\begin{align*}
	\rho_3^{-1}(\beta^2G\circ_2F)
	&=(\rho_2^{-1}\beta^2G)\circ F
	& \text{(\Cref{cor:rho-preserves-subst})} \\
	&= (\lambda_2^{-1}\beta^2\beta^2G)\circ F
	=\lambda_2^{-1}G\circ F
	& \text{(definition $\rho_2$ and \Cref{cor:beta-2-self-inv})}. 
\end{align*}
And then the right hand side:
\begin{align*}
	\rho_3^{-1}\beta^3_1\beta^3_2(G\circ_1F)
	&=\lambda_3^{-1}\beta^3_2\beta^3_1\beta^3_1\beta^3_2(G\circ_1F)
	& \text{(definition of $\rho_3^{-1}$, see \ref{rmk:def-right-closedness})} \\
	&=\lambda_3^{-1}(G\circ_1F)
	=\lambda_2^{-1}G\circ F
	& \text{(\Cref{prop:beta-3-self-inverses} and definition $\circ_1$, see \Cref{rmk:subst-in-terms-of-lambdas})}
\end{align*}
\item Finally, we prove the last equality by postcomposing with $\lambda_3^{-1}$. 
By definition of $\circ_1$, we have $\lambda_3^{-1}(G\circ_1\beta^2F)=\lambda_2^{-1}G\circ \beta^2F$.
Then, using (in order) the definition of $\beta^3_1$, the definition of $\circ_1$, and naturality of $\beta^2$, we obtain:
\begin{center}
	$\lambda_3^{-1}\beta^3_1(G\circ_1 F)=\beta^2\lambda_3^{-1}(G\circ_1F)=\beta^2(\lambda_2^{-1}G\circ F)=\lambda_2^{-1}G\circ \beta^2F$.
\end{center}
\end{enumerate}
\end{proof}

\subsection{The $4$-ary braidings}\label{sec:4-braidings}

The final information we need is the explicit description of the $4$-ary braidings.  These are described in \Cref{tab:4-ary-brd-descr}, noting that a 4-ary multimap is fully specified by its ternary components.

\begin{table}
		\caption{4-ary braidings for $F\colon\ca,\cb,\cc,\cd\to\ce$}
		\label{tab:4-ary-brd-descr}
		\begin{minipage}{.325\linewidth}
			\caption{$\beta^4_1$}
		\label{tab:beta-4-1}
		\centering
		\renewcommand\arraystretch{1.25}
		\begin{tabular}{| c : c |}
			\hline
			at $b\in\cb$ & $F\midscript{b}$\\
			at $a\in\ca$ & $F^a$\\
			at $c\in\cc$ & $\beta^3_1F\midscript{c}$\\
			at $d\in\cd$ & $\beta^3_1F_d$\\
			\hline
		\end{tabular}
		
	\end{minipage}
	\begin{minipage}{.325\linewidth}
		\caption{$\beta^4_2$}
	\label{tab:beta-4-2}
	\centering
	\renewcommand\arraystretch{1.25}
	\begin{tabular}{| c : c |}
		\hline
		at $a\in\ca$ & $\beta^3_1F^a$\\
		at $c\in\cc$ & $F\midscript{c}$\\
		at $b\in\cb$ & $F\midscript{b}$\\
		at $d\in\cd$ & $\beta^3_2F_d$\\
		\hline
	\end{tabular}
	
\end{minipage}%
\begin{minipage}{.325\linewidth}
	\centering
	\caption{$\beta^4_3$}
\label{tab:beta-4-3}
\renewcommand\arraystretch{1.25}
\begin{tabular}{| c : c |}
	\hline
	at $a\in\ca$ & $\beta^3_2F^a$\\
	at $b\in\cb$ & $\beta^3_2F_b$\\
	at $d\in\cd$ & $F_d$\\
	at $c\in\cc$ & $F\midscript{c}$\\
	\hline
\end{tabular}
\end{minipage} 
\end{table}

To obtain the descriptions in \Cref{tab:4-ary-brd-descr} simply requires tracing through the definitions of the 4-ary braidings from \Cref{sec:higher-braidings}, applying the definitions of $\lambda_4$ and $\lambda_3$ (see \Cref{lem:left-closedness-psd}), ternary braidings (Tables~\ref{tab:beta-3-1} and~\ref{tab:beta-3-2}), and the dualities in \Cref{prop:dualities-psd}.  Here we demonstrate how this is done, giving the complete calculations for one variable each (the rest are completely analogous).  

Let $F\colon\ca,\cb,\cc,\cd\to\ce$ be a pseudo 4-ary map, then:


\begin{itemize}
	\item $(\beta^4_1F)^b:=[\lambda_4\beta^3_1{\lambda_4}^{-1}F]^b
			=\lambda_3[(\beta^3_1{\lambda_4}^{-1}F)^b]
		=\lambda_3[({\lambda_4}^{-1}F)\midscript{b}]
			=\lambda_3^p[{\lambda_3}^{-1}(F\midscript{b})]
			=F\midscript{b}$; \vspace{0.15cm}
	
	\item $	(\beta^4_2F)^a :=[\lambda_4\beta^3_2{\lambda_4}^{-1}F]^a
	=\lambda_3[(\beta^3_2{\lambda_4}^{-1}F)^a]
	=\lambda_3[\beta^2({\lambda_4}^{-1}F)^a]
	=\lambda_3\beta^2{\lambda_3}^{-1}(F^a)
	=\beta^3_1F^a$; \vspace{0.15cm}
	
	\item $(\beta^4_3F)^a 
	:=[d^4\beta^4_1d^4F]^a
	=d^3[(\beta^4_1d^4F)_a]
	=d^3[\beta^3_1((d^4F)_a)]
	=d^3\beta^3_1d^3(F^a)
	=\beta^3_2F^a$.
\end{itemize}

\begin{prop}\label{prop:beta-4-2_3-restr-to-tight}
Both  $\beta^4_2$ and $\beta^4_3$ preserve tightness.
\end{prop} 
\begin{proof}
By definition, a pseudo $4$-ary map $\mathbf{F}\colon\ca,\cb,\cc,\cd\to\ce$ is tight precisely when its associated pseudo $3$-ary maps $F\midscript{b}$, $F\midscript{c}$ and $F_d$ are tight.  Given the description of $\beta^4_2$ and $\beta^4_3$ in Tables~\ref{tab:beta-4-2} and~\ref{tab:beta-4-3}, and using that $\beta^3_2$ preserves tightness (\Cref{prop:beta-3-2-restr-to-tight}), it then follows that also $\beta^4_2$ and $\beta^4_3$ preserve tightness. 
\end{proof}
		
\begin{prop}
	\label{prop:brd-resp-sub-2-in-3}
	The braidings $\beta^4_2$ and $\beta^4_3$ are compatible with substitution of binary maps into ternary maps, i.e.
	\begin{enumerate}
		\item For any pseudo binary $F\colon\ca,\cb\to\cx$ and pseudo ternary map $G\colon \cx,\cc,\cd\to\ce$,
		$$\beta^3_2G\circ_1F=\beta_3^4(G\circ_1F).$$
		
		\item For any pseudo binary $F\colon\cc,\cd\to\cx$ and pseudo ternary map $G\colon \ca,\cb,\cx\to\ce$,
		$$\beta^3_2G\circ_2F=\beta_3^4\beta^4_2(G\circ_3F).$$
		
		\item For any pseudo binary $F\colon\cb,\cc\to\cx$ and pseudo ternary map $G\colon \ca,\cx,\cd\to\ce$,
		$$\beta^3_2G\circ_3F=\beta_2^4\beta^4_3(G\circ_2F).$$
	\end{enumerate}
\end{prop}

\begin{proof}
	\begin{enumerate}
		\item[]
		
		\item One verifies that the ternary maps associated with $\beta^3_2G\circ_1F$ and $\beta_3^4(G\circ_1F)$ are the same using the descriptions of the braidings in Tables~\ref{tab:beta-4-2} and~\ref{tab:beta-4-3}. 
		
		
		\item Similarly as above, we just need to use the descriptions of $\beta^4_2$ and $\beta^4_3$, this time together with Part (i) of \Cref{prop:symm-2-2}.
		
		\item As above with Part (iii) of \Cref{prop:symm-2-2}.
	\end{enumerate}
\end{proof}

\begin{prop}
	\label{prop:brd-resp-sub-3-in-2}
	The braidings $\beta^4_2$ and $\beta^4_3$ are compatible with substitution of ternary maps into binary maps, i.e.
	\begin{enumerate}
		\item For any pseudo ternary $F\colon\ca,\cb,\cc\to\cx$ and pseudo binary map $G\colon \cx,\cd\to\ce$,
		$$G\circ_1\beta^3_2F=\beta_2^4(G\circ_1F).$$
		
		\item For any pseudo ternary $F\colon\cb,\cc,\cd\to\cx$ and pseudo binary map $G\colon \ca,\cx\to\ce$,
		$$G\circ_2\beta^3_2F=\beta_3^4\beta^4_2(G\circ_2F).$$
		
	\end{enumerate}
\end{prop}

\begin{proof}
	\begin{enumerate}
		\item[]
		\item We use again the explicit descriptions of the $4$-ary braidings in Tables~\ref{tab:beta-4-2} and~\ref{tab:beta-4-3} together with Part (iv) of \Cref{prop:symm-2-2}.
		
		\item As above with Part (ii) of \Cref{prop:symm-2-2}.
	\end{enumerate}
\end{proof}

\begin{prop}
	\label{prop:symm-4}
	For any pseudo 4-ary map $F\colon\ca,\cb,\cc,\cd\to\ce$, the following equation hold.
	$$ \beta^4_2\beta^4_3\beta^4_2(H)=\beta^4_3\beta^4_2\beta^4_3(H)$$
\end{prop}

\begin{proof}
	This follows directly from definitions and \Cref{prop:symm-3}. 
\end{proof}

\subsection{The symmetric skew monoidal structure on $\GCat$}
\label{sec:symm-skw-mon-structure}

Finally, in the next theorem, we prove that the flexible skew structure on $\GCat$ is symmetric using the results in \Cref{sec:4-braidings} and \cite[Theorem~5.29]{Lob:short-multi}. 

\begin{thm}\label{maintheorem}
	$(\GCat,\otimes_p,\mathbf{1})$ is a symmetric skew monoidal category, in the sense of \cite{BourkeLack:symmetric}.
\end{thm}
\begin{proof}
	It suffices by \cite[Theorem~5.29]{Lob:short-multi} to show that we have a symmetric short multicategory structure on $\GCatmp$ (see \Cref{sec:brd-skew-multi-mon}), 
	which consists of  braidings on tight maps $\beta^3_2$, $\beta^4_2$, and $\beta^4_3$ satisfying the equations 
	\cref{ax:sh-braid-4-ary,ax:sh-braid-3-in-2-first,ax:sh-braid-3-in-2-secd,ax:sh-braid-2-in-3-first,ax:sh-braid-2-in-3-secd,ax:sh-braid-2-in-3-thrd}, and moreover $\beta^3_2\beta^3_2=1$.
%
	
	In Remarks~\ref{prop:beta-3-2-restr-to-tight} and~\ref{prop:beta-4-2_3-restr-to-tight}, we defined the required braidings. 
	Then, 
	Propositions~\ref{prop:brd-resp-sub-2-in-3},~\ref{prop:brd-resp-sub-3-in-2}, and~\ref{prop:symm-4} prove the equations \cref{ax:sh-braid-4-ary,ax:sh-braid-3-in-2-first,ax:sh-braid-3-in-2-secd,ax:sh-braid-2-in-3-first,ax:sh-braid-2-in-3-secd,ax:sh-braid-2-in-3-thrd} respectively.
	Finally, $\beta^3_2\beta^3_2=1$ is proven in \Cref{prop:beta-3-self-inverses}.
\end{proof}

\begin{cor}
	We have a symmetric skew multicategory $\GCatmp$ --- in particular, its underlying multicategory $\GCatmp^l$ of loose multimaps is symmetric, and so biclosed.
\end{cor}

\begin{proof}
	This is a direct consequence of \cite[Theorem~5.9]{BourkeLack:symmetric}. 
\end{proof}

\appendix

\section{Axioms for Oplax Transformations and Modifications}\label{app:ax-oplax}
\subsection{Oplax Transformations} \label{subsec:ax-oplax-transf}
Let $F,G\colon\ca\rightrightarrows\cb$ be pseudomaps between Gray-categories.  An oplax transformation $\alpha^o\colon F\to G$ has to satisfy the following axioms, corresponding, to the axioms in A.2 of \cite{BourkeLobbia}, respectively.

\begin{enumerate}[(i)]
	\item\label{ax:oplax-tr.deg-2-cell} \emph{Degeneracy for the 2-cell component:} at an object $x\in\ca$, then $\alpha^o_{1_x}=1_{\alpha^o_x}$.
	\item\label{ax:oplax-tr.3-cell} \emph{Compatibility with 3-cells:} at $1$-cells $f,f'\colon x\to y$, $2$-cells $\phi,\phi' \colon f\to f'$ and a $3$-cell $\Gamma\colon\phi\to\phi'$, we require:
	\begin{center}
		\scalebox{0.9}{
			\begin{tikzcd}[ampersand replacement=\&]
				{\alpha_yFf} \&\& {\alpha_yFf'} \\
				\\
				{Gf\alpha_x} \&\& {Gf'\alpha_x}
				\arrow[""{name=0, anchor=center, inner sep=0}, "{1\circ F\phi}", curve={height=-12pt}, from=1-1, to=1-3]
				\arrow[""{name=1, anchor=center, inner sep=0}, "{1\circ F\phi'}"', curve={height=12pt}, from=1-1, to=1-3]
				\arrow["{\alpha^o_f}"', from=1-1, to=3-1]
				\arrow["{\alpha^o_{f'}}", from=1-3, to=3-3]
				\arrow[""{name=2, anchor=center, inner sep=0}, "{G\phi'\circ1}"', curve={height=12pt}, from=3-1, to=3-3]
				\arrow["{1\circ F\Gamma}", shift right=3, between={0.3}{0.7}, Rightarrow, from=0, to=1]
				\arrow["{\alpha^o_{\phi'}}", between={0.4}{0.6}, Rightarrow, from=1, to=2]
			\end{tikzcd}
		}
		=
				\scalebox{0.9}{
			\begin{tikzcd}[ampersand replacement=\&]
				{\alpha_yFf} \&\& {\alpha_yFf'} \\
				\\
				{Gf\alpha_x} \&\& {Gf'\alpha_x}
				\arrow[""{name=0, anchor=center, inner sep=0}, "{1\circ F\phi}", curve={height=-12pt}, from=1-1, to=1-3]
				\arrow["{\alpha^o_f}"', from=1-1, to=3-1]
				\arrow["{\alpha^o_{f'}}", from=1-3, to=3-3]
				\arrow[""{name=1, anchor=center, inner sep=0}, "{G\phi'\circ1}"', curve={height=12pt}, from=3-1, to=3-3]
				\arrow[""{name=2, anchor=center, inner sep=0}, "{G\phi\circ1}", curve={height=-12pt}, from=3-1, to=3-3]
				\arrow["{\alpha^o_\phi}", between={0.4}{0.6}, Rightarrow, from=0, to=2]
				\arrow["{G\Gamma\circ1}", shift right=3, between={0.3}{0.7}, Rightarrow, from=2, to=1]
			\end{tikzcd}
		}
	\end{center}
	
	\item\label{ax:oplax-tr.comp-2-cell} \emph{Compatibility with vertical composition of 2-cells:}
at $1$-cells $f,f',f''\colon x\to y$ and $2$-cells $\phi_1\colon f\to f',\phi_2\colon f'\to f''$ in $\ca$, we require:
\begin{center}
		\begin{tikzcd}[ampersand replacement=\&,column sep=scriptsize]
			{\alpha_yFf} \&\& {\alpha_yFf'} \&\& {\alpha_yFf''} \\
			{Gf\alpha_x} \&\& {Gf'\alpha_x} \&\& {Gf'\alpha_x}
			\arrow[""{name=0, anchor=center, inner sep=0}, "{1\circ F\phi_1}", from=1-1, to=1-3]
			\arrow["{\alpha^o_f}"', from=1-1, to=2-1]
			\arrow[""{name=1, anchor=center, inner sep=0}, "{1\circ F\phi_2}", from=1-3, to=1-5]
			\arrow["{\alpha^o_{f'}}"{description}, from=1-3, to=2-3]
			\arrow["{\alpha^o_{f''}}", from=1-5, to=2-5]
			\arrow[""{name=2, anchor=center, inner sep=0}, "{G\phi_1\circ1}"', from=2-1, to=2-3]
			\arrow[""{name=3, anchor=center, inner sep=0}, "{G\phi_2\circ1}"', from=2-3, to=2-5]
			\arrow["{\alpha^o_{\phi_1}}", between={0.4}{0.6}, Rightarrow, from=0, to=2]
			\arrow["{\alpha^o_{\phi_2}}", between={0.4}{0.6}, Rightarrow, from=1, to=3]
		\end{tikzcd}
		= 
		\begin{tikzcd}[ampersand replacement=\&,column sep=small]
			{\alpha_yFf} \&\& {\alpha_yFf'} \&\& {\alpha_yFf''} \\
			{Gf\alpha_x} \&\& {Gf'\alpha_x} \&\& {Gf'\alpha_x}
			\arrow["{1\circ F\phi_1}", from=1-1, to=1-3]
			\arrow["{\alpha^o_f}"', from=1-1, to=2-1]
			\arrow["{1\circ F\phi_2}", from=1-3, to=1-5]
			\arrow["{\alpha^o_{\phi_2\cdot\phi_1}}", between={0.3}{0.7}, Rightarrow, from=1-3, to=2-3]
			\arrow["{\alpha^o_{f''}}", from=1-5, to=2-5]
			\arrow["{G\phi_1\circ1}"', from=2-1, to=2-3]
			\arrow["{G\phi_2\circ1}"', from=2-3, to=2-5]
		\end{tikzcd}
	\end{center}

	\item \label{ax:oplax-tr.deg-3-cell}\emph{Degeneracy for the 3-cell component:} at a $1$-cell $f \colon x \to y$ we require $\alpha^o_{1_f}=1_{\alpha^o_f}$.

	\item\label{ax:oplax-tr.comp-1-cell} \emph{Compatibility with composition of 1-cells:} at composable $1$-cells $f_1\colon w\to x$, $f_2\colon x\to y$ and $f_3\colon y\to z$ in $\ca$, we require:
	\begin{center}
		\scalebox{0.75}{
\begin{tikzcd}[ampersand replacement=\&,column sep=small]
	{\alpha_w Ff_3Ff_2Ff_1} \&\& {\alpha_w F(f_3f_2)Ff_1} \&\& {\alpha_w F(f_3f_2f_1)} \\
	{Gf_3\alpha_z Ff_2Ff_1} \\
	{Gf_3Gf_2\alpha_y Ff_1} \&\& {G(f_3f_2)\alpha_y Ff_1} \\
	{Gf_3G{f_2}G{f_1}\alpha_x} \&\& {G(f_3{f_2})G{f_1} \alpha_x} \&\& {G(f_3{f_2}{f_1}) \alpha_x}
	\arrow[""{name=0, anchor=center, inner sep=0}, "{1\circ F^2_{f_3,f_2}\circ1}", from=1-1, to=1-3]
	\arrow["{{\alpha^o_{f_3}}\circ1}"', from=1-1, to=2-1]
	\arrow[""{name=1, anchor=center, inner sep=0}, "{1\circ F^2_{f_3f_2,f_1}}", from=1-3, to=1-5]
	\arrow["{{\alpha^o_{f_3f_2}}\circ1}", from=1-3, to=3-3]
	\arrow["{{\alpha^o_{f_3f_2f_1}}}", from=1-5, to=4-5]
	\arrow["{1\circ{\alpha^o_{f_2}}\circ1}"', from=2-1, to=3-1]
	\arrow[""{name=2, anchor=center, inner sep=0}, "{G^2_{f_3,f_2}\circ1}"{description}, from=3-1, to=3-3]
	\arrow["{1\circ\alpha^o_{f_1}}"', from=3-1, to=4-1]
	\arrow["{1\circ\alpha^o_{f_1}}", from=3-3, to=4-3]
	\arrow[""{name=3, anchor=center, inner sep=0}, "{G^2_{f_3,f_2}\circ1}"', from=4-1, to=4-3]
	\arrow[""{name=4, anchor=center, inner sep=0}, "{ G^2_{f_3f_2,f_1}\circ1}"', from=4-3, to=4-5]
	\arrow["{{\alpha^o}_2^{f_3,f_2}\circ1}", between={0.4}{0.6}, Rightarrow, nfold, from=0, to=2]
	\arrow["{{\alpha^o}_2^{f_3f_2,f_1}}", between={0.4}{0.6}, Rightarrow, nfold, from=1, to=4]
	\arrow[between={0.3}{0.7}, Rightarrow, nfold, dashed, from=2, to=3]
\end{tikzcd}
		}
		$=$
		\scalebox{0.75}{
\begin{tikzcd}[ampersand replacement=\&,column sep=small]
	{\alpha_w Ff_3Ff_2Ff_1} \&\& {\alpha_w Ff_3F(f_2f_1)} \&\& {\alpha_w F(f_3f_2f_1)} \\
	{Gf_3\alpha_z Ff_2Ff_1} \&\& {Gf_3\alpha_zF(f_2f_1)} \\
	{Gf_3Gf_2\alpha_y Ff_1} \\
	{Gf_3G{f_2}G{f_1}\alpha_x} \&\& {Gf_3G({f_2}{f_1}) \alpha_x} \&\& {G(f_3{f_2}{f_1}) \alpha_x}
	\arrow[""{name=0, anchor=center, inner sep=0}, "{1\circ F^2_{f_2,f_1}}", from=1-1, to=1-3]
	\arrow["{{\alpha^o_{f_3}}\circ1}"', from=1-1, to=2-1]
	\arrow[""{name=1, anchor=center, inner sep=0}, "{1\circ F^2_{f_3,f_2f_1}}", from=1-3, to=1-5]
	\arrow["{{\alpha^o_{f_3}}\circ1}", from=1-3, to=2-3]
	\arrow["{{\alpha^o_{f_3f_2f_1}}}", from=1-5, to=4-5]
	\arrow[""{name=2, anchor=center, inner sep=0}, "{1\circ F^2_{f_2,f_1}}"{description}, from=2-1, to=2-3]
	\arrow["{1\circ{\alpha^o_{f_2}}\circ1}"', from=2-1, to=3-1]
	\arrow["{1\circ{\alpha^o_{f_2f_1}}}", from=2-3, to=4-3]
	\arrow["{1\circ\alpha^o_{f_1}}"', from=3-1, to=4-1]
	\arrow[""{name=3, anchor=center, inner sep=0}, "{1\circ G^2_{f_2,f_1}\circ1}"', from=4-1, to=4-3]
	\arrow[""{name=4, anchor=center, inner sep=0}, "{G^2_{f_3,f_2f_1}\circ1}"', from=4-3, to=4-5]
	\arrow[between={0.3}{0.7}, Rightarrow, nfold, dashed, from=0, to=2]
	\arrow["{{\alpha^o}_2^{f_3,f_2f_1}}", between={0.4}{0.6}, Rightarrow, nfold, from=1, to=4]
	\arrow["{1\circ({\alpha^o})_2^{f_2,f_1}}", between={0.4}{0.6}, Rightarrow, nfold, from=2, to=3]
\end{tikzcd}
		}
	\end{center}

\item\label{ax:oplax-tr.deg-alpha2} \emph{Degeneracy for $(\alpha^o)_2^{f_2,f_1}$:} at a $1$-cell $f\colon x\to y$ in $\ca$, we require
$(\alpha^o)_2^{1_y,f}=(\alpha^o)_2^{f,1_x}=1_{\alpha^o_f}$.

\item\label{ax:oplax-tr.whisk-alpha2} \emph{Compatibility of $\alpha^2_{-,-}$ with whiskering of 2-cells:} at a diagram in $\ca$ as below
\[\begin{tikzcd}
	x & y & z,
	\arrow["{f_1}", from=1-1, to=1-2]
	\arrow[""{name=0, anchor=center, inner sep=0}, "{f_2}", curve={height=-12pt}, from=1-2, to=1-3]
	\arrow[""{name=1, anchor=center, inner sep=0}, "{f_2'}"', curve={height=12pt}, from=1-2, to=1-3]
	\arrow["\psi"{xshift=0.1cm}, shorten <=3pt, shorten >=3pt, Rightarrow, from=0, to=1]
\end{tikzcd}\] we require:
\begin{center}
	(L)
\scalebox{0.8}{
\begin{tikzcd}[ampersand replacement=\&,column sep=small]
	{\alpha_z Ff_2Ff_1} \&\& {\alpha_zF(f_2f_1)} \&\& {\alpha_zF(f'_2f_1)} \\
	{Gf_2\alpha_y Ff_1} \\
	{G{f_2}G{f_1}\alpha_x} \&\& {G({f_2}{f_1}) \alpha_x} \&\& {G({f'_2}{f_1}) \alpha_x}
	\arrow[""{name=0, anchor=center, inner sep=0}, "{1\circ F^2_{f_2,f_1}}", from=1-1, to=1-3]
	\arrow["{{\alpha^o_{f_2}}\circ1}"', from=1-1, to=2-1]
	\arrow[""{name=1, anchor=center, inner sep=0}, "{1\circ F(\psi f_1)}", from=1-3, to=1-5]
	\arrow["{{\alpha^o_{f_2f_1}}}", from=1-3, to=3-3]
	\arrow["{{\alpha^o_{f'_2f_1}}}", from=1-5, to=3-5]
	\arrow["{1\circ\alpha^o_{f_1}}"', from=2-1, to=3-1]
	\arrow[""{name=2, anchor=center, inner sep=0}, "{G^2_{f_2,f_1}\circ1}"', from=3-1, to=3-3]
	\arrow[""{name=3, anchor=center, inner sep=0}, "{G(\psi f_1)\circ1}"', from=3-3, to=3-5]
	\arrow["{({\alpha^o})_2^{f_2,f_1}}", between={0.4}{0.6}, Rightarrow, nfold, from=0, to=2]
	\arrow["{\alpha^o_{\psi f_1}}", between={0.4}{0.6}, Rightarrow, nfold, from=1, to=3]
\end{tikzcd}}
$=$
\scalebox{0.8}{
\begin{tikzcd}[ampersand replacement=\&,column sep=small]
	{\alpha_z Ff_2Ff_1} \&\& {\alpha_z Ff'_2Ff_1} \&\& {\alpha_zF(f'_2f_1)} \\
	{Gf_2\alpha_y Ff_1} \&\& {Gf'_2\alpha_y Ff_1} \\
	{G{f_2}G{f_1}\alpha_x} \&\& {G{f'_2}G{f_1}\alpha_x} \&\& {G(f'_2f_1) \alpha_x}
	\arrow[""{name=0, anchor=center, inner sep=0}, "{1\circ F\psi\circ1}", from=1-1, to=1-3]
	\arrow["{{\alpha^o_{f_2}}\circ1}"', from=1-1, to=2-1]
	\arrow[""{name=1, anchor=center, inner sep=0}, "{1\circ F^2_{f'_2,f_1}}", from=1-3, to=1-5]
	\arrow["{{\alpha^o_{f'_2}}\circ1}", from=1-3, to=2-3]
	\arrow["{{\alpha^o_{f'_2f_1}}}", from=1-5, to=3-5]
	\arrow[""{name=2, anchor=center, inner sep=0}, "{G\psi\circ1}"{description}, from=2-1, to=2-3]
	\arrow["{1\circ\alpha^o_{f_1}}"', from=2-1, to=3-1]
	\arrow["{1\circ\alpha^o_{f_1}}", from=2-3, to=3-3]
	\arrow[""{name=3, anchor=center, inner sep=0}, "{1\circ G\phi}"', from=3-1, to=3-3]
	\arrow[""{name=4, anchor=center, inner sep=0}, "{G^2_{f'_2,f_1}\circ1}"', from=3-3, to=3-5]
	\arrow["{\alpha^o_\psi\circ1}", between={0.4}{0.6}, Rightarrow, nfold, from=0, to=2]
	\arrow["{{(\alpha^o)}_2^{f_2,f'_1}}", between={0.4}{0.6}, Rightarrow, nfold, from=1, to=4]
	\arrow[between={0.3}{0.7}, Rightarrow, nfold, dashed, from=2, to=3]
\end{tikzcd}
}
\end{center}
Similarly, given a diagram in $\ca$ as below
\[\begin{tikzcd}
	x & y & z,
	\arrow[""{name=0, anchor=center, inner sep=0}, "{f_1}", curve={height=-12pt}, from=1-1, to=1-2]
	\arrow["{f_2}", from=1-2, to=1-3]
	\arrow[""{name=1, anchor=center, inner sep=0}, "{f'_1}"', curve={height=12pt}, from=1-1, to=1-2]
	\arrow["\phi"{xshift=0.1cm}, shorten <=3pt, shorten >=3pt, Rightarrow, from=0, to=1]
\end{tikzcd}\]
we require:
	\begin{center}
		(R)
	\scalebox{0.8}{
			\begin{tikzcd}[ampersand replacement=\&,column sep=small]
				{\alpha_z Ff_2Ff_1} \&\& {\alpha_zF(f_2f_1)} \&\& {\alpha_zF(f_2f'_1)} \\
				{Gf_2\alpha_y Ff_1} \\
				{G{f_2}G{f_1}\alpha_x} \&\& {G({f_2}{f_1}) \alpha_x} \&\& {G({f_2}{f'_1}) \alpha_x}
				\arrow[""{name=0, anchor=center, inner sep=0}, "{1\circ F^2_{f_2,f_1}}", from=1-1, to=1-3]
				\arrow["{{\alpha^o_{f_2}}\circ1}"', from=1-1, to=2-1]
				\arrow[""{name=1, anchor=center, inner sep=0}, "{1\circ F(f_2\phi)}", from=1-3, to=1-5]
				\arrow["{{\alpha^o_{f_2f_1}}}", from=1-3, to=3-3]
				\arrow["{{\alpha^o_{f_2f'_1}}}", from=1-5, to=3-5]
				\arrow["{1\circ\alpha^o_{f_1}}"', from=2-1, to=3-1]
				\arrow[""{name=2, anchor=center, inner sep=0}, "{G^2_{f_2,f_1}\circ1}"', from=3-1, to=3-3]
				\arrow[""{name=3, anchor=center, inner sep=0}, "{G(f_2\phi)\circ1}"', from=3-3, to=3-5]
				\arrow["{({\alpha^o})_2^{f_2,f_1}}", between={0.4}{0.6}, Rightarrow, from=0, to=2]
				\arrow["{\alpha^o_{f_2\phi}}", between={0.4}{0.6}, Rightarrow, from=1, to=3]
			\end{tikzcd}}
		$=$
		\scalebox{0.8}{
\begin{tikzcd}[ampersand replacement=\&,column sep=small]
	{\alpha_z Ff_2Ff_1} \&\& {\alpha_z Ff_2Ff'_1} \&\& {\alpha_zF(f_2f'_1)} \\
	{Gf_2\alpha_y Ff_1} \&\& {Gf_2\alpha_y Ff'_1} \\
	{G{f_2}G{f_1}\alpha_x} \&\& {G{f_2}G{f'_1}\alpha_x} \&\& {G({f_2}{f'_1}) \alpha_x}
	\arrow[""{name=0, anchor=center, inner sep=0}, "{1\circ F\phi}", from=1-1, to=1-3]
	\arrow["{{\alpha^o_{f_2}}\circ1}"', from=1-1, to=2-1]
	\arrow[""{name=1, anchor=center, inner sep=0}, "{1\circ F^2_{f_2,f'_1}}", from=1-3, to=1-5]
	\arrow["{{\alpha^o_{f_2}}\circ1}", from=1-3, to=2-3]
	\arrow["{{\alpha^o_{f_2f'_1}}}", from=1-5, to=3-5]
	\arrow[""{name=2, anchor=center, inner sep=0}, "{1\circ F\phi}"{description}, from=2-1, to=2-3]
	\arrow["{1\circ\alpha^o_{f_1}}"', from=2-1, to=3-1]
	\arrow["{1\circ\alpha^o_{f_1}}", from=2-3, to=3-3]
	\arrow[""{name=3, anchor=center, inner sep=0}, "{1\circ G\phi\circ1}"', from=3-1, to=3-3]
	\arrow[""{name=4, anchor=center, inner sep=0}, "{G^2_{f_2,f'_1}\circ1}"', from=3-3, to=3-5]
	\arrow[between={0.3}{0.7}, Rightarrow, nfold, dashed, from=0, to=2]
	\arrow["{{(\alpha^o)}_2^{f_2,f'_1}}", between={0.4}{0.6}, Rightarrow, nfold, from=1, to=4]
	\arrow["{1\circ\alpha^o_\phi}", between={0.4}{0.6}, Rightarrow, nfold, from=2, to=3]
\end{tikzcd} }
	\end{center}
\end{enumerate}

\subsection{Oplax Modifications}
\label{subsec:ax-oplax-mod}

An oplax modification $\Gamma^o\colon\alpha^o\to\beta^o$ has to satisfy the following axioms.

\begin{enumerate}[(i)]
	\item\label{ax:op-mod.deg-3} \emph{Degeneracy for the 3-cell component:} at an object $x\in\ca$, $\Gamma^o_{1_x}=1_{\Gamma^o_x}$. 

	\item\label{ax:op-mod.comp-1} \emph{Compatibility with composition of 1-cells:}
at $1$-cells $f_1\colon x\to y$ and $f_2\colon y\to z$ in $\ca$, 
\begin{center}
	\adjustbox{width=\textwidth}{
			\begin{tikzcd}[ampersand replacement=\&,column sep=tiny,row sep=small]
				{\alpha_z Ff_2Ff_1} \&\& {\alpha_zF(f_2f_1)} \&\& \\
				\& {Gf_2\alpha_y Ff_1} \&\& \textcolor{white}{{Gf_2\alpha_y Ff_1}} \\
				{\beta_z Ff_2Ff_1} \&\& {G{f_2}G{f_1}\alpha_x} \&\& {G({f_2}{f_1}) \alpha_x} \\
				\& {Gf_2\beta_y Ff_1} \\
				\&\& {G{f_2}G{f_1}\beta_x} \&\& {G({f_2}{f_1}) \beta_x}
				\arrow[""{name=0, anchor=center, inner sep=0}, "{1\circ F^2_{f_2,f_1}}", from=1-1, to=1-3]
				\arrow[""{name=1, anchor=center, inner sep=0}, "{{\alpha^o_{f_2}}\circ1}"{pos=0.7}, from=1-1, to=2-2]
				\arrow["{\Lambda_z\circ1}"', from=1-1, to=3-1]
				\arrow["{{\alpha^o_{f_2f_1}}}", from=1-3, to=3-5]
				\arrow[""{name=2, anchor=center, inner sep=0}, "{1\circ\alpha^o_{f_1}}"{pos=0.7}, from=2-2, to=3-3]
				\arrow["{1\circ\Lambda_y\circ1}"{description}, from=2-2, to=4-2]
				\arrow[""{name=3, anchor=center, inner sep=0}, "{{\beta^o_{f_2}}\circ1}"', from=3-1, to=4-2]
				\arrow[""{name=4, anchor=center, inner sep=0}, "{G^2_{f_2,f_1}\circ1}"{description}, from=3-3, to=3-5]
				\arrow["{1\circ\Lambda_x}", from=3-3, to=5-3]
				\arrow["{1\circ\Lambda_x}", from=3-5, to=5-5]
				\arrow[""{name=5, anchor=center, inner sep=0}, "{1\circ\beta^o_{f_1}}"', from=4-2, to=5-3]
				\arrow[""{name=6, anchor=center, inner sep=0}, "{G^2_{f_2,f_1}\circ1}"{description}, from=5-3, to=5-5]
				\arrow["{({\alpha^o})_2^{f_2,f_1}}", between={0.4}{0.6}, Rightarrow, from=0, to=4]
				\arrow["{\Lambda^o_{f_2}\circ1}", shift right=3, between={0.4}{0.6}, Rightarrow, from=1, to=3]
				\arrow["{1\circ\Lambda^o_{f_1}}", shift right=3, between={0.4}{0.6}, Rightarrow, from=2, to=5]
				\arrow[between={0.3}{0.7}, Rightarrow, dashed, from=4, to=6]
			\end{tikzcd}
		$=$ 
			\begin{tikzcd}[ampersand replacement=\&,column sep=tiny,row sep=small]
				{\alpha_z Ff_2Ff_1} \&\& {\alpha_zF(f_2f_1)} \&\& \\
				\&\&\& \textcolor{white}{{Gf_2\alpha_y Ff_1}} \\
				{\beta_z Ff_2Ff_1} \&\& {\beta_zF(f_2f_1)} \&\& {G({f_2}{f_1}) \alpha_x} \\
				\& {Gf_2\beta_y Ff_1} \\
				\&\& {G{f_2}G{f_1}\beta_x} \&\& {G({f_2}{f_1}) \beta_x}
				\arrow[""{name=0, anchor=center, inner sep=0}, "{1\circ F^2_{f_2,f_1}}", from=1-1, to=1-3]
				\arrow["{\Lambda_z\circ1}"', from=1-1, to=3-1]
				\arrow["{\Lambda_z\circ1}"{description}, from=1-3, to=3-3]
				\arrow[""{name=1, anchor=center, inner sep=0}, "{{\alpha^o_{f_2f_1}}}", from=1-3, to=3-5]
				\arrow[""{name=2, anchor=center, inner sep=0}, "{1\circ F^2_{f_2,f_1}}", from=3-1, to=3-3]
				\arrow["{{\beta^o_{f_2}}\circ1}"', from=3-1, to=4-2]
				\arrow[""{name=3, anchor=center, inner sep=0}, "{{\beta^o_{f_2f_1}}}"{description}, from=3-3, to=5-5]
				\arrow["{1\circ\Lambda_x}", from=3-5, to=5-5]
				\arrow["{1\circ\beta^o_{f_1}}"', from=4-2, to=5-3]
				\arrow[""{name=4, anchor=center, inner sep=0}, "{G^2_{f_2,f_1}\circ1}"{description}, from=5-3, to=5-5]
				\arrow[between={0.3}{0.7}, Rightarrow, dashed, from=0, to=2]
				\arrow["{\Lambda^o_{f_2f_1}}", between={0.4}{0.6}, Rightarrow, from=1, to=3]
				\arrow["{{\beta^o}_2^{f_2,f_1}}", between={0.4}{0.6}, Rightarrow, from=2, to=4]
			\end{tikzcd}
		}
	\end{center}

	\item\label{ax:op-mod.2-cell} \emph{Compatibility with 2-cells:}
at $1$-cells $f,f'\colon x\to y$ and a $2$-cell $\phi\colon f\to f'$, we require:
\begin{center}
\scalebox{0.8}{
\begin{tikzcd}[ampersand replacement=\&]
	{\alpha_yFf} \&\& {\alpha_yFf'} \&\& \\
	\&\& {Gf\alpha_x} \&\& {Gf'\alpha_x} \\
	{\beta_yFf} \\
	\&\& {Gf\beta_x} \&\& {Gf'\beta_x}
	\arrow[""{name=0, anchor=center, inner sep=0}, "{1\circ F\phi}", from=1-1, to=1-3]
	\arrow[""{name=1, anchor=center, inner sep=0}, "{\alpha^o_f}"{description}, from=1-1, to=2-3]
	\arrow["{\Lambda_y\circ1}"', from=1-1, to=3-1]
	\arrow["{\alpha^o_{f'}}", from=1-3, to=2-5]
	\arrow[""{name=2, anchor=center, inner sep=0}, "{G\phi\circ1}"{description}, from=2-3, to=2-5]
	\arrow["{1\circ\Lambda_x}"{description}, from=2-3, to=4-3]
	\arrow["{1\circ\Lambda_x}", from=2-5, to=4-5]
	\arrow[""{name=3, anchor=center, inner sep=0}, "{\beta^o_f}"', from=3-1, to=4-3]
	\arrow[""{name=4, anchor=center, inner sep=0}, "{G\phi\circ1}"', from=4-3, to=4-5]
	\arrow["{\alpha^o_\phi}"{pos=0.6}, shift right, between={0.4}{0.6}, Rightarrow, nfold, from=0, to=2]
	\arrow["{{\Lambda^o_f}}", between={0.4}{0.6}, Rightarrow, nfold, from=1, to=3]
	\arrow[between={0.4}{0.6}, Rightarrow, nfold, dashed, from=2, to=4]
\end{tikzcd}
}
$=$
\scalebox{0.8}{
\begin{tikzcd}[ampersand replacement=\&]
	{\alpha_yFf} \&\& {\alpha_yFf'} \&\& \\
	\&\&\&\& {Gf'\alpha_x} \\
	{\beta_yFf} \&\& {\beta_yFf'} \\
	\&\& {Gf\beta_x} \&\& {Gf'\beta_x}
	\arrow[""{name=0, anchor=center, inner sep=0}, "{1\circ F\phi}", from=1-1, to=1-3]
	\arrow["{\Lambda_y\circ1}"', from=1-1, to=3-1]
	\arrow[""{name=1, anchor=center, inner sep=0}, "{\alpha^o_{f'}}", from=1-3, to=2-5]
	\arrow["{\Lambda_y\circ1}"{description}, from=1-3, to=3-3]
	\arrow["{1\circ\Lambda_x}", from=2-5, to=4-5]
	\arrow[""{name=2, anchor=center, inner sep=0}, "{1\circ F\phi}", from=3-1, to=3-3]
	\arrow["{\beta^o_f}"', from=3-1, to=4-3]
	\arrow[""{name=3, anchor=center, inner sep=0}, "{\beta^o_{f'}}"{description}, from=3-3, to=4-5]
	\arrow[""{name=4, anchor=center, inner sep=0}, "{G\phi\circ1}"', from=4-3, to=4-5]
	\arrow[between={0.4}{0.6}, Rightarrow, nfold, dashed, from=0, to=2]
	\arrow["{{\Lambda^o_{f'}}}", between={0.4}{0.6}, Rightarrow, nfold, from=1, to=3]
	\arrow["{\beta^o_\phi}"{pos=0.6}, shift right, between={0.4}{0.6}, Rightarrow, nfold, from=2, to=4]
\end{tikzcd}
}
\end{center}

\end{enumerate}

\subsection{Oplax Perturbations}
\label{ax:oplax-pert}
An oplax perturbation $\theta^o\colon\Lambda^o\to\Gamma^o$ has to satisfy the following axiom.
\begin{center}
		\adjustbox{scale=1}{
			\begin{tikzcd}[ampersand replacement=\&,column sep=scriptsize,row sep=2.25em]
				{\alpha_yFf} \&\& {Gf\alpha_x} \\
				{\beta_yFf} \&\& {Gf\beta_x}
				\arrow["{\alpha^o_f}", from=1-1, to=1-3]
				\arrow[""{name=0, anchor=center, inner sep=0}, "{\Gamma_y\circ1}"', curve={height=12pt}, from=1-1, to=2-1]
				\arrow[""{name=1, anchor=center, inner sep=0}, "{\Lambda_y\circ1}", curve={height=-12pt}, from=1-1, to=2-1]
				\arrow[""{name=2, anchor=center, inner sep=0}, "{1\circ\Lambda_x}", curve={height=-12pt}, from=1-3, to=2-3]
				\arrow["{\beta^o_f}"', from=2-1, to=2-3]
				\arrow["{\theta_y\circ1}"', between={0.3}{0.7}, Rightarrow, from=1, to=0]
				\arrow["{\Lambda^o_f}"', between={0.4}{0.6}, Rightarrow, from=2, to=1]
			\end{tikzcd}
		}
		=
		\adjustbox{scale=1}{
			\begin{tikzcd}[ampersand replacement=\&,column sep=scriptsize,row sep=2.25em]
				{\alpha_yFf} \&\& {Gf\alpha_x} \\
				{\beta_yFf} \&\& {Gf\beta_x}
				\arrow["{\alpha^o_f}", from=1-1, to=1-3]
				\arrow[""{name=0, anchor=center, inner sep=0}, "{\Gamma_y\circ1}"', curve={height=12pt}, from=1-1, to=2-1]
				\arrow[""{name=1, anchor=center, inner sep=0}, "{1\circ\Lambda_x}", curve={height=-12pt}, from=1-3, to=2-3]
				\arrow[""{name=2, anchor=center, inner sep=0}, "{1\circ\Gamma_x}"', curve={height=12pt}, from=1-3, to=2-3]
				\arrow["{\beta^o_f}"', from=2-1, to=2-3]
				\arrow["{\Gamma_f^o}"', between={0.4}{0.6}, Rightarrow, from=2, to=0]
				\arrow["{1\circ\theta_x}"', between={0.3}{0.7}, Rightarrow, from=1, to=2]
			\end{tikzcd}
		}
	\end{center}

\section{Local adjoint equivalences in $\psd(\ca,\cb)$}
\label{app:loc-adj-psd}

Let $\ca$ and $\cb$ be Gray-categories and $F,G\colon\ca\ \rightrightarrows \cb$ be pseudomaps between them. In this appendix we characterise the invertible $2$-cells and adjoint equivalence $1$-cells in the 2-category $\psd(\ca,\cb)(F,G)$, showing that they are computed pointwise.  We begin with the simpler result.

\begin{prop}\label{invertpert}
	A perturbation $\theta \colon \Gamma \to \Delta \in \psd(\ca,\cb)(F,G)(\alpha,\beta)$ is an invertible $3$-cell in $\psd(\ca,\cb)$ if and only if each $3$-cell $\theta_x$ is invertible in $\cb$.
\end{prop}
\begin{proof}
For the non-trivial direction, suppose that each $\theta_x$ is invertible.  We must show that their inverses $(\theta_x)^{-1} \colon \Delta_x \to \Gamma_x$ form a perturbation, which is to say that the equality in the diagram below holds.
	
	\[\begin{tikzcd}[ampersand replacement=\&]
		{Gf\alpha_X} \&\& {Gf\beta_X} \&\& {Gf\alpha_X} \&\& {Gf\beta_X} \\
		\&\&\& {=} \\
		{\alpha_YFf} \&\& {\beta_YFf} \&\& {\alpha_YFf} \&\& {\beta_YFf}
		\arrow[""{name=0, anchor=center, inner sep=0}, "{1\circ\Lambda_X}", curve={height=-18pt}, from=1-1, to=1-3]
		\arrow[""{name=1, anchor=center, inner sep=0}, "{1\circ\Gamma_X}"{description}, curve={height=18pt}, from=1-1, to=1-3]
		\arrow["{\alpha_f}"', from=1-1, to=3-1]
		\arrow["{\beta_f}", from=1-3, to=3-3]
		\arrow[""{name=2, anchor=center, inner sep=0}, "{1\circ\Lambda_X}", curve={height=-18pt}, from=1-5, to=1-7]
		\arrow["{\alpha_f}"{description}, from=1-5, to=3-5]
		\arrow["{\beta_f}", from=1-7, to=3-7]
		\arrow[""{name=3, anchor=center, inner sep=0}, "{\Gamma_Y\circ1}"', curve={height=18pt}, from=3-1, to=3-3]
		\arrow[""{name=4, anchor=center, inner sep=0}, "{\Lambda_Y\circ1}"{description}, curve={height=-18pt}, from=3-5, to=3-7]
		\arrow[""{name=5, anchor=center, inner sep=0}, "{\Gamma_Y\circ1}"', curve={height=18pt}, from=3-5, to=3-7]
		\arrow["{1\circ (\theta_X)^{-1}}"{description}, between={0.2}{0.8}, Rightarrow, from=0, to=1]
		\arrow["{\Gamma_f}"{description}, between={0.4}{0.6}, Rightarrow, from=1, to=3]
		\arrow["{\Lambda_f}"{description}, between={0.4}{0.6}, Rightarrow, from=2, to=4]
		\arrow["{(\theta_Y)^{-1}\circ1}"{description}, between={0.2}{0.8}, Rightarrow, from=4, to=5]
	\end{tikzcd}\]
	which is equally to say the composite pasting below left
	\[\begin{tikzcd}[ampersand replacement=\&]
		{Gf\alpha_X} \&\& {Gf\beta_X} \&\& {Gf\alpha_X} \&\& {Gf\beta_X} \\
		\&\&\& {=} \\
		{\alpha_YFf} \&\& {\beta_YFf} \&\& {\alpha_YFf} \&\& {\beta_YFf}
		\arrow[""{name=0, anchor=center, inner sep=0}, "{1\circ\Lambda_X}", curve={height=-18pt}, from=1-1, to=1-3]
		\arrow[""{name=1, anchor=center, inner sep=0}, "{1\circ\Gamma_X}"{description}, curve={height=18pt}, from=1-1, to=1-3]
		\arrow["{\alpha_f}"', from=1-1, to=3-1]
		\arrow["{\beta_f}", from=1-3, to=3-3]
		\arrow[""{name=2, anchor=center, inner sep=0}, "{1\circ\Lambda_X}", curve={height=-24pt}, from=1-5, to=1-7]
		\arrow[""{name=3, anchor=center, inner sep=0}, "{1\circ\Gamma_X}"{description}, from=1-5, to=1-7]
		\arrow[""{name=4, anchor=center, inner sep=0}, "{1\circ\Lambda_X}"', curve={height=24pt}, from=1-5, to=1-7]
		\arrow["{\alpha_f}"', from=1-5, to=3-5]
		\arrow["{\beta_f}", from=1-7, to=3-7]
		\arrow[""{name=5, anchor=center, inner sep=0}, "{\Gamma_Y\circ1}", curve={height=-12pt}, from=3-1, to=3-3]
		\arrow[""{name=6, anchor=center, inner sep=0}, "{\Lambda_Y \circ 1}"{description}, curve={height=12pt}, from=3-1, to=3-3]
		\arrow[""{name=7, anchor=center, inner sep=0}, "{\Lambda_Y \circ 1}"{description}, curve={height=12pt}, from=3-5, to=3-7]
		\arrow["{1\circ (\theta_X)^{-1}}"{description}, between={0.2}{0.8}, Rightarrow, from=0, to=1]
		\arrow["{\Gamma_f}"{description}, between={0.4}{0.6}, Rightarrow, from=1, to=5]
		\arrow["{1\circ (\theta_X)^{-1}}"{description}, between={0.2}{0.8}, Rightarrow, from=2, to=3]
		\arrow["{1 \circ \theta_X}"{description}, between={0.2}{0.8}, Rightarrow, from=3, to=4]
		\arrow["{\Lambda_f}"{description}, between={0.2}{0.8}, Rightarrow, from=4, to=7]
		\arrow["{\theta_X \circ 1}"{description}, between={0.2}{0.8}, Rightarrow, from=5, to=6]
	\end{tikzcd}\]
	equals $\Gamma_f$.  Indeed, applying the perturbation axiom for $\theta$, it equals the central pasting above which equals $\Gamma_f$ on cancelling inverses.
\end{proof}

We now turn to adjoint equivalences in $\psd(\ca,\cb)(F,G)$, which consist of pseudo-transformations $\alpha$ and $\beta$ together with modifications $\Lambda\colon\alpha\to\beta$ and $\Gamma\colon\beta\to\alpha$ and  invertible perturbations $\eta\colon1_\beta\to\Lambda\Gamma$ and $\epsilon\colon\Gamma\Lambda\to1_\alpha$ in $\psd(\ca,\cb)$, satisfying the  triangle equations for an adjunction.  By \Cref{invertpert}, a perturbation is invertible just when its components are.  Fixing $\Gamma$ and $\Lambda$, therefore an adjoint equivalence $\Gamma \dashv \Lambda$ is specified by families of invertible $3$-cells $\eta_x\colon 1_{\beta_x}\to \Lambda_x\Gamma_x$ and $\epsilon_x\colon\Gamma_x\Lambda_x\to1_{\alpha_x}$ giving adjoint equivalences in $\Gamma_x \dashv \Lambda_x$ in $\cb(Fx,Gx)$ and moreover satisfying the perturbation equations below.


\begin{equation}
	\label{eq:loca-adj-psd-1}
	\begin{tikzcd}[ampersand replacement=\&]
		{Gf\beta_x} \&\& {Gf\beta_x} \\
		\& {Gf\alpha_x} \\
		{\beta_yFf} \&\& {\beta_yFf} \\
		\& {\alpha_yFf}
		\arrow["{\beta_f}"', from=1-1, to=3-1]
		\arrow["{\beta_f}", from=1-3, to=3-3]
		\arrow["{\alpha_f}"{description}, from=2-2, to=4-2]
		\arrow[""{name=0, anchor=center, inner sep=0}, "{\Gamma_y\circ1}"', curve={height=12pt}, from=3-1, to=4-2]
		\arrow[""{name=1, anchor=center, inner sep=0}, "{\Lambda_y\circ1}"', curve={height=12pt}, from=4-2, to=3-3]
		\arrow[""{name=2, anchor=center, inner sep=0}, "{1\circ\Gamma_x}"{description}, curve={height=12pt}, from=1-1, to=2-2]
		\arrow[""{name=3, anchor=center, inner sep=0}, "{1\circ\Lambda_x}"{description}, curve={height=12pt}, from=2-2, to=1-3]
		\arrow[""{name=4, anchor=center, inner sep=0}, curve={height=-18pt}, Rightarrow, no head, from=1-1, to=1-3]
		\arrow["{\Lambda_f}"', shorten <=17pt, shorten >=17pt, Rightarrow, from=3, to=1]
		\arrow["{1\circ\eta_x}"{description}, shorten <=8pt, shorten >=8pt, Rightarrow, from=4, to=2-2]
		\arrow["{\Gamma_f}", shorten <=17pt, shorten >=17pt, Rightarrow, from=2, to=0]
	\end{tikzcd}
	=
	\begin{tikzcd}[ampersand replacement=\&]
		{Gf\beta_x} \&\& {Gf\beta_x} \\
		\\
		{\beta_yFf} \&\& {\beta_yFf} \\
		\& {\alpha_yFf}
		\arrow["{\beta_f}"', from=1-1, to=3-1]
		\arrow["{\beta_f}", from=1-3, to=3-3]
		\arrow["{\Gamma_y\circ1}"', curve={height=12pt}, from=3-1, to=4-2]
		\arrow["{\Lambda_y\circ1}"', curve={height=12pt}, from=4-2, to=3-3]
		\arrow[curve={height=-18pt}, Rightarrow, no head, from=1-1, to=1-3]
		\arrow[""{name=0, anchor=center, inner sep=0}, curve={height=-18pt}, Rightarrow, no head, from=3-1, to=3-3]
		\arrow["{\eta_y\circ1}"{description}, shorten <=8pt, shorten >=8pt, Rightarrow, from=0, to=4-2]
	\end{tikzcd}
\end{equation}

\begin{equation}
	\label{eq:loca-adj-psd-2}
	\begin{tikzcd}[ampersand replacement=\&]
		\& {Gf\beta_x} \\
		{Gf\alpha_x} \&\& {Gf\alpha_x} \\
		\\
		{\alpha_yFf} \&\& {\alpha_yFf}
		\arrow["{\alpha_f}"', from=2-1, to=4-1]
		\arrow["{\alpha_f}", from=2-3, to=4-3]
		\arrow["{1\circ\Lambda_x}", curve={height=-12pt}, from=2-1, to=1-2]
		\arrow["{1\circ\Gamma_x}", curve={height=-12pt}, from=1-2, to=2-3]
		\arrow[curve={height=18pt}, Rightarrow, no head, from=4-1, to=4-3]
		\arrow[""{name=0, anchor=center, inner sep=0}, curve={height=18pt}, Rightarrow, no head, from=2-1, to=2-3]
		\arrow["{1\circ\epsilon_x}"{description}, shorten <=8pt, shorten >=8pt, Rightarrow, from=1-2, to=0]
	\end{tikzcd}
	=
	\begin{tikzcd}[ampersand replacement=\&]
		\& {Gf\beta_x} \\
		{Gf\alpha_x} \&\& {Gf\alpha_x} \\
		\& {\beta_yFf} \\
		{\alpha_yFf} \&\& {\alpha_yFf}
		\arrow["{\alpha_f}"', from=2-1, to=4-1]
		\arrow["{\alpha_f}", from=2-3, to=4-3]
		\arrow["{\beta_f}"{description}, from=1-2, to=3-2]
		\arrow[""{name=0, anchor=center, inner sep=0}, "{\Lambda_y\circ1}"{description}, curve={height=-12pt}, from=4-1, to=3-2]
		\arrow[""{name=1, anchor=center, inner sep=0}, "{\Gamma_y\circ1}"{description}, curve={height=-12pt}, from=3-2, to=4-3]
		\arrow[""{name=2, anchor=center, inner sep=0}, "{1\circ\Lambda_x}", curve={height=-12pt}, from=2-1, to=1-2]
		\arrow[""{name=3, anchor=center, inner sep=0}, "{1\circ\Gamma_x}", curve={height=-12pt}, from=1-2, to=2-3]
		\arrow[""{name=4, anchor=center, inner sep=0}, curve={height=18pt}, Rightarrow, no head, from=4-1, to=4-3]
		\arrow["{\Lambda_f}", shorten <=17pt, shorten >=17pt, Rightarrow, from=2, to=0]
		\arrow["{\Gamma_f}"', shorten <=17pt, shorten >=17pt, Rightarrow, from=3, to=1]
		\arrow["{\epsilon_y\circ1}"{description}, shorten <=8pt, shorten >=8pt, Rightarrow, from=3-2, to=4]
	\end{tikzcd}
\end{equation}

Now pasting both equations with ${\Gamma_f}^{-1}$ we can rewrite equations \eqref{eq:loca-adj-psd-1} and \eqref{eq:loca-adj-psd-2} as follows. 

\begin{equation}\tag{\ref{eq:loca-adj-psd-1}'}
	\label{eq:loca-adj-psd-1'}
	\begin{tikzcd}[ampersand replacement=\&]
		{Gf\beta_x} \\
		{Gf\alpha_x} \& {Gf\beta_x} \\
		{\alpha_yFf} \& {\beta_yFf}
		\arrow["{\beta_f}", from=2-2, to=3-2]
		\arrow["{\alpha_f}"', from=2-1, to=3-1]
		\arrow[""{name=0, anchor=center, inner sep=0}, "{\Lambda_y\circ1}"', from=3-1, to=3-2]
		\arrow["{1\circ\Gamma_x}"', from=1-1, to=2-1]
		\arrow[""{name=1, anchor=center, inner sep=0}, "{1\circ\Lambda_x}"{description}, from=2-1, to=2-2]
		\arrow[""{name=2, anchor=center, inner sep=0}, Rightarrow, no head, from=1-1, to=2-2]
		\arrow["{\Lambda_f}"', shorten <=10pt, shorten >=10pt, Rightarrow, from=1, to=0]
		\arrow["{1\circ\eta_x}"', shorten <=5pt, shorten >=5pt, Rightarrow, from=2, to=1]
	\end{tikzcd}
	\quad=\quad
	\begin{tikzcd}[ampersand replacement=\&]
		{Gf\beta_x} \& {\beta_yFf} \\
		{Gf\alpha_x} \& {\alpha_yFf} \& {\beta_yFf}
		\arrow[""{name=0, anchor=center, inner sep=0}, "{\beta_f}", from=1-1, to=1-2]
		\arrow["{\Gamma_y\circ1}"{description}, from=1-2, to=2-2]
		\arrow[""{name=1, anchor=center, inner sep=0}, "{\Lambda_y\circ1}"', from=2-2, to=2-3]
		\arrow[""{name=2, anchor=center, inner sep=0}, Rightarrow, no head, from=1-2, to=2-3]
		\arrow["{1\circ\Gamma_x}"', from=1-1, to=2-1]
		\arrow[""{name=3, anchor=center, inner sep=0}, "{\alpha_f}"', from=2-1, to=2-2]
		\arrow["{\eta_y\circ1}"', shorten <=5pt, shorten >=5pt, Rightarrow, from=2, to=1]
		\arrow["{{\Gamma_f}^{-1}}"', shorten <=10pt, shorten >=10pt, Rightarrow, from=0, to=3]
	\end{tikzcd}
\end{equation}

\begin{equation}\tag{\ref{eq:loca-adj-psd-2}'}
	\label{eq:loca-adj-psd-2'}
	\begin{tikzcd}[ampersand replacement=\&]
		{Gf\alpha_x} \& {Gf\beta_x} \& {\beta_yFf} \\
		\& {Gf\alpha_x} \& {\alpha_yFf}
		\arrow[""{name=0, anchor=center, inner sep=0}, "{\alpha_f}"', from=2-2, to=2-3]
		\arrow[""{name=1, anchor=center, inner sep=0}, "{1\circ\Lambda_x}", from=1-1, to=1-2]
		\arrow["{1\circ\Gamma_x}"{description}, from=1-2, to=2-2]
		\arrow[""{name=2, anchor=center, inner sep=0}, Rightarrow, no head, from=1-1, to=2-2]
		\arrow[""{name=3, anchor=center, inner sep=0}, "{\beta_f}", from=1-2, to=1-3]
		\arrow["{\Gamma_y\circ1}", from=1-3, to=2-3]
		\arrow["{1\circ\epsilon_x}", shorten <=5pt, shorten >=5pt, Rightarrow, from=1, to=2]
		\arrow["{{\Gamma_f}^{-1}}", shorten <=10pt, shorten >=10pt, Rightarrow, from=3, to=0]
	\end{tikzcd}
	\quad=\quad
	\begin{tikzcd}[ampersand replacement=\&]
		{Gf\alpha_x} \& {Gf\beta_x} \\
		{\alpha_yFf} \& {\beta_yFf} \\
		\& {\alpha_yFf}
		\arrow["{\alpha_f}"', from=1-1, to=2-1]
		\arrow["{\beta_f}", from=1-2, to=2-2]
		\arrow[""{name=0, anchor=center, inner sep=0}, "{\Lambda_y\circ1}"{description}, from=2-1, to=2-2]
		\arrow["{\Gamma_y\circ1}", from=2-2, to=3-2]
		\arrow[""{name=1, anchor=center, inner sep=0}, "{1\circ\Lambda_x}", from=1-1, to=1-2]
		\arrow[""{name=2, anchor=center, inner sep=0}, Rightarrow, no head, from=2-1, to=3-2]
		\arrow["{\Lambda_f}", shorten <=10pt, shorten >=10pt, Rightarrow, from=1, to=0]
		\arrow["{\epsilon_y\circ1}", shorten <=5pt, shorten >=5pt, Rightarrow, from=0, to=2]
	\end{tikzcd}
\end{equation}

Notice that both equations \eqref{eq:loca-adj-psd-1'} and \eqref{eq:loca-adj-psd-2'} are equivalent to the single equation below: the former by pasting with $1 \circ \epsilon_x$ and using a triangle equation, and the latter by pasting with $\eta_y \circ 1$ and using a triangle equation.

\begin{equation}\label{eq:loca-adj-psd-final2}
\begin{tikzcd}
	{Gf\alpha_x} & {Gf\beta_x} \\
	{\alpha_yFf} & {\beta_yFf}
	\arrow[""{name=0, anchor=center, inner sep=0}, "{1\circ\Lambda_x}", from=1-1, to=1-2]
	\arrow["{\alpha_f}"', from=1-1, to=2-1]
	\arrow["{\beta_f}", from=1-2, to=2-2]
	\arrow[""{name=1, anchor=center, inner sep=0}, "{\Lambda_y\circ1}"', from=2-1, to=2-2]
	\arrow["{\Lambda_f}", between={0.4}{0.6}, Rightarrow, nfold, from=0, to=1]
\end{tikzcd}
		\quad=\quad
\begin{tikzcd}[column sep=scriptsize]
	{Gf\alpha_x} && {Gf\beta_x} && {\beta_yFf} && \\
	&& {Gf\alpha_x} && {\alpha_yFf} && {\beta_yFf}
	\arrow[""{name=0, anchor=center, inner sep=0}, "{1\circ\Lambda_x}", from=1-1, to=1-3]
	\arrow[""{name=1, anchor=center, inner sep=0}, equals, nfold, from=1-1, to=2-3]
	\arrow[""{name=2, anchor=center, inner sep=0}, "{\beta_f}", from=1-3, to=1-5]
	\arrow["{1\circ\Gamma_x}"{description}, from=1-3, to=2-3]
	\arrow["{\Gamma_y\circ1}"{description}, from=1-5, to=2-5]
	\arrow[""{name=3, anchor=center, inner sep=0}, equals, nfold, from=1-5, to=2-7]
	\arrow[""{name=4, anchor=center, inner sep=0}, "{\alpha_f}"', from=2-3, to=2-5]
	\arrow[""{name=5, anchor=center, inner sep=0}, "{\Lambda_y\circ1}"', from=2-5, to=2-7]
	\arrow["{1\circ\epsilon_x}", between={0.3}{0.7}, Rightarrow, nfold, from=0, to=1]
	\arrow["{{\Gamma_f}^{-1}}", between={0.4}{0.6}, Rightarrow, nfold, from=2, to=4]
	\arrow["{\eta_y\circ1}"', between={0.3}{0.7}, Rightarrow, nfold, from=3, to=5]
\end{tikzcd}
\end{equation}

Equally, both equations \eqref{eq:loca-adj-psd-1'} and \eqref{eq:loca-adj-psd-2'} are equivalent to the single equation below: in the first case by pasting by $ \epsilon_y \circ 1$, using a triangle equation and taking inverses, and in the second case by pasting with $1 \circ \eta_x$, using a triangle equation and taking inverses.

\begin{equation}\label{eq:loca-adj-psd-final}
\begin{tikzcd}[ampersand replacement=\&]
	{Gf\beta_x} \& {Gf\alpha_x} \\
	{\beta_yFf} \& {\alpha_yFf}
	\arrow[""{name=0, anchor=center, inner sep=0}, "{1\circ\Gamma_x}", from=1-1, to=1-2]
	\arrow["{\beta_f}"', from=1-1, to=2-1]
	\arrow["{\alpha_f}", from=1-2, to=2-2]
	\arrow[""{name=1, anchor=center, inner sep=0}, "{\Gamma_y\circ1}"', from=2-1, to=2-2]
	\arrow["{\Gamma_f}", between={0.2}{0.8}, Rightarrow, nfold, from=0, to=1]
\end{tikzcd}
		\quad=\quad
\begin{tikzcd}[ampersand replacement=\&]
	{Gf\beta_x} \&\& {Gf\alpha_x} \& {\alpha_yFf} \&\& \\
	\&\& {Gf\beta_x} \& {\beta_yFf} \&\& {\alpha_yFf}
	\arrow[""{name=0, anchor=center, inner sep=0}, "{1\circ\Gamma_x}", from=1-1, to=1-3]
	\arrow[""{name=1, anchor=center, inner sep=0}, equals, nfold, from=1-1, to=2-3]
	\arrow[""{name=2, anchor=center, inner sep=0}, "{\alpha_f}", from=1-3, to=1-4]
	\arrow["{1\circ\Lambda_x}"{description}, from=1-3, to=2-3]
	\arrow["{\Lambda_y\circ1}"{description}, from=1-4, to=2-4]
	\arrow[""{name=3, anchor=center, inner sep=0}, equals, nfold, from=1-4, to=2-6]
	\arrow[""{name=4, anchor=center, inner sep=0}, "{\beta_f}"', from=2-3, to=2-4]
	\arrow[""{name=5, anchor=center, inner sep=0}, "{\Gamma_y\circ1}"', from=2-4, to=2-6]
	\arrow["{1\circ\eta_x^{-1}}", between={0.3}{0.7}, Rightarrow, nfold, from=0, to=1]
	\arrow["{\Lambda_f^{-1}}", between={0.4}{0.6}, Rightarrow, nfold, from=2, to=4]
	\arrow["{\epsilon_y^{-1}\circ1}"', between={0.3}{0.7}, Rightarrow, nfold, from=3, to=5]
\end{tikzcd}
\end{equation}

\begin{prop}
	\label{prop:loc-adj-psd-char}
	Adjoint equivalences in $\psd(\ca,\cb)(F,G)$ between pseudo-transformations $\alpha,\beta\colon F\to G$ corresponds bijectively to modifications $\Lambda\colon\alpha\to\beta$ equipped with, for any object $x\in\ca$, an adjoint equivalence $\Gamma_x\dashv\Lambda_x$. 
\end{prop}

\begin{proof}
The equation \eqref{eq:loca-adj-psd-final} shows that in such an adjoint equivalence $\Gamma \dashv \Lambda$, the components $\Gamma_f$ of $\Gamma$ are uniquely determined by $\Lambda$ and the adjoint equivalences  $\Gamma_x\dashv\Lambda_x$.  Conversely, consider a modification $\Lambda\colon\alpha\to\beta$ equipped with, for any object $x\in\ca$, an adjoint equivalence $\Gamma_x\dashv\Lambda_x$.  Defining $\Gamma_f$ as in \eqref{eq:loca-adj-psd-final}, it remains to show that $\Gamma$ forms a modification.
	For any pair of composable 1-cells $x\xrightarrow{f_1}y\xrightarrow{f_2}z$, we consider the following chain of equalities, which follow from (in order): definition of $\Gamma_f$; triangle identity for $\Gamma_y\dashv\Lambda_y$; \Cref{lemma:interchange-property} for $\epsilon_z^{-1}$; equivalent description of the corresponding axiom for $\Lambda$; \Cref{lemma:interchange-property} for $\eta_x^{-1}$. 
	\begin{center}
		\scalebox{0.7}{
			\begin{tikzcd}[ampersand replacement=\&,column sep=scriptsize]
				{Gf_2Gf_1\beta_x} \&\&\&\&\&\& \\
				\&\& {Gf_2Gf_1\alpha_x} \\
				{Gf_2Gf_1\beta_x} \&\&\& {Gf_2\alpha_yFf_1} \\
				\& {Gf_2\beta_yFf_1} \&\&\& {Gf_2\alpha_yFf_1} \\
				{G(f_2f_1)\beta_x} \&\& {Gf_2\beta_yFf_1} \&\&\& {\alpha_zFf_2Ff_1} \\
				\&\&\& {\beta_zFf_2Ff_1} \&\&\& {\alpha_zFf_2Ff_1} \\
				\\
				\&\&\& {\beta_zF(f_2f_1)} \&\&\& {\alpha_zF(f_2f_1)}
				\arrow[""{name=0, anchor=center, inner sep=0}, "{1\circ\Gamma_x}", from=1-1, to=2-3]
				\arrow[""{name=1, anchor=center, inner sep=0}, equals, from=1-1, to=3-1]
				\arrow["{1\circ\Lambda_x}"{description}, from=2-3, to=3-1]
				\arrow[""{name=2, anchor=center, inner sep=0}, "{1\circ\alpha^l_{f_1}}", from=2-3, to=3-4]
				\arrow[""{name=3, anchor=center, inner sep=0}, "{1\circ\beta^l_{f_1}}"{description}, from=3-1, to=4-2]
				\arrow["{G^2_{f_2,f_1}\circ1}"', from=3-1, to=5-1]
				\arrow["{1\circ\Lambda_y\circ1}"{description}, from=3-4, to=4-2]
				\arrow[""{name=4, anchor=center, inner sep=0}, equals, from=3-4, to=4-5]
				\arrow[""{name=5, anchor=center, inner sep=0}, "{1\circ\Gamma_y\circ1}"{description}, from=4-2, to=4-5]
				\arrow[""{name=6, anchor=center, inner sep=0}, equals, from=4-2, to=5-3]
				\arrow["{1\circ\Lambda_y\circ1}"{description}, from=4-5, to=5-3]
				\arrow[""{name=7, anchor=center, inner sep=0}, "{\alpha^l_{f_2}\circ1}", from=4-5, to=5-6]
				\arrow[""{name=8, anchor=center, inner sep=0}, "{\beta^l_{f_2f_1}}"', from=5-1, to=8-4]
				\arrow[""{name=9, anchor=center, inner sep=0}, "{\beta^l_{f_2}\circ1}"{description}, from=5-3, to=6-4]
				\arrow["{\Lambda_z\circ1}"{description}, from=5-6, to=6-4]
				\arrow[""{name=10, anchor=center, inner sep=0}, equals, from=5-6, to=6-7]
				\arrow[""{name=11, anchor=center, inner sep=0}, "{\Gamma_z\circ1}"', from=6-4, to=6-7]
				\arrow["{1\circ F^2_{f_2,f_1}}"{description}, from=6-4, to=8-4]
				\arrow["{1\circ F^2_{f_2,f_1}}", from=6-7, to=8-7]
				\arrow[""{name=12, anchor=center, inner sep=0}, "{\Gamma_z\circ1}"', from=8-4, to=8-7]
				\arrow["{1\circ\eta_x^{-1}}", between={0.4}{0.6}, Rightarrow, from=0, to=1]
				\arrow["{1\circ\Lambda_{f_1}^{-1}}", shift right=3, between={0.4}{0.6}, Rightarrow, from=2, to=3]
				\arrow["{1\circ\epsilon_y^{-1}\circ1}"'{pos=0.6}, between={0.3}{0.5}, Rightarrow, from=4, to=5]
				\arrow["{1\circ\eta_y^{-1}\circ1}"{pos=0.4}, between={0.5}{0.7}, Rightarrow, from=5, to=6]
				\arrow["{\beta^2_{f_2,f_1}}", between={0.4}{0.6}, Rightarrow, from=6, to=8]
				\arrow["{\Lambda_{f_2}^{-1}\circ1}", shift right=2, between={0.4}{0.6}, Rightarrow, from=7, to=9]
				\arrow["{\epsilon_z^{-1}\circ1}"', between={0.3}{0.5}, Rightarrow, from=10, to=11]
				\arrow[between={0.3}{0.7}, Rightarrow, dashed, from=11, to=12]
			\end{tikzcd}
		}\\
		$=$
		\scalebox{0.8}{
			\begin{tikzcd}[ampersand replacement=\&,column sep=small,row sep=scriptsize]
				{Gf_2Gf_1\beta_x} \&\&\&\& \\
				\&\& {Gf_2Gf_1\alpha_x} \\
				{Gf_2Gf_1\beta_x} \&\&\& {Gf_2\alpha_yFf_1} \\
				\& {Gf_2\beta_yFf_1} \&\&\& {\alpha_zFf_2Ff_1} \\
				{G(f_2f_1)\beta_x} \&\& {\beta_zFf_2Ff_1} \\
				\&\&\&\& {\alpha_zF(f_2f_1)} \\
				\&\& {\beta_zF(f_2f_1)} \\
				\&\&\&\& {\alpha_zF(f_2f_1)}
				\arrow[""{name=0, anchor=center, inner sep=0}, "{1\circ\Gamma_x}", from=1-1, to=2-3]
				\arrow[""{name=1, anchor=center, inner sep=0}, equals, from=1-1, to=3-1]
				\arrow["{1\circ\Lambda_x}"{description}, from=2-3, to=3-1]
				\arrow[""{name=2, anchor=center, inner sep=0}, "{1\circ\alpha^l_{f_1}}", from=2-3, to=3-4]
				\arrow[""{name=3, anchor=center, inner sep=0}, "{1\circ\beta^l_{f_1}}"{description}, from=3-1, to=4-2]
				\arrow["{G^2_{f_2,f_1}\circ1}"', from=3-1, to=5-1]
				\arrow["{1\circ\Lambda_y\circ1}"{description}, from=3-4, to=4-2]
				\arrow[""{name=4, anchor=center, inner sep=0}, "{\alpha^l_{f_2}\circ1}", from=3-4, to=4-5]
				\arrow[""{name=5, anchor=center, inner sep=0}, "{\beta^l_{f_2}\circ1}"{description}, from=4-2, to=5-3]
				\arrow["{\Lambda_z\circ1}"{description}, from=4-5, to=5-3]
				\arrow[""{name=6, anchor=center, inner sep=0}, "{1\circ F^2_{f_2,f_1}}", from=4-5, to=6-5]
				\arrow[""{name=7, anchor=center, inner sep=0}, "{\beta^l_{f_2f_1}}"', from=5-1, to=7-3]
				\arrow[""{name=8, anchor=center, inner sep=0}, "{1\circ F^2_{f_2,f_1}}"{description}, from=5-3, to=7-3]
				\arrow["{\Lambda_z\circ1}"', from=6-5, to=7-3]
				\arrow[""{name=9, anchor=center, inner sep=0}, equals, from=6-5, to=8-5]
				\arrow[""{name=10, anchor=center, inner sep=0}, "{\Gamma_z\circ1}"', from=7-3, to=8-5]
				\arrow["{1\circ\eta_x^{-1}}", between={0.4}{0.6}, Rightarrow, from=0, to=1]
				\arrow["{1\circ\Lambda_{f_1}^{-1}}", shift right=3, between={0.4}{0.6}, Rightarrow, from=2, to=3]
				\arrow["{\Lambda_{f_2}^{-1}\circ1}", shift right=2, between={0.4}{0.6}, Rightarrow, from=4, to=5]
				\arrow["{\beta^2_{f_2,f_1}}", between={0.3}{0.7}, Rightarrow, from=4-2, to=7]
				\arrow[between={0.4}{0.6}, Rightarrow, dashed, from=6, to=8]
				\arrow["{\epsilon_z^{-1}\circ1}"', between={0.3}{0.5}, Rightarrow, from=9, to=10]
			\end{tikzcd}
		} \\
		$=$
		\scalebox{0.8}{
			\begin{tikzcd}[ampersand replacement=\&,column sep=small,row sep=scriptsize]
				{Gf_2Gf_1\beta_x} \&\&\&\& \\
				\&\& {Gf_2Gf_1\alpha_x} \\
				{Gf_2Gf_1\beta_x} \&\&\& {Gf_2\alpha_yFf_1} \\
				\&\& {G(f_2f_1)\alpha_x} \&\& {\alpha_zFf_2Ff_1} \\
				{G(f_2f_1)\beta_x} \& \textcolor{white}{{Gf_2\alpha_yFf_1}} \\
				\&\&\&\& {\alpha_zF(f_2f_1)} \\
				\&\& {\beta_zF(f_2f_1)} \\
				\&\&\&\& {\alpha_zF(f_2f_1)}
				\arrow[""{name=0, anchor=center, inner sep=0}, "{1\circ\Gamma_x}", from=1-1, to=2-3]
				\arrow[""{name=1, anchor=center, inner sep=0}, equals, from=1-1, to=3-1]
				\arrow["{1\circ\Lambda_x}"{description}, from=2-3, to=3-1]
				\arrow["{1\circ\alpha^l_{f_1}}", from=2-3, to=3-4]
				\arrow[""{name=2, anchor=center, inner sep=0}, from=2-3, to=4-3]
				\arrow[""{name=3, anchor=center, inner sep=0}, "{G^2_{f_2,f_1}\circ1}"', from=3-1, to=5-1]
				\arrow["{\alpha^l_{f_2}\circ1}", from=3-4, to=4-5]
				\arrow["{1\circ\Lambda_x}", from=4-3, to=5-1]
				\arrow[""{name=4, anchor=center, inner sep=0}, from=4-3, to=6-5]
				\arrow["{1\circ F^2_{f_2,f_1}}", from=4-5, to=6-5]
				\arrow[""{name=5, anchor=center, inner sep=0}, "{\beta_{f_2f_1}}"', from=5-1, to=7-3]
				\arrow["{\Lambda_z\circ1}"', from=6-5, to=7-3]
				\arrow[""{name=6, anchor=center, inner sep=0}, equals, from=6-5, to=8-5]
				\arrow[""{name=7, anchor=center, inner sep=0}, "{\Gamma_z\circ1}"', from=7-3, to=8-5]
				\arrow["{1\circ\eta_x^{-1}}", between={0.4}{0.6}, Rightarrow, from=0, to=1]
				\arrow[between={0.4}{0.6}, Rightarrow, dashed, from=2, to=3]
				\arrow["{\alpha^2_{f_2,f_1}}", between={0.3}{0.7}, Rightarrow, from=3-4, to=4]
				\arrow["{\Lambda^{-1}_{f_2f_1}}", between={0.4}{0.6}, Rightarrow, from=4, to=5]
				\arrow["{\epsilon_z^{-1}\circ1}"', between={0.3}{0.5}, Rightarrow, from=6, to=7]
			\end{tikzcd}
		} \\
		$=$
		\scalebox{0.8}{
			\begin{tikzcd}[ampersand replacement=\&,column sep=small,row sep=scriptsize]
				{Gf_2Gf_1\beta_x} \&\& {Gf_2Gf_1\alpha_x} \&\& \\
				\&\&\& {Gf_2\alpha_yFf_1} \\
				{G(f_2f_1)\alpha_x} \&\& {G(f_2f_1)\alpha_x} \&\& {\alpha_zFf_2Ff_1} \\
				\& {G(f_2f_1)\alpha_x} \\
				\&\&\&\& {\alpha_zF(f_2f_1)} \\
				\&\&\& {\beta_zF(f_2f_1)} \\
				\&\&\&\& {\alpha_zF(f_2f_1)}
				\arrow[""{name=0, anchor=center, inner sep=0}, "{1\circ\Gamma_x}", from=1-1, to=1-3]
				\arrow["{G^2_{f_2,f_1}\circ1}"', from=1-1, to=3-1]
				\arrow["{1\circ\alpha^l_{f_1}}", from=1-3, to=2-4]
				\arrow["{G^2_{f_2,f_1}\circ1}"{description}, from=1-3, to=3-3]
				\arrow["{\alpha^l_{f_2}\circ1}", from=2-4, to=3-5]
				\arrow[""{name=1, anchor=center, inner sep=0}, "{1\circ\Gamma_x}", from=3-1, to=3-3]
				\arrow[""{name=2, anchor=center, inner sep=0}, equals, from=3-1, to=4-2]
				\arrow[""{name=3, anchor=center, inner sep=0}, "{1\circ\Lambda_x}", from=3-3, to=4-2]
				\arrow[""{name=4, anchor=center, inner sep=0}, from=3-3, to=5-5]
				\arrow["{1\circ F^2_{f_2,f_1}}", from=3-5, to=5-5]
				\arrow[""{name=5, anchor=center, inner sep=0}, "{\beta_{f_2f_1}}"', from=4-2, to=6-4]
				\arrow["{\Lambda_z\circ1}"', from=5-5, to=6-4]
				\arrow[""{name=6, anchor=center, inner sep=0}, equals, from=5-5, to=7-5]
				\arrow[""{name=7, anchor=center, inner sep=0}, "{\Gamma_z\circ1}"', from=6-4, to=7-5]
				\arrow[between={0.4}{0.6}, Rightarrow, dashed, from=0, to=1]
				\arrow["{\alpha^2_{f_2,f_1}}", between={0.3}{0.7}, Rightarrow, from=2-4, to=4]
				\arrow["{1\circ\eta_x^{-1}}"', between={0.3}{0.7}, Rightarrow, from=3, to=2]
				\arrow["{\Lambda^{-1}_{f_2f_1}}", between={0.4}{0.6}, Rightarrow, from=4, to=5]
				\arrow["{\epsilon_z^{-1}\circ1}"', between={0.3}{0.5}, Rightarrow, from=6, to=7]
			\end{tikzcd}
		}
	\end{center}
	
	Then, if we consider $\phi\colon f\to f'$ be a 2-cell in $\ca$, we get the following chain of equalities using (in order): definition of $\Gamma_f$; \Cref{lemma:interchange-property} for $\epsilon_y^{-1}$; reformulation of the corresponding axiom for $\Lambda$ (pasting the equality with $\Lambda_f*{-1}$ and an interchange isomorphism); \Cref{lemma:interchange-property} for $\eta_x^{-1}$.
	\begin{center}
		\scalebox{0.7}{
			\begin{tikzcd}[ampersand replacement=\&,sep=scriptsize]
				{Gf\beta_x} \&\& {Gf\alpha_x} \&\& \\
				\&\& {\beta_yFf} \&\& {\beta_yFf} \\
				{Gf'\beta_x} \\
				\&\& {\beta_yFf'} \&\& {\beta_yFf'}
				\arrow[""{name=0, anchor=center, inner sep=0}, "{1\circ\Gamma_x}", from=1-1, to=1-3]
				\arrow[""{name=1, anchor=center, inner sep=0}, "{\beta_f}"{description}, from=1-1, to=2-3]
				\arrow["{G\phi\circ1}"', from=1-1, to=3-1]
				\arrow["{\alpha_f}", from=1-3, to=2-5]
				\arrow[""{name=2, anchor=center, inner sep=0}, "{\Gamma_y\circ1}"{description}, from=2-3, to=2-5]
				\arrow["{1\circ F\phi}"{description}, from=2-3, to=4-3]
				\arrow["{1\circ F\phi}", from=2-5, to=4-5]
				\arrow[""{name=3, anchor=center, inner sep=0}, "{\beta_{f'}}"', from=3-1, to=4-3]
				\arrow[""{name=4, anchor=center, inner sep=0}, "{\Gamma_y\circ1}"', from=4-3, to=4-5]
				\arrow["{\Gamma_f}", between={0.4}{0.6}, Rightarrow, from=0, to=2]
				\arrow["{\beta_\phi}", between={0.4}{0.6}, Rightarrow, from=1, to=3]
				\arrow[between={0.4}{0.6}, Rightarrow, dashed, from=2, to=4]
			\end{tikzcd}
		}
		$=$ 
		\scalebox{0.8}{
			\begin{tikzcd}[ampersand replacement=\&,column sep=scriptsize]
				{Gf\beta_x} \&\& {Gf\alpha_x} \&\& {\alpha_yFf} \& {} \& \\
				\& {} \& {Gf\beta_x} \&\& {\beta_yFf} \&\& {\beta_yFf} \\
				\&\& {Gf'\beta_x} \&\& {\beta_yFf'} \&\& {\beta_yFf'}
				\arrow[""{name=0, anchor=center, inner sep=0}, "{1\circ\Gamma_x}", from=1-1, to=1-3]
				\arrow[curve={height=12pt}, equals, from=1-1, to=2-3]
				\arrow[""{name=1, anchor=center, inner sep=0}, "{\alpha_f}", from=1-3, to=1-5]
				\arrow["{1\circ\Lambda_x}"{description}, from=1-3, to=2-3]
				\arrow["{\Lambda_y\circ1}"{description}, from=1-5, to=2-5]
				\arrow[curve={height=-12pt}, equals, from=1-5, to=2-7]
				\arrow[""{name=2, anchor=center, inner sep=0}, "{\beta_f}"{description}, from=2-3, to=2-5]
				\arrow["{G\phi\circ1}"', from=2-3, to=3-3]
				\arrow[""{name=3, anchor=center, inner sep=0}, "{\Gamma_y\circ1}"{description}, from=2-5, to=2-7]
				\arrow["{1\circ F\phi}"{description}, from=2-5, to=3-5]
				\arrow["{1\circ F\phi}", from=2-7, to=3-7]
				\arrow[""{name=4, anchor=center, inner sep=0}, "{\beta_{f'}}"', from=3-3, to=3-5]
				\arrow[""{name=5, anchor=center, inner sep=0}, "{\Gamma_y\circ1}"', from=3-5, to=3-7]
				\arrow["{1\circ\eta_x^{-1}}"{pos=0.4}, shift right=2, between={0.2}{0.6}, Rightarrow, from=0, to=2-2]
				\arrow["{\Lambda^{-1}_f}", shift right=2, between={0.3}{0.7}, Rightarrow, from=1, to=2]
				\arrow["{\epsilon_y^{-1}\circ1}"'{pos=0.4}, shift left=2, between={0.3}{0.7}, Rightarrow, from=1-6, to=3]
				\arrow["{\beta_\phi}", shift right, between={0.3}{0.7}, Rightarrow, from=2, to=4]
				\arrow[between={0.3}{0.7}, Rightarrow, dashed, from=3, to=5]
			\end{tikzcd}
		}
		\\
		$=$
		\scalebox{0.8}{
			\begin{tikzcd}[ampersand replacement=\&,column sep=scriptsize]
				{Gf\beta_x} \&\&\& {Gf\alpha_x} \&\& {\alpha_yFf} \&\& \\
				\& {} \& {Gf\beta_x} \&\& {\beta_yFf} \\
				\&\&\&\&\& {\alpha_yFf'} \& {} \\
				\&\& {Gf'\beta_x} \&\& {\beta_yFf'} \&\&\& {\beta_yFf'}
				\arrow[""{name=0, anchor=center, inner sep=0}, "{1\circ\Gamma_x}", from=1-1, to=1-4]
				\arrow[curve={height=12pt}, equals, from=1-1, to=2-3]
				\arrow[""{name=1, anchor=center, inner sep=0}, "{\alpha_f}", from=1-4, to=1-6]
				\arrow["{1\circ\Lambda_x}"{description}, from=1-4, to=2-3]
				\arrow["{\Lambda_y\circ1}"{description}, from=1-6, to=2-5]
				\arrow[""{name=2, anchor=center, inner sep=0}, "{1\circ F\phi}", from=1-6, to=3-6]
				\arrow[""{name=3, anchor=center, inner sep=0}, "{\beta_f}"{description}, from=2-3, to=2-5]
				\arrow["{G\phi\circ1}"', from=2-3, to=4-3]
				\arrow[""{name=4, anchor=center, inner sep=0}, "{1\circ F\phi}"', from=2-5, to=4-5]
				\arrow["{\Lambda_y\circ1}"{description}, from=3-6, to=4-5]
				\arrow[curve={height=-12pt}, equals, from=3-6, to=4-8]
				\arrow[""{name=5, anchor=center, inner sep=0}, "{\beta_{f'}}"', from=4-3, to=4-5]
				\arrow[""{name=6, anchor=center, inner sep=0}, "{\Gamma_y\circ1}"', from=4-5, to=4-8]
				\arrow["{1\circ\eta_x^{-1}}"{pos=0.4}, shift right=2, between={0.3}{0.7}, Rightarrow, from=0, to=2-2]
				\arrow["{\Lambda^{-1}_f}"{pos=0.4}, shift right=2, between={0.3}{0.7}, Rightarrow, from=1, to=3]
				\arrow[between={0.3}{0.7}, Rightarrow, dashed, from=2, to=4]
				\arrow["{\beta_\phi}"', shift right, between={0.3}{0.7}, Rightarrow, from=3, to=5]
				\arrow["{\epsilon_y^{-1}\circ1}"'{pos=0.6}, shift left=3, between={0.2}{0.6}, Rightarrow, from=3-7, to=6]
			\end{tikzcd}
		}
		\\
		$=$
		\scalebox{0.8}{
			\begin{tikzcd}[ampersand replacement=\&,column sep=scriptsize]
				{Gf\beta_x} \&\&\& {Gf\alpha_x} \&\& {\alpha_yFf} \&\& \\
				\& {} \& {Gf\beta_x} \\
				\&\&\& {Gf'\alpha_x} \&\& {\alpha_yFf'} \& {} \\
				\&\& {Gf'\beta_x} \&\& {\beta_yFf'} \&\&\& {\beta_yFf'}
				\arrow[""{name=0, anchor=center, inner sep=0}, "{1\circ\Gamma_x}", from=1-1, to=1-4]
				\arrow[curve={height=12pt}, equals, from=1-1, to=2-3]
				\arrow[""{name=1, anchor=center, inner sep=0}, "{\alpha_f}", from=1-4, to=1-6]
				\arrow["{1\circ\Lambda_x}"{description}, from=1-4, to=2-3]
				\arrow[""{name=2, anchor=center, inner sep=0}, "{G\phi\circ1}", from=1-4, to=3-4]
				\arrow["{1\circ F\phi}", from=1-6, to=3-6]
				\arrow[""{name=3, anchor=center, inner sep=0}, "{G\phi\circ1}"', from=2-3, to=4-3]
				\arrow[""{name=4, anchor=center, inner sep=0}, "{\alpha_{f'}}", from=3-4, to=3-6]
				\arrow["{1\circ\Lambda_x}"{description}, from=3-4, to=4-3]
				\arrow["{\Lambda_y\circ1}"{description}, from=3-6, to=4-5]
				\arrow[curve={height=-12pt}, equals, from=3-6, to=4-8]
				\arrow[""{name=5, anchor=center, inner sep=0}, "{\beta_{f'}}"', from=4-3, to=4-5]
				\arrow[""{name=6, anchor=center, inner sep=0}, "{\Gamma_y\circ1}"', from=4-5, to=4-8]
				\arrow["{1\circ\eta_x^{-1}}"{pos=0.4}, shift right=2, between={0.4}{0.7}, Rightarrow, from=0, to=2-2]
				\arrow["{\alpha_\phi}", between={0.3}{0.7}, Rightarrow, from=1, to=4]
				\arrow[between={0.3}{0.7}, Rightarrow, dashed, from=2, to=3]
				\arrow["{\Lambda^{-1}_f}"{pos=0.4}, shift right=2, between={0.4}{0.8}, Rightarrow, from=4, to=5]
				\arrow["{\epsilon_y^{-1}\circ1}"'{pos=0.6}, shift left=3, between={0.2}{0.6}, Rightarrow, from=3-7, to=6]
			\end{tikzcd}
		}
		\\
		$=$
		\scalebox{0.8}{
			\begin{tikzcd}[ampersand replacement=\&,column sep=scriptsize]
				{Gf\beta_x} \&\& {Gf\alpha_x} \&\& {\alpha_yFf} \&\& \\
				{Gf'\beta_x} \&\& {Gf'\alpha_x} \&\& {\alpha_yFf'} \& {} \\
				\& {} \& {Gf'\beta_x} \&\& {\beta_yFf'} \&\& {\beta_yFf'}
				\arrow[""{name=0, anchor=center, inner sep=0}, "{1\circ\Gamma_x}", from=1-1, to=1-3]
				\arrow["{G\phi\circ1}"', from=1-1, to=2-1]
				\arrow[""{name=1, anchor=center, inner sep=0}, "{\alpha_f}", from=1-3, to=1-5]
				\arrow["{G\phi\circ1}", from=1-3, to=2-3]
				\arrow["{1\circ F\phi}", from=1-5, to=2-5]
				\arrow[""{name=2, anchor=center, inner sep=0}, "{1\circ\Gamma_x}"{description}, from=2-1, to=2-3]
				\arrow[curve={height=12pt}, equals, from=2-1, to=3-3]
				\arrow[""{name=3, anchor=center, inner sep=0}, "{\alpha_{f'}}"{description}, from=2-3, to=2-5]
				\arrow["{1\circ\Lambda_x}"{description}, from=2-3, to=3-3]
				\arrow["{\Lambda_y\circ1}"{description}, from=2-5, to=3-5]
				\arrow[curve={height=-12pt}, equals, from=2-5, to=3-7]
				\arrow[""{name=4, anchor=center, inner sep=0}, "{\beta_{f'}}"', from=3-3, to=3-5]
				\arrow[""{name=5, anchor=center, inner sep=0}, "{\Gamma_y\circ1}"', from=3-5, to=3-7]
				\arrow[between={0.3}{0.7}, Rightarrow, dashed, from=0, to=2]
				\arrow["{\alpha_\phi}", between={0.3}{0.7}, Rightarrow, from=1, to=3]
				\arrow["{1\circ\eta_x^{-1}}"{pos=0.4}, shift right=3, between={0.2}{0.6}, Rightarrow, from=2, to=3-2]
				\arrow["{\Lambda^{-1}_f}", shift right=2, between={0.3}{0.7}, Rightarrow, from=3, to=4]
				\arrow["{\epsilon_y^{-1}\circ1}"'{pos=0.4}, shift left=3, between={0.3}{0.7}, Rightarrow, from=2-6, to=5]
			\end{tikzcd}
		}
		$=$
		\scalebox{0.7}{
			\begin{tikzcd}[ampersand replacement=\&,sep=scriptsize]
				{Gf\beta_x} \&\& {Gf\alpha_x} \&\& \\
				\&\&\&\& {\beta_yFf} \\
				{Gf'\beta_x} \&\& {Gf'\alpha_x} \\
				\&\& {\beta_yFf'} \&\& {\beta_yFf'}
				\arrow[""{name=0, anchor=center, inner sep=0}, "{1\circ\Gamma_x}", from=1-1, to=1-3]
				\arrow["{G\phi\circ1}"', from=1-1, to=3-1]
				\arrow[""{name=1, anchor=center, inner sep=0}, "{\alpha_f}", from=1-3, to=2-5]
				\arrow["{G\phi\circ1}"', from=1-3, to=3-3]
				\arrow["{1\circ F\phi}", from=2-5, to=4-5]
				\arrow[""{name=2, anchor=center, inner sep=0}, "{1\circ\Gamma_x}", from=3-1, to=3-3]
				\arrow["{\beta_{f'}}"', from=3-1, to=4-3]
				\arrow[""{name=3, anchor=center, inner sep=0}, "{\alpha_{f'}}", from=3-3, to=4-5]
				\arrow[""{name=4, anchor=center, inner sep=0}, "{\Gamma_y\circ1}"', from=4-3, to=4-5]
				\arrow[between={0.3}{0.7}, Rightarrow, dashed, from=0, to=2]
				\arrow["{\alpha_\phi}", between={0.3}{0.7}, Rightarrow, from=1, to=3]
				\arrow["{\Gamma_f}", shift right, between={0.2}{0.6}, Rightarrow, from=2, to=4]
			\end{tikzcd}
		}
	\end{center}
	This proves that $(\Gamma,\Gamma_f)$ form a modification. 
\end{proof}

\begin{cor}
	\label{cor:char_mod-equ_inv}
	A modification $\Lambda\colon\alpha\to\beta$ is:
	\begin{enumerate}[(a)]
		\item an equivalence if and only if its 0-cells components $\Lambda_x$ are equivalences;
		\item an isomorphism if and only if its 0-cells  components $\Lambda_x$ are invertible. 
	\end{enumerate}
\end{cor}

\begin{proof}
	\begin{enumerate}[(a)]
	\item[] 
	\item This follows directly from the fact that in any 2-category an equivalence can be lifted to an adjoint equivalence plus \Cref{prop:loc-adj-psd-char}.
	\item Clearly, if $\Lambda$ is invertible, then its 0-cells components must be so as well.  On the other hand, let us consider a pseudo-modification $\Lambda\colon\alpha\to\beta$ such that each $\Lambda_x$ is invertible, which is to say for each $x$ we have $\Gamma_x\colon\beta_x\to\alpha_x$  with unit and counit identities.  Then, by \Cref{prop:loc-adj-psd-char}, $\Lambda$ is an adjoint equivalence where the unit and counit perturbations forming the adjoint equivalence are componentwise identities, hence they are the identities. 
	Therefore, we can conclude that the left adjoint $\Gamma$ of $\Lambda$ is actually an inverse. \qedhere
	\end{enumerate}
\end{proof} 

\section{Incubator for substitution of binary into binary}

\subsection{First variable}
\label{app:inc-sub-b-in-b-1st}
Let $G\colon\ca,\cc\to\ce$ and $F\colon\ca,\cb\to\cx$ be binary maps. 
The incubator of $G\circ_1F$ is defined as the following pasting. 

\[\hspace{-1.5cm}\begin{tikzcd}[ampersand replacement=\&]
	{G^{F^{a'}_{b'}}_{C}G_{c}^{F^{a'}_B}G^{F^A_{b}}_{c}} \&\&\& {G^{F^{a'}_{b'}}_{C}G^{F^{a'}_BF^A_{b}}_{c}} \&\&\& {G^{F^{a'}_{b'}}_{C}G^{F^A_{b'}F^{a'}_B}_{c}} \&\&\& {G^{F^{a'}_{b'}}_{C}G_{c}^{F^A_{b'}}G^{F^{a'}_B}_{c}} \\
	{G_{c'}^{F^{a'}_B}G^{F^{a'}_{b}}_{C}G^{F^A_{b}}_{c}} \&\&\&\&\&\&\&\&\& {G^{F^A_{b'}}_{c'}G^{F^{a}_{b'}}_{C}G^{F^{a'}_B}_{c}} \\
	{G_{c'}^{F^{a'}_B}G^{F^A_{b}}_{c'}G^{F^{a}_{b}}_{C}} \&\&\& {G^{F^{a'}_BF^A_{b}}_{c'}G^{F^{a}_{b}}_{C}} \&\&\& {G^{F^A_{b'}F^{a'}_B}_{c'}G^{F^{a}_{b}}_{C}} \&\&\& {G^{F^A_{b'}}_{c'}G^{F^{a'}_B}_{c'}G^{F^{a}_{b}}_{C}}
	\arrow[""{name=0, anchor=center, inner sep=0}, "{1\circ G^{F^{a'}_B,F^A_{b}}_{c}}", from=1-1, to=1-4]
	\arrow[""{name=1, anchor=center, inner sep=0}, "{1\circ G_c(A_B)}", from=1-4, to=1-7]
	\arrow["{{(F^{a'}_B)}_C\circ1}"', from=1-1, to=2-1]
	\arrow["{1\circ {(F^A_{b})}_C}"', from=2-1, to=3-1]
	\arrow[""{name=2, anchor=center, inner sep=0}, "{G^{F^{a'}_B,F^A_{b}}_{c'}\circ1}"', from=3-1, to=3-4]
	\arrow[""{name=3, anchor=center, inner sep=0}, "{(G^{F^A_{b'},F^{a}_B}_{c'})^{-1}\circ1}"', from=3-7, to=3-10]
	\arrow["{{(F^{A}_{b'})}_C\circ1}", from=1-10, to=2-10]
	\arrow["{1\circ {(F^a_{B})}_C}", from=2-10, to=3-10]
	\arrow["{(F^{a'}_BF^A_{b})_C}"{description}, from=1-4, to=3-4]
	\arrow[""{name=4, anchor=center, inner sep=0}, "{G_{c'}(A_B)\circ1}"', from=3-4, to=3-7]
	\arrow[""{name=5, anchor=center, inner sep=0}, "{1\circ (G^{F^A_{b'},F^{a}_B}_{c})^{-1}}", from=1-7, to=1-10]
	\arrow["{(F^A_{b'}F^{a'}_B)_C}"{description}, from=1-7, to=3-7]
	\arrow["{{(A_B)}_C}"{description}, shorten <=20pt, shorten >=20pt, Rightarrow, from=1, to=4]
	\arrow["{G_C^{F^{a'}_B,F^A_{b}}}"{description}, shorten <=20pt, shorten >=20pt, Rightarrow, from=0, to=2]
	\arrow["{(G_C^{F^A_{b'},F^{a}_B})\sharp}"{description}, shorten <=20pt, shorten >=20pt, Rightarrow, from=5, to=3]
\end{tikzcd}\]

From the diagram above, we can see that if $G_c$ is a Gray-functor, then the incubator is just $(A_B)_C$.  

\subsection{Second variable}
\label{app:inc-sub-b-in-b-2nd}
Let $G\colon\ca,\cx\to\ce$ and $F\colon\cb,\cc\to\cx$ be binary maps. 
The incubator of $G\circ_2F$ is defined as the following pasting. 

\[\begin{tikzcd}[ampersand replacement=\&]
	{G_{F^{b'}_{C}}^{a'}G_{F^{B}_{c}}^{a'}G_{F^{b}_{c}}^{A}} \&\& {G_{F^{b'}_{C}}^{a'}G_{F^{b'}_{c}}^{A}G_{F^{B}_{c}}^{a}} \&\& {G_{F^{b'}_{c'}}^{A}G_{F^{b'}_{C}}^{a}G_{F^{B}_{c}}^{a}} \\
	{G_{F^{b'}_{C}F^{B}_{c}}^{a'}G_{F^{b}_{c}}^{A}} \&\&\&\& {G_{F^{b'}_{c'}}^{A}G_{F^{b'}_{C}F^{B}_{c}}^{a}} \\
	{G_{F^{B}_{c'}F^{b}_{C}}^{a'}G_{F^{b}_{c}}^{A}} \&\&\&\& {G_{F^{b'}_{c'}}^{A}G_{F^{B}_{c'}F^{b}_{C}}^{a}} \\
	{G_{F^{B}_{c'}}^{a'}G_{F^{b}_{C}}^{a'}G_{F^{b}_{c}}^{A}} \&\& {G_{F^{B}_{c'}}^{a'}G_{F^{b}_{c'}}^{A}G_{F^{b}_{C}}^{a}} \&\& {G_{F^{b'}_{c'}}^{A}G^{a}_{F^{B}_{c'}}G_{F^{b}_{C}}^{a}}
	\arrow["{G_{F^{b'}_{C},F^{B}_c}^{a'}\circ1}"', from=1-1, to=2-1]
	\arrow["{G^{a'}(B_C)\circ1}"', from=2-1, to=3-1]
	\arrow["{1\circ A_{(F^{B}_c)}}", from=1-1, to=1-3]
	\arrow["{A_{(F^{b'}_{C})}\circ1}", from=1-3, to=1-5]
	\arrow["{1\circ G_{F^{b'}_{C},F^{B}_c}^{a}}", from=1-5, to=2-5]
	\arrow["{1\circ A_{(F^{b}_C)}}"', from=4-1, to=4-3]
	\arrow["{A_{(F^{B}_{c'})}\circ1}"', from=4-3, to=4-5]
	\arrow["{(G_{F^B_{c'},F^{b}_C}^{c})^{-1}\circ1}"', from=3-1, to=4-1]
	\arrow[""{name=0, anchor=center, inner sep=0}, "{A_{(F^{B}_{c'}F^{b}_{C})}}"{description}, from=3-1, to=3-5]
	\arrow[""{name=1, anchor=center, inner sep=0}, "{A_{(F^{b'}_{C}F^{B}_c)}}"{description}, from=2-1, to=2-5]
	\arrow["{1\circ (G_{F^B_{c'},F^{b}_C}^{c})^{-1}}", from=3-5, to=4-5]
	\arrow["{1\circ G^{a}(B_C)}", from=2-5, to=3-5]
	\arrow["{G_{F^{b'}_{C},F^{B}_c}^{A}}"{xshift=0.1cm,pos=0.3}, shorten <=4pt, shorten >=12pt, Rightarrow, from=1-3, to=1]
	\arrow["{A_{(B_C)}}"{xshift=0.1cm}, shorten <=12pt, shorten >=12pt, Rightarrow, from=1, to=0]
	\arrow["{(G_{F^{B}_{c'},F^{b}_{C}}^{A})\flat}"{xshift=0.1cm,pos=0.6}, shorten <=12pt, shorten >=4pt, Rightarrow, from=0, to=4-3]
\end{tikzcd}\]

\black 


\section{Details of proof of \Cref{prop:equ-psd-transf}}
\label{app:big-diagram}

Using the definitions of $(\alpha^o)_2^{f_2,f_1}$ and $(\alpha^o)_2^{f_3,f_2f_1}$, we obtain the following diagram. 

\begin{center}
	\adjustbox{width=\textwidth}{
\begin{tikzcd}[ampersand replacement=\&,sep=scriptsize]
	{\alpha_w Ff_3Ff_2Ff_1} \&\&\&\&\&\&\&\& \\
	\\
	{Gf_3\alpha_z Ff_2Ff_1} \\
	\\
	{Gf_3Gf_2\alpha_y Ff_1} \& {} \& {} \&\& {\alpha_w Ff_3Ff_2Ff_1} \\
	\\
	{Gf_3G{f_2}G{f_1}\alpha_x} \&\& {Gf_3Gf_2\alpha_y Ff_1} \&\& {Gf_3\alpha_z Ff_2Ff_1} \&\& {\alpha_w Ff_3F(f_2f_1)} \\
	\\
	\& {Gf_3G({f_2}{f_1}) \alpha_x} \&\&\& {} \&\& {Gf_3\alpha_zF(f_2f_1)} \& {} \\
	\\
	\&\&\&\& {Gf_3G({f_2}{f_1}) \alpha_x} \& {} \& {Gf_3\alpha_zF(f_2f_1)} \&\& {\alpha_w Ff_3F(f_2f_1)} \\
	\\
	\&\&\&\& {G(f_3{f_2}{f_1}) \alpha_x} \&\&\&\& {\alpha_w F(f_3f_2f_1)} \\
	\\
	\&\&\&\&\&\& {} \&\& {G(f_3{f_2}{f_1}) \alpha_x}
	\arrow["{{\alpha^o_{f_3}}\circ1}"', from=1-1, to=3-1]
	\arrow[curve={height=-12pt}, equals, from=1-1, to=5-5]
	\arrow["{1\circ{\alpha^o_{f_2}}\circ1}"', from=3-1, to=5-1]
	\arrow[curve={height=-12pt}, equals, from=3-1, to=7-5]
	\arrow["{1\circ\alpha^o_{f_1}}"', from=5-1, to=7-1]
	\arrow[curve={height=-12pt}, equals, from=5-1, to=7-3]
	\arrow["{1\circ\eta_{f_2}\circ1}", between={0.4}{0.6}, Rightarrow, from=5-3, to=7-3]
	\arrow["{{\alpha^o_{f_3}}\circ1}"', from=5-5, to=7-5]
	\arrow[""{name=0, anchor=center, inner sep=0}, "{1\circ F^2_{f_2,f_1}}", from=5-5, to=7-7]
	\arrow[""{name=1, anchor=center, inner sep=0}, "{1\circ\alpha_{f_1}}"', from=7-1, to=7-3]
	\arrow["{1\circ G^2_{f_2,f_1}\circ1}"', from=7-1, to=9-2]
	\arrow["{1\circ{\alpha^l_{f_2}}\circ1}"', from=7-3, to=7-5]
	\arrow["{1\circ\alpha^2_{f_2.f_1}}"{pos=0.6}, between={0.3}{0.7}, Rightarrow, from=7-3, to=9-5]
	\arrow[""{name=2, anchor=center, inner sep=0}, "{1\circ F^2_{f_2,f_1}}"{description}, from=7-5, to=9-7]
	\arrow["{{\alpha^o_{f_3}}\circ1}", from=7-7, to=9-7]
	\arrow[curve={height=-12pt}, equals, from=7-7, to=11-9]
	\arrow["{1\circ\alpha_{f_2f_1}}"{description}, from=9-2, to=9-7]
	\arrow[curve={height=6pt}, equals, from=9-2, to=11-5]
	\arrow["{1\circ\epsilon_{f_2f_1}}"', between={0.3}{0.7}, Rightarrow, from=9-5, to=11-5]
	\arrow[""{name=3, anchor=center, inner sep=0}, "{1\circ{\alpha^o_{f_2f_1}}}"', from=9-7, to=11-5]
	\arrow[equals, from=9-7, to=11-7]
	\arrow["{1\circ{\alpha_{f_2f_1}}}"', from=11-5, to=11-7]
	\arrow["{G^2_{f_3,f_2f_1}\circ1}"', from=11-5, to=13-5]
	\arrow[""{name=4, anchor=center, inner sep=0}, "{{\alpha_{f_3}}\circ1}"', from=11-7, to=11-9]
	\arrow["{1\circ F^2_{f_3,f_2f_1}}", from=11-9, to=13-9]
	\arrow[""{name=5, anchor=center, inner sep=0}, "{{\alpha_{f_3f_2f_1}}}"{description}, from=13-5, to=13-9]
	\arrow[curve={height=12pt}, equals, from=13-5, to=15-9]
	\arrow["{{\alpha^o_{f_3f_2f_1}}}", from=13-9, to=15-9]
	\arrow["{1\circ\eta_{f_1}}"'{pos=0.6}, between={0.4}{0.7}, Rightarrow, from=5-2, to=1]
	\arrow[between={0.3}{0.7}, Rightarrow, dashed, from=0, to=2]
	\arrow["{1\circ\eta_{f_2f_1}}", between={0.3}{0.8}, Rightarrow, from=3, to=11-6]
	\arrow["{\eta_{f_3}\circ1}", between={0.3}{0.7}, Rightarrow, from=9-8, to=4]
	\arrow["{{\alpha^o}_2^{f_3,f_2f_1}}", between={0.3}{0.7}, Rightarrow, from=11-7, to=5]
	\arrow["{\epsilon_{f_3f_2f_1}}"{pos=0.4}, between={0.3}{0.6}, Rightarrow, from=5, to=15-7]
\end{tikzcd}
}
\end{center}

Then, using the triangle identity for  $\alpha^o_{f_2f_1}\dashv\alpha_{f_2f_1}$   

\begin{center}
\begin{tikzcd}[ampersand replacement=\&]
	{Gf_3G({f_2}{f_1}) \alpha_x} \& {} \&\& {Gf_3\alpha_zF(f_2f_1)} \\
	\\
	\& {Gf_3G({f_2}{f_1}) \alpha_x} \& {} \& {Gf_3\alpha_zF(f_2f_1)}
	\arrow["{1\circ\alpha_{f_2f_1}}"{description}, from=1-1, to=1-4]
	\arrow[curve={height=6pt}, equals, from=1-1, to=3-2]
	\arrow["{1\circ\epsilon_{f_2f_1}}"', between={0.3}{0.7}, Rightarrow, from=1-2, to=3-2]
	\arrow[""{name=0, anchor=center, inner sep=0}, "{1\circ{\alpha^o_{f_2f_1}}}"', from=1-4, to=3-2]
	\arrow[equals, from=1-4, to=3-4]
	\arrow["{1\circ{\alpha_{f_2f_1}}}"', from=3-2, to=3-4]
	\arrow["{1\circ\eta_{f_2f_1}}", between={0.3}{0.8}, Rightarrow, from=0, to=3-3]
\end{tikzcd}
$=\quad 1_{1\circ\alpha_{f_2f_1}}$
\end{center}

and the Gray-interchange axiom interacting with $\eta_{f_3}\circ1$ (see \Cref{lemma:interchange-property})

\begin{center}
\begin{tikzcd}[ampersand replacement=\&,column sep=small,row sep=scriptsize]
	{\alpha_w Ff_3Ff_2Ff_1} \&\& {\alpha_w Ff_3F(f_2f_1)} \& {} \& \\
	\\
	{Gf_3\alpha_z Ff_2Ff_1} \&\& {Gf_3\alpha_zF(f_2f_1)} \&\& {\alpha_w Ff_3F(f_2f_1)}
	\arrow[""{name=0, anchor=center, inner sep=0}, "{1\circ F^2_{f_2,f_1}}", from=1-1, to=1-3]
	\arrow["{{\alpha^o_{f_3}}\circ1}"', from=1-1, to=3-1]
	\arrow["{{\alpha^o_{f_3}}\circ1}"{description}, from=1-3, to=3-3]
	\arrow[curve={height=-12pt}, equals, from=1-3, to=3-5]
	\arrow[""{name=1, anchor=center, inner sep=0}, "{1\circ F^2_{f_2,f_1}}"', from=3-1, to=3-3]
	\arrow[""{name=2, anchor=center, inner sep=0}, "{{\alpha_{f_3}}\circ1}"', from=3-3, to=3-5]
	\arrow[between={0.3}{0.7}, Rightarrow, dashed, from=0, to=1]
	\arrow["{\eta_{f_3}\circ1}"', between={0.4}{0.6}, Rightarrow, from=1-4, to=2]
\end{tikzcd} \\
$=$
\begin{tikzcd}[ampersand replacement=\&,column sep=small,row sep=scriptsize]
	{\alpha_w Ff_3Ff_2Ff_1} \& {} \& \\
	\\
	{Gf_3\alpha_z Ff_2Ff_1} \&\& {\alpha_w Ff_3F(f_2f_1)} \\
	\\
	{Gf_3\alpha_zF(f_2f_1)} \&\& {\alpha_w Ff_3F(f_2f_1)}
	\arrow["{{\alpha^o_{f_3}}\circ1}"', from=1-1, to=3-1]
	\arrow[curve={height=-12pt}, equals, from=1-1, to=3-3]
	\arrow[""{name=0, anchor=center, inner sep=0}, "{{\alpha_{f_3}}\circ1}"', from=3-1, to=3-3]
	\arrow["{1\circ F^2_{f_2,f_1}}"', from=3-1, to=5-1]
	\arrow["{1\circ F^2_{f_2,f_1}}", from=3-3, to=5-3]
	\arrow[""{name=1, anchor=center, inner sep=0}, "{{\alpha_{f_3}}\circ1}"', from=5-1, to=5-3]
	\arrow["{\eta_{f_3}\circ1}"'{pos=0.6}, between={0.4}{0.8}, Rightarrow, from=1-2, to=0]
	\arrow[between={0.4}{0.7}, Rightarrow, dashed, from=0, to=1]
\end{tikzcd}
\end{center}

we obtain the following diagram, which has in the middle one side of the corresponding axiom for $\alpha$ as a lax transformation. 

\begin{center}
\adjustbox{width=\textwidth}{
\begin{tikzcd}[ampersand replacement=\&,column sep=scriptsize,row sep=small]
	{\alpha_w Ff_3Ff_2Ff_1} \&\&\&\&\&\& \\
	\\
	{Gf_3\alpha_z Ff_2Ff_1} \\
	\&\&\& {} \\
	{Gf_3Gf_2\alpha_y Ff_1} \&\& {} \\
	\& {} \&\& {} \\
	{Gf_3G{f_2}G{f_1}\alpha_x} \&\& {Gf_3Gf_2\alpha_y Ff_1} \&\& {Gf_3\alpha_z Ff_2Ff_1} \&\& {\alpha_w Ff_3Ff_2Ff_1} \\
	\\
	{Gf_3G({f_2}{f_1}) \alpha_x} \&\& {} \& {} \& {Gf_3\alpha_zF(f_2f_1)} \&\& {\alpha_w Ff_3F(f_2f_1)} \\
	\\
	{G(f_3{f_2}{f_1}) \alpha_x} \&\&\& {} \& {} \&\& {\alpha_w F(f_3f_2f_1)} \\
	\\
	\&\&\&\& {} \&\& {G(f_3{f_2}{f_1}) \alpha_x}
	\arrow["{{\alpha^o_{f_3}}\circ1}"', from=1-1, to=3-1]
	\arrow[curve={height=-18pt}, equals, from=1-1, to=7-7]
	\arrow["{1\circ{\alpha^o_{f_2}}\circ1}"', from=3-1, to=5-1]
	\arrow[curve={height=-12pt}, equals, from=3-1, to=7-5]
	\arrow["{\eta_{f_3}\circ1}"{pos=0.4}, between={0.2}{0.6}, Rightarrow, from=4-4, to=6-4]
	\arrow["{1\circ\alpha^o_{f_1}}"', from=5-1, to=7-1]
	\arrow[curve={height=-12pt}, equals, from=5-1, to=7-3]
	\arrow["{1\circ\eta_{f_2}\circ1}"{pos=0.4}, between={0.2}{0.6}, Rightarrow, from=5-3, to=7-3]
	\arrow[""{name=0, anchor=center, inner sep=0}, "{1\circ\alpha_{f_1}}"', from=7-1, to=7-3]
	\arrow["{1\circ G^2_{f_2,f_1}\circ1}"', from=7-1, to=9-1]
	\arrow["{1\circ{\alpha^l_{f_2}}\circ1}"', from=7-3, to=7-5]
	\arrow["{1\circ\alpha^2_{f_2.f_1}}", between={0.2}{0.8}, Rightarrow, from=7-3, to=9-3]
	\arrow[""{name=1, anchor=center, inner sep=0}, "{{\alpha_{f_3}}\circ1}", from=7-5, to=7-7]
	\arrow["{1\circ F^2_{f_2,f_1}}"', from=7-5, to=9-5]
	\arrow["{1\circ F^2_{f_2,f_1}}", from=7-7, to=9-7]
	\arrow["{1\circ\alpha_{f_2f_1}}"', from=9-1, to=9-5]
	\arrow["{G^2_{f_3,f_2f_1}\circ1}"', from=9-1, to=11-1]
	\arrow["{{\alpha^o}_2^{f_3,f_2f_1}}", between={0.3}{0.7}, Rightarrow, from=9-4, to=11-4]
	\arrow[""{name=2, anchor=center, inner sep=0}, "{{\alpha_{f_3}}\circ1}"', from=9-5, to=9-7]
	\arrow["{1\circ F^2_{f_3,f_2f_1}}", from=9-7, to=11-7]
	\arrow["{{\alpha_{f_3f_2f_1}}}"{description}, from=11-1, to=11-7]
	\arrow[curve={height=12pt}, equals, from=11-1, to=13-7]
	\arrow["{\epsilon_{f_3f_2f_1}}"{pos=0.4}, between={0.2}{0.6}, Rightarrow, from=11-5, to=13-5]
	\arrow["{{\alpha^o_{f_3f_2f_1}}}", from=11-7, to=13-7]
	\arrow["{1\circ\eta_{f_1}}"'{pos=0.3}, between={0.1}{0.6}, Rightarrow, from=6-2, to=0]
	\arrow[between={0.3}{0.7}, Rightarrow, dashed, from=1, to=2]
\end{tikzcd}
}
\end{center}

Analogously, the right hand side of the sought-for equation can be reduced to this diagram as well.
\black

\bibliographystyle{plain}

\begin{thebibliography}{10}

\bibitem{Altenkirch}
{\sc Altenkirch, T., Chapman, J., and Uustalu, T.} 
\newblock Monads need not be endofunctors. 
\newblock{\em Logical Methods in Computer Science} 11(1:3) (2015), 1--40.

\bibitem{Bourke2015A-cocategorical}
{\sc Bourke, J., and Gurski, N.}
\newblock A cocategorical obstruction to tensor products of Gray-categories.
\newblock {\em Theory and Applications of Categories}, 11 (2015), 387--409.

\bibitem{BourkeLack:skew}
{\sc Bourke, J. and Lack, S.}
\newblock Skew monoidal categories and skew multicategories.
\newblock {\em Journal of Algebra}, volume 56, 237--266, 2018.

\bibitem{BourkeLack:symmetric}
{\sc Bourke, J. and Lack, S.}
\newblock Braided skew monoidal categories.
\newblock {\em Theory and Applications of Categories}, Vol. 35 (2020) 19--63.

\bibitem{BourkeLobbia}
{\sc Bourke, J. and Lobbia, G.}
\newblock A skew approach to enrichment for Gray-categories.
\newblock {\em Advances in Mathematics}, Vol. 434:109327, 2023.

\bibitem{BourkeLobbiaCorrected}
{\sc Bourke, J. and Lobbia, G.}
\newblock A skew approach to enrichment for Gray-categories.
\newblock {Arxiv preprint 2026, \url{https://arxiv.org/abs/2212.12358v4}}

\bibitem{BourkeLobbia2}
{\sc Bourke, J. and Lobbia, G.}
\newblock Semi-strict 4-categories as enriched categories.
\newblock {Arxiv preprint 2026, \url{https://arxiv.org/abs/2610.31448}}

\bibitem{GambLob:psm}
{\sc Gambino, N. and Lobbia, G.}
\newblock On the formal theory of pseudomonads and pseudodistributive laws.
\newblock {\em Theory and Applications of Categories} 37(2) (2021), 14--56.

\bibitem{GohlaB:mapp}
{\sc Gohla, B.}
\newblock Mapping Spaces of Gray-Categories.
\newblock {\em Theory and Applications of Categories}, 29(5) (2014), 100-187.

\bibitem{Gordon1995Coherence}
{\sc Gordon, R., Power, J., and Street, R.}
\newblock Coherence for tricategories.
\newblock {\em Memoirs of the American Mathematical Society 117\/} (1995).

\bibitem{Gurskibook}
{\sc Gurski, N.}
\newblock {\em Coherence in Three-Dimensional Category Theory}
\newblock Cambridge University Press, 2013.

\bibitem{Hermida2000Representable}
C.~Hermida.
\newblock Representable multicategories.
\newblock {\em Adv. Math.}, 151(2):164--225, 2000.

  

\bibitem{KellyStreet:rew-el-2cats}
{\sc Kelly, G. M. and Street, R.}
\newblock Review of the elements of 2-categories.
\newblock {\em Category Seminar}. Lecture Notes in Mathematics, vol 420. Springer, Berlin, Heidelberg, 1974.

\bibitem{Lob:short-multi}
{\sc Lobbia, G.}
\newblock A finite approach to representable multicategories and related structures.
\newblock {\em Theory and Applications of Categories}, vol. 45, 2026, No. 10, pp 332-390, 2026.

\bibitem{Szlachanyi2012Skew}
{\sc Szlach\'{a}nyi, K.}
\newblock Skew-monoidal categories and bialgebroids.
\newblock {\em Advances in Mathematics 231}, 3-4 (2012), 1694--1730.



\end{thebibliography}

\end{document}